\documentclass[12pt]{amsart}
\usepackage{mathrsfs}
\usepackage{amsfonts}
\usepackage[centertags]{amsmath}
\usepackage{amssymb}
\usepackage{amsthm}
\usepackage{graphicx}
\usepackage{float}
\usepackage[all]{xy}
\usepackage{tikz-cd}
\usepackage{caption}
\usepackage[colorlinks]{hyperref}
\hypersetup{
bookmarksnumbered,
pdfstartview={FitH},
breaklinks=true,
linkcolor=blue,
urlcolor=blue,
citecolor=blue,
bookmarksdepth=2
}

\usepackage{amsxtra}

\usepackage{xcolor}
\usepackage[all]{xy}
\usepackage{tikz}
\usepackage{tkz-graph}
\usepackage{enumerate}
\tikzstyle{vertex}=[circle, draw, inner sep=0pt, minimum size=6pt]

\usepackage{enumerate}
\usepackage[a4paper, portrait, margin=0.9in]{geometry}
\usepackage{tikz-cd}
\usepackage{rotating}
\usetikzlibrary{matrix,arrows,decorations.pathmorphing}
\usepackage{listing}
\usepackage{color}

\theoremstyle{plain}
   \newtheorem{thm}{Theorem}[section]
   \newtheorem{theorem}{Theorem}[section]
   \newtheorem{proposition}[theorem]{Proposition}
   
   \newtheorem{lemma}[theorem]{Lemma}
   \newtheorem{corollary}[theorem]{Corollary}
   \newtheorem{conjecture}[theorem]{Conjecture}
   
\theoremstyle{definition}
   \newtheorem{definition}[theorem]{Definition}
   
   \newtheorem{example}[theorem]{Example}

   \newtheorem{remark}[theorem]{Remark}

\swapnumbers \theoremstyle{theorems}

\def\endproof{\hfill$\square$\medskip}

\begin{document}

\setlength{\baselineskip}{1.2\baselineskip}
%%%%%%%%%%%%%%%%%%%%%%%%%%%%%%%%%%%%%%%%%%%%%%%%%%%%%%%%%%%%%%%%%%%%%%%%%%%
%%%%%%%%%%%%%%%%%%%%%% Title %%%%%%%%%%%%%%%%%%%%%%%%%%%%%%%%%%%%%%%%
%%%%%%%%%%%%%%%%%%%% Introduction %%%%%%%%%%%%%%%%%%%%%%%%%%%%%%%%%%%%%%%%
\vspace{-.35in}
\title[Positivity for QCA from orbifolds]
{Positivity for quantum cluster algebras from orbifolds}
\author{Min Huang}
\address{
School of Mathematics (Zhuhai) \\
Sun Yat-sen University \\
%Pokfulam Road \\
%Hong Kong
}
\email{huangm97@mail.sysu.edu.cn}

\date{}

\begin{abstract}
Let $(S,M,U)$ be a marked orbifold with or without punctures and let $\mathcal A_v$ be a quantum cluster algebra from $(S,M,U)$ with arbitrary coefficients and quantization. We provide combinatorial formulas for quantum Laurent expansion of quantum cluster variables of $\mathcal A_v$ concerning an arbitrary quantum seed. Consequently, the positivity for the quantum cluster algebra $\mathcal A_v$ is proved.
\end{abstract}
\maketitle

\tableofcontents

\section{Introduction}

Cluster algebras are commutative algebras introduced by Fomin-Zelevinsky \cite{fz1} around the year 2000. Quantum cluster algebras were later introduced by Berenstein-Zelevinsky, \cite{BZ}. The theory of cluster algebras is related to numerous other fields, including Lie theory, representation theory of algebras, dynamical systems, Teichm\"{u}ller theory, combinatorics, number theory, topology, and mathematical physics.

A cluster algebra is a subalgebra of a rational function field with a distinguished set of generators, called cluster variables. Different cluster variables are related by an iterated procedure, called mutation. Cluster variables are rational functions by construction. In \cite{fz1}, Fomin and Zelevinsky proved that they are Laurent polynomials of initial cluster variables, known as the Laurent phenomenon. Laurent polynomials were proven to have non-negative coefficients, known as positivity, see \cite{LS,GHKK}.

Berenstein-Zelevinsky \cite{BZ} proved that the Laurent phenomenon has a quantum version in the quantum setting. The cluster variables are quantum Laurent polynomials of initial ones. The coefficients, were conjectured to be in $\mathbb N[q^{\pm 1/2}]$, known as positivity conjecture for quantum cluster algebras, where $q$ is the quantum parameter. Kimura and Qin \cite{KQ} proved the positivity conjecture for the acyclic skew-symmetric quantum cluster algebras, and Davison \cite{D} proved this conjecture for the skew-symmetric case. The author \cite{H1} proved the case for quantum cluster algebras from unpunctured orbifolds. There is no noteworthy work
performed for skew-symmetrizable cases in general.

The original motivation of Fomin and Zelevinsky was to provide a combinatorial characterization of the (dual) canonical bases in quantum groups (see \cite{L,K}) and the total positivity in algebraic groups. They conjectured that the cluster structure should serve as an algebraic framework for the study of the ``dual canonical bases" in various coordinate rings and their $q$-deformations. In Particular, they conjectured all cluster monomials belong to the dual canonical bases. This was recently proved in \cite{KKKO,Q,Q1}. Generally, it can be very difficult to write the dual canonical bases explicitly. From this point of view, finding an explicit (quantum) Laurent expansion formula for (quantum)
cluster monomials represents a noteworthy step in research in cluster theory.

Finding an explicit formula for cluster variables attracted a lot of attention. The Caldero-Chapoton map or cluster character is a way to provide the expansion formula in terms of Euler-Poincar\'{e} characteristics, see \cite{CC,CK,CK1,FK,P,PP,PP1,DE,HL,CHL}. Rupel \cite{R} generalized the Caldero-Chapoton map to the quantum case.
However, since Euler-Poincar\'{e} can be negative, the
positivity of the coefficients is not immediately implied. Extensive research has been conducted in the other direction
on (quantum) cluster algebras from surfaces, introduced by Fomin-Shapiro-Thurston \cite{FST}. The
expansion formula is well-studied by Musiker, Schiffler, Thomas, Williams, et al. in
a series of studies, see \cite{ST,S,MS,MSW,Y}. Rupel \cite{R1} gave the explicit quantum Laurent expansion formula for the quantum cluster algebras of type $A$. Canakci-Lampe \cite{CL} gave an explicit quantum Laurent expansion formula for the quantum cluster algebras of type $A$ and of the Kronecker type. The author \cite{H,H1} provided an expansion formula for quantum cluster algebras from unpunctured surfaces (orbifolds). Berenstein-Retakh \cite{BR} deals with a fully noncommutative version of cluster algebras arising from surfaces.

(Quantum) cluster algebras from orbifolds are an essential class of (quantum) cluster algebras. Almost all skew-symmetrizable (quantum) cluster algebras of finite mutation type are in this class, see \cite{FST3}. Using the unfolding method, the positivity for cluster algebras from orbifolds can be deduced from the positivity for the cluster algebras from surfaces, see \cite{FST3}. However, the unfolding method does not provide any information on the $q$-coefficients. Thus, we do not know the quantum Laurent expansion formula for quantum cluster algebras from orbifolds, even that is known for quantum cluster algebras from surfaces. The study solves the positivity conjecture for such a class of quantum cluster algebras by giving an expansion formula of any quantum cluster variable for any quantum cluster. We would generalize the methods in \cite{H2} from the commutative case to the quantum case.

The study is organized as follows. Sections \ref{sec:pre} are devoted to preliminaries. In Subsection \ref{sec:poset}, we recall the key tool partition bijection which is defined in \cite{H} and give a criterion for two partition bijections are inverse to each other, Proposition \ref{Pro-parin}. Subsections \ref{sec:C}, \ref{sec:QC} and \ref{sec:QCO} are preliminaries on (quantum) cluster algebras and (quantum) cluster algebras from orbifolds. In Subsection \ref{sec:OMCP}, we recall orbifold morphism and canonical polygon defined in \cite{BR}. Subsections \ref{sec:P}, \ref{Sec-CPB} are preliminaries on some heart combinatoric notations, snake graphs, perfect matchings, and complete $(T^o,\gamma)$-path.

In Section \ref{Sec-iso}, for an orbifold $\Sigma$, we generalize the result in \cite{MSW} and show that changing the tags at a puncture $p$ gives an isomorphism between two quantum cluster algebras from $\Sigma$, see Proposition \ref{Pro-iso}. To give a Laurent expansion formula for any quantum cluster variables concerning any cluster, by Proposition \ref{Pro-iso}, it suffices to restrict to three cases, see Remark \ref{Rem-3cases}.

In Sections \ref{Sec-threesets} and \ref{sec:VM}, given an ideal triangulation $T^o$, we introduce three classes of lattices for ordinary arcs, one end tagged arcs and two ends tagged arcs, and construct valuation maps on them. These lattices are the index sets for the Laurent expansion formulas and the valuation maps provide the powers of the quantum parameter.

The main results in this paper are given in Section \ref{sec:EFP}. To be precise, we provide Laurent expansion formulas in Theorems \ref{Thm-1}, \ref{Thm-2}, \ref{Thm-M3} and the positivity for this class of quantum cluster algebras, Theorem \ref{thm:P}.

We prove Theorems \ref{Thm-1}, \ref{Thm-2}, \ref{Thm-M3} in Section \ref{sec:PROOF}. Our strategy is the following: for a given tagged arc $\beta$, $\beta^{(q)}$ or $\beta^{(p,q)}$, we first show that the expansion formulas given in Theorems \ref{Thm-1}, \ref{Thm-2}, \ref{Thm-M3} do not depend on the choice of ideal triangulations. Since any two ideal triangulations are connected by a sequence of flips, it suffices to show the expansion formulas given in Theorems \ref{Thm-1}, \ref{Thm-2}, \ref{Thm-M3} are invariant concerning two ideal triangulations related by a flip, see Theorem \ref{thm-sum4}. We then show that the expansion formulas hold for a specially chosen ideal triangulation, see Propositions \ref{prop-equ4}, \ref{lem:step2}, \ref{prop-equ5}, \ref{lem:step3} and \ref{prop:ll}.

Sections \ref{Sec-parcom}, \ref{sec:PB}, \ref{sec-pb}, \ref{sec:CPP} and \ref{sec:TEQ} are devoted to the proof of Theorem \ref{thm-sum4}. The naive idea is the following: for any ideal triangulation $T^o$ and a non-self-folded arc $\alpha$, denote $T'^o=\mu_\alpha(T^o)$ and by $LP(T^o)$ the Laurent expansion formula concerning $T^o$ for a given tagged arc $\beta$, $\beta^{(q)}$ or $\beta^{(p,q)}$. For each term $X$ in $LP(T^o)$, assume that the index of the quantum cluster variable $X_\alpha$ in $X$ is $n(X)$. By the exchange formula of quantum cluster variables, we have $X$ is the sum of $2^{n(x)}$ terms in $LP(T'^o)$ if $n(X)\geq 0$ and the sum of $X$ with $2^{-n(x)}-1$ terms in $LP(T^o)$ is one term in $LP(T'^o)$ if $n(X)<0$. This allows us to give a partition bijection between the terms of $LP(T^o)$ and $LP(T'^o)$, accordingly, we have a partition bijection between the index sets. We realize this in Section \ref{sec-pb}. Sections \ref{Sec-parcom} and \ref{sec:PB} are preliminaries to give the partition bijection in Section \ref{sec-pb}. We prove in \ref{sec:CPP} that has a nice compatible property with the lattice structures in Section \ref{Sec-threesets}. In Section \ref{sec:TEQ}, we deal with the powers of the quantum parameter and coefficients under the partition bijection.

\medskip

{\bf Acknowledgements:}\;
The author would like to thank Prof. Arkady Berenstein, Prof. Vladimir Retakh and Dr. Eugen Rogozinnikov for inspirit discussion. This project is supported by
the National Natural Science Foundation of China (No.12101617).

\newpage

\section{Preliminary and some notation}\label{sec:pre}

\subsection{Poset and partition bijection}\label{sec:poset}

\subsubsection{Poset} Herein, we recall some notation in poset.

\begin{definition}
A binary relation $\leq$ is a \emph{partial order} on a set $\mathcal P$ if it satisfies the following properties.
\begin{enumerate}[(1)]
    \item (reflexivity) $P\leq P$ for all $P\in \mathcal P$;
    \item (antisymmetry) $P\leq Q$ and $Q\leq P$ imply $P=Q$;
    \item (transitivity) $P\leq Q$ and $Q\leq R$ imply $P\leq R$.
\end{enumerate}

Denote $P\geq Q$ if $Q\leq P$. A \emph{poset} is a set with a partial order.
\end{definition}

We say that a poset $\mathcal P$ is \emph{connected} if for any $P,Q\in \mathcal P$ there exists a sequence of elements $P_0=P,P_1,\cdots, P_s=Q$ such that $P_i\leq P_{i+1}$ if $i$ is even and $P_i\geq P_{i+1}$ if $i$ is odd.

\begin{definition} Let $\mathcal P$ be a poset. For any $P,Q\in \mathcal P$, we say that $Q$ \emph{covers} $P$ if $P<Q$ and there is no other element $R$ such that $P<R<Q$.
\end{definition}

\subsubsection{Partition bijection}

\begin{definition}\label{Def-parbi} Let $\mathcal P$ be a finite set.
\begin{enumerate}[$(1)$]
    \item A \emph{partition} of $\mathcal P$ is a finite collection of subsets $\mathcal P_i,i\in I$ such that $\cup_{i\in I} \mathcal P_i=\mathcal P$ and $\mathcal P_i\cap \mathcal P_j=\emptyset$ for any different $i,j\in I$.
    \item Let $\mathcal P, \mathcal P'$ be finite subsets. A \emph{partition bijection} from $\mathcal P$ to $\mathcal P'$ is a bijection from some partition of $\mathcal P$ to some partition of $\mathcal P'$, denoted by $\varphi:\mathcal P\overset{par}{\to} \mathcal P'$.
    \item Let $\varphi:\mathcal P\overset{par}{\to} \mathcal P'$ and $\varphi':\mathcal P'\overset{par}{\to} \mathcal P$ be two partition bijections. We say that $\varphi$ and $\varphi'$ are \emph{inverse} to each other if they are inverse to each other as bijective maps.
\end{enumerate}

\end{definition}

\begin{example}
Let $\mathcal P=\{x,y,x^{-2}y^2,x^{-2}yz,x^{-2}zy,x^{-2}z^2\}$ and $\mathcal P'=\{x'^{-1}y,x'^{-1}z,y,x'^2\}$. Then $\{\{x\},\{y\},\{x^{-2}y^2,x^{-2}yz,x^{-2}zy,x^{-2}z^2\}\}$ is a partition of $\mathcal P$ and $\{\{x'^{-1}y,x'^{-1}z\},\{y\},\{x'^2\}\}$ is a partition of $\mathcal P'$.

One see that $\{x\}\to \{x'^{-1}y,x'^{-1}z\}$, $\{y\}\to \{y\}$ and $\{x^{-2}y^2,x^{-2}yz,x^{-2}zy,x^{-2}z^2\} \to \{x'^2\}$ give a partition bijection from $\mathcal P$ to $\mathcal P'$.
\end{example}

\begin{remark}\label{Rmk-par}
Let $\mathcal P,\mathcal P'$ be finite sets. From Definition \ref{Def-parbi}, to give a partition bijection from $\mathcal P$ to $\mathcal P'$ is equivalent to associating each $P\in \mathcal P'$ with a non-empty subset $\varphi(P)\in \mathcal P'$ such that the following conditions hold.
\begin{enumerate}[(i)]
    \item For any $P,Q\in \mathcal P$, we have either $\varphi(P)\cap \varphi(Q)=\emptyset$ or $\varphi(P)=\varphi(Q)$;
    \item $\bigcup_{P\in \mathcal P}\varphi(P)=\mathcal P'$.
\end{enumerate}
\end{remark}

We will frequently use this equivalence definition throughout this paper. The following result is useful to judge if two given partition bijections are inverse to each other.

\begin{proposition}\label{Pro-parin}
Let $\mathcal P,\mathcal P'$ be finite sets. Let $\varphi:\mathcal P\overset{par}{\to} \mathcal P'$ and $\varphi':\mathcal P'\overset{par}{\to} \mathcal P$ be two partition bijections. Then $\varphi$ and $\varphi'$ are inverse to each other is equivalent to the following condition: for any $P\in \mathcal P,P'\in\mathcal P'$ we have $P'\in \varphi(P)$ if and only if $P\in \varphi'(P')$.
\end{proposition}

\begin{proof}
The necessity is clear. Now we prove the sufficiency. As $\varphi:\mathcal P\overset{par}{\to} \mathcal P'$ is a partition bijection, $\varphi$ gives a bijection between the partition $\bigsqcup_{P\in \mathcal P}\{Q\mid \varphi(Q)=\varphi(P)\}$ of $\mathcal P$ and the partition bijection $\bigsqcup_{P\in \mathcal P}\varphi(P)$ of $\mathcal P'$. Similarly, $\varphi'$ gives a bijection between the partition $\bigsqcup_{P'\in \mathcal P'}\{Q'\mid \varphi'(Q')=\varphi'(P')\}$ of $\mathcal P'$ and the partition bijection $\bigsqcup_{P'\in \mathcal P'}\varphi'(P')$ of $\mathcal P$. According to the condition, for any $P\in \mathcal P$ and $P'\in \varphi(P)$ we have $\{Q\mid \varphi(Q)=\varphi(P)\}=\varphi'(P')$. Thus the partitions $\bigsqcup_{P\in \mathcal P}\{Q\mid \varphi(Q)=\varphi(P)\}$ and $\bigsqcup_{P'\in \mathcal P'}\varphi'(P')$ of $\mathcal P$ coincide. Similarly, the partitions $\bigsqcup_{P'\in \mathcal P'}\{Q'\mid \varphi'(Q')=\varphi'(P')\}$ and $\bigsqcup_{P\in \mathcal P}\varphi(P)$ of $\mathcal P'$ coincide. The result follows.
\end{proof}

\subsection{Cluster algebra}\label{sec:C} In this subsection, we recall the definitions of cluster algebra and geometric type cluster algebra in \cite{fz1}. Throughout, for any $a\in \mathbb Z$, denote by $[a]_+=max\{0,a\}$. Given a vector ${\bf a}=(a_1,\cdots,a_k)\in \mathbb Z^k$ for some $k\in \mathbb N$, denote by ${\bf a}_+=([a_1]_+,\cdots,[a_k]_+)$ and denote by ${\bf a}_-=(-{\bf a})_+$.

A triple $(\mathbb P,\oplus,\cdot)$ is called a \emph{semifield} if $(\mathbb P,\cdot)$ is an abelian multiplicative group and $(\mathbb P,\oplus)$ is a commutative semigroup such that $``\oplus"$ is distributive with respect to $``\cdot"$. A \emph{tropical semifield} ${\rm trop}(u_1,\cdots,u_l)$ is a semifield freely generated by $u_1,\cdots,u_l$ as abelian groups with $\oplus$ defined by $\prod_ju_j^{a_j}\oplus \prod_ju_j^{b_j}=\prod_ju_j^{min(a_j,b_j)}$. Let $(\mathbb P,\oplus,\cdot)$ be a semifield. The group ring $\mathbb {ZP}$ will be used as \emph{ground ring}. Give an integer $n$, let $\mathcal F$ be the rational functions field in $n$ independent variables, with coefficients in $\mathbb {QP}$.

\begin{definition}

A seed $t$ in $\mathcal F$ consists a triple $(x(t),y(t),B(t))$, where

\begin{itemize}
\item $x(t)=\{x_1(t),\cdots,x_n(t)\}$ such that $\mathcal F$ is freely generated by $x(t)$ over $\mathbb {QP}$.
\item $y(t)=\{y_1(t),\cdots,y_n(t)\}\subseteq \mathbb P$.
\item $B(t)=(b_{ij})$ is an $n\times n$ skew-symmetrizable integer matrix.
\end{itemize}
\end{definition}

\begin{definition}

Given a seed $t$ in $\mathcal F$, for any $k\in [1,n]$, we define the \emph{mutation} of $t$ at the $k$-th direction to be the new seed $t'=\mu_k(t)=(x(t'),y(t'),B(t'))$, where

\begin{itemize}
\item \[\begin{array}{ccl} x_i(t') &=&
\left\{\begin{array}{ll}
x_i(t), &\mbox{if $i\neq k$}, \\
\frac{y_k(t)\prod x_i(t)^{[b_{ik}]_{+}}+\prod x_i(t)^{[-b_{ik}]_{+}}}{(y_k(t)\oplus 1)x_k(t)}, &\mbox{otherwise}.
\end{array}\right.
\end{array}\]
\item \begin{equation}\label{eq-y} \begin{array}{ccl}y_i(t') &=&
\left\{\begin{array}{ll}
y^{-1}_k(t), &\mbox{if $i=k$}, \\
y_i(t)y_k(t)^{[b_{ki}]_{+}}(1\oplus y_k(t))^{-b_{ki}}, &\mbox{otherwise}.
\end{array}\right.
\end{array}\end{equation}
\item $B(t')=(b'_{ij})$ is determined by $B(t)=(b_{ij})$:
\[\begin{array}{ccl} b'_{ij} &=&
\left\{\begin{array}{ll}
-b_{ij}, &\mbox{if $i=k$ or $j=k$}, \\
b_{ij}+[b_{ik}]_{+}[b_{kj}]_{+}-[-b_{ik}]_{+}[-b_{kj}]_{+}, &\mbox{otherwise}.
\end{array}\right.
\end{array}\]
\end{itemize}

\end{definition}

\begin{definition}

A \emph{cluster algebra} $\mathcal A$ (of rank $n$) over $\mathbb P$ is defined by the following steps.
\begin{enumerate}[$(1)$]
\item Choose an initial seed $t_0=(x(t_0),y(t_0),B(t_0))$.
\item Get all the seeds $t$ which are obtained from $t_0$ by iterated mutations at directions $k\in[1,n]$.
\item Define $\mathcal A=\mathbb{ZP}[x_i(t)]_{t,i\in [1,n]}$.
\item $x(t)$ is called a \emph{cluster} for any $t$.
\item $x_i(t)$ is called a \emph{cluster variable} for any $i\in [1,n]$ and $t$.
\item A monomial in $x(t)$ is called a \emph{cluster monomial} for any $t$.
\item $y(t)$ is called a \emph{coefficient tuple} for any $t$.
\item $B(t)$ is called an \emph{exchange matrix} for any $t$.
\end{enumerate}
In particular, when $\mathbb P={\rm Trop}(u_1,\cdots,u_l)$, we say that $\mathcal A$ is of \emph{geometric type}.

\end{definition}

In case $\mathbb P={\rm Trop}(u_1,\cdots,u_l)$, denote by $m=n+l$. For a seed $t$ in $\mathcal F$, $y_j(t)=\prod u_i^{c_{ij}}$ for some integers $c_{ij}$. Thus we can write $t$ as $\widetilde x(t),\widetilde B(t)$, where
\begin{enumerate}[$(1)$]
\item $\widetilde x(t)=\{x_1(t),\cdots,x_n(t),x_{n+1}(t)=u_1,\cdots,x_m(t)=u_l\}$.
\item $\widetilde B(t)=(b_{ij})$ is an $m\times n$ with $b_{ij}=c_{i-n,j}$ for $i\in [n+1,m]$.
\end{enumerate}

Under this convention, the mutation of $t$ at direction $k$ is $t'=\mu_k(t)=(\widetilde x(t'),\widetilde B(t'))$, where

\begin{enumerate}[$(1)$]
\item \[\begin{array}{ccl} x_i(t') &=&
\left\{\begin{array}{ll}
x_i(t), &\mbox{if $i\neq k$}, \\
\frac{\prod x_i(t)^{[b_{ik}]_{+}}+\prod x_i(t)^{[-b_{ik}]_{+}}}{x_k(t)}, &\mbox{otherwise}.
\end{array}\right.
\end{array}\]
\item $\widetilde B(t')=(b'_{ij})$ is determined by $\widetilde B(t)=(b_{ij})$:
\[\begin{array}{ccl} b'_{ij} &=&
\left\{\begin{array}{ll}
-b_{ij}, &\mbox{if $i=k$ or $j=k$}, \\
b_{ij}+[b_{ik}]_{+}[b_{kj}]_{+}-[-b_{ik}]_{+}[-b_{kj}]_{+}, &\mbox{otherwise}.
\end{array}\right.
\end{array}\]
\end{enumerate}

\subsection{Quantum cluster algebra}\label{sec:QC} In this subsection, we recall the definition of quantum cluster algebra in \cite{BZ}.

Fix two integers $n\leq m$. Let $\widetilde B$ be an $m\times n$ integer matrix. Let $\Lambda$ be an $m\times m$ skew-symmetric integer matrix. We call $(\widetilde B,\Lambda)$ \emph{compatible} if $\widetilde B^T \Lambda=(D\;\;0)$ for some diagonal matrix $D$ with positive entries, where $\widetilde B^T$ is the transpose of $\widetilde B$. Note that in this case, the upper $n\times n$ submatrix of $\widetilde B$ is skew-symmetrizable and $\widetilde B$ is full rank.

\begin{definition}

Let $v$ be the quantum parameter.

\begin{enumerate}[$(1)$]
    \item A \emph{quantum seed} $t$ consists a compatible pair $(\widetilde B(t),\Lambda(t))$ and a collection of indeterminate $X_i(t),i\in [1,m]$, called \emph{quantum cluster variables}, $\widetilde B(t)$ is called \emph{extended exchange matrix}, $\Lambda(t)$ is called \emph{quantum commutative matrix}.
    \item Let $\{{\bf e}_i\}$ be the standard basis of $\mathbb Z^m$ and $X(t)^{{\bf e}_i}=X_i(t)$. Define the corresponding \emph{quantum torus} $\mathcal T(t)$ to be the algebra which is freely generated by $X(t)^{{\bf a}},{\bf a}\in \mathbb Z^m$ as $\mathbb Z[v^{\pm 1}]$-module, with multiplication on these elements defined by
    \begin{equation*}
        X(t)^{{\bf a}}X(t)^{{\bf b}}=v^{\Lambda(t)({\bf a},{\bf b})}X(t)^{{\bf a}+{\bf b}},
    \end{equation*}
where $\Lambda(t)(,)$ means the bilinear form on $\mathbb Z^m$ such that
\begin{equation*}
    \Lambda(t)({\bf e}_i,{\bf e}_j)=\Lambda(t)_{ij}.
\end{equation*}

\end{enumerate}

\end{definition}

\begin{definition}

For any $k\in [1,n]$, we define the \emph{mutation} of $t$ at the $k$-th direction to be the new seed $t'=\mu_k(t)=((X_i(t')_{i\in [1,m]}), \widetilde B(t'),\Lambda(t'))$, where
\begin{enumerate}[$(1)$]
\item $X_i(t')=X_i(t)$ for $i\neq t$,
\item $X_k(t')=X(t)^{-{\bf e}_k+\sum_i[b_{ik}]_{+}{\bf e}_i}+X(t)^{-{\bf e}_k+\sum_i[-b_{ik}]_{+}{\bf e}_i}$.
\item $\widetilde{B}(t')=(b'_{ij})$ is determined by $\widetilde B(t)=(b_{ij})$ such that
\[\begin{array}{ccl} b'_{ij} &=&
\left\{\begin{array}{ll}
-b_{ij}, &\mbox{if $i=k$ or $j=k$}, \\
b_{ij}+[b_{ik}]_{+}[b_{kj}]_{+}-[-b_{ik}]_{+}[-b_{kj}]_{+}, &\mbox{otherwise}.
\end{array}\right.
\end{array}\]
\item $\Lambda(t')$ is skew-symmetric and satisfies:
\[\begin{array}{ccl} \Lambda(t')_{ij} &=&
\left\{\begin{array}{ll}
\Lambda(t)_{ij}, &\mbox{if $i,j\neq k$}, \\
\Lambda(t)({\bf e}_i,-{\bf e}_k+\sum_{l}[b_{lk}]_{+}{\bf e}_l), &\mbox{if $i\neq k=j$},
\end{array}\right.
\end{array}\]
\end{enumerate}

\end{definition}

One see that $(\widetilde B(t'), \Lambda(t'))$ is compatible since $\widetilde B^T(t)\Lambda(t)=\widetilde B^T(t')\Lambda(t')$. The quantum torus $\mathcal T(t')$ for the new seed $t'$ can be defined similarly.

\begin{definition}

A \emph{quantum cluster algebra} $\mathcal A_v$ is defined by the following steps.
\begin{enumerate}[$(1)$]
\item Choose an initial seed $t_0=((X_1,\cdots,X_m), \widetilde B,\Lambda)$.
\item Get all the seeds $t$ are obtained from $t_0$ by iterated mutations at directions $k\in [1,n]$.
\item Define $\mathcal A_v=\mathbb Z[v^{\pm1}]\langle X_i(t)\rangle_{t,i\in [1,m]}$.
\item $X_{n+1},\cdots, X_m$ are called \emph{frozen variables} or \emph{coefficients}.
\item A quantum cluster variable in $t$ is called a \emph{quantum cluster variable} of $\mathcal A_v$.
\item $X(t)^{{\bf a}}$ for some $t$ and ${\bf a}\in \mathbb N^m$ is called a \emph{quantum cluster monomial}.
\end{enumerate}

\end{definition}

Note that by specializing $v$ to 1, we get a commutative cluster algebra $\mathcal A_v|_{v=1}$.

\begin{theorem}(Quantum Laurent Phenomenon, \cite{BZ})
Let $\mathcal A_v$ be a quantum cluster algebra and $t$ be a seed. For any quantum cluster variable $X$, we have $X\in \mathcal T(t)$.

\end{theorem}

\begin{conjecture}(Positivity Conjecture, \cite{BZ})\label{conj-positive}
Let $\mathcal A_v$ be a quantum cluster algebra and $t$ be a seed. For any quantum cluster variable $X$ of $\mathcal A_v$,
$$X\in \mathbb N[v^{\pm1}]\langle X(t)^{{\bf a}}\mid {\bf a}\in \mathbb Z^m\rangle.$$

\end{conjecture}

\begin{remark}
The positivity conjecture was proved by Davison \cite{D} for the skew-symmetric case and by the author for the unpunctured orbifold case \cite{H1}.
\end{remark}

\subsection{(Quantum) Cluster algebras from orbifolds}\label{sec:QCO} In this subsection, we recall some combinatorial notation from orbifolds, please refer to \cite{FST,FST3,FST4} for more details.

Let $S$ be a connected oriented Riemann orbifold with boundary. Fix a non-empty set $M$ of marked points in the closure of $S$ with at least one marked point on each boundary component. Fix a finite set $U$ in the interior of $S$ such that $U\cap M=\emptyset$. We call the triple $(S, M, U)$ an \emph{orbifold}. Marked points in the interior of $S$ are called \emph{punctures}. The points in $U$ are called \emph{orbifold points.} Each orbifold point in $U$ comes with a weight $\frac{1}{2}$ or $2$. Throughout this paper, we assume that $U$ contains no orbifold point with weight $2$.

\begin{definition}\cite{FST3,FST4}
An arc $\beta$ in $(S,M,U)$ is a curve in $S$ considered up to relative isotopy (of $S\setminus (M\cup U)$) modulo endpoints such that
\begin{enumerate}[$\bullet$]
    \item one of the following holds:
    \begin{enumerate}[$-$]
        \item either both endpoints of $\beta$ belong to $M$ (and then $\beta$ is called an \emph{ordinary arc})
        \item or one endpoint belongs to $M$ and another belongs to $U$ (then $\beta$ is called a \emph{pending arc});
    \end{enumerate}
    \item $\beta$ has no self-intersections, except that its endpoints may coincide;
    \item except for the endpoints, $\beta$ and $M\cup U\cup \partial S$ are disjoint;
    \item $\beta$ does not cut out a monogon not containing points of $M$;
    \item $\beta$ is not homotopic to a boundary segment.
\end{enumerate}

 If $\gamma$ is an arc incident to a puncture $p$, denote by $l_p(\gamma)$ the loop based on another endpoint of $\gamma$ and encloses $\gamma$.

An \emph{oriented arc} $\overrightarrow \beta$ is an arc $\beta$ with an orientation. Denote by $s(\overrightarrow \beta)$ and $t(\overrightarrow \beta)$ the starting point and ending point, respectively, of $\overrightarrow\beta$.

\end{definition}

Note that we do not allow both endpoints of $\beta$ to be in $U$.

\begin{definition}\cite{FST3,FST4}
Two arcs $\beta$ and $\beta'$ are \emph{compatible} if the following hold:
\begin{enumerate}[$\bullet$]
    \item they do not intersect in the interior of $S$;
    \item if both $\beta$ and $\beta'$ are pending arcs, then the endpoints of $\beta$ and $\beta'$ that are orbifold points do not coincide (i.e., two pending arcs may share a marked point, but not an orbifold point).
\end{enumerate}
\end{definition}

\begin{definition}\cite{FST,FST3,FST4}
An \emph{ideal triangulation} of $(S,M,U)$ is a maximal collection of distinct pairwise compatible arcs.
\end{definition}

The arcs of an ideal triangulation cut $S$ into triangles. See Fig. \ref{Fig:tri} for a list of possible ideal triangles.

\begin{figure}[ht]
\includegraphics{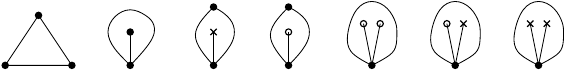}

\caption{List of possible ideal triangles}\label{Fig:tri}
\end{figure}

A triangle with two arcs folded is called a \emph{self-folded triangle}. The folded arc is called the \emph{radius}. Let $T^o$ be an ideal triangulation of $(S,M,U)$. If $\alpha$ is not the radius of some self-folded triangles, then there is a unique arc $\alpha'\neq \alpha$ such that $(T^o\setminus \{\alpha\})\cup \{\alpha'\}$ is an ideal triangulation. Denote $\mu_\alpha(T^o)=(T^o\setminus \{\alpha\})\cup \{\alpha'\}$ and we call $\mu_\alpha(T^o)$ is the \emph{flip} of $T^o$ at $\alpha$ and $\alpha'$ is obtained from $T^o$ by flip at $\alpha$.

\begin{lemma}\label{lem:flip-p}
For each puncture $p$, let $T_1$ and $T_2$ be ideal triangulations without self-folded triangles enclose $p$. Then there exists a sequence of ideal triangulations $T_1=T^0, T^1,\cdots, T^n=T_2$ without self-folded triangles enclose $p$ such that $T^i$ and $T^{i+1}$ are related by a flip for all $i=0,\cdots, n-1$.
\end{lemma}

\begin{proof}
By \cite[Proposition 3.8]{FST}, $T_1$ and $T_2$ are connected by a sequence of flips. We may assume that there exists a sequence of ideal triangulations $T_1=\hat T^0, \hat T^1,\cdots, \hat T^m=T_2$ without self-folded triangles enclose $p$ such that $\hat T^i$ and $\hat T^{i+1}$ are related by a flip for all $i=0,\cdots, m-1$. We prove the result by induction on the number $N$ of the loops in $\bigcup_{i=0}^m \hat T^i$ enclosed only $p$. If $N=0$ then we are done. Assume $N>1$ and $l_p\in \hat T_\ell\cap \hat T_{\ell'}$ is a loop enclose only $p$ such that $\ell<\ell'$, $l_p\notin \hat T^i$ for $0\leq i<\ell$, $l_p\in \hat T^i$ for $\ell\leq i\leq \ell'$ and $l_p\notin \hat T_{\ell'+1}$. Then $l_p$ is in some self-folded triangle in $\hat T^i$ for $\ell\leq i\leq \ell'$. Denote by $\gamma$ the radius. Thus $\gamma\in \hat T_{\ell-1}\cap \hat T_{\ell'+1}$. Cutting along $\gamma$, we obtain a new orbifold $\widetilde S$. $\gamma$ becomes to two boundary arcs $\gamma^\pm$ in $\hat S$. Since $\gamma\in \hat T_{\ell-1}\cap \hat T_{ell'+1}$, we have $(\hat T_{\ell-1}\setminus\{\gamma\})\cup \{\gamma^\pm\}$ and $(\hat T_{\ell+1}\setminus\{\gamma\})\cup \{\gamma^\pm\}$ are two triangulations of $\widetilde S$. By \cite[Proposition 3.8]{FST}, $(\hat T_{\ell-1}\setminus\{\gamma\})\cup \{\gamma^\pm\}$ and $(\hat T_{\ell+1}\setminus\{\gamma\})\cup \{\gamma^\pm\}$ are connected by a sequence of flips. Assume that $\widetilde T_1=(\hat T_{\ell-1}\setminus\{\gamma\})\cup \{\gamma^\pm\}, \widetilde T_2,\cdots, \widetilde T_{m'}=(\hat T_{\ell+1}\setminus\{\gamma\})\cup \{\gamma^\pm\}$ is a sequence of ideal triangulations of $\widetilde S$ such that $\widetilde T_{i+1}$ and $\widetilde T_i$ are related by a flip for $1\leq i\leq m'-1$. For each $1\leq i\leq m'-1$, let $\bar T_i=(\widetilde T_i\setminus \{\gamma^\pm\})\cup \{\gamma\}$. Then $\bar T_i$ is an ideal triangulation of $S$ and $\bar T_i,\bar T_{i+1}$ are related by a flip, and $\bar T_1=\hat T_{\ell-1}, \bar T_{m'}=\hat T_{\ell+1}$. Consider the flip sequence $\hat T^0=T_1,\cdots, \hat T_{\ell-1}=\bar T_1, \bar T_2,\cdots \bar T_{m'}=\hat T_{\ell+1},\cdots, \hat T_m=T_2$. We have the number of the loops in $\bigcup_{i=0}^{\ell-1} \hat T^i\cup \bigcup_{i=\ell+1}^{m} \hat T^i\cup \bigcup_{i=0}^{m'} \bar T^i$ enclosed only $p$ is strictly than $N$. Then the result is followed by induction.
\end{proof}

The following lemma is similar to Lemma \ref{lem:flip-p}.

\begin{lemma}\label{lem:flip-pq}
For punctures $p$ and $q$, let $T_1$ and $T_2$ be ideal triangulations without self-folded triangles enclose $p$ or $q$. Then there exists a sequence of ideal triangulations $T_1=T^0, T^1,\cdots, T^n=T_2$ without self-folded triangles enclose $p$ or $q$ such that $T^i$ and $T^{i+1}$ are related by a flip for all $i=0,\cdots, n-1$.
\end{lemma}

\begin{definition}\cite[Definition 7.1]{FST}(Tagged arc)
Each arc $\beta$ in $(S,M,U)$ has two ends obtained by arbitrarily cutting $\beta$ into three pieces, then throwing out the middle one. We think of the two ends as locations near the endpoints to be used for labeling (``tagging") an arc. A tagged arc is an ordinary arc in which each end has been tagged in one of two ways, plain or notched, so that the following conditions are satisfied:
\begin{enumerate}[$\bullet$]
    \item the arc does not cut out a once-punctured monogon;
    \item an endpoint lying on the boundary is tagged plain;
    \item ends of a pending arc being orbifold points are always tagged plain;  and
    \item both ends of a loop are tagged in the same way.
\end{enumerate}
\end{definition}

\begin{definition}\cite[Definition 7.4]{FST} (Compatibility of tagged arcs) Two tagged arcs $\beta$ and $\beta'$ are called \emph{compatible} if the following conditions are satisfied:
\begin{enumerate}[$\bullet$]
    \item The untagged versions of $\beta$ and $\beta'$ are compatible;
    \item If the untagged versions of $\beta$ and $\beta'$ are different, and $\beta$ and $\beta'$ share an endpoint $a$, then the ends of $\beta$ and $\beta'$ connecting to $a$ must be tagged in the same way;
    \item If the untagged versions of $\beta$ and $\beta'$ coincide, then at least one end of $\beta$ must be tagged in the same way as the corresponding end of $\beta'$.
\end{enumerate}

\end{definition}

\begin{definition}\cite{FST,FST3,FST4}
A \emph{tagged triangulation} of $(S,M,U)$ is a maximal collection of distinct pairwise compatible tagged arcs.
\end{definition}

\begin{definition}

Let $o$ be an orbifold point. For a tagged pending arc $\beta$ connecting a marked point $i$ and $o$, denote by $sl(\beta)$ the tagged loop at $i$ around $o$ which tagged the same with $\beta$ at $i$. We call $sl(\beta)$ the \emph{special loop} associated with $\beta$. See Fig. \ref{Fig:SL}.

\begin{figure}[ht]

\includegraphics{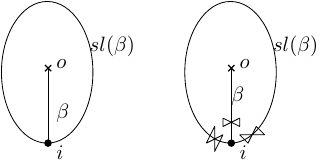}

\caption{Special loop}\label{Fig:SL}
\end{figure}

\end{definition}

Note that special loop is not an arc in $\Sigma=(S,M,U)$.

\smallskip

Throughout this paper, the notation $\gamma$, $\gamma^{(q)}$ and $\gamma^{(p,q)}$ respectively mean the following three cases,
\begin{enumerate}
    \item $\gamma$ is an ordinary arc or a pending arc starting from $p$ and ending at $q$;
    \item $\gamma^{(p)}$ is the tagged arc with underlying arc $\gamma$ and tagged notched at $q$;
    \item $\gamma^{(p,q)}$ is the tagged arc with underlying arc $\gamma$ and tagged notched at $p,q$.
\end{enumerate}

In particular, if $p=q$ then $\gamma^{(p)}=\gamma^{(p,q)}$.

\begin{figure}[h]
\includegraphics{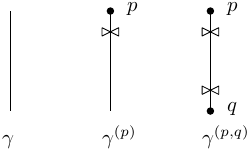}

\end{figure}

For any ideal triangulation $T^o$, replace all the loops $l_p(\gamma)$ in self-folded triangles enclose puncture $p$ with radius $\gamma$ by $\gamma^{(p)}$, we obtain a tagged triangulation $T$, we call $T$ the tagged triangulation corresponding to $T^o$.

Conversely, for any tagged triangulation $T$, first replace all pairs of tagged arcs of the form $\gamma,\gamma^{(p)}$ by $\gamma,l_p(\gamma)$ and then replace $\gamma^{(p)}$ and $\gamma^{(p,q)}$ by $\gamma$, we obtain an ideal triangulation, denoted by $l(T)$.

If $\gamma$ is a pending arc which incident an orbifold point $o$, denote by $l(\gamma)$ the loop cutting out a monogon enclosing $o$ and radius $\gamma$, then the \emph{crossing number} $N(\gamma,\gamma')$ of $\gamma$ with $\gamma'$ is the minimum of the numbers of crossings of arcs $\alpha$ and $\alpha'$, where $\alpha$ is isotopic to $l(\gamma)$ and $\alpha'$ is isotopic to $\gamma'$. Note that $N(\gamma,\gamma')\neq N(\gamma,\gamma')$ generally if one is an ordinary arc and another is a pending arc.

 For $\tau\in T$, denote by $w(\tau)$ the weight of $\tau$.

For two non-boundary non-self-folded arcs $\tau$, $\tau'$ in an ideal triangulation $T^o$ and a non-self-folded triangle $\Delta$ in $T^o$, define
\[\begin{array}{ccl} b^{T^o,\Delta}_{\tau\tau'}=
\left\{\begin{array}{ll}
w(\tau'), &\mbox{if $\tau,\tau'$ are sides of $\Delta$ and $\tau'$ following $\tau$ in the clockwise order}, \\
-w(\tau'), &\mbox{if $\tau,\tau'$ are sides of $\Delta$ and $\tau$ following $\tau'$ in the clockwise order}, \\
0, &\mbox{otherwise.}
\end{array}\right.
\end{array}\]
and $b^{T^o}_{\tau\tau'}=\sum_{\Delta}b^{T^o,\Delta}_{\tau\tau'}$, where $\Delta$ runs over all non-self-folded triangle $\Delta$ in $T^o$. Denote by $l(\tau)$ the loop encloses $\tau$ if $\tau$ is a self-fold arc and $\tau$ otherwise. For any two non-boundary arcs $\tau,\tau'\in T^o$, define $b^{T^o,\Delta}_{\tau\tau'}=b^{T^o,\Delta}_{l(\tau)l(\tau')}$.

We say the matrix $B^{T^o}=(b^{T^o}_{\tau,\tau'})$ the \emph{signed adjacency matrix} of $T$, see \cite{FST3,FST4}. Then $B^{T^o}$ is a skew-symmetrizable matrix. In fact, let $D^{T^o}=diag(w(\tau))_{\tau\in T^o}$, then $D^{T^o}B^{T^o}$ is skew-symmetric.

For any non-self-folded $\tau\in T^o$, we have $B^{\mu_{\tau}(T^o)}=\mu_{\tau}(B^{T^o})$.

\begin{definition} Let $(S,M,U)$ be an orbifold.
\begin{enumerate}[$(1)$]
    \item We say that a cluster algebra $\mathcal A$ is \emph{coming from $(S, M, U)$} if there exists a tagged triangulation $T$ such that $B^{T}$ is an exchange matrix of $\mathcal A$.
    \item We say that a quantum cluster algebra $\mathcal A_v$ is \emph{coming from $(S, M, U)$} if the specialized cluster algebra $\mathcal A_v|_{v=1}$ is coming from $(S,M,U)$.
\end{enumerate}
\end{definition}

\begin{proposition}\label{Pro-11}\cite{FST,FST3,BZ}
Let $\mathcal A_v$ be a quantum cluster algebra from $\Sigma=(S,M,U)$.
\begin{enumerate}[$(1)$]
\item If $\Sigma$ is not a closed surface with one puncture, then there are bijections
\begin{equation*}
    \{\text{Tagged arcs in }\Sigma\} \rightarrow \{\text{Quantum cluster variables of }\mathcal A_v\}, \gamma\mapsto X_\gamma.
\end{equation*}

\begin{equation*}
    \{\text{Tagged triangulation of }\Sigma\} \rightarrow \{\text{Quantum clusters of }\mathcal A\}, T\mapsto X_T.
\end{equation*}

\item If $\Sigma$ is a closed surface with exactly one puncture, then there are bijections
\begin{equation*}
    \{\text{Ordinary arcs in }\Sigma\} \rightarrow \{\text{Quantum cluster variables of }\mathcal A\}, \gamma\mapsto X_\gamma.
\end{equation*}

\begin{equation*}
    \{\text{Ideal triangulation of }\Sigma\} \rightarrow \{\text{Quantum clusters of }\mathcal A\}, T\mapsto X_T.
\end{equation*}
\end{enumerate}
\end{proposition}

\subsection{Orbifold morphism and canonical polygon}\label{sec:OMCP}

\begin{definition}\cite[Definition 3.5]{BR}
For two orbifolds $(S,M,U)$ and $(S',M',U')$, we say that a continuous map $f:S\to S'$ is a \emph{morphism} if
\begin{enumerate}[$\bullet$]
\item $f^{-1}(M'\cup U')\subseteq M\cup U$ and $f(U)\subseteq U'$;
\item For each point $p\in S\setminus I^f$, there exists a neighborhood $S_p$ of $p$ in $S$ such that the restriction of $f$ to $S_p$ is injective, where $I^f:=f^{-1}(U')\setminus U$;
\item For each point $p\in I^f$, there exists a neighborhood $S_p$ of $p$ in $S$ such that the restriction of $f$ to $S_p$ is a two-fold cover of $f(S_p)$ ramified at $p$.
\end{enumerate}
\end{definition}

\begin{theorem}\cite[Theorem 3.21]{BR}
Let $T^o$ be a triangulation of $(S,M,U)$. Then for each $r$-gon $Q=(\gamma'_1,\cdots,\gamma'_r$ in $(S,M,U)$, there exists an $n$-gon $P=(\gamma_1,\cdots,\gamma_n)\in (T_0(Q,T^o))^n$ for some $n\geq r$, a triangulation $\Delta$ of $P_n$ (the $n$-gon), and an order-preserving embedding $\iota:[r]\hookrightarrow [n]$ such that:
\begin{enumerate}[$(a)$]
\item $\gamma_{ij}\in T_0(Q,T^o)$ if and only if $(i,j)\in \Delta$;
\item $\gamma'_k=\gamma_{\iota(k),\iota(k^+)}$ for all $k\in [r]$ (i.e., $Q$ is a ``sub-polygon" of $P$).
\end{enumerate}
\end{theorem}

In particular, if $Q=(\gamma,\bar\gamma)$ is the digon formed by an arc $\gamma$ and its inverse, then we call the polygon $P_{T^o}(\gamma)$ the \emph{canonical polygon} of $\gamma$ with respect to $T^o$.

\begin{example}
In Figure \ref{F-cp}, the left picture is an annulus with one marked point on each boundary and triangulation $\{1,2\}$. The arc $\gamma$ crosses through triangles $(1,3,2), (2,4,1),(1,3,2)$ and $(2,4,1)$ consecutively. Thus the canonical polygon for $\gamma$ is as shown in the right picture.
\begin{figure}[h]

\includegraphics{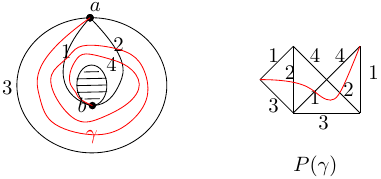}

\caption{Canonical polygon}\label{F-cp}
\end{figure}
\end{example}

\subsection{Snake graph and perfect matching}\label{sec:P}

We first recall the definition of an abstract snake graph \cite{CS}. A tile is considered as a graph with four vertices and four edges in the obvious way. Throughout this paper, we denote by $N(G)$ (resp.  $S(G),W(G),E(G)$) the north (resp. south, west, east) edge of a tile $G$.

\begin{figure}[h]

\includegraphics{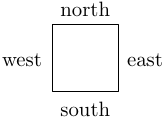}

\end{figure}

\begin{definition}\cite{CS}
A \emph{snake graph} $G$ is a connected graph consisting of a finite sequence of tiles $G_1,G_2,\cdots, G_c$ with $c\geq 1$, such that for each $i=1,2,\cdots,c-1$

\begin{enumerate}[(i)]
    \item  $G_i$ and $G_{i+1}$ share exactly one edge which is either the north edge of $G_i$ and the south edge of $G_{i+1}$ or the east edge of $G_i$ and the west edge of $G_{i+1}$.
    \item  $G_i$ and $G_j$ have no edge in common whenever $|i-j|\geq  2$.
    \item  $G_i$ and $G_j$ are disjoint whenever $|i-j|\geq 3$.
\end{enumerate}
\end{definition}

Let $H$ be a graph. In this paper, we denote by $edge(H)$ the edge set of $H$.

\begin{definition} (\cite[Definition 4.6]{MSW})
A \emph{perfect matching} of a graph $G$ is a subset $P$ of the edges of $G$ such that each vertex of $G$ is incident to exactly one edge of $P$. Denote by $\mathcal P(G)$ the set of all perfect matchings of $G$.

\end{definition}

\begin{definition}(\cite{MSW1,H,H1})
Let $P$ be a perfect matching of a snake graph $G$. We say that $P$ can \emph{twist} on a tile $G_i$ if there are two edges of $G_i$ in $P$. The perfect matching obtained by replacing the two edges with the remaining two edges of $G_i$ is called the \emph{twist} of $P$ at $G_i$, denoted by $\mu_{G_i}(P)$.
\end{definition}

\begin{figure}[h]

\includegraphics{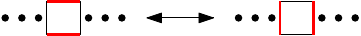}

\centerline{Twist of perfect matching}

\end{figure}

\begin{definition}
For a perfect matching $P$, we call an edge in $P$ \emph{$\tau$-mutable} if it is an edge of some tile with diagonal labeled $\tau$ that $P$ can twist on.
\end{definition}

Next, we recall the construction of the snake graph $G_{T,\gamma}$ and its perfect matching. For more details, see \cite[Section 4]{MSW}, \cite{CS}.

Let $T^o$ be an ideal triangulation and $\gamma$ be a curve connecting two marked points. Let $p_0$ be the starting point of $\gamma$, and let $p_{c+1}$ be its endpoint. Assume that $\gamma$ crosses $T^o$ at $p_1,\cdots,p_c$ sequentially.

Let $\tau_{i_j}$ be the arc in $T^o$ containing $p_j$. Let $\Delta_{j-1}$ and $\Delta_{j}$ be the two ideal triangles in $T$ on either side of $\tau_{i_j}$.

For each $p_j$, we associate a \emph{tile} $G_j$ as follows. Define $\Delta_1^j$ and $\Delta_2^j$ to be two triangles with edges labeled as in $\Delta_{j-1}$ and $\Delta_{j}$, further, the orientations of $\Delta_1^j$ and $\Delta_2^j$ both agree with those of $\Delta_{j-1}$ and $\Delta_{j}$ if $j$ is odd; the orientations of $\Delta_1^j$ and $\Delta_2^j$ both disagree with those of $\Delta_{j-1}$ and $\Delta_{j}$ otherwise. We glue $\Delta_1^j$ and $\Delta_2^j$ at the edge labeled $\tau_{i_j}$, so that the orientations of $\Delta_1^j$ and $\Delta_2^j$ both either agree or disagree with those of $\Delta_{j-1}$ and $\Delta_{j}$. We say the edge labeled $\tau_{i_j}$ the \emph{diagonal} of $G_j$.

The two arcs $\tau_{i_j}$ and $\tau_{i_{j+1}}$ form two edges of the triangle $\Delta_j$. Denote the third edge of $\Delta_j$ by $\tau_{[\gamma_j]}$. After gluing the tiles $G(p_j)$ and $G(p_{j+1})$ at the edge labeled $\tau_{[\gamma_j]}$ for $1\leq j<d-1$ step by step, we obtain a graph, denote as $\overline{G_{T^o,\gamma}}$. Let $G_{T^o,\gamma}$ be the graph obtained from $\overline{G_{T^o,\gamma}}$ by removing the diagonal of each tile.

In particular, when $\gamma\in T$, let $G_{T^o,\gamma}$ be the graph with one only edge labeled $\gamma$.

Denote \[\begin{array}{ccl} rel(G_i,T^o)=
\left\{\begin{array}{ll}
1, &\mbox{if the orientations of $G_i$ and $T^o$ coincide}, \\
-1, &\mbox{if the orientations of $G_i$ and $T^o$ are different.}
\end{array}\right.
\end{array}\]

\begin{definition} (\cite[Definition 4.7]{MSW})
Let $a_1$ and $a_2$ be the two edges of $G_{T^o,\gamma}$ which lie in the counterclockwise direction from the diagonal of $G_1$. Then the \emph{minimal matching} $P_{-}$ is defined as the unique matching which contains only boundary edges and does not contain edges $a_1$ or $a_2$. The \emph{maximal matching} $P_{+}$ is the other matching with only boundary edges.
\end{definition}

\begin{lemma}\cite[Lemma 2.4]{H}\label{max-min}
Let $a$ be an edge of the tile $G_j$. If $a$ is in the maximal/minimal perfect matching of $G_{T,\gamma}$, then $a$ lies in the counterclockwise/clockwise direction from the diagonal of $G_j$ when $j$ is odd and lies in the clockwise/counterclockwise direction from the diagonal of $G_j$ when $j$ is even.
\end{lemma}

\subsection{Complete $(T^o,\gamma)$-path}\label{Sec-CPB}

Let $T^o$ be an ideal triangulation and $\gamma$ be an arc. Choose an orientation of $\gamma$, assume that $\overrightarrow\gamma$ crosses $T^o$ at $p_1,\cdots,p_c$ sequentially. Denote $p_0=s(\overrightarrow\gamma)$ and $p_{c+1}=t(\overrightarrow\gamma)$. Suppose that $p_j\in \tau_{i_j}$ for $j=1,\cdots,c$.

\begin{definition} \cite{MS,S} A sequences of oriented arcs $\overrightarrow\xi=(\overrightarrow{\xi}_{\hspace{-2pt}1},\cdots,\overrightarrow{\xi}_{\hspace{-2pt}2c+1})$ is called a \emph{complete $(T^o,\gamma)$-path} if the following
axioms hold:
\begin{enumerate}[$(T1)$]
\item $\xi_i\in T^o$ for all $i\in \{1,\cdots,2c+1\}$;
\item $s(\overrightarrow\xi_{\hspace{-2pt}1})=s(\overrightarrow\gamma)$, $t(\overrightarrow\xi_{\hspace{-2pt}2c+1})=t(\overrightarrow\gamma)$ and $s(\overrightarrow\xi_{\hspace{-2pt}i+1})=t(\overrightarrow\xi_{\hspace{-2pt}i})$ for all $i\in \{1,\cdots, 2c\}$;
\item The even arcs are precisely the arcs crossed by $\gamma$ in order, that is, $\xi_{2k} = \tau_{i_k}$ for all $k\in\{1,\cdots,c\}$;
\item For all $k=0,1,2,\cdots,c$, the segment $\overrightarrow\gamma_{\hspace{-2pt}k}$ of $\overrightarrow\gamma$ starting from $p_k$ and ending at $p_{k+1}$ is homotopic to the segment of the path $\overrightarrow\xi$ starting at the point $p_k$ following $\overrightarrow\xi_{\hspace{-2pt}2k}, \overrightarrow\xi_{\hspace{-2pt}2k+1}$ and $\overrightarrow\xi_{\hspace{-2pt}2k+2}$ until the point $p_{k+1}$.
\end{enumerate}
Denote by $\mathcal {CP}(T^o,\gamma)$ the set of all complete $(T^o,\gamma)$-paths.
\end{definition}

For any complete $(T^o,\gamma)$-path $\overrightarrow\xi=(\overrightarrow\xi_{\hspace{-2pt}1},\cdots,\overrightarrow\xi_{\hspace{-2pt}2c+1})$ and arc $\zeta\in T^o$, denote by
\begin{equation}\label{equ-m1}
m(\overrightarrow\xi,\zeta)=\sum_{i=1}^{2c+1}(-1)^{i-1}\delta_{\xi_i,\zeta},
\end{equation}
where
\begin{equation}\label{equ-delta}
        \delta_{\xi_i,\zeta}=
        \begin{cases}
        1, &\mbox{if $\xi_i=\zeta$},\\
        0, &\mbox{if $\xi_i\neq \zeta$}.
        \end{cases}
    \end{equation}

By \cite[Theorem 4.4]{MS}, there is a natural bijective map from $\mathcal P(G_{T^o,\gamma})$ to $\mathcal {CP}(T^o,\gamma)$. Roughly speaking, for any perfect matching $P$, by taking the diagonals of all the tiles, we get a complete path, see the following figure for an illustration.

\begin{figure}[h]

\includegraphics{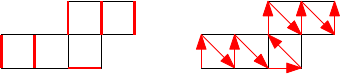}

\end{figure}

\begin{definition}
We say that two complete $(T^o,\gamma)$-paths $\overrightarrow\xi$ and $\overrightarrow\zeta$ are related by a \emph{twist at $\alpha\in T^o$} if there is an even $j$ such that $\overrightarrow\xi_{\hspace{-2pt}i}=\overrightarrow\zeta_{\hspace{-2pt}i}$ for $i\neq j-1,j,j+1$ and $(\overrightarrow\xi_{\hspace{-2pt}j-1},\overrightarrow\zeta_{\hspace{-2pt}{j+1}},\overleftarrow\xi_{\hspace{-2pt}j+1},\overleftarrow\zeta_{\hspace{-2pt}j-1})$ is a quadrilateral in $T^o$ with diagonal $\alpha$. In this case, we see that $\xi_{j-1}, \xi_{j+1}$ are $\alpha$-twist-able edges in $\overrightarrow\xi$.
\end{definition}

\newpage

\section{An isomorphism of quantum cluster algebras}\label{Sec-iso}

Let $\Sigma$ be an orbifold and $\mathcal A_v(\Sigma)$ be a quantum cluster algebra from $\Sigma$. For any puncture $q$ and tagged arc $\beta$, denote by $\beta^{(q)}$ the tagged arc obtained from $\beta$ by changing the tags at $q$. Let $T$ be a tagged triangulation of $\Sigma$. Denote by $T^{(q)}=\{\beta^{(q)}\mid \beta\in T\}$. Then $T^{(q)}$ is a tagged triangulation and $B(T)=B(T^{(q)})$. Suppose that the extended exchange matrix and the quantum commutative matrix of $\mathcal A_v(\Sigma)$ at $T$ are $\widetilde B(T)=\left(\begin{array}{c} B(T) \\ C(T)\end{array}\right)$ and $\Lambda(T)$, respectively. Assume that $\Sigma$ is not a closed surface with exactly one puncture, let $\mathcal A^{(q)}_v$ be the quantum cluster algebra from $\Sigma$ such that the extended exchange matrix and quantum commutative matrix at $T^{(q)}$ are $\widetilde B(T)$ and $\Lambda(T)$, respectively.

By Proposition \ref{Pro-11}, for any tagged arc $\beta$, denote by $X_\beta$ and $X^{(q)}_\beta$ the quantum cluster variables of $\mathcal A_v(\Sigma)$ and $\mathcal A^{(q)}_v(\Sigma)$, respectively corresponding to $\beta$. Then we have the following proposition. Denote by $X_{n+1},\cdots,X_m$ the frozen quantum cluster variables of $\mathcal A_v(\Sigma)$, denote by $X^{(q)}_{n+1},\cdots,X^{(q)}_m$ the frozen quantum cluster variables of $\mathcal A^{(q)}_v(\Sigma)$.

\begin{lemma}\label{Lem-comq}
Let $T$ be a tagged triangulation and $q$ be a puncture. For any tagged arc $\beta\in T$, we have
$\mu_{\beta^{(q)}}(T^{(q)})=(\mu_\beta(T))^{(q)}.$
\end{lemma}

\begin{proof}
Suppose that $\mu_\beta(T)=(T\setminus \{\beta\})\cup \{\beta'\}$ for some tagged arc $\beta'$. Then for any $\alpha\in T\setminus \{\beta\}$ we have $\beta'$ is compatible with $\alpha$, by \cite[Remark 5.13]{FT}, $\beta'^{(q)}$ is compatible with $\alpha^{(q)}$. It follows that
$\mu_{\beta^{(q)}}(T^{(q)})=(\mu_\beta(T))^{(q)}.$
\end{proof}

\begin{proposition}\label{Pro-iso}
Assume that $\Sigma$ is not a closed surface with exactly one puncture. Then there is a $\mathbb Z[v^{\pm1}]$-algebra isomorphism $\sigma: \mathcal A_v(\Sigma)\to \mathcal A^{(q)}_v(\Sigma)$ which satisfies
\begin{enumerate}[$(1)$]
\item $\sigma(X_\beta)=X^{(q)}_{\beta^{(q)}}$ for all tagged arc $\beta$.
\item $\sigma$ preserves mutations.
\end{enumerate}
\end{proposition}

\begin{proof}
For any tagged triangulation $T'$, denote by $\widetilde B^{(q)}(T')$ and $\Lambda^{(q)}(T')$ the extended exchange matrix and quantum commutative matrix, respectively of $\mathcal A^{(q)}_v(\Sigma)$ at $T'$, denote by $\mathcal T(T')$ and $\mathcal T^{(q)}(T')$ the quantum torus at $T'$ for $\mathcal A_v(\Sigma)$ and $\mathcal A^{(q)}_v(\Sigma)$, respectively. Then we have $\widetilde B^{(q)}(T^{(q)})=\widetilde B(T)$ and $\Lambda^{(q)}(T^{(q)})=\Lambda(T)$, moreover, for any $\beta\in T$, the column (resp. row) of $\widetilde B(T)$ which is indexed by $X_\beta$ equals the column (resp. row) of $\widetilde B^{(q)}(T^{(q)})$ which is indexed by $X_{\beta^{(q)}}$; for any $i\in \{n+1,\cdots,m\}$ the row of $\widetilde B(T)$ which is indexed by $X_\beta$ equals the row of $\widetilde B^{(q)}(T^{(q)})$ which is indexed by $X_{\beta^{(q)}}$. It is also true for the quantum commutative matrices $\Lambda(T)$ and $\Lambda^{(q)}(T^{(q)})$.

Clearly, we have a $\mathbb Z[v^{\pm1}]$-algebra isomorphism of quantum torus
\begin{equation*}
\sigma: \mathcal T(T) \to \mathcal T^{(q)}(T^{(q)}), X_\alpha\mapsto X^{(q)}_{\alpha^{(q)}}.
\end{equation*}
%\hspace{7.1cm}$X_{\alpha}\to X^{(q)}_{\alpha^{(q)}}$.

For any $\beta\in T$, suppose that $\mu_\beta(T)=(T\setminus \{\beta\})\cup \{\beta'\}$ for some tagged arc $\beta'$. As $\widetilde B^{(q)}(T^{(q)})=\widetilde B(T)$, by Lemma \ref{Lem-comq} we have $\sigma(X_{\beta'})=X_{\beta'^{(q)}}$, moreover, we have
\begin{equation*}
\widetilde B^{(q)}(\mu_{\beta^{(q)}}T^{(q)})=\mu_{X^{(q)}_{\beta^{(q)}}}(B^{(q)}(T^{(q)}))=\mu_{X_{\beta}}(B(T))=B(\mu_\beta T).
\end{equation*}
Similarly, we have
$
\Lambda^{(q)}(\mu_{\beta^{(q)}}T^{(q)})=\Lambda(\mu_\beta T).$

Therefore, $\sigma$ induces a $\mathbb Z[v^{\pm1}]$-algebra isomorphism of quantum torus
\begin{equation*}
\sigma: \mathcal T(\mu_\tau T) \to \mathcal T^{(q)}(\mu_{\tau^{(q)}}T^{(q)}), X_{\alpha}\mapsto X^{(q)}_{\alpha^{(q)}}.
\end{equation*}
%\hspace{7.1cm}$X_{\alpha}\to X^{(q)}_{\alpha^{(q)}}$.

As $\Sigma$ is not a closed surface with exactly one puncture, any two tagged triangulations of $\Sigma$ are connected by a sequence of flips. By induction we see that for any tagged triangulation $T'$, $\sigma$ induces a $\mathbb Z[v^{\pm1}]$-algebra isomorphism of quantum torus
\begin{equation*}
\sigma: \mathcal T(T') \to \mathcal T^{(q)}(T'^{(q)}), X_{\alpha}\mapsto X^{(q)}_{\alpha^{(q)}}.
\end{equation*}
%\hspace{7.1cm}$$.
The result follows.
\end{proof}

\begin{remark}\label{Rem-3cases}
Let $T^o$ be an ideal triangulation of $\Sigma$ and $\beta$ be any arc. Assume that $\widetilde\beta$ starts from $p$ and ends at $q$. Let $T$ be the tagged triangulation corresponding to $T^o$. By Proposition \ref{Pro-iso}, to give an expansion formula of $X_\gamma$ concerning a quantum seed $X_\Delta$ for any tagged arc $\gamma$ and tagged triangulation $\Delta$, it suffices to restrict to the following three cases:
\begin{enumerate}[$(i)$]
    \item $\gamma=\widetilde\beta$ and $\Delta=T$;
    \item $q$ is a puncture with $q\neq p$,
    $\gamma=\widetilde\beta^{(q)}$, $\Delta=T$ contains no arcs tagged notched at $q$;
    \item $p,q$ are punctures, $\gamma=\widetilde\beta^{(p,q)}$, $\Delta=T$ contains no arcs tagged notched at $p$ or $q$.
\end{enumerate}
\end{remark}

\newpage

\section{Three lattices associated with tagged arcs}\label{Sec-threesets}

Fix an ideal triangulation $T^o$ of $\Sigma$ and an arc $\beta$. Let $T$ be the corresponding tagged triangulation of $T^o$.
 Let \begin{equation}\label{Eq-wideg}
        \widetilde \beta=
            \begin{cases}
                \beta, & \mbox{if $\beta$ is not a pending arc of weight $1/2$},\vspace{1mm}\\
                sl(\beta), & \mbox{if $\beta$ is a pending arc of weight $1/2$ }.
            \end{cases}
\end{equation}

Denote by $G_1,\cdots, G_c$ the tiles of $G_{T^o,\widetilde\beta}$ in order. Assume that $\widetilde\beta$ starts from $p$ and ends at $q$. In this section, we recall the three lattices $\mathcal L(T^o,\widetilde\beta)$, $\mathcal L(T^o,\widetilde\beta^{(q)})$ and $\mathcal L(T^o,\widetilde\beta^{(p,q)})$ constructed in \cite{H2}, which will be the index sets for our expansion formulas in the quantum case. These three lattices also appear in \cite{BHR,H} for providing expansion formulas for tagged curves in (non-commutative) cluster algebras from $\Sigma$.

\subsection{$\mathcal L(T^o,\widetilde\beta)$}

Let $\mathcal L(T^o,\widetilde\beta)=\mathcal P(G_{T^o,\widetilde\beta})$. For any $P\in \mathcal L(T^o,\widetilde\beta)$ can twist on a tile $G_i$, let $P<\mu_{G_i}(P)$ if $W(G_i),E(G_i)\in P, rel(G_i,T^o)=1$ or $N(G_i),S(G_i)\in P, rel(G_i,T^o)=-1$.

\begin{proposition}\label{Pro-la}\cite{CS1,MSW1}
$\mathcal L(T^o,\widetilde\beta)$ with above order forms a lattice.
\end{proposition}

The following proposition is an immediate consequence.

\begin{proposition}\label{prop-mm}
\begin{enumerate}[$(1)$]
\item The maximum element $P_+$ in $\mathcal P(G_{T^o,\widetilde\beta})$ is the perfect matching contains only boundary edges such that for any $e\in P_+\cap edge(G_i)$ we have $e\in \{N(G_i),S(G_i)\}$ if $rel(G_i,T^o)=1$ and $e\in \{W(G_i),E(G_i)\}$ if $rel(G_i,T^o)=-1$;
\item The minimum element $P_-$ in $\mathcal P(G_{T^o,\widetilde\beta})$ is the perfect matching contains only boundary edges such that for any $e\in P_+\cap edge(G_i)$ we have $e\in \{W(G_i),E(G_i)\}$ if $rel(G_i,T^o)=1$ and $e\in \{N(G_i),S(G_i)\}$ if $rel(G_i,T^o)=-1$.
\end{enumerate}
\end{proposition}

It follows that any two perfect matchings are related by a sequence of twists.

\begin{lemma}\label{lem-H1}
The Hasse graph of $\mathcal L(T^o,\widetilde\beta)$ is connected.
\end{lemma}

\subsection{$\mathcal L(T^o,\widetilde\beta^{(q)})$} \label{delta1}
Herein we assume that $q$ is a puncture in $\Sigma$ and $q\neq p$. We label clockwise the triangles in $T^o$ incident to $q$ by $\Delta_1(q),\Delta_2(q),\cdots, \Delta_t(q)$ such that either $\widetilde\beta$ crosses $\Delta_1(q)$ or $\widetilde \beta$ is the common side of $\Delta_1(q)$ and $\Delta_t(q)$, see Figure \ref{F-tq}. Denote
$$\Delta(T^o,q)=\{\Delta_1(q),\Delta_2(q),\cdots, \Delta_t(q)\}.$$

\begin{figure}[h]

\centerline{\includegraphics[width=7cm]{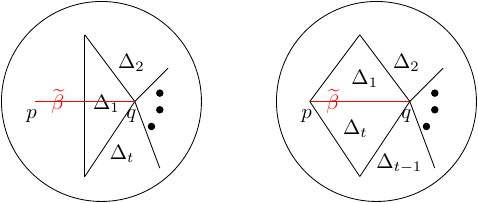}}

\caption{Triangles incident to $q$}\label{F-tq}
\end{figure}

Let
\begin{equation}\label{Eq-L1}
 \mathcal L(T^o,\widetilde\beta^{(q)})=\mathcal P(G_{T^o,\widetilde\beta})\times\{\Delta_1(q),\cdots, \Delta_t(q)\}.
\end{equation}

Set $\Delta_{t+1}(q)=\Delta_1(q)$. Denote by $\tau_j(q)$ the common side of $\Delta_j(q)$ and $\Delta_{j+1}(q)$ for $j=1,\cdots, t$. Let $\tau_0(q)=\tau_t(q)$. Thus $\tau_{j-1}(q)$ and $\tau_j(q)$ are two sides of $\Delta_j(q)$. Denote by $\tau_{[j]}(q)$ the third side of $\Delta_j(q)$. We have $\tau_t(q)=\widetilde\beta$ when $\widetilde\beta\in T^o$.

\subsubsection{$\widetilde\beta\notin T^o$} Herein we consider the case that $\widetilde\beta\notin T^o$. Define
\begin{equation}\label{E1q}
        E_1(q)=
            \begin{cases}
                E(G_c), & \mbox{if $rel( G_c,  T^o)=1$},\vspace{1mm}\\
                N(G_c), & \mbox{if $rel( G_c,  T^o)=-1$,}
            \end{cases}\;\;\;
        E_2(q)=
            \begin{cases}
                N(G_c), & \mbox{if $rel(G_c,  T^o)=1$},\vspace{1mm}\\
                E(G_c), & \mbox{if $rel(G_c,  T^o)=-1$.}
            \end{cases}
\end{equation}

Thus $E_1(q)$ is labeled $\tau_t(q)$ and $E_2(q)$ is labeled $\tau_1(q)$. By Proposition \ref{prop-mm}, we have $E_1(q)\in P_-$ and $E_2(q)\in P_+$.

For any $P\in \mathcal P(G_{T^o,\widetilde\beta})$, it is clear that either $E_1(q)\in P$ or $E_2(q)\in P$.

It is proved in \cite{H2} that $\mathcal L(  T^o,\widetilde\beta^{(q)})$ forms a lattice with minimum element $(P_-,\Delta_1(q))$ and maximum element $(P_+,\Delta_1(q))$, under the partial order induced by the following.

\begin{enumerate}
\item For any $P\in \mathcal P(G_{  T^o,\widetilde\beta})$,
\begin{enumerate}
\item if $E_1(q)\in P$ then
\begin{equation*}
 (P,\Delta_1(q))<(P,\Delta_2(q))<(P,\Delta_3(q))<\cdots<(P,\Delta_t(q));
\end{equation*}
\item if $E_2(q)\in P$ then
\begin{equation*}
(P,\Delta_2(q))<\cdots<(P,\Delta_{t-1}(q))<(P,\Delta_t(q))<(P,\Delta_1(q)).
\end{equation*}
\end{enumerate}
\item For any $j\in \{1,\cdots, t\}$, $(P,\Delta_j(q))< (Q,\Delta_j(q))$ if $P<Q$.
\end{enumerate}

The following lemmas are immediate.

\begin{lemma}\label{Lem-cover}
Let $P$ be a perfect matching which can twist on $G_l$ such that $P>\mu_{G_l}P$. Then for any $\Delta_j(q)$, we have $(P,\Delta_j(q))$ covers $(\mu_{G_l}P,\Delta_j(q))$ unless $j=1$ and $l=c$.
\end{lemma}

\begin{lemma}\label{Lem-cover2}
For any $(P,\Delta_j(q))\in \mathcal L(T^o,\widetilde\beta^{(q)})$,
\begin{enumerate}[$(1)$]
\item if $E_1(q)\in P$, then $(P,\Delta_j(q))$ covers $(P,\Delta_{j-1}(q))$ for all $j\neq 1$;
\item if $E_2(q)\in P$, then $(P,\Delta_j(q))$ covers $(P,\Delta_{j-1}(q))$ for all $j\neq 2$.
\end{enumerate}
\end{lemma}

\begin{example}
In Figure \ref{F-l2}, we show an example of the Hasse graph of $\mathcal L(T^o,\widetilde\beta^{(q)})$.

\begin{figure}[h]
\centerline{\includegraphics[width=7cm]{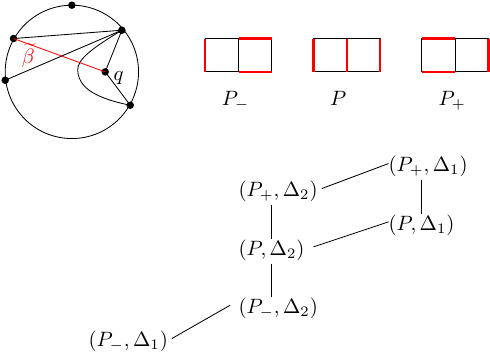}}
\caption{Hasse graph of $\mathcal L(T^o,\widetilde\beta^{(q)})$}\label{F-l2}
\end{figure}

\end{example}

\subsubsection{$\widetilde\beta\in   T^o$.} Herein we consider the case that $\widetilde\beta\in   T^o$. It is clear that $\mathcal L(T^o,\widetilde\beta^{(q)})$ forms a lattice under the following partial order.
\begin{equation*}
 (P_{\widetilde\beta},\Delta_1(q))<(P_{\widetilde\beta},\Delta_2(q))<\cdots<(P_{\widetilde\beta},\Delta_{t-1}(q))<(P_{\widetilde\beta},\Delta_t(q)).
\end{equation*}

The following Lemma follows by Lemma \ref{lem-H1}.

\begin{lemma}\label{lem-H2}
The Hasse graph of $\mathcal L(T^o,\widetilde\beta^{(q)})$ is connected.
\end{lemma}

\begin{lemma}\label{lem:diamond}
Assume that ${\bf P}=(P,\Delta_{j}(q))$ covers ${\bf Q}=(\mu_{G_l}P,\Delta_j(q))$ such that $\tau_{i_l}=\alpha_k\neq \alpha_{k-1},\alpha_{k+1}$ for some $k\in \{1,2,3,4\}$.
\begin{enumerate}[$(1)$]
\item In case $k\in \{1,3\}$, if ${\bf P}_1=(P,\Delta_{j+1}(q))$ covers ${\bf P}$ with $\tau_j(q)=\alpha$, then ${\bf Q}_1=(\mu_{G_l}P,\Delta_{j+1}(q))$ covers ${\bf Q}=(\mu_{G_l}P,\Delta_j(q))$ and ${\bf P}_1$ covers ${\bf Q}_1$.
\item In case $k\in \{2,4\}$, if ${\bf P}$ covers ${\bf P}_1=(P,\Delta_{j-1}(q))$ with $\tau_{j-1}(q)=\alpha$, then ${\bf Q}=(\mu_{G_l}P,\Delta_j(q))$ covers ${\bf Q}_1=(\mu_{G_l}P,\Delta_{j-1}(q))$ covers and ${\bf P}_1$ covers ${\bf Q}_1$.
\end{enumerate}
\end{lemma}

\begin{proof}
We shall only prove (1). We may assume that $k=1$. We first prove ${\bf Q}_1$ covers ${\bf Q}$. Otherwise, by Lemma \ref{Lem-cover2}, we have either $E_1(q)\in \mu_{G_l}P, \Delta_{j+1}(q)=\Delta_1(q)$ or $E_2(q)\in \mu_{G_l}P, \Delta_{j}(q)=\Delta_1(q)$. In case $E_1(q)\in \mu_{G_l}P, \Delta_{j+1}(q)=\Delta_1(q)$, as ${\bf P}_1$ covers ${\bf P}$ we have $E_2(q)\in P$ by Lemma \ref{Lem-cover2}. Thus $G_l=G_c$. As $\tau_{i_c}=\alpha_1$, we have $\tau_t(q)=\alpha_4$. It contradicts to $\tau_j(q)=\tau_t(q)=\alpha$.
We then prove ${\bf P}_1$ covers ${\bf Q}_1$. Otherwise, by Lemma \ref{Lem-cover}, we have $G_l=G_c$ and $\Delta_{j+1}(q)=\Delta_1(q)$, it is impossible from the above discussion.
\end{proof}

\begin{lemma}\label{lem:cop}
Assume that ${\bf P}=(P,\Delta_{j}(q))$ covers ${\bf Q}=(\mu_{G_l}P,\Delta_j(q))$ such that $\tau_{i_l}=\alpha_k\neq \alpha_{k-1},\alpha_{k+1}$ for some $k\in \{1,2,3,4\}$.
\begin{enumerate}[$(1)$]
\item In case $k\in \{1,3\}$, if $\tau_j(q)=\alpha$ but $(P,\Delta_{j+1}(q))$ does not cover ${\bf P}$, then $G_l\neq G_c$ and $m_{\eta_\alpha}({\bf P})=m_{\eta_\alpha}({\bf Q})=1$.
\item In case $k\in \{2,4\}$, if $\tau_{j-1}(q)=\alpha$ but ${\bf P}$ does not cover $(P,\Delta_{j-1}(q))$, then $G_l\neq G_c$ and $m_{\eta_\alpha}({\bf P})=m_{\eta_\alpha}({\bf Q})=1$.
\end{enumerate}
\end{lemma}

\begin{proof}
We shall only prove one. As ${\bf P}_1=(P,\Delta_{j+1}(q))$ does not cover ${\bf P}$, by Lemma \ref{Lem-cover2}, we have either $E_1(q)\in P, \Delta_{j+1}(q)=\Delta_1(q)$ or $E_2(q)\in P, \Delta_{j}(q)=\Delta_1(q)$.

In case $E_1(q)\in P, \Delta_{j+1}(q)=\Delta_1(q)$, as $\tau_j(q)=\alpha$, we have $\tau_{[1]}(q)=\alpha_2$ or $\alpha_4$. Thus $\tau_{i_c}=\tau_{[1]}(q)=\alpha_2$ or $\alpha_4$. It follows that $E_1(q)$ is labeled $\alpha$ and thus $m_{\eta_\alpha}({\bf P})=1$. $E_1(q)\in P$ implies $G_l\neq G_c$. Thus, $m_{\eta_\alpha}({\bf Q})=1$.

In case $E_2(q)\in P, \Delta_{j}(q)=\Delta_1(q)$, as $\tau_j(q)=\alpha$, we have $\tau_{[1]}(q)=\alpha_1$ or $\alpha_3$. Thus $\tau_{i_c}=\tau_{[1]}(q)=\alpha_1$ or $\alpha_3$. It follows that $E_2(q)$ is labeled $\alpha$ and thus $m_{\eta_\alpha}({\bf P})=1$. As ${\bf P}=(P,\Delta_{j}(q))$ covers ${\bf Q}=(\mu_{G_l}P,\Delta_j(q))$, $G_l\neq G_c$. Thus, $m_{\eta_\alpha}({\bf Q})=1$.

From $\tau_j(q)=\alpha$, we see that $m_{\eta_\alpha+1}({\bf P})=-1$. Therefore, in both cases we have $(\eta_\alpha,\eta_\alpha+1)$ is an ${\bf m}({\bf P};\alpha)$ and ${\bf m}({\bf Q};\alpha)$-pair.
\end{proof}

\begin{lemma}\label{lem:diamond1}
Assume that ${\bf P}=(P,\Delta_{j+1}(q))$ covers ${\bf Q}=(P,\Delta_j(q))$ such that $\tau_{j}(q)=\alpha_k\neq \alpha_{k-1},\alpha_{k+1}$ for some $k\in \{1,2,3,4\}$. Assume that $P$ can twist on some $G_l$ with $\tau_{i_l}=\alpha$.
\begin{enumerate}[$(1)$]
\item In case $k\in \{1,3\}$, if $P<\mu_{G_l}P$, then $(\mu_{G_l}P,\Delta_{j+1}(q))$ covers ${\bf P}$ and $(\mu_{G_l}P,\Delta_j(q))$, $(\mu_{G_l}P,\Delta_j(q))$ covers ${\bf Q}$.
\item In case $k\in \{2,4\}$, if $P>\mu_{G_l}P$, then ${\bf P}$ covers $(\mu_{G_l}P,\Delta_{j+1}(q))$, and $(\mu_{G_l}P,\Delta_{j}(q))$ is covered by ${\bf Q}$ and $(\mu_{G_l}P,\Delta_{j+1}(q))$.
\end{enumerate}
\end{lemma}

\begin{proof}
We shall only prove (1).

We first prove $(\mu_{G_l}P,\Delta_{j+1}(q))$ covers ${\bf P}$. Otherwise, by Lemma \ref{Lem-cover} we have $G_l=G_c$ and $\Delta_{j+1}(q)=\Delta_1(q)$. As $P<\mu_{G_l}P$, we see that $E_1(q)\in P$. Thus ${\bf P}=(P,\Delta_1(q))$ does not cover ${\bf Q}=(P,\Delta_t(q))$, a contradiction.

We then prove $(\mu_{G_l}P,\Delta_{j+1}(q))$ covers $(\mu_{G_l}P,\Delta_j(q))$. Otherwise, by Lemma \ref{Lem-cover2} we have either $E_1(q)\in \mu_{G_l}P, \Delta_{j+1}(q)=\Delta_1(q)$ or $E_2(q)\in \mu_{G_l}P, \Delta_{j}(q)=\Delta_1(q)$. If $E_1(q)\in \mu_{G_l}P, \Delta_{j+1}(q)=\Delta_1(q)$, as $P<\mu_{G_l}P$, we have $E_1(q)\in P, \Delta_{j+1}(q)=\Delta_1(q)$, it is contradictions to ${\bf P}$ covers to ${\bf Q}$. If $E_2(q)\in \mu_{G_l}P, \Delta_{j}(q)=\Delta_1(q)$, as ${\bf P}$ covers to ${\bf Q}$, we see that $G_l=G_c$. Thus $\tau_{[j]}(q)=\alpha$ and hence $\tau_{j}(q)\in \{\alpha_2,\alpha_4\}$, contradicts to $\tau_{j}(q)\in \{\alpha_1,\alpha_3\}$.

Last, we prove $(\mu_{G_l}P,\Delta_j(q))$ covers ${\bf Q}$. Otherwise, by Lemma \ref{Lem-cover} we have $G_l=G_c$ and $\Delta_{j}(q)=\Delta_1(q)$. Thus $\tau_{[j]}(q)=\alpha$ and hence $\tau_{j}(q)\in \{\alpha_2,\alpha_4\}$, contradicts to $\tau_{j}(q)\in \{\alpha_1,\alpha_3\}$.

The proof is complete.
\end{proof}

\subsection{$\mathcal L(  T^o,\widetilde \beta^{(p,q)})$} \label{delta2}
In this section, we assume that $p,q$ are punctures. As the once notched case, we label clockwise the triangles in $  T^o$ incident to $p$ by $\Delta_1(p),\Delta_2(p),\cdots, \Delta_s(p)$ such that either $\widetilde\beta$ crosses $\Delta_1(p)$ or $\widetilde\beta$ is the common side of $\Delta_1(p)$ and $\Delta_s(p)$. See Figure \ref{F-tq1}.
Denote $$\Delta(T^o,p)=\{\Delta_1(p),\Delta_2(p),\cdots, \Delta_s(p)\}.$$

Set $\Delta_{s+1}(p)=\Delta_1(p)$. Denote by $\tau_i(p)$ the common side of the triangles $\Delta_i(p)$ and $\Delta_{i+1}(p)$ for $i=1,\cdots, s$. Set $\tau_0(p)=\tau_s(p)$. Thus $\tau_{i-1}(p)$ and $\tau_i(p)$ are two sides of $\Delta_i(p)$. Denote by $\tau_{[i]}(p)$ the third edge of $\Delta_i(p)$.

In particular, if $\widetilde\beta\in T^o$ then $\tau_s(p)=\tau_t(q)=\widetilde\beta$, $\tau_{[s]}(p)=\tau_1(q), \tau_{[t]}(q)=\tau_1(p)$.

\begin{figure}[h]

\centerline{\includegraphics[width=7cm]{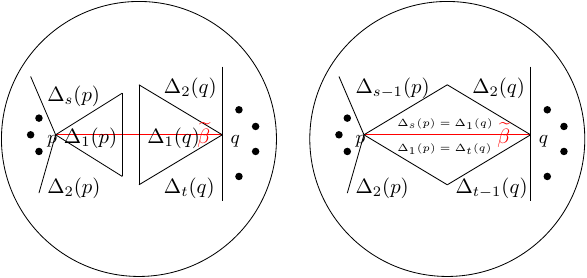}}
\caption{Triangles incident to $p$ and $q$}\label{F-tq1}
\end{figure}

Let
\begin{equation}\label{Eq-L2}
\begin{array}{rcl}
\mathcal L(  T^o,\widetilde\beta^{(p,q)})\hspace{-2mm}&= &\hspace{-2mm}
\{\Delta_1(p),\cdots, \Delta_s(p)\}\times \mathcal P(G_{  T^o,\widetilde\beta})\times\{\Delta_1(q),\cdots, \Delta_t(q)\}\vspace{2mm}\\
\hspace{-2mm}&=&\hspace{-2mm}
\{\Delta_1(p),\cdots, \Delta_s(p)\}\times \mathcal L(  T^o,\widetilde\beta^{(q)}).
\end{array}
\end{equation}

\subsubsection{$\widetilde\beta\notin   T^o$.} Herein we suppose that $\widetilde\beta\notin   T^o$. Similar to (\ref{E1q}) in Section \ref{delta1}, define
\begin{equation}\label{E1p}
        E_1(p)=
            \begin{cases}
                W(G_1), & \mbox{if $rel(G_1,  T^o)=1$},\vspace{1mm}\\
                S(G_1), & \mbox{if $rel(G_1,  T^o)=-1$,}
            \end{cases} \;\;\;
                    E_2(p)=
            \begin{cases}
                S(G_1), & \mbox{if $rel(G_1,  T^o)=1$},\vspace{1mm}\\
                W(G_1), & \mbox{if $rel(G_1,  T^o)=-1$.}
            \end{cases}
\end{equation}

Then $E_1(p)$ is labeled $\tau_s(p)$ and $E_2(p)$ is labeled $\tau_1(p)$. By Proposition \ref{prop-mm}, we have $E_1(p)\in P_-$ and $E_2(p)\in P_+$.

For any $P\in \mathcal P(G_{  T^o,\beta})$, it is clear that either $E_1(p)\in P$ or $E_2(p)\in P$.

It is proved in \cite{H2} that $\mathcal L(  T^o,\widetilde\beta^{(p,q)})$ forms a lattice with minimum element $(\Delta_1(p),P_-,\Delta_1(q))$ and maximum element $(\Delta_1(p),P_+,\Delta_1(q))$, under the partial order induced by the following relations.

\begin{enumerate}
\item For any $(P,\Delta_j(q))\in \mathcal P(G_{T^o,\widetilde\beta})\times \Delta(T^o,q)$,
\begin{enumerate}
\item if $E_1(p)\in P$ then
\begin{equation*}
 (\Delta_1(q),P,\Delta_j(q))<(\Delta_2(q),P,\Delta_j(q))<\cdots<(\Delta_s(q),P,\Delta_j(q));
\end{equation*}
\item if $E_2(q)\in P$ then
\begin{equation*}
 (\Delta_2(q),P,\Delta_j(q))<(\Delta_3(q),P,\Delta_j(q))<\cdots<(\Delta_s(q),P,\Delta_j(q))<(\Delta_1(q),P,\Delta_j(q));
\end{equation*}
\end{enumerate}
\item For any $i\in \{1,\cdots, t\}$, $(\Delta_i(p),P,\Delta_a(q))< (\Delta_i(p),Q,\Delta_b(q))$ if $(P,\Delta_a(q))<(Q,\Delta_b(q))$ in $\mathcal L(T^o,\widetilde\beta^{(q)})$.
\end{enumerate}

The following two lemmas follow immediately by the partial order defined above, Lemma \ref{Lem-cover} and Lemma \ref{Lem-cover2}.

\begin{lemma}\label{Lem-cover1}
For any $P$ can twist on $G_l$ with $\mu_{G_l}P>P$, then for any $\Delta_i(p)$ and $\Delta_j(q)$, we have $(\Delta_i(p),\mu_{G_l}P,\Delta_j(q))$ covers $(\Delta_i(p),P,\Delta_j(q))$ unless $i=1, l=1$ or $j=1, l=c$.
\end{lemma}

\begin{lemma}\label{Lem-cover3}
For any $(\Delta_i(p),P,\Delta_j(q))\in \mathcal L(T^o,\widetilde\beta^{(p,q)})$,
\begin{enumerate}[$(1)$]
\item if $E_1(q)\in P$, then $(\Delta_i(p),P,\Delta_j(q))$ covers $(\Delta_i(p),P,\Delta_{j-1}(q))$ for all $j\neq 1$;
\item if $E_2(q)\in P$, then $(\Delta_i(p),P,\Delta_j(q))$ covers $(\Delta_i(p),P,\Delta_{j-1}(q))$ for all $j\neq 2$;
\item if $(\Delta_i(p),P,\Delta_j(q))>(\Delta_i(p),P,\Delta_{j-1}(q))$ then $(\Delta_i(p),P,\Delta_j(q))$ covers $(\Delta_i(p),P,\Delta_{j-1}(q))$;
\item if $E_1(p)\in P$, then $(\Delta_i(p),P,\Delta_j(q))$ covers $(\Delta_{i-1}(p),P,\Delta_{j}(q))$ for all $i\neq 1$;
\item if $E_2(p)\in P$, then $(\Delta_i(p),P,\Delta_j(q))$ covers $(\Delta_{i-1}(p),P,\Delta_{j}(q))$ for all $i\neq 2$;
\item if $(\Delta_i(p),P,\Delta_j(q))>(\Delta_{i-1}(p),P,\Delta_{j}(q))$ then $(\Delta_i(p),P,\Delta_j(q))$ covers $(\Delta_{i-1}(p),P,\Delta_{j}(q))$.
\end{enumerate}
\end{lemma}

See the following example of the Hasse diagram for $\mathcal L(T^o,\widetilde\beta^{(p,q)})$, where $\Sigma$ is the twice punctured digon and $T^o$ is shown in the left below Figure.

\begin{figure}[h]

\centerline{\includegraphics{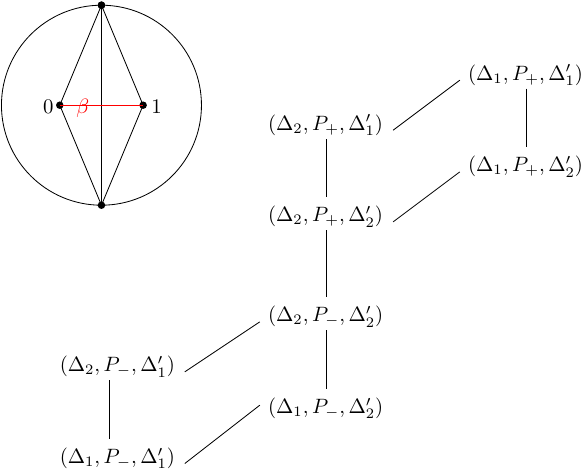}}

\caption{{\rm Hasse diagram of $\mathcal L(T^o,\widetilde\beta^{(0,1)})$ }}\label{Fig-Hasse1}
\end{figure}

\subsubsection{$\widetilde\beta\in T^o$} Herein we suppose that $\widetilde\beta\in  T^o$.

$\bullet$ If $s,t\geq 2$, then $\mathcal L(  T^o,\widetilde\beta^{(p,q)})$ forms a lattice with the following order.

\begin{enumerate}
\item  \begin{equation*}
 (\Delta_1(p),P_{\widetilde\beta},\Delta_t(q))<(\Delta_1(p),P_{\widetilde\beta},\Delta_{1}(q))<\cdots<(\Delta_1(p),P_{\widetilde\beta},\Delta_{t-2}(q))<(\Delta_1(p),P_{\widetilde\beta},\Delta_{t-1}(q)).
\end{equation*}
\item For $i=2,\cdots,s-1$,
\begin{equation*}
 (\Delta_i(p),P_{\widetilde\beta},\Delta_1(q))<(\Delta_i(p),P_{\widetilde\beta},\Delta_2(q))<\cdots<(\Delta_i(p),P_{\widetilde\beta},\Delta_{t-1}(q))<(\Delta_i(p),P_{\widetilde\beta},\Delta_t(q)).
\end{equation*}
\item
\begin{equation*}
 (\Delta_s(p),P_{\widetilde\beta},\Delta_2(q))<(\Delta_s(p),P_{\widetilde\beta},\Delta_{3}(q))<\cdots<(\Delta_s(p),P_{\widetilde\beta},\Delta_{t}(q))<(\Delta_s(p),P_{\widetilde\beta},\Delta_{1}(q)).
\end{equation*}
\item For $j=1,\cdots,t$,
\begin{equation*}
 (\Delta_1(p),P_{\widetilde\beta},\Delta_j(q))<(\Delta_2(p),P_{\widetilde\beta},\Delta_{j}(q))<\cdots<(\Delta_{s-1}(p),P_{\widetilde\beta},\Delta_{j}(q))<(\Delta_s(p),P_{\widetilde\beta},\Delta_{j}(q)).
\end{equation*}
\end{enumerate}

See the Hasse diagram of $\mathcal L(T^o,\widetilde\beta^{(p,q)})$ in Figure \ref{Fig-Hasse}.

\begin{figure}[h]

\centerline{\includegraphics{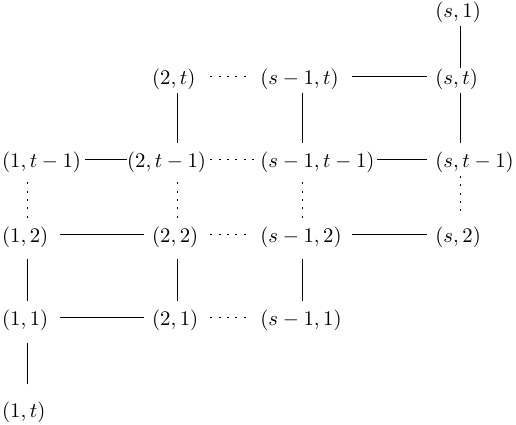}}

\caption{{\rm Hasse diagram of $\mathcal L(T^o,\widetilde\beta^{(p,q)})$ in case $\widetilde\beta\in   T^o$} and $s\geq 2$}\label{Fig-Hasse}
\end{figure}

$\bullet$ If $s=1$ or $t=1$, we may assume that $s=1$, then $\mathcal L(  T^o,\widetilde\beta^{(p,q)})$ forms a lattice with the following order.
\begin{equation*}
 (\Delta_1(p),P_{\widetilde\beta},\Delta_t(q))<(\Delta_1(p),P_{\widetilde\beta},\Delta_{1}(q))<\cdots<(\Delta_1(p),P_{\widetilde\beta},\Delta_{t-2}(q))<(\Delta_1(p),P_{\widetilde\beta},\Delta_{t-1}(q)).
\end{equation*}

The following lemma follows by Lemma \ref{lem-H2}.

\begin{lemma}\label{lem-H3}
The Hasse graph of $\mathcal L(T^o,\widetilde\beta^{(p,q)})$ is connected.
\end{lemma}

The following two lemmas are similar to Lemma \ref{lem:cop}.

\begin{lemma}\label{lem:cop1}
Assume that ${\bf P}=(\Delta_i(p),P,\Delta_{j}(q))$ covers ${\bf Q}=(\Delta_i(p),\mu_{G_l}P,\Delta_j(q))$ such that $\tau_{i_l}=\alpha_k\neq \alpha_{k-1},\alpha_{k+1}$ for some $k\in \{1,2,3,4\}$.
\begin{enumerate}[$(1)$]
\item In case $k\in \{1,3\}$, if $\tau_i(p)=\alpha$ but $(\Delta_{i+1}(p),P,\Delta_{j}(q))$ does not cover ${\bf P}$, then $G_l\neq G_1$ and $m_{1}({\bf P})=m_{1}({\bf Q})=1$. %In particular, $(0,1)$ is an ${\bf m}({\bf P};\alpha)$ and ${\bf m}({\bf Q};\alpha)$-pair.
\item In case $k\in \{1,3\}$, if $\tau_j(q)=\alpha$ but $(\Delta_i(p),P,\Delta_{j+1}(q))$ does not cover ${\bf P}$, then $G_l\neq G_c$ and $m_{\eta_\alpha}({\bf P})=m_{\eta_\alpha}({\bf Q})=1$.
    %In particular, $(\eta_\alpha,\eta_\alpha+1)$ is an ${\bf m}({\bf P};\alpha)$ and ${\bf m}({\bf Q};\alpha)$-pair.
\item In case $k\in \{2,4\}$, if $\tau_{i-1}(p)=\alpha$ but ${\bf P}$ does not cover $(\Delta_{i-1}(p),P,\Delta_{j}(q))$, then $G_l\neq G_1$ and $m_{1}({\bf P})=m_{1}({\bf Q})=1$.
    %In particular, $(0,1)$ is an ${\bf m}({\bf P};\alpha)$ and ${\bf m}({\bf Q};\alpha)$-pair.
\item In case $k\in \{2,4\}$, if $\tau_{j-1}(q)=\alpha$ but ${\bf P}$ does not cover $(\Delta_i(p),P,\Delta_{j-1}(q))$, then $G_l\neq G_c$ and $m_{\eta_\alpha}({\bf P})=m_{\eta_\alpha}({\bf Q})=1$.
    %In particular, $(\eta_\alpha,\eta_\alpha+1)$ is an ${\bf m}({\bf P};\alpha)$ and ${\bf m}({\bf Q};\alpha)$-pair.
\end{enumerate}
\end{lemma}

\begin{lemma}\label{lem:cop2}
Assume that ${\bf P}=(\Delta_{i+1}(p),P,\Delta_{j}(q))$ covers ${\bf Q}=(\Delta_{i}(p),P,\Delta_j(q))$ such that $\tau_{i}(p)=\alpha_k\neq \alpha_{k-1},\alpha_{k+1}$ for some $k\in \{1,2,3,4\}$.
\begin{enumerate}[$(1)$]
\item In case $k\in \{1,3\}$, if $\tau_j(q)=\alpha$ but $(\Delta_{i+1}(p),P,\Delta_{j+1}(q))$ does not cover ${\bf P}$, then $m_{\eta_\alpha}({\bf P})=m_{\eta_\alpha}({\bf Q})=1$.
    %In particular, $(\eta_\alpha,\eta_\alpha+1)$ is an ${\bf m}({\bf P};\alpha)$ and ${\bf m}({\bf Q};\alpha)$-pair.
\item In case $k\in \{2,4\}$, if $\tau_{j-1}(q)=\alpha$ but ${\bf P}$ does not cover $(\Delta_{i+1}(p),P,\Delta_{j-1}(q))$, then $m_{\eta_\alpha}({\bf P})=m_{\eta_\alpha}({\bf Q})=1$.
    %In particular, $(\eta_\alpha,\eta_\alpha+1)$ is an ${\bf m}({\bf P};\alpha)$ and ${\bf m}({\bf Q};\alpha)$-pair.
\end{enumerate}
\end{lemma}

\begin{lemma}\label{lem:cop3}
Assume that ${\bf P}=(\Delta_{i}(p),P,\Delta_{j+1}(q))$ covers ${\bf Q}=(\Delta_{i}(p),P,\Delta_j(q))$ such that $\tau_{j}(q)=\alpha_k\neq \alpha_{k-1},\alpha_{k+1}$ for some $k\in \{1,2,3,4\}$.
\begin{enumerate}[$(1)$]
\item In case $k\in \{1,3\}$, if $\tau_i(p)=\alpha$ but $(\Delta_{i+1}(p),P,\Delta_{j+1}(q))$ does not cover ${\bf P}$, then $m_{1}({\bf P})=m_{1}({\bf Q})=1$.
%In particular, $(0,1)$ is an ${\bf m}({\bf P};\alpha)$ and ${\bf m}({\bf Q};\alpha)$-pair.
\item In case $k\in \{2,4\}$, if $\tau_{i-1}(p)=\alpha$ but ${\bf P}$ does not cover $(\Delta_{i-1}(p),P,\Delta_{j+1}(q))$, then $m_{1}({\bf P})=m_{1}({\bf Q})=1$.
%In particular, $(0,1)$ is an ${\bf m}({\bf P};\alpha)$ and ${\bf m}({\bf Q};\alpha)$-pair.
\end{enumerate}
\end{lemma}

The following two lemmas can be proved similar to Lemma \ref{lem:diamond}.

\begin{lemma}\label{lem:diamond4}
Assume that ${\bf P}=(\Delta_{i+1}(p),P,\Delta_{j}(q))$ covers ${\bf Q}=(\Delta_i(p),P,\Delta_j(q))$ such that $\tau_{i}(p)=\alpha_k\neq \alpha_{k-1},\alpha_{k+1}$ for some $k\in \{1,2,3,4\}$.
\begin{enumerate}[$(1)$]
\item In case $k\in \{1,3\}$, if $(\Delta_{i+1}(p),P,\Delta_{j+1}(q))$ covers ${\bf P}$ with $\tau_j(q)=\alpha$, then $(\Delta_{i+1}(p),P,\Delta_{j+1}(q))$ covers $(\Delta_i(p),P,\Delta_{j+1}(q))$ and $(\Delta_i(p),P,\Delta_{j+1}(q))$ covers ${\bf Q}$.
\item In case $k\in \{2,4\}$, if ${\bf P}$ covers $(\Delta_{i+1}(p),P,\Delta_{j-1}(q))$ with $\tau_{j-1}(q)=\alpha$, then $(\Delta_{i}(p),P,\Delta_{j-1}(q))$ is covered by ${\bf Q}$ and $(\Delta_{i+1}(p),P,\Delta_{j-1}(q))$.
\end{enumerate}
\end{lemma}

\begin{lemma}\label{lem:diamond5}
Assume that ${\bf P}=(\Delta_{i}(p),P,\Delta_{j+1}(q))$ covers ${\bf Q}=(\Delta_i(p),P,\Delta_j(q))$ such that $\tau_{j}(q)=\alpha_k\neq \alpha_{k-1},\alpha_{k+1}$ for some $k\in \{1,2,3,4\}$.
\begin{enumerate}[$(1)$]
\item In case $k\in \{1,3\}$, if $(\Delta_{i+1}(p),P,\Delta_{j+1}(q))$ covers ${\bf P}$ with $\tau_i(p)=\alpha$, then $(\Delta_{i+1}(p),P,\Delta_{j+1}(q))$ covers $(\Delta_{i+1}(p),P,\Delta_{j}(q))$ and $(\Delta_{i+1}(p),P,\Delta_{j}(q))$ covers ${\bf Q}$.
\item In case $k\in \{2,4\}$, if ${\bf P}$ covers $(\Delta_{i-1}(p),P,\Delta_{j+1}(q))$ with $\tau_{i-1}(p)=\alpha$, then $(\Delta_{i-1}(p),P,\Delta_{j}(q))$ is covered by ${\bf Q}$ and $(\Delta_{i-1}(p),P,\Delta_{j+1}(q))$.
\end{enumerate}
\end{lemma}

The following two lemmas can be proved similar to Lemma \ref{lem:diamond1}.

\begin{lemma}\label{lem:diamond2}
Assume that ${\bf P}=(\Delta_i(p),P,\Delta_{j+1}(q))$ covers ${\bf Q}=(\Delta_i(p),P,\Delta_j(q))$ such that $\tau_{j}(q)=\alpha_k\neq \alpha_{k-1},\alpha_{k+1}$ for some $k\in \{1,2,3,4\}$. Assume that $P$ can twist on some $G_l$ with $\tau_{i_l}=\alpha$.
\begin{enumerate}[$(1)$]
\item In case $k\in \{1,3\}$, if $P<\mu_{G_l}P$, then $(\Delta_i(p),\mu_{G_l}P,\Delta_{j+1}(q))$ covers ${\bf P}$ and $(\Delta_i(p),\mu_{G_l}P,\Delta_j(q))$, $(\Delta_i(p),\mu_{G_l}P,\Delta_j(q))$ covers ${\bf Q}$.
\item In case $k\in \{2,4\}$, if $P>\mu_{G_l}P$, then ${\bf P}$ covers $(\Delta_i(p),\mu_{G_l}P,\Delta_{j+1}(q))$, and $(\Delta_i(p),\mu_{G_l}P,\Delta_{j}(q))$ is covered by ${\bf Q}$ and $(\Delta_i(p),\mu_{G_l}P,\Delta_{j+1}(q))$.
\end{enumerate}
\end{lemma}

\begin{lemma}\label{lem:diamond3}
Assume that ${\bf P}=(\Delta_{i+1}(p),P,\Delta_{j}(q))$ covers ${\bf Q}=(\Delta_i(p),P,\Delta_j(q))$ such that $\tau_{i}(p)=\alpha_k\neq \alpha_{k-1},\alpha_{k+1}$ for some $k\in \{1,2,3,4\}$. Assume that $P$ can twist on some $G_l$ with $\tau_{i_l}=\alpha$.
\begin{enumerate}[$(1)$]
\item In case $k\in \{1,3\}$, if $P<\mu_{G_l}P$, then $(\Delta_{i+1}(p),\mu_{G_l}P,\Delta_{j}(q))$ covers ${\bf P}$ and $(\Delta_i(p),\mu_{G_l}P,\Delta_j(q))$, $(\Delta_i(p),\mu_{G_l}P,\Delta_j(q))$ covers ${\bf Q}$.
\item In case $k\in \{2,4\}$, if $P>\mu_{G_l}P$, then ${\bf P}$ covers $(\Delta_{i+1}(p),\mu_{G_l}P,\Delta_{j}(q))$, and $(\Delta_i(p),\mu_{G_l}P,\Delta_{j}(q))$ is covered by ${\bf Q}$ and $(\Delta_i(p),\mu_{G_l}P,\Delta_{j+1}(q))$.
\end{enumerate}
\end{lemma}

\newpage

\section{Valuation maps}\label{sec:VM}

Let $T^o$ be an ideal triangulation and $\beta$ be an arc. Assume that $s(\widetilde\beta)=p, t(\widetilde \beta=q)$. In this section, for $\mathcal L=\mathcal L(  T^o,\widetilde\beta),\mathcal L(  T^o,\widetilde\beta^{(q)})$ or $\mathcal L(  T^o,\widetilde\beta^{(p,q)})$ we construct a map
$w:\mathcal L\to \mathbb Z$,
which will act as the quantum coefficients in the expansion formula for quantum cluster variables. To this end, we should introduce some integers.

Let $\alpha$ be an arc in $T^o$. For any perfect matching $P\in \mathcal P(G_{T^o,\widetilde\beta})$, denote
\begin{equation}\label{Equ-mp}
m(P;\alpha)=\text{number of edges labeled } \alpha \text{ in } P.
\end{equation}

For any snake subgraph $G$ of $G_{T^o,\widetilde\beta}$, denote
\begin{equation}\label{Equ-n}
n(G;\alpha)=\text{number of diagonals labeled } \alpha \text{ of } G.
\end{equation}

In the case that $\widetilde\beta\in T^o$, $G_{T^o,\widetilde\beta}$ is the graph with two vertices and one edge labeled $\widetilde\beta$ connecting them, it has a unique perfect matching $P_{\widetilde\beta}$. Set $m(P_{\widetilde\beta};\alpha)=1$ and $n(G_{T^o,\widetilde\beta};\alpha)=0$.

With the notation in Subsection \ref{delta1}, for any $j=1,\cdots,t$, denote
\begin{equation}\label{Eq-mj}
\begin{array}{rcl}
m(\Delta_j(q);\alpha)\hspace{-2mm}&= &\hspace{-2mm}
\text{number of edges labeled } \alpha \text{ in } \{\tau_{[j]}(q)\}\vspace{2mm}\\
\hspace{-2mm}&-&\hspace{-2mm}
\text{number of edges labeled } \alpha \text{ in } \{\tau_{j-1}(q),\tau_{j}(q)\}.
\end{array}
\end{equation}

Similarly, with the notation in Subsection \ref{delta2}, for any $i=1,\cdots,s$, denote
\begin{equation}\label{Eq-mi}
\begin{array}{rcl}
m(\Delta_i(p);\alpha)\hspace{-2mm}&= &\hspace{-2mm}
\text{number of edges labeled } \alpha \text{ in } \{\tau_{[i]}(p)\}\vspace{2mm}\\
\hspace{-2mm}&-&\hspace{-2mm}
\text{number of edges labeled } \alpha \text{ in } \{\tau_{i-1}(p),\tau_{i}(p)\}.
\end{array}
\end{equation}

For any $l=1,\cdots,c$, denote by $G^+_l$ (resp. $G^-_l$) the snake subgraph of $G_{  T^o,\beta}$ formed by the tiles $G_{l+1},\cdots,G_c$ (resp. $G_{1},\cdots,G_{l-1}$).

Moreover, if $P$ can twist on a tile $G_l$ for some $l=1,\cdots,c$, denote
\begin{equation*}
m^{\pm}(P,G_l;\alpha)=\text{number of edges labeled } \alpha \text{ in } P\cap (edge(G^{\pm}_l)\setminus edge(G_l)).
\end{equation*}

\subsection{Valuation map on $\mathcal L(  T^o,\widetilde\beta)$}

In this section, we construct a map $$w:\mathcal L(  T^o,\widetilde\beta)\to \mathbb Z.$$

\begin{definition}\label{Def-ome}
Let $P$ be a perfect matching of $G_{  T^o,\widetilde\beta}$. Suppose that $P$ can twist on a tile $G_l$ for some $l=1,\cdots,c$. Assume that the diagonal of $G_l$ is labeled $\tau_{i_l}$. The \emph{gradient number $\Omega(P;G_l)$ of $P$ at $G_l$} is defined to be
\begin{equation*}
\Omega(P;G_l)=d(\tau_{i_l})[\left(m^{+}(P,G_l;\tau_{i_l})-m^{-}(P,G_l;\tau_{i_l})\right)-\left(n(G^+_l;\tau_{i_l})-n(G^-_l;\tau_{i_l})\right)].
\end{equation*}
\end{definition}

\begin{lemma}\label{lem-gra}
Suppose that $P$ can twist on two tiles $G_l$ and $G_r$ with $r-l>1$ such that $\mu_{G_l}P,\mu_{G_r}P>P$. Then
$$\Omega(P;G_l)+\Omega(\mu_{G_l}P;G_r)=\Omega(P;G_r)+\Omega(\mu_{G_r}P;G_l).$$
\end{lemma}

\begin{proof}
We first consider the case that $b^{T^o}_{\tau_{i_l}\tau_{i_r}}=0$. Then we have
$$m^{\pm}(P,G_l;\tau_{i_l})=m^{\pm}(\mu_{G_r}P,G_l;\tau_{i_l}),\hspace{5mm} m^{\pm}(P,G_r;\tau_{i_r})=m^{\pm}(\mu_{G_l}P,G_r;\tau_{i_r}).$$
Thus the result follows in this case.

We then consider the case that $b^{T^o}_{\tau_{i_l}\tau_{i_r}}\neq 0$. We may assume that $b^{T^o}_{\tau_{i_l}\tau_{i_r}}>0$. As $\mu_{G_l}P,\mu_{G_r}P>P$, the edge labeled $\tau_{i_r}$ of $G_{i_l}$ is in $P$ and the edge labeled $\tau_{i_l}$ of $G_{i_r}$ is not in $P$. Therefore, we have
$$m^{+}(\mu_{G_r}P,G_l;\tau_{i_l})=m^{+}(P,G_l;\tau_{i_l})+b^{T^o}_{\tau_{i_l}\tau_{i_r}},\hspace{5mm} m^{-}(\mu_{G_r}P,G_l;\tau_{i_l})=m^{-}(P,G_l;\tau_{i_l}),$$
$$m^{-}(\mu_{G_l}P,G_r;\tau_{i_l})=m^{-}(P,G_r;\tau_{i_l})+b^{T^o}_{\tau_{i_r}\tau_{i_l}},\hspace{5mm} m^{+}(\mu_{G_l}P,G_r;\tau_{i_l})=m^{+}(P,G_r;\tau_{i_l}).$$
Then the result follows by $d(\tau_{i_l})b^{T^o}_{\tau_{i_l}\tau_{i_r}}=-d(\tau_{i_r})b^{T^o}_{\tau_{i_r}\tau_{i_l}}$, as $B(T^o)$ is skew-symmetrizable.
\end{proof}

\begin{proposition}\label{Prop-v1}
There is a unique map
$$w:\mathcal L(  T^o,\widetilde\beta)\to \mathbb Z$$ satisfies the following conditions:
\begin{enumerate}[$(1)$]
\item (Initial condition) $w(P_-)=0$, where $P_-$ is the minimum element in $\mathcal L(T^o,\widetilde\beta)$,
\item (Recurrence condition)
%\begin{equation*}
$w(\mu_{G_l}(P))-w(P)=\Omega(P;G_l)$
for any $P\in \mathcal L(  T^o,\beta)$ such that $P$ can twist on $G_l$ with $P<\mu_{G_l}(P)$.
\end{enumerate}

\end{proposition}

\begin{proof}
We shall only consider the case that $\widetilde\beta\notin T^o$. The uniqueness of $w$ is clear since the Hasse graph of $\mathcal P(G_{T^o,\widetilde\beta})$ is connected (by Lemma \ref{lem-H1}).

We now show the existence. For any $P\in\mathcal P(G_{T^o,\beta})$ and two chains $P_-=P_0<P_1<\cdots<P_a=P$, $P_-=P'_0<P'_1<\cdots<P'_{a'}=P$ such that $P_{i+1}$ covers $P_i$ and $P'_{i+1}$ covers $P'_i$ for any $i$, it suffices to prove that $w_1(P)=w_2(P)$, where $w_1(P)$ (resp. $ w_2(P)$) is obtained via the chain $P_-=P_0<P_1<\cdots<P_{a}=P$ (resp. $P_-=P'_0<P'_1<\cdots<P'_{a'}=P$). Suppose that $P\ominus P_-$ encloses the union of tiles $\bigcup_{j\in J} G_{j}$ for some $J\subseteq \{1,\cdots, c\}$. Denote $\Psi(P)=\prod_{j\in J}\chi_j$. It is clear that $\frac{\Psi(\mu_{G_l}Q)}{\Psi(Q)}=\chi_l$ for any $Q$ can twist on $G_l$ with $Q<\mu_{G_l}Q$. Therefore $a=a'=|J|$. We say that $a=|J|$ is the distance between $P$ and $P_-$.

Next, we show that $w_1(P)=w_2(P)$ by induction on the distance $a$ between $P$ and $P_-$. If $a=0$ then $P=P_-$ and thus $w_1(P)=w_2(P)=0$. Suppose that $P_{a-1}=\mu_{G_l}P$ and $P'_{a-1}=\mu_{G_r}P$. Assume that $w_1(P)=w_2(P)$ holds for all the cases that the distance is less than $a$. In particular, we have $w_1(P_{a-1})=w_2(P'_{a-1})$.

If $P_{a-1}=P'_{a-1}$ then we have
$$w_1(P)=w_1(P_{a-1})+\Omega(P_{a-1};G_l)=w_2(P_{a-1})+\Omega(P_{a-1};G_l)=w_2(P).$$

If $P_{a-1}\neq P'_{a-1}$, we have $|l-r|>1$ as $P_{a-1}, P'_{a-1}<P$. Thus $\mu_{G_l}\mu_{G_r}P=\mu_{G_r}\mu_{G_l}P$. Therefore,
\begin{equation*}
\begin{array}{rcl}
w_1(P)&=&w_1(P_{s-1})+\Omega(P_{a-1};G_l)\vspace{2mm}\\
\hspace{-2mm}&=&\hspace{-2mm}
\left(w_1(\mu_{G_r}\mu_{G_l}P)+\Omega(\mu_{G_r}\mu_{G_l}P;G_r)\right)+\Omega(\mu_{G_l}P;G_l)\vspace{2mm}\\
\hspace{-2mm}&=&\hspace{-2mm}
\left(w_2(\mu_{G_l}\mu_{G_r}P)+\Omega(\mu_{G_l}\mu_{G_r}P;G_l)\right)+\Omega(\mu_{G_r}P;G_r)\vspace{2mm}\\
\hspace{-2mm}&=&\hspace{-2mm}
w_2(\mu_{G_r}P)+\Omega(\mu_{G_r}P;G_r)\vspace{2mm}\\
\hspace{-2mm}&=&\hspace{-2mm}
w_2(P'_{a-1})+\Omega(\mu_{G_r}P;G_r)\vspace{2mm}\\
\hspace{-2mm}&=&\hspace{-2mm}
w_2(P),
\end{array}
\end{equation*}
where the third equality follows by induction hypothesis and Lemma \ref{lem-gra}. The result follows.
\end{proof}

\subsection{Valuation map on $\mathcal L(T^o,\widetilde\beta^{(q)})$}
Suppose that $q$ is a puncture and $q\neq p$. Assume that $T$ contains no arc tagged notched at $q$. In this section, we construct a map
$$w:\mathcal L(  T^o,\widetilde\beta^{(q)})\to \mathbb Z.$$

\begin{definition}\label{Def-gra1}
For $(P,\Delta_j(q))\in \mathcal L(  T^o,\widetilde\beta^{(q)})$ and $\zeta\in T^o$, the \emph{gradient number $\Omega(P,\underline{\Delta_j(q)};\zeta)$ of $(P,\Delta_j(q))$ at $\Delta_j(q)$} is defined to be
\begin{equation*}
\Omega(P,\underline{\Delta_j(q)};\zeta)=d(\zeta)\left(-m(P;\zeta)+n(G_{T^o,\widetilde\beta};\zeta)\right)=-d(\zeta)\hat m(P;\zeta).
\end{equation*}
\end{definition}

In particular, if $\zeta=\tau_j(q)$, we write $\Omega(P,\underline{\Delta_j(q)};\zeta)$ as $\Omega(P,\underline{\Delta_j(q)})$.

\begin{definition}\label{Def-gra2}
For $(P,\Delta_j(q))\in \mathcal L(T^o,\widetilde\beta^{(q)})$, assume that $P$ can twist on a tile $G_l$ with diagonal labeled $\tau_{i_l}$, the \emph{gradient number $\Omega(P,\Delta_j(q);G_l)$ of $(P,\Delta_j(q))$ at $G_l$} is defined to be
\begin{equation*}
\Omega(P,\Delta_j(q);G_l)=\Omega(P;G_l)+d(\tau_{i_l})m(\Delta_j(q);\tau_{i_l}).
\end{equation*}
\end{definition}

\begin{lemma}\label{Lem-det}
Suppose that $P$ can twist on two tiles $G_l$ and $G_r$ with $r-l>1$ such that $\mu_{G_l}P,\mu_{G_r}P>P$. Then for any $\Delta_j(q)\in \Delta(T^o,q)$ we have
$$\Omega(P,\Delta_j(q);G_l)+\Omega(\mu_{G_l}P,\Delta_j(q);G_r)=\Omega(P,\Delta_j(q);G_r)+\Omega(\mu_{G_r}P,\Delta_j(q);G_l).$$
\end{lemma}

\begin{proof}
It follows by Lemma \ref{lem-gra} and Definition \ref{Def-gra2}.
\end{proof}

\begin{lemma}\label{Lem-det1}
Suppose that $P$ can twist on a tile $G_l$ with $\mu_{G_l}P>P$. Then for any $\Delta_j(q)\in \Delta(T^o,q)$ we have
$$\Omega(P,\Delta_j(q);G_l)+\Omega(\mu_{G_l}P,\underline{\Delta_j(q)})=\Omega(P,\underline{\Delta_j(q)})+\Omega(P,\Delta_{j+1}(q);G_l).$$
\end{lemma}

\begin{proof}
By definition, we have
\begin{equation*}
\begin{array}{rcl}
\Omega(P,\Delta_j(q);G_l)+\Omega(\mu_{G_l}P,\underline{\Delta_j(q)})\hspace{-2mm}&=&\hspace{-2mm} \Omega(P;G_l)+d(\tau_{i_l})m(\Delta_j(q);\tau_{i_l})\vspace{2mm}\\
\hspace{-2mm}&+&\hspace{-2mm}
d(\tau_{j}(q))(-m(\mu_{G_l}(P);\tau_j(q))+n(G_{T^o,\widetilde\beta};\tau_j(q))),
\end{array}
\end{equation*}
\begin{equation*}
\begin{array}{rcl}
\Omega(P,\underline{\Delta_j(q)})+\Omega(P,\Delta_{j+1}(q);G_l)\hspace{-2mm}&=&\hspace{-2mm} d(\tau_{j}(q))(-m(P;\tau_j(q))+n(G_{T^o,\widetilde\beta};\tau_j(q)))\vspace{2mm}\\
\hspace{-2mm}&+&\hspace{-2mm}
 \Omega(P;G_l)+d(\tau_{i_l})m(\Delta_{j+1}(q);\tau_{i_l}),
\end{array}
\end{equation*}

We first consider the case that $b^{T^o}_{\tau_{i_l},\tau_{j}(q)}=0$. Then we have
$$m(P;\tau_{j}(q))=m(\mu_{G_l}P;\tau_{j}(q)),\hspace{5mm} m(\Delta_j(q);\tau_{i_l})=m(\Delta_{j+1}(q);\tau_{i_l}).$$
Thus the result follows in this case.

We then consider the case that $b^{T^o}_{\tau_{i_l},\tau_{j}(q)}\neq 0$. We may assume that $b^{T^o}_{\tau_{i_l},\tau_{j}(q)}>0$. As $\mu_{G_l}P>P$, the edge labeled $\tau_{j}(q)$ of $G_{i_l}$ is in $P$. Therefore, we have
$$m(P;\tau_{j}(q))=m(\mu_{G_l}P;\tau_{j}(q))-b^{T^o}_{\tau_j(q),\tau_{i_l}},\hspace{5mm} m(\Delta_j(q);\tau_{i_l})=m(\Delta_{j+1}(q);\tau_{i_l})-b^{T^o}_{\tau_{i_l},\tau_{j}(q)}.$$
Then the result follows by $d(\tau_{i_l})b^{T^o}_{\tau_{i_l}\tau_{j}(q)}=-d(\tau_{j}(q))b^{T^o}_{\tau_{j}(q)\tau_{i_l}}$, as $B(T^o)$ is skew-symmetrizable.

The proof is complete.
\end{proof}

\begin{proposition}\label{Pro-vm}
Assume that $q$ is a puncture and $q\neq p$. Assume that $T$ contains no arc tagged notched at $q$. Then there is a unique map
$$w:\mathcal L(T^o,\widetilde\beta^{(q)})\to \mathbb Z$$ satisfies the following conditions:
\begin{enumerate}[$(1)$]
\item (Initial condition) $w({\bf P}_-)=0$, where ${\bf P}_-=(P_-,\Delta_1(q))$ is the minimum element in $\mathcal L(  T^o,\widetilde\beta^{(q)})$,
\item (Recurrence condition)
\begin{enumerate}
\item
For any $(P,\Delta_j(q))$ such that $(P,\Delta_{j+1}(q))$ covers $(P,\Delta_j(q))$,
\begin{equation*}
w(P,\Delta_{j+1}(q))-w(P,\Delta_j(q))=\Omega(P,\underline{\Delta_j(q)}).
\end{equation*}
\item
For any $(P,\Delta_j(q)), (Q,\Delta_j(q))\in \mathcal L(T^o,\beta^{(q)})$ such that $(Q,\Delta_j(q))$ covers $(P,\Delta_j(q))$, in particular $P<Q$ are related by a twist on some tile $G_l$,
\begin{equation*}
w(Q,\Delta_j(q))-w(P,\Delta_j(q))=\Omega(P,\Delta_j(q);G_l).
\end{equation*}
\end{enumerate}
\end{enumerate}
\end{proposition}

\begin{proof}
The result holds for the case $\widetilde\beta\in T^o$ as the Hasse graph of $\mathcal L(T^o,\widetilde\beta^{(q)})$ does not have any un-oriented cycles in this case.

We now consider the case that $\widetilde\beta\notin T^o$. The uniqueness holds as the Hasse graph of $\mathcal L(T^o,\widetilde\beta^{(q)})$ is connected (By Lemma
\ref{lem-H2}). We now show the existence. For any $\textbf{P}=(P,\Delta_j(q))\in \mathcal L(T^o,\widetilde\beta^{(q)})$, for any two chains $\textbf{P}_0=\textbf{P}_-<\textbf {P}_1<\cdots<\textbf P_a=\textbf P$ and $\textbf{P}'_0=\textbf{P}_-<\textbf P'_1<\cdots<\textbf P'_{a'}=\textbf P$ such that $\textbf P_{i+1}$ covers $\textbf P_i$ and $\textbf P'_{i+1}$ covers $\textbf P'_i$ for all $i$, it suffices to show that $w_1(\textbf P)=w_2(\textbf P)$, where $w_1(\textbf P)$ (resp. $w_1(\textbf P)$) is determined by the chain $\textbf{P}_0=\textbf{P}_-<\textbf {P}_1<\cdots<\textbf P_a=\textbf P$ (resp. $\textbf{P}'_0=\textbf{P}_-<\textbf P'_1<\cdots<\textbf P'_{a'}=\textbf P$).

Suppose that $P\ominus P_-$ encloses the union of tiles $\bigcup_{i\in J} G_{i}$ for some $J\subseteq \{1,\cdots, c\}$. Denote
   \begin{equation*}
        \Psi(\textbf P)=
        \begin{cases}
        \prod_{i\in J}\chi_i,  &\mbox{if $E_1(q)\in P$ and $j=1$},\\
        \prod_{i\in J}\chi_i(\kappa_{1}\kappa_{2}\cdots \kappa_{t-1}\kappa_{t}),  &\mbox{if $E_2(q)\in P$ and $j=1$},\\
        \prod_{i\in J}\chi_i(\kappa_{1}\kappa_{2}\cdots \kappa_{j-1}),  &\mbox{otherwise},
        \end{cases}
    \end{equation*}
where $E_1(q)$ and $E_2(q)$ are the edges of $G_c$ determined by (\ref{E1q}) in Section \ref{Sec-threesets}.
Thus $\Psi(\mu_{G_l}Q,\Delta)=\Psi(Q,\Delta)\chi_l$ if $(\mu_{G_l}Q,\Delta)$ covers $(Q,\Delta)$ and $\Psi(Q,\Delta_{l+1}(q))=\Psi(Q,\Delta_l(q))\kappa_l$ if $(Q,\Delta_{l+1}(q))$ covers $(Q,\Delta_l(q))$. It follows that $a=a'$ equals the exponent of $\Psi(\textbf P)$. We say that $a=a'$ is the distance between $\textbf P_-$ and $\textbf P$. We prove that $w_1(\textbf P)=w_2(\textbf P)$ by induction on the distance $a$.

In case $a=0$ we have $\textbf P=\textbf P_-$, it follows that $w_1(\textbf P)=w_2(\textbf P)=0$. Assume that $w_1(\textbf P)=w_2(\textbf P)$ holds for any case that the distance is less than $a$.

If $\textbf P_{a-1}=\textbf P'_{a-1}$ then $$w_1(\textbf P)-w_1(\textbf P_{a-1})=w_2(\textbf P)-w_2(\textbf P_{a-1}).$$
Thus by hypothesis, we have
$$w_1(\textbf P)=w_2(\textbf P).$$

If $\textbf P_{a-1}\neq \textbf P'_{a-1}$, we have either $\textbf P_{a-1}=(\mu_{G_{l}}P,\Delta_j(q)), \textbf P'_{a-1}=(\mu_{G_{r}}P,\Delta_j(q))$ for some $l,r$ with $P$ can twist on $G_l$ and $G_r$ or $\textbf P_{a-1}=(\mu_{G_{l}}P,\Delta_j(q)), \textbf P'_{a-1}=(P,\Delta_{j-1}(q))$ for some $l$.

(1)~ If $\textbf P_{a-1}=(\mu_{G_{l}}P,\Delta_j(q)), \textbf P'_{a-1}=(\mu_{G_{r}}P,\Delta_j(q))$, as $\textbf P_{a-1}, \textbf P'_{a-1}<\textbf P$, we have that $\mu_{G_l}P, \mu_{G_r}P<P$ and $|l-r|>1$. Consequently, $\mu_{G_l}\mu_{G_r}P=\mu_{G_r}\mu_{G_l}P$, moreover, as $\textbf P$ covers $\textbf P_{s-1}$, $l=c$ and $j=1$ can not hold simultaneously by Lemma \ref{Lem-cover}. Similarly, $r=c$ and $j=1$ can not hold simultaneously. Thus by Lemma \ref{Lem-cover}, we have both $\textbf P_{a-1}$ and $\textbf P'_{a-1}$ cover $(\mu_{G_r}\mu_{G_l}P,\Delta_j(q))$. Therefore,
\begin{equation*}
\begin{array}{rcl}
w_1(\textbf P)&=&w_1(\textbf P_{a-1})+\Omega(\mu_{G_l}P,\Delta_j(q);G_l)\vspace{2mm}\\
\hspace{-2mm}&=&\hspace{-2mm}
\left(w_1(\mu_{G_r}\mu_{G_l}P,\Delta_j(q))+\Omega(\mu_{G_r}\mu_{G_l}P,\Delta_j(q);G_r)\right)+\Omega(\mu_{G_l}P,\Delta_j(q);G_l)
\vspace{2mm}\\
\hspace{-2mm}&=&\hspace{-2mm}
w_2(\mu_{G_l}\mu_{G_r}P,\Delta_j(q))+\left(\Omega(\mu_{G_l}\mu_{G_r}P,\Delta_j(q);G_l)+\Omega(\mu_{G_r}P,\Delta_j(q);G_r)\right)\vspace{2mm}\\
\hspace{-2mm}&=&\hspace{-2mm}
w_2(\mu_{G_r}P,\Delta_j(q))+\Omega(\mu_{G_r}P,\Delta_j(q);G_r)\vspace{2mm}\\
\hspace{-2mm}&=&\hspace{-2mm}
w_2(\textbf P'_{a-1})+\Omega(\mu_{G_r}P,\Delta_j(q);G_r)\vspace{2mm}\\
\hspace{-2mm}&=&\hspace{-2mm}
w_2(\textbf P),
\end{array}
\end{equation*}
where the third equality follows by the induction hypothesis and Lemma \ref{Lem-det}.

(2)~ If $\textbf P_{a-1}=(\mu_{G_{l}}P,\Delta_j(q)), \textbf P'_{a-1}=(P,\Delta_{j-1}(q))$, we see that $l=c$ and $j-1=1$ can not hold simultaneously, since otherwise $\textbf P=(P,\Delta_j(q))<\textbf P'_{a-1}=(P,\Delta_{j-1}(q))$. Thus, $\textbf P'_{a-1}=(P,\Delta_{j-1}(q))$ covers $(\mu_{G_l}P,\Delta_{j-1}(q))$ by Lemma \ref{Lem-cover}. We also see that $l=c$ and $j=1$ can not hold simultaneously, since otherwise $\textbf P=(P,\Delta_j(q))$ does not cover $\textbf P_{a-1}=(\mu_{G_l}P,\Delta_{j}(q))$. As $\textbf P=(P,\Delta_{j}(q))$ covers $\textbf P'_{a-1}=(P,\Delta_{j-1}(q))$, by Lemma \ref{Lem-cover2}, we thus have $\textbf P_{a-1}=(\mu_{G_l}P,\Delta_{j}(q))$ covers $(\mu_{G_l}P,\Delta_{j-1}(q))$.

 Therefore,
\begin{equation*}
\begin{array}{rcl}
w_1(\textbf P)&=&w_1(\textbf P_{a-1})+\Omega(\mu_{G_l}P,\Delta_j(q);G_l)\vspace{2mm}\\
\hspace{-2mm}&=&\hspace{-2mm}
\left(w_1(\mu_{G_l}P,\Delta_{j-1}(q))+\Omega(\mu_{G_l}P,\underline{\Delta_{j-1}(q)})\right)+\Omega(\mu_{G_l}P,\Delta_j(q);G_l)
\vspace{2mm}\\
\hspace{-2mm}&=&\hspace{-2mm}
w_2(\mu_{G_l}P,\Delta_{j-1}(q))+\Omega(\mu_{G_l}P,\Delta_{j-1}(q);G_l)+\Omega(P,\underline{\Delta_{j-1}(q)})\vspace{2mm}\\
\hspace{-2mm}&=&\hspace{-2mm}
w_2(P,\Delta_{j-1}(q))+\Omega(P,\underline{\Delta_{j-1}(q)})\vspace{2mm}\\
\hspace{-2mm}&=&\hspace{-2mm}
w_2(\textbf P'_{a-1})+\Omega(P,\underline{\Delta_{j-1}(q)})\vspace{2mm}\\
\hspace{-2mm}&=&\hspace{-2mm}
w_2(\textbf P),
\end{array}
\end{equation*}
where the third equality follows by induction hypothesis and Lemma \ref{Lem-det1}.

The proof is complete.
\end{proof}

\subsection{Valuation map on $\mathcal L(T^o,\widetilde\beta^{(p,q)})$}

Suppose that $p$ and $q$ are punctures. Assume that $T$ contains no arcs tagged notched at $p$ or $q$. In this section, we construct a map
$$w:\mathcal L(T^o,\widetilde\beta^{(p,q)})\to \mathbb Z.$$

\begin{definition}\label{Def-gra5}
For any $(\Delta_i(p),P,\Delta_j(q))\in \mathcal L(T^o,\widetilde\beta^{(p,q)})$ and $\zeta\in T^o$, the \emph{gradient number $\Omega(\underline{\Delta_i(p)},P,\Delta_j(q);\zeta)$ of $(\Delta_i(p),P,\Delta_j(q))$ at $\Delta_i(p)$} is defined to be
\begin{equation*}
\Omega(\underline{\Delta_i(p)},P,\Delta_j(q);\zeta)=d(\zeta)[m(P;\zeta)-n(G_{T^o,\widetilde\beta};\zeta)+m(\Delta_j(q);\zeta)].
\end{equation*}
\end{definition}

In particular, if $\zeta=\tau_i(p)$, we write $\Omega(\underline{\Delta_i(p)},P,\Delta_j(q);\zeta)$ as $\Omega(\underline{\Delta_i(p)},P,\Delta_j(q))$.

\begin{definition}\label{Def-gra4}
For any $(\Delta_i(p),P,\Delta_j(q))\in \mathcal L(T^o,\widetilde\beta^{(p,q)})$ and $\zeta\in T^o$, the \emph{gradient number $\Omega(\Delta_i(p),P,\underline{\Delta_j(q)};\zeta)$ of $(\Delta_i(p),P,\Delta_j(q))$ at $\Delta_j(q)$} is defined to be
\begin{equation*}
\Omega(\Delta_i(p),P,\underline{\Delta_j(q)};\zeta)=-d(\zeta)[m(P;\zeta)-n(G_{T^o,\widetilde\beta};\zeta)+m(\Delta_i(p);\zeta)].
\end{equation*}
\end{definition}

In particular, if $\zeta=\tau_j(q)$, we write $\Omega(\Delta_i(p),P,\underline{\Delta_j(q)};\zeta)$ as $\Omega(\Delta_i(p),P,\underline{\Delta_j(q)})$.

\begin{definition}\label{Def-gra3}
Assume that $(\Delta_i(p),P,\Delta_j(q))\in \mathcal L(T^o,\widetilde\beta^{(p,q)})$ and $P$ can twist on a tile $G_l$ for some $l\in \{1,\cdots,c\}$. Assume that the diagonal of $G_l$ is labeled $\tau_{i_l}$. The \emph{gradient number $\Omega(\Delta_i(p),P,\Delta_j(q);G_l)$ of $(\Delta_i(p),P,\Delta_j(q))$ at $G_l$} is defined to be
\begin{equation*}
\Omega(\Delta_i(p),P,\Delta_j(q);G_l)=\Omega(P;G_l)+d(\tau_{i_l})[m(\Delta_j(q);\tau_{i_l})-m(\Delta_i(p);\tau_{i_l})].
\end{equation*}
\end{definition}

\begin{lemma}\label{Lem-det2}
Suppose that $P$ can twist on two tiles $G_l$ and $G_r$ with $r-l>1$ such that $\mu_{G_l}P,\mu_{G_r}P>P$. Then for any $\Delta_i(p)\in \Delta(T^o,p)$, $\Delta_j(q)\in \Delta(T^o,q)$ we have
\begin{equation*}
\begin{array}{rcl}
&&
\Omega(\Delta_i(p),P,\Delta_j(q);G_l)+\Omega(\Delta_i(p),\mu_{G_l}P,\Delta_j(q);G_r)\vspace{2mm}\\
\hspace{-2mm}&=&\hspace{-2mm}
\Omega(\Delta_i(p),P,\Delta_j(q);G_r)+\Omega(\Delta_i(p),\mu_{G_r}P,\Delta_j(q);G_l).
\end{array}
\end{equation*}
\end{lemma}

\begin{proof}
It follows by Lemma \ref{lem-gra} and Definition \ref{Def-gra3}.
\end{proof}

\begin{lemma}\label{Lem-det3}
Suppose that $P$ can twist on a tile $G_l$ such that $\mu_{G_l}P>P$. Then for any $\Delta_i(p)\in \Delta(T^o,p)$, $\Delta_j(q)\in \Delta(T^o,q)$ we have
\begin{enumerate}[$(1)$]
\item \begin{equation*}
\begin{array}{rcl}
&&
\Omega(\Delta_i(p),P,\Delta_j(q);G_l)+\Omega(\Delta_i(p),\mu_{G_l}P,\underline{\Delta_j(q)})\vspace{2mm}\\
\hspace{-2mm}&=&\hspace{-2mm}
\Omega(\Delta_i(p),P,\underline{\Delta_j(q)})+\Omega(\Delta_i(p),P,\Delta_{j+1}(q);G_l).
\end{array}
\end{equation*}
\item \begin{equation*}
\begin{array}{rcl}
&&
\Omega(\Delta_i(p),P,\Delta_j(q);G_l)+\Omega(\underline{\Delta_i(p)},\mu_{G_l}P,\Delta_j(q))\vspace{2mm}\\
\hspace{-2mm}&=&\hspace{-2mm}
\Omega(\underline{\Delta_i(p)},P,\Delta_j(q))+\Omega(\Delta_{i+1}(p),P,\Delta_{j}(q);G_l).
\end{array}
\end{equation*}
\end{enumerate}
\end{lemma}

\begin{proof}
We shall only prove (1) as the proof of (2) is similar.

By definition, we have
\begin{equation*}
\begin{array}{rcl}
&& \Omega(\Delta_i(p),P,\Delta_j(q);G_l)+\Omega(\Delta_i(p),\mu_{G_l}P,\underline{\Delta_j(q)})\vspace{2mm}\\
\hspace{-2mm}&=&\hspace{-2mm} \Omega(P;G_l)+d(\tau_{i_l})[m(\Delta_j(q);\tau_{i_l})-m(\Delta_i(p);\tau_{i_l})]\vspace{2mm}\\
\hspace{-2mm}&-&\hspace{-2mm}
d(\tau_j(q))[m(\mu_{G_l}P;\tau_j(q))-n(G_{T^o,\widetilde\beta};\tau_j(q))+m(\Delta_i(p);\tau_j(q))]
\end{array}
\end{equation*}
and
\begin{equation*}
\begin{array}{rcl}
&& \Omega(\Delta_i(p),P,\underline{\Delta_j(q)})+\Omega(\Delta_i(p),P,\Delta_{j+1}(q);G_l)\vspace{2mm}\\
\hspace{-2mm}&=&\hspace{-2mm}
-d(\tau_j(q))[m(P;\tau_j(q))-n(G_{T^o,\widetilde\beta};\tau_j(q))+m(\Delta_i(p);\tau_j(q))]\vspace{2mm}\\
\hspace{-2mm}&+&\hspace{-2mm}
\Omega(P;G_l)+d(\tau_{i_l})[m(\Delta_{j+1}(q);\tau_{i_l})-m(\Delta_i(p);\tau_{i_l})].
\end{array}
\end{equation*}

We first consider the case that $b^{T^o}_{\tau_{i_l},\tau_{j}(q)}=0$. Then we have
$$m(P;\tau_{j}(q))=m(\mu_{G_l}P;\tau_{j}(q)),\hspace{5mm} m(\Delta_j(q);\tau_{i_l})=m(\Delta_{j+1}(q);\tau_{i_l}).$$
Thus the result follows in this case.

We then consider the case that $b^{T^o}_{\tau_{i_l},\tau_{j}(q)}\neq 0$. We may assume that $b^{T^o}_{\tau_{i_l},\tau_{j}(q)}>0$. As $\mu_{G_l}P>P$, the edge labeled $\tau_{j}(q)$ of $G_{i_l}$ is in $P$. Therefore, we have
$$m(P;\tau_{j}(q))=m(\mu_{G_l}P;\tau_{j}(q))-b^{T^o}_{\tau_j(q),\tau_{i_l}},\hspace{5mm} m(\Delta_j(q);\tau_{i_l})=m(\Delta_{j+1}(q);\tau_{i_l})-b^{T^o}_{\tau_{i_l},\tau_{j}(q)}.$$
Then the result follows by $d(\tau_{i_l})b^{T^o}_{\tau_{i_l},\tau_j(q)}=-d(\tau_{j}(q))b^{T^o}_{\tau_{j}(q),\tau_{i_l}}$, as $B(T^o)$ is skew-symmetrizable.

The proof is complete.
\end{proof}

\begin{lemma}\label{Lem-det4}
For any $(\Delta_i(p),P,\Delta_j(q))\in \mathcal L(T^o,\widetilde\beta^{(p,q)})$ we have
\begin{equation*}
\begin{array}{rcl}
&&
\Omega(\underline{\Delta_i(p)},P,\Delta_j(q))+\Omega(\Delta_{i+1}(p),P,\underline{\Delta_j(q)})\vspace{2mm}\\
\hspace{-2mm}&=&\hspace{-2mm}
\Omega(\Delta_i(p),P,\underline{\Delta_j(q)})+\Omega(\underline{\Delta_i(p)},P,\Delta_{j+1}(q)).
\end{array}
\end{equation*}
\end{lemma}

\begin{proof}
By definition, we have
\begin{equation*}
\begin{array}{rcl}
&& \Omega(\underline{\Delta_i(p)},P,\Delta_j(q))+\Omega(\Delta_{i+1}(p),P,\underline{\Delta_j(q)})\vspace{2mm}\\
\hspace{-2mm}&=&\hspace{-2mm}
d(\tau_i(p))[m(P;\tau_i(p))-n(G_{T^o,\widetilde\beta};\tau_i(p))+m(\Delta_j(q);\tau_i(p))]\vspace{2mm}\\
\hspace{-2mm}&-&\hspace{-2mm}
d(\tau_j(q))[m(P;\tau_j(q))-n(G_{T^o,\widetilde\beta};\tau_j(q))+m(\Delta_{i+1}(p);\tau_j(q))],
\end{array}
\end{equation*}

\begin{equation*}
\begin{array}{rcl}
&& \Omega(\Delta_i(p),P,\underline{\Delta_j(q)})+\Omega(\underline{\Delta_i(p)},P,\Delta_{j+1}(q))\vspace{2mm}\\
\hspace{-2mm}&=&\hspace{-2mm}
-d(\tau_j(q))[m(P;\tau_j(p))-n(G_{T^o,\widetilde\beta};\tau_j(q))+m(\Delta_i(p);\tau_j(q))]\vspace{2mm}\\
\hspace{-2mm}&+&\hspace{-2mm}
d(\tau_i(p))[m(P;\tau_i(p))-n(G_{T^o,\widetilde\beta};\tau_i(p))+m(\Delta_{j+1}(q);\tau_i(p))],
\end{array}
\end{equation*}

We first consider the case that $b^{T^o}_{\tau_{i}(p),\tau_{j}(q)}=0$. Then we have
$$m(\Delta_i(p);\tau_{j}(q))=m(\Delta_{i+1}(p);\tau_{j}(q)),\hspace{5mm} m(\Delta_j(q);\tau_{i}(p))=m(\Delta_{j+1}(q);\tau_{i}(p)).$$
Thus the result follows in this case.

We then consider the case that $b^{T^o}_{\tau_{i}(p),\tau_{j}(q)}\neq 0$. We may assume that $b^{T^o}_{\tau_{i}(p),\tau_{j}(q)}>0$. Then we have
$$ m(\Delta_i(p);\tau_{j}(q))=m(\Delta_{i+1}(p);\tau_{j}(q))-b^{T^o}_{\tau_{j}(q),\tau_{i}(p)},\hspace{2mm} m(\Delta_j(q);\tau_{i}(p))=m(\Delta_{j+1}(q);\tau_{i}(p))-b^{T^o}_{\tau_{i}(p),\tau_{j}(q)}.$$
Then the result follows by $d(\tau_{i}(p))b^{T^o}_{\tau_{i}(p),\tau_{j}(q)}=-d(\tau_{j}(q))b^{T^o}_{\tau_{j}(q),\tau_{i}(p)}$, as $B(T^o)$ is skew-symmetrizable.

The proof is complete.
\end{proof}

\begin{proposition}\label{pro-vpq}
With the foregoing notation. Suppose that $p$ and $q$ are punctures and $T$ contains no arcs tagged notched at $p$ or $q$. There is a unique map
$$w:\mathcal L(T^o,\widetilde\beta^{(p,q)})\to \mathbb Z$$
satisfies the following conditions:

\hspace{-1cm}(1) (Initial condition) $w({\bf P}_-)=0$, where ${\bf P}_-$ is the minimum element in $\mathcal L(T^o,\widetilde\beta^{(p,q)})$,

\hspace{-1cm}(2) (Recurrence condition)

\hspace{-1cm}(a) For any $(\Delta_i(p),P,\Delta_j(q))$ such that $(\Delta_{i+1}(p),P,\Delta_{j}(q))$ covers $(\Delta_i(p),P,\Delta_j(q))$,
\begin{equation*}
w(\Delta_{i+1}(p),P,\Delta_{j}(q))-w(\Delta_i(p),P,\Delta_j(q))=\Omega(\underline{\Delta_i(p)},P,\Delta_j(q)).
\end{equation*}
\hspace{-0.5cm}(b) For any $(\Delta_i(p),P,\Delta_j(q))$ such that $(\Delta_i(p),P,\Delta_{j+1}(q))$ covers $(\Delta_i(p),P,\Delta_j(q))$,
\begin{equation*}
w(\Delta_i(p),P,\Delta_{j+1}(q))-w(\Delta_i(p),P,\Delta_j(q))=\Omega(\Delta_i(p),P,\underline{\Delta_j(q)}).
\end{equation*}
\hspace{-0.5cm}(c) For any $(\Delta_i(p),Q,\Delta_j(q)),(\Delta_i(p),P,\Delta_j(q))\in \mathcal L(T^o,\widetilde\beta^{(p,q)})$ such that $(\Delta_i(p),Q,\Delta_j(q))$ covers $(\Delta_i(p),P,\Delta_j(q))$, in particular, $Q>P$ are related by a twist on some tile $G_l$,
\begin{equation*}
w(\Delta_i(p),Q,\Delta_j(q))-w(\Delta_i(p),P,\Delta_j(q))=\Omega(\Delta_i(p),P,\Delta_j(q);G_l).
\end{equation*}
\end{proposition}

\begin{proof}
The uniqueness holds as the Hasse graph of $\mathcal L(T^o,\widetilde\beta^{(p,q)})$ is connected (by Lemma \ref{lem-H3}). We now show the existence. For any $\textbf{P}=(\Delta_i(p),P,\Delta_j(q))\in \mathcal L(T^o,\widetilde\beta^{(p,q)})$, for any two chains $\textbf{P}_0=\textbf{P}_-<\textbf {P}_1<\cdots<\textbf P_a=\textbf P$ and $\textbf{P}'_0=\textbf{P}_-<\textbf P'_1<\cdots<\textbf P'_{a'}=\textbf P$ such that $\textbf P_{i+1}$ covers $\textbf P_i$ and $\textbf P'_{i+1}$ covers $\textbf P'_i$ for all $i$, it suffices to show that $w_1(\textbf P)=w_2(\textbf P)$, where $w_1(\textbf P)$ (resp. $w_2(\textbf P)$) is determined by the chain $\textbf{P}_0=\textbf{P}_-<\textbf {P}_1<\cdots<\textbf P_a=\textbf P$ (resp. $\textbf{P}'_0=\textbf{P}_-<\textbf P'_1<\cdots<\textbf P'_{a'}=\textbf P$).

If $\widetilde\beta\notin T^o$, denote
   \begin{equation*}
        \Psi(\Delta_i(p),P,\Delta_j(q))=
        \begin{cases}
        \Psi(P,\Delta_j(q)), &\mbox{if $E_1(p)\in P$ and $i=1$},\\
        \Psi(P,\Delta_j(q))\vartheta_{1}\vartheta_{2}\cdots \vartheta_{s-1}\vartheta_{s},  &\mbox{if $E_2(p)\in P$ and $i=1$},\\
        \Psi(P,\Delta_j(q))\vartheta_{1}\vartheta_{2}\cdots \vartheta_{i-1},  &\mbox{otherwise},
        \end{cases}
    \end{equation*}
    where $E_1(p)$ and $E_2(p)$ are the edges of $G_1$ determined by (\ref{E1p}) in Section \ref{Sec-threesets} and $\Psi(P,\Delta_j(q))$ is given in the proof of Proposition \ref{Pro-vm}.

If $\widetilde\beta\in T^o$, denote
  \begin{equation}\label{Eq-hpq}
        \Psi(\Delta_i(p),P_{\widetilde\beta},\Delta_j(q))=
        \begin{cases}
        1, &\mbox{if $i=1$ and $j=t$},\\
        \vartheta_{1}\vartheta_{2}\cdots \vartheta_{s}\kappa_1\cdots \kappa_t,  &\mbox{if $i=s$ and $j=1$},\\
        \vartheta_{1}\vartheta_{2}\cdots \vartheta_{i-1}\kappa_1\cdots \kappa_{j-1}\kappa_t,  &\mbox{otherwise}.
        \end{cases}
    \end{equation}

Thus, $\Psi(\Delta,\mu_{G_l}Q,\Delta')=\Psi(\Delta,Q,\Delta')\chi_l$ if $(\Delta,\mu_{G_l}Q,\Delta')$ covers $(\Delta,Q,\Delta')$, $\Psi(\Delta,Q,\Delta_{l+1}(q))=\Psi(\Delta,Q,\Delta_l(q))\kappa_l$ if $(\Delta,Q,\Delta_{l+1}(q))$ covers $(\Delta, Q,\Delta_l(q))$, and $\Psi(\Delta_{l+1}(p),Q,\Delta')=\Psi(\Delta_l(p),Q,\Delta')\vartheta_l$ if $(\Delta_{l+1}(p),Q,\Delta')$ covers $(\Delta_l(p), Q,\Delta')$. It follows that $a=a'$ equals the exponent of $\Psi(\textbf P)$. We say that $a=a'$ is the distance between $\textbf P_-$ and $\textbf P$. We prove that $w_1(\textbf P)=w_2(\textbf P)$ by induction on the distance $a$.

In case $a=0$ we have $\textbf P=\textbf P_-$. It follows that $w_1(\textbf P)=w_2(\textbf P)=0$ in this case. Assume that $w_1(\textbf P)=w_2(\textbf P)$ holds for any case that the distance is less than $a$.

If $\textbf P_{a-1}=\textbf P'_{a-1}$ then $w_1(\textbf P)-w_1(\textbf P_{a-1})=w_2(\textbf P)-w_2(\textbf P_{a-1}).$
Thus by the hypothesis we have
$w_1(\textbf P)=w_2(\textbf P).$

We now consider the case that $\textbf P_{a-1}\neq \textbf P'_{a-1}$. There are four possibilities:
\begin{enumerate}[$(1)$]
\item $\textbf P_{a-1}=(\Delta_{i}(p),\mu_{G_l}P,\Delta_j(q))$ and $\textbf P'_{a-1}=(\Delta_i(p),\mu_{G_r}P,\Delta_j(q))$ for some $G_l$ and $G_r$ such that $P$ can twist on $G_l$ and $G_r$;
\item $\textbf P_{a-1}=(\Delta_{i-1}(p),P,\Delta_j(q))$ and $\textbf P'_{a-1}=(\Delta_i(p),\mu_{G_l}P,\Delta_j(q))$ for some $G_l$ such that $P$ can twist on $G_l$;
\item $\textbf P_{a-1}=(\Delta_{i}(p),P,\Delta_{j-1}(q))$ and $\textbf P'_{a-1}=(\Delta_i(p),\mu_{G_l}P,\Delta_j(q))$ for some $G_l$ such that $P$ can twist on $G_l$;
\item $\textbf P_{a-1}=(\Delta_{i-1}(p),P,\Delta_{j}(q))$ and $\textbf P'_{a-1}=(\Delta_i(p),P,\Delta_{j-1}(q))$.
\end{enumerate}

We shall only prove cases (1) (3) and (4) as cases (2) and (3) are dual.

For the case (1), we have $|l-r|>1$ and $\textbf P_{a-1}$ covers $(\Delta_i(p),\mu_{G_r}\mu_{G_l}P,\Delta_j(q))$ and $\textbf P'_{a-1}$ covers $(\Delta_i(p),\mu_{G_l}\mu_{G_r}P,\Delta_j(q))$. Thus we have
\begin{equation*}
\begin{array}{rcl}
w_1(\textbf P)&=&w_1(\textbf P_{a-1})+\Omega(\Delta_i(p),\mu_{G_l}P,\Delta_j(q);G_l)\vspace{2mm}\\
\hspace{-2mm}&=&\hspace{-2mm}
\left(w_1(\Delta_i(p),\mu_{G_r}\mu_{G_l}P,\Delta_j(q))+\Omega(\Delta_i(p),\mu_{G_r}\mu_{G_l}P,\Delta_j(q);G_r)\right)\vspace{2mm}\\
&+&\Omega(\Delta_i(p),\mu_{G_l}P,\Delta_j(q);G_l)
\vspace{2mm}\\
\hspace{-2mm}&=&\hspace{-2mm}
w_2(\Delta_i(p),\mu_{G_l}\mu_{G_r}P,\Delta_j(q))+(\Omega(\Delta_i(p),\mu_{G_l}\mu_{G_r}P,\Delta_j(q);G_l)
\vspace{2mm}\\
&+&\Omega(\Delta_i(p),\mu_{G_r}P,\Delta_j(q);G_r))\vspace{2mm}\\
\hspace{-2mm}&=&\hspace{-2mm}
w_2(\Delta_i(p),\mu_{G_r}P,\Delta_j(q))+\Omega(\Delta_i(p),\mu_{G_r}P,\Delta_j(q);G_r))\vspace{2mm}\\
\hspace{-2mm}&=&\hspace{-2mm}
w_2(\textbf P'_{a-1})+\Omega(\Delta_i(p),\mu_{G_r}P,\Delta_j(q);G_r))\vspace{2mm}\\
\hspace{-2mm}&=&\hspace{-2mm}
w_2(\textbf P),
\end{array}
\end{equation*}
where the third equality follows by the induction hypothesis and Lemma \ref{Lem-det2}.

For the case (3), as $\textbf P=(\Delta_i(p),P,\Delta_j(q))$ covers $\textbf P_{a-1}=(\Delta_i(p),\mu_{G_l}P,\Delta_{j}(q))$, we have $E_2(q)\in P$ and $j\neq 1$.
Thus, $E_1(q)\in \mu_{G_l}P$ and $\textbf P_{a-1}=(\Delta_i(p),\mu_{G_l}P,\Delta_{j}(q))$ covers $(\Delta_i(p),\mu_{G_l}P,\Delta_{j-1}(q))$ by Lemma \ref{Lem-cover3}. We also have that $l=c$ and $j-1=1$ can not hold simultaneously, since otherwise $\textbf P=(\Delta_i(p),P,\Delta_j(q))<\textbf P'_{a-1}=(\Delta_i(p),P,\Delta_{j-1}(q))$.
%Similarly, $l=1$ and $i=1$ can not hold simultaneously.
Thus by Lemma \ref{Lem-cover1}, $\textbf P'_{a-1}=(\Delta_i(p),P,\Delta_{j-1}(q))$ covers $(\Delta_i(p),\mu_{G_l}P,\Delta_{j-1}(q))$.

Therefore,
\begin{equation*}
\begin{array}{rcl}
w_1(\textbf P)\hspace{-2mm}&=&\hspace{-2mm} w_1(\textbf P_{a-1})+\Omega(\Delta_i(p),\mu_{G_l}P,\Delta_j(q);G_l)\vspace{2mm}\\
\hspace{-2mm}&=&\hspace{-2mm}
w_1(\Delta_i(p),\mu_{G_l}P,\Delta_{j-1}(q))+\Omega(\Delta_i(p),\mu_{G_l}P,\underline{\Delta_{j-1}(q)})
\vspace{2mm}\\
\hspace{-2mm}&+&\hspace{-2mm}
\Omega(\Delta_i(p),\mu_{G_l}P,\Delta_j(q);G_l)
\vspace{2mm}\\
\hspace{-2mm}&=&\hspace{-2mm}
w_2(\Delta_i(p),\mu_{G_l}P,\Delta_{j-1}(q))+\Omega(\Delta_i(p),\mu_{G_l}P,\Delta_{j-1}(q);G_l)
\vspace{2mm}\\
\hspace{-2mm}&+&\hspace{-2mm}
\Omega(\Delta_i(p),P,\underline{\Delta_{j-1}(q)})\vspace{2mm}\\
\hspace{-2mm}&=&\hspace{-2mm}
w_2(\Delta_i(p),P,\Delta_{j-1}(q))+\Omega(\Delta_i(p),P,\underline{\Delta_{j-1}(q)})\vspace{2mm}\\
\hspace{-2mm}&=&\hspace{-2mm}
w_2(\textbf P'_{a-1})+\Omega(\Delta_i(p),P,\underline{\Delta_{j-1}(q)})\vspace{2mm}\\
\hspace{-2mm}&=&\hspace{-2mm}
w_2(\textbf P),
\end{array}
\end{equation*}
where the third equality follows by induction hypothesis and Lemma \ref{Lem-det3}.

For the case (4), as $\textbf P$ covers $\textbf P_{a-1}=(\Delta_{i-1}(p),P,\Delta_{j}(q))$ and $\textbf P'_{a-1}=(\Delta_i(p),P,\Delta_{j-1}(q))$, we have $\textbf P_{a-1}=(\Delta_{i-1}(p),P,\Delta_{j}(q))$ and $\textbf P'_{a-1}=(\Delta_i(p),P,\Delta_{j-1}(q))$ covers $(\Delta_{i-1}(p),P,\Delta_{j-1}(q))$ by Lemma \ref{Lem-cover3}.

Therefore,
\begin{equation*}
\begin{array}{rcl}
w_1(\textbf P)\hspace{-2mm}&=&\hspace{-2mm} w_1(\textbf P_{a-1})+\Omega(\underline{\Delta_{i-1}(p)},\mu_{G_l}P,\Delta_j(q))\vspace{2mm}\\
\hspace{-2mm}&=&\hspace{-2mm}
w_1(\Delta_{i-1}(p),P,\Delta_{j-1}(q))+\Omega(\Delta_{i-1}(p),P,\underline{\Delta_{j-1}(q)})
\vspace{2mm}\\
\hspace{-2mm}&+&\hspace{-2mm}
\Omega(\underline{\Delta_{i-1}(p)},\mu_{G_l}P,\Delta_j(q))
\vspace{2mm}\\
\hspace{-2mm}&=&\hspace{-2mm}
w_2(\Delta_{i-1}(p),P,\Delta_{j-1}(q))+\Omega(\underline{\Delta_{i-1}(p)},P,\Delta_{j-1}(q))
\vspace{2mm}\\
\hspace{-2mm}&+&\hspace{-2mm}
\Omega(\Delta_{i}(p),\mu_{G_l}P,\underline{\Delta_{j-1}(q)})\vspace{2mm}\\
\hspace{-2mm}&=&\hspace{-2mm}
w_2(\Delta_i(p),P,\Delta_{j-1}(q))+\Omega(\Delta_{i}(p),\mu_{G_l}P,\underline{\Delta_{j-1}(q)})\vspace{2mm}\\
\hspace{-2mm}&=&\hspace{-2mm}
w_2(\textbf P'_{a-1})+\Omega(\Delta_{i}(p),\mu_{G_l}P,\underline{\Delta_{j-1}(q)})\vspace{2mm}\\
\hspace{-2mm}&=&\hspace{-2mm}
w_2(\textbf P),
\end{array}
\end{equation*}
where the third equality follows by induction hypothesis and Lemma \ref{Lem-det4}.

The proof is complete.
\end{proof}

\newpage

\section{Expansion formulas and the positivity}\label{sec:EFP}
Let $\Sigma$ be an orbifold and $\mathcal A_v(\Sigma)$ be a quantum cluster algebra arising from $\Sigma$. Let $\beta$ be an arc in $\Sigma$ and $T^o$ be an ideal triangulation. Let $\widetilde\beta$ be the curve associated with $\beta$ given by (\ref{Eq-wideg}). Let $T$ be the corresponding tagged triangulation. Assume that $\widetilde\beta$ starts from $p$ and ends at $q$. In this section using the lattices $\mathcal L(T^o,\widetilde\beta)$, $\mathcal L(T^o,\widetilde\beta^{(q)})$ and $\mathcal L(T^o,\widetilde\beta^{(p,q)})$, respectively in Section \ref{Sec-threesets}, we give expansion formulas for quantum cluster variables $\widetilde\beta$, $\widetilde\beta^{(q)}$ and $\widetilde\beta^{(p,q)}$, respectively. As a corollary, we obtain the positivity for quantum cluster algebras from orbifolds.

\subsection{Expansion formula for ordinary arcs.}

\begin{definition}\cite{MSW}\label{Def-weight} Suppose that the diagonal of $G_{T^o,\widetilde\beta}$ are labeled $\tau_{i_1},\cdots, \tau_{i_c}$. Let $P$ be a perfect matching of $G_{T^o,\widetilde\beta}$ with edges labeled by $\tau_{k_1},\cdots, \tau_{k_r}$. Suppose that $P\ominus P_-$ encloses the union of tiles $\bigcup_{j\in J} G_{j}$ for some $J\subseteq \{1,\cdots, c\}$.
\begin{enumerate}[$(1)$]
    \item The \emph{weight} $x^{T}(P)$ of $P$ is defined to be
    \begin{equation*}
        x^{T}(P):=\frac{x_{\tau_{k_1}}\cdots x_{\tau_{k_r}}}{x_{\tau_{i_1}}\cdots x_{\tau_{i_c}}},
    \end{equation*}
    \item The \emph {height monomial} $h^T(P)$ of $P$ is defined to be
    \begin{equation*}
        h^T(P):=\prod_{j\in J}h^T_{i_j},
    \end{equation*}
    \item The \emph {specialized height monomial} $y^T(P)$ of $P$ is defined to be
    \begin{equation*}
        y^T(P):=\Phi(h^T(P)),
    \end{equation*}
    where $\Phi$ is given by
    \begin{equation}\label{Eq-phi}
        \Phi(h^T_{\tau_i})=
            \begin{cases}
                y^T_{\tau_i}, & \mbox{if $\tau_i$ is not a side of a self-folded triangle},\vspace{1mm}\\
                \frac{y^T_{r}\vspace{2mm}}{\vspace{2mm}y^T_{r^{(p)}}}, & \mbox{if $\tau_i$ is a radius $r$ to puncture $p$ in a self-folded triangle},\vspace{1mm}\\
                y^T_{r^{(p)}}, &  \mbox{if $\tau_i$ is a loop in a self-folded triangle with radius $r$ to a puncture $p$.}
            \end{cases}
    \end{equation}
    \item The \emph {quantum weight} $X^{T}(P)$ of $P$ is defined to be the element $X^{a(P)}$ such that
    $$X^{a(P)}|_{v=1}=\frac{x^{T}(P)y^T(P)}{\bigoplus_{R\in \mathcal P(G_{T^o,\widetilde\beta})}y^T(R)}.$$%|_{trop(y^T_\tau,\tau\in T)}}.$$
\end{enumerate}

\end{definition}

In particular, if $\widetilde\beta\in T^o$ we have $x^{T}(P_{\widetilde\beta})=x_{\widetilde\beta}, y^T(P_{\widetilde\beta})=1$, where $P_{\widetilde\beta}$ is the unique perfect matching of $G_{T^o,\widetilde\beta}$.

For an arc $\gamma$ incident to a puncture $p$, we let the weight $x_{l_q(\gamma)}=x_{\gamma}x_{\gamma^{(q)}}$.

\begin{remark}
\begin{enumerate}[$(1)$]
\item Note that the weight of a perfect matching slightly differs from the original definition in \cite{MSW}, here we quotient the crossing monomial.
\item For convenience, we denote
\begin{equation*}
        y^{T^o}_{\tau_i}=
            \begin{cases}
                y^T_{\tau_i}, & \mbox{if $\tau_i$ is not a side of a self-folded triangle},\vspace{1mm}\\
                \frac{y^T_{r}\vspace{2mm}}{\vspace{2mm}y^T_{r^{(p)}}}, & \mbox{if $\tau_i$ is a radius $r$ to puncture $p$ in a self-folded triangle},\vspace{1mm}\\
                y^T_{r^{(p)}}, &  \mbox{if $\tau_i$ is a loop in a self-folded triangle with radius $r$ to a puncture $p$.}
            \end{cases}
    \end{equation*}
    Then $\Phi(h^T_{\tau_i})=y^{T^o}_{\tau_i}.$
\end{enumerate}
\end{remark}

For $i\in \{1,2,3,4\}$, if $\alpha_i$ is a radius of some self-folded triangle in $T^o$, then $\alpha$ is the loop.

\begin{lemma}\label{lem-yf1}
Let $T^o$ be an ideal triangulation and $T$ be the corresponding tagged triangulation. For any arc $\alpha\in T^o$ which can be flipped, denote $T'^o=(T^o\setminus \{\alpha\})\cup \{\alpha'\}$ be flipped ideal triangulation and $T'$ the corresponding tagged triangulation. Suppose that $\alpha$ is a diagonal of the quadrilateral $(\alpha_1,\alpha_2,\alpha_3,\alpha_4)$ in $T^o$ and $\alpha_1,\alpha_3$ are in the clockwise direction of $\alpha$. Then
\begin{enumerate}[$(1)$]
\item if $\tau\neq \alpha,\alpha_1,\alpha_2,\alpha_3,\alpha_4$, then $y^{T'^o}_{\tau}=y^{T^o}_{\tau}$,
\item $y^{T'^o}_{\alpha'}=(y^{T^o}_{\alpha})^{-1}$,
\item For $i\in \{1,2,3,4\}$,
\begin{enumerate}
\item if $\alpha_i$ is not a radius of any self-folded triangles in $T^o$ or $T'^o$, i.e., $\alpha_i\neq \alpha_{i-1},\alpha_{i+1}$, then we have
$y_{\alpha_i}^{T^o}=y_{\alpha_i}^{T'^o}(y_{\alpha'}^{T'^o})^{[-b^{T^o}_{\alpha \alpha_i}]_+}(1\oplus y_{\alpha'}^{T'^o})^{b^{T^o}_{\alpha \alpha_i}},$
\item if $\alpha_i$ is a radius of some self-folded triangle in $T'^o$ to some puncture $p$, then we have
$$y_{\alpha_i}^{T^o}=y_{\alpha_i}^{T'^o}y_{l_p(\alpha_i)}^{T'^o},$$
\end{enumerate}
%\item For $i\in \{2,4\}$, we have $y_{\alpha_i}^{T^o}=y_{\alpha_i}^{T'^o}(1\oplus y_{\alpha}^{T^o})^{b^{T^o}_{\alpha \alpha_i}}$,
\item $x_{\alpha}=\frac{x_{\alpha_2}x_{\alpha_4}+x_{\alpha_1}x_{\alpha_3}y^{T'^o}_{\alpha'}}{x_{\alpha'}\cdot(1\oplus y^{T'^o}_{\alpha'})}$.
%\item $y_{a_1}^{T^o}=y_{a_1}^{T'^o}(1\oplus y_{\alpha'}^{T'^o})^{b^{T^o}_{\alpha a_1}}$ \huang{refer latter}
\end{enumerate}
\end{lemma}

\begin{proof}
(1) It is clear.

(2) It is clear if $\alpha,\alpha'$ are not sides of self-folded triangles. We may assume one of $\alpha,\alpha'$ is a side of a self-folded triangle, without loss of generality, we may assume that $\alpha$ is a loop of a self-folded triangle with the radius $\alpha_3=\alpha_4$ to a puncture $p$. Then $y^{T'^o}_{\alpha'}=y^{T'}_{\alpha'}=(y^{T}_{\alpha_3^{(p)}})^{-1}=(y^{T^o}_{\alpha})^{-1}$.

(3) We may assume $i=1$.
For the statement (a), if $\alpha_1$ and $\alpha$ are not a side of a self-folded triangle in $T^o$ or $T'^o$, then
$$y^{T'^o}_{\alpha_1}=y^{T'}_{\alpha_1}=y^{T}_{\alpha_1}(y^{T}_{\alpha})^{b^{T^o}_{\alpha\alpha_1}}(1\oplus y^{T^o}_{\alpha})^{-b^{T^o}_{\alpha\alpha_1}}=y^{T^o}_{\alpha_1}(1\oplus (y^T_{\alpha})^{-1})^{-b^{T^o}_{\alpha\alpha_1}}=y^{T^o}_{\alpha_1}(1\oplus y^{T'^o}_{\alpha'})^{-b^{T^o}_{\alpha\alpha_1}}.$$

If $\alpha_1$ is not a side of a self-folded triangle in $T^o$ or $T'^o$ but $\alpha$ is a side of a self-folded triangle, then $\alpha_2=\alpha_3$ is the radius to some puncture $p$, then
$$y^{T'^o}_{\alpha_1}=y^{T'}_{\alpha_1}=y^{T}_{\alpha_1}(y^{T}_{\alpha_2^{(p)}})^{b^{T}_{\alpha_2^{(p)}\alpha_1}}(1\oplus y^T_{\alpha_2^{(p)}})^{-b^{T}_{\alpha_2^{(p)}\alpha_1}}=y^{T^o}_{\alpha_1}(1\oplus (y^{T^o}_{\alpha})^{-1})^{-b^{T^o}_{\alpha\alpha_1}}=y^{T^o}_{\alpha_1}(1\oplus y^{T'^o}_{\alpha'})^{-b^{T^o}_{\alpha\alpha_1}}.$$

If $\alpha_1$ is a side of a self-folded triangle in $T^o$, but $\alpha$ is not a side of a self-folded triangle in $T^o$, then $\alpha_1$ is a loop of a self-folded triangle with radius $r$ to a puncture $p$. Thus we have
$$y^{T'^o}_{\alpha_1}=y^{T'}_{r^{(p)}}=y^{T}_{r^{(p)}}(y^{T}_{\alpha})^{b^{T}_{\alpha r^{(p)}}}(1\oplus y^T_{\alpha})^{-b^{T}_{\alpha r^{(p)}}}=y^{T^o}_{\alpha_1}(1\oplus (y^{T^o}_{\alpha})^{-1})^{-b^{T^o}_{\alpha\alpha_1}}=y^{T^o}_{\alpha_1}(1\oplus y^{T'^o}_{\alpha'})^{-b^{T^o}_{\alpha\alpha_1}}.$$

If $\alpha_1$ and $\alpha$ are sides of some self-folded triangles in $T^o$, then $\alpha_1, \alpha$ are loops of some self-folded triangles. Assume that $\alpha_2=\alpha_3$ is the radius to some puncture $p$ and $\alpha_1$ is the loop of a self-folded triangle with radius $r$ to a puncture $q$. Thus we have
$$y^{T'^o}_{\alpha_1}=y^{T'}_{r^{(q)}}=y^{T}_{r^{(q)}}(y^{T}_{\alpha_2^{(p)}})^{b^{T}_{\alpha_2^{(p)} r^{(p)}}}(1\oplus y^T_{\alpha_2^{(p)}})^{-b^{T}_{\alpha_2^{(p)} r^{(p)}}}=y^{T^o}_{\alpha_1}(1\oplus (y^{T^o}_{\alpha})^{-1})^{-b^{T^o}_{\alpha\alpha_1}}=y^{T^o}_{\alpha_1}(1\oplus y^{T'^o}_{\alpha'})^{-b^{T^o}_{\alpha\alpha_1}}.$$

The result follows.

The statement (b) follows immediately by $y_{\alpha_1}^{T'}=y_{\alpha_1}^{T}$.

(4) As $x_{\alpha'}=\frac{y^{T^o}_{\alpha}x_{\alpha_2}x_{\alpha_4}+x_{\alpha_1}x_{\alpha_3}}{x_{\alpha}\cdot(1\oplus y^{T^o}_{\alpha})}$, we have
$$x_{\alpha}=\frac{y^{T^o}_{\alpha}x_{\alpha_2}x_{\alpha_4}+x_{\alpha_1}x_{\alpha_3}}{x_{\alpha'}\cdot(1\oplus y^{T^o}_{\alpha})}=\frac{x_{\alpha_2}x_{\alpha_4}+(y^{T^o}_{\alpha})^{-1}x_{\alpha_1}x_{\alpha_3}}{x_{\alpha'}\cdot((y^{T^o}_{\alpha})^{-1}\oplus 1)}=
\frac{x_{\alpha_2}x_{\alpha_4}+x_{\alpha_1}x_{\alpha_3}y^{T'^o}_{\alpha'}}{x_{\alpha'}\cdot(1\oplus y^{T'^o}_{\alpha'})}.$$
\end{proof}

The following lemma follows by Definition \ref{Def-weight} (2) (3).

\begin{lemma}\label{Lem-yf}
Let $P,Q$ be two perfect matchings. If $P<Q$ are related by a twist on a tile with diagonal labeled $\tau_i$ then
$\frac{y^T(Q)}{y^T(P)} =
            y^{T^o}_{\tau_i}.$
\end{lemma}

The following theorem is the first main result of this paper.

\begin{theorem}\label{Thm-1}
Let $\Sigma$ be an orbifold without orbifold points of weight $2$. Let $T^o$ be an ideal triangulation of $\Sigma$ and $\beta$ be an arc. Then in the quantum cluster algebra $\mathcal A_v(\Sigma)$ we have
$$X_\beta=\sum_{P\in \mathcal L(T^o,\widetilde\beta)} v^{w(P)}X^{T}(P),$$
where $w:\mathcal L(  T^o,\widetilde\beta)\to \mathbb Z$ is given by Proposition \ref{Prop-v1}.
\end{theorem}

\subsection{Expansion formula for one end tagged notched arcs.}
Suppose that $q$ is a puncture and $q\neq p$. Thus $\widetilde\beta=\beta$. Assume $T$ contains no arc tagged notched at $q$. In this section, we give a quantum expansion formula for $\beta^{(q)}=\widetilde\beta^{(q)}$.

For any triangle $\Delta_j(q)$ incident to $q$, recall that the three sides are $\tau_{j-1}(q),\tau_{j}(q)$ and $\tau_{[j]}(q)$.

\begin{definition}\label{Def-wei1}
With the notation in Section \ref{delta1}. For any $(P,\Delta_j(q))\in \mathcal L(T^o,\widetilde\beta^{(q)})$ with $P\in \mathcal P(G_{T^o,\widetilde\beta})$ and $j\in\{1,\cdots,t\}$,
\begin{enumerate}[$(1)$]
\item The \emph{weight} $x^{T}(\Delta_j(q))$ of $\Delta_j(q)$ is defined to be
\begin{equation*}
       x^{T}(\Delta_j(q)):=x^{-1}_{\tau_{j-1}(q)}x_{\tau_{[j]}(q)}x^{-1}_{\tau_{j}(q)}.
    \end{equation*}
\item The \emph{weight} $x^{T}(P,\Delta_j(q))$ of $(P,\Delta_j(q))$ is defined to be
    \begin{equation}\label{Eq-xPi}
       x^{T}(P,\Delta_j(q)):=x^{T}(P)x^{T}(\Delta_j(q))=x^{T}(P)x^{-1}_{\tau_{j-1}(q)}x_{\tau_{[j]}(q)}x^{-1}_{\tau_{j}(q)}.
    \end{equation}
\item The \emph {height monomial} $h^T(P,\Delta_j(q))$ of $(P,\Delta_j(q))$ is defined as follows:
\begin{enumerate}
  \item if $\widetilde\beta(=\beta)\notin T^o$ then
   \begin{equation}\label{Eq-hPi}
        h^T(P,\Delta_j(q)):=
        \begin{cases}
        h^T(P),  &\mbox{if $E_1(q)\in P$ and $j=1$},\\
        h^T(P)h^{T}_{\tau_1(q)}h^{T}_{\tau_2(q)}\cdots h^{T}_{\tau_{t-1}(q)}h^{T}_{\tau_t(q)},  &\mbox{if $E_2(q)\in P$ and $j=1$},\\
        h^T(P)h^{T}_{\tau_1(q)}h^{T}_{\tau_2(q)}\cdots h^{T}_{\tau_{j-1}(q)},  &\mbox{otherwise},
        \end{cases}
    \end{equation}
    where $E_1(q)$ and $E_2(q)$ are the edges of $G_c$ determined by (\ref{E1q}) in Section \ref{Sec-threesets}.
  \item if $\widetilde\beta(=\beta)\in T^o$ then
  \begin{equation*}
  h^T(P,\Delta_j(q)):=h^{T}_{\tau_1(q)}h^{T}_{\tau_2(q)}\cdots h^{T}_{\tau_{j-1}(q)}.
  \end{equation*}
  \end{enumerate}
    \item The \emph {specialized height monomial} $y^T(P,\Delta_j(q))$ of $(P,\Delta_j(q))$ is defined to be
    \begin{equation}\label{Eq-yPi}
        y^T(P,\Delta_j(q)):=\Phi(h^T(P,\Delta_j(q))),
    \end{equation}
    where $\Phi$ is given by (\ref{Eq-phi}) in Definition \ref{Def-weight}.
    \item The \emph {quantum weight} $X^{T}(P,\Delta_j(q))$ of $(P,\Delta_j(q))$ is defined to be the element $X^{a(P,\Delta_j(q))}$ such that
    $$X^{a(P,\Delta_j(q))}|_{v=1}=\frac{x^{T}(P,\Delta_j(q))y^T(P,\Delta_j(q))}{\bigoplus_{(R,\Delta)\in \mathcal L(T^o,\widetilde\beta^{(q)})}y^T(R,\Delta)}.$$
\end{enumerate}
\end{definition}

The following lemmas follow by Definition \ref{Def-wei1}.

\begin{lemma}\label{lem-ycover}
Assume $P>Q$ are related by a twist on a tile $G$ with diagonal labeled $\tau$ and $(P,\Delta_j(q))$ covers $(Q,\Delta_j(q))$. Then $\frac{y^T(P,\Delta_j(q))}{y^T(Q,\Delta_j(q))}=y^{T^o}_{\tau}$
\end{lemma}

\begin{lemma}\label{lem-ycover1}
If $(P,\Delta_{j+1}(q))$ covers $(P,\Delta_j(q))$ then
        $\frac{y^T(P,\Delta_{j+1}(q))}{y^T(P,\Delta_j(q))}
    =  y^{T^o}_{\tau_j(q)}.$
\end{lemma}

The following theorem is the second main result of this paper.

\begin{theorem}\label{Thm-2}
Let $\Sigma$ be an orbifold without orbifold points of weight $2$. Let $T^o$ be an ideal triangulation of $\Sigma$ and $T$ be the corresponding tagged triangulation. For any $\beta$, let $\widetilde\beta$ be the curve associated with $\beta$  by (\ref{Eq-wideg}). Assume that $\widetilde\beta$ connects $p$ and a puncture $q$ with $q\neq p$. If $T$ contains no arc tagged notched at $q$, then in the quantum cluster algebra $\mathcal A_v(\Sigma)$ we have
$$X_{\beta^{(q)}}=\sum_{(P,\Delta_j(q))\in \mathcal L(T^o,\widetilde\beta^{(q)})} v^{w(P,\Delta_j(q))}X^{T}(P,\Delta_j(q)),$$
where $w:\mathcal L(T^o,\widetilde\beta^{(q)})\to \mathbb Z$ is given by Proposition \ref{Pro-vm}.
\end{theorem}

\subsection{Expansion formula for two ends tagged notched arcs.}
Suppose that $p$ and $q$ are punctures. Assume that $T$ contains no arc tagged notched at $p$ or $q$.
If $\beta$ is a pending arc incident to puncture $p$, let $X_{\widetilde\beta^{(p,q)}}=X_{\beta^{(p)}}$. Herein, we give an expansion formula for $\widetilde\beta^{(p,q)}$.

For any triangle $\Delta_i(p)$ incident to $p$, recall that the three sides are $\tau_{i-1}(p),\tau_{i}(p)$ and $\tau_{[i]}(p)$.

\begin{definition}\label{Def-wei2}
With the notation in Section \ref{delta2}. Assume that $p$ and $q$ are punctures. For any $(\Delta_i(p),P,\Delta_j(q))\in \mathcal L(T^o,\widetilde\beta^{(p,q)})$ with $P\in \mathcal P(G_{T^o,\widetilde\beta})$ and $i\in\{1,\cdots,s\}, j\in \{1,\cdots, t\}$,
\begin{enumerate}[$(1)$]
\item The \emph{weight} $x^{T}(\Delta_i(p))$ of $\Delta_i(p)$ is defined to be
\begin{equation*}
       x^{T}(\Delta_i(p)):=x^{-1}_{\tau_{i-1}(p)}x_{\tau_{[i]}(p)}x^{-1}_{\tau_{i}(p)}.
    \end{equation*}
\item The \emph{weight} $x^{T}(\Delta_j(q))$ of $\Delta_j(q)$ is defined to be
\begin{equation*}
       x^{T}(\Delta_j(q)):=x^{-1}_{\tau_{j-1}(q)}x_{\tau_{[j]}(q)}x^{-1}_{\tau_{j}(q)}.
    \end{equation*}
\item The \emph{weight} $x^T\left(\Delta_i(p), P,\Delta_j(q)\right)$ of $(\Delta_i(p),P,\Delta_j(q))$ is defined to be
    \begin{equation}\label{Eq-xdelta}
    x^T\left(\Delta_i(p),P,\Delta_j(q)\right):=x^{-1}_{\tau_{i-1}(p)}x_{\tau_{[i]}(p)}x^{-1}_{\tau_{i}(p)}x^{T}(P)x^{-1}_{\tau_{j-1}(q)}x_{\tau_{[j]}(q)}x^{-1}_{\tau_{j}(q)}.
    \end{equation}
\item The \emph {height monomial} $h^T(\Delta_i(p),P,\Delta_j(q))$ of $(\Delta_i(p),P,\Delta_j(q))$ is defined as follows:
\begin{enumerate}
  \item if $\widetilde\beta\notin T^o$ then
   \begin{equation*}
        h^T(\Delta_i(p),P,\Delta_j(q)):=
        \begin{cases}
        h^T(P,\Delta_j(q)), &\mbox{if $E_1(p)\in P$ and $i=1$},\\
        h^T(P,\Delta_j(q))h^{T}_{\tau_1(p)}h^{T}_{\tau_2(p)}\cdots h^{T}_{\tau_{s-1}(p)}h^{T}_{\tau_s(p)},  &\mbox{if $E_2(p)\in P$ and $i=1$},\\
        h^T(P,\Delta_j(q))h^{T}_{\tau_1(p)}h^{T}_{\tau_2(p)}\cdots h^{T}_{\tau_{i-1}(p)},  &\mbox{otherwise},
        \end{cases}
    \end{equation*}
    where $E_1(p)$ is the edge of $G_1$ determined by (\ref{E1p}) in Section \ref{Sec-threesets}.
  \item if $\widetilde\beta\in T^o$ and $s,t\geq 2$ then
  \begin{equation}\label{Eq-hpq}
        h^T(\Delta_i(p),P_\beta,\Delta_j(q)):=
        \begin{cases}
        1, &\mbox{if $i=1$ and $j=t$},\\
        h^{T}_{\tau_1(p)}h^{T}_{\tau_2(p)}\cdots h^{T}_{\tau_{s-1}(p)}h^{T}_{\tau_1(q)}\cdots h^{T}_{\tau_{t-1}(q)}(h^T_{\beta})^2,  &\mbox{if $i=s$ and $j=1$},\\
        h^{T}_{\tau_1(p)}h^{T}_{\tau_2(p)}\cdots h^{T}_{\tau_{i-1}(p)}h^{T}_{\tau_1(q)}\cdots h^{T}_{\tau_{j-1}(q)}h^{T}_\beta,  &\mbox{otherwise}.
        \end{cases}
    \end{equation}
\end{enumerate}
    \item The \emph {specialized height monomial} $y^T(\Delta_i(p),P,\Delta_j(q))$ of $(\Delta_i(p),P,\Delta_j(q))$ is defined to be
    \begin{equation*}
        y^T(\Delta_i(p),P,\Delta_j(q)):=\Phi(h^T(\Delta_i(p),P,\Delta_j(q))),
    \end{equation*}
    where $\Phi$ is given by (\ref{Eq-phi}) in Definition \ref{Def-weight}.
    \item The \emph {quantum weight} $X^{T}(\Delta_i(p),P,\Delta_j(q))$ of $(\Delta_i(p),P,\Delta_j(q))$ is defined to be the element $X^{a(\Delta_i(p),P,\Delta_j(q))}$ such that
    $$X^{a(\Delta_i(p),P,\Delta_j(q))}|_{v=1}=\frac{x^{T}(\Delta_i(p),P,\Delta_j(q))y^T(\Delta_i(p),P,\Delta_j(q))}{{\bigoplus_{(\Delta,R,\Delta')\in \mathcal L(T^o,\widetilde\beta^{(p,q)})}y^T(\Delta,R,\Delta')}}.$$
\end{enumerate}
\end{definition}

From the definition of weight $  x^T\left(\Delta_i(p),P,\Delta_j(q)\right)$ and $  x^T\left(P,\Delta_j(q)\right)$, we have
\begin{equation*}
      x^T\left(\Delta_i(p),P,\Delta_j(q)\right)=x^{T}(\Delta_i(p))  x^T\left(P,\Delta_j(q)\right).
\end{equation*}

The following lemmas follow by Definition \ref{Def-wei2} (4) (5).

\begin{lemma}\label{lem-dcover1}
Let $P,Q$ be two perfect matchings such that $P>Q$ are related by a twist on a tile $G$ with diagonal labeled $\tau$. If $(\Delta_i(p),P,\Delta_j(q))$ covers $(\Delta_i(p),Q,\Delta_j(q))$ then
 \begin{equation*}
        \frac{y^T(\Delta_i(p),P,\Delta_j(q))}{y^T(\Delta_i(p),Q,\Delta_j(q))}=
                y^{T^o}_{\tau}.
  \end{equation*}
\end{lemma}

\begin{lemma}\label{lem-dcover2}
If $(\Delta_{i+1}(p),P,\Delta_{j}(q))$ covers $(\Delta_i(p),P,\Delta_j(q))$ then
 \begin{equation*}
        \frac{y^T(\Delta_{i+1}(p),P,\Delta_{j}(q))}{y^T(\Delta_i(p),P,\Delta_j(q))}=
y^{T^o}_{\tau_i(p)}.
    \end{equation*}

\end{lemma}

\begin{lemma}\label{lem-dcover3}
If $(\Delta_i(p),P,\Delta_{j+1}(q))$ covers $(\Delta_i(p),P,\Delta_j(q))$ then
 \begin{equation*}
        \frac{y^T(\Delta_i(p),P,\Delta_{j+1}(q))}{y^T(\Delta_i(p),P,\Delta_j(q))}
        =
        y^{T^o}_{\tau_j(q)}.
    \end{equation*}

\end{lemma}

The following theorem is the third main result of this paper.

\begin{theorem}\label{Thm-M3}
Let $\Sigma$ be an orbifold without orbifold points of weight $2$. Let $T^o$ be an ideal triangulation of $\Sigma$ and $\beta$ be an arc. Let $T$ be the corresponding tagged triangulation of $T^o$. Let $\widetilde\beta$ be the curve associated with $\beta$ given by (\ref{Eq-wideg}). Assume that $\widetilde\beta$ connects two punctures $p$ and $q$.
If $T$ contains no arc tagged notched at $p$ or $q$, then in the quantum cluster algebra $\mathcal A_v(\Sigma)$ we have
$$X_{\widetilde\beta^{(p,q)}}=\sum_{(\Delta_i(p),P,\Delta_j(q))\in \mathcal L(T^o,\widetilde\beta^{(p,q)})} v^{w(\Delta_i(p),P,\Delta_j(q))}X^{T}(\Delta_i(p),P,\Delta_j(q)),$$
where $w:\mathcal L(T^o,\widetilde\beta^{(p,q)})\to \mathbb Z$ is given by Proposition \ref{pro-vpq}.
 %($T$ contains no arc tagged notched at $p$ or $q$) and Remark \ref{Rem-v2} ($T$ contains some arc tagged notched at $p$ or $q$).
\end{theorem}

As an immediate corollary of Theorem \ref{Thm-1}, Theorem \ref{Thm-2}, and Theorem \ref{Thm-M3}, we obtain the positivity for quantum cluster algebras from orbifolds.

\begin{theorem}\label{thm:P}
Positivity conjecture holds for quantum cluster algebras from orbifolds without orbifold points of weight $2$.
\end{theorem}

\begin{remark}
In case $\Sigma$ be an orbifold with or without orbifold points of weight $2$, our method can be also applied to its associated orbifold $\hat\Sigma$, which is obtained from $\Sigma$ by making all orbifold points of weight $2$ to punctures.
\end{remark}

\newpage

\section{Partition bijection between perfect matchings}\label{Sec-parcom}

Let $T^o$ be an ideal triangulation and $\gamma$ be a curve connecting two marked points $p$ and $q$. Choose an orientation of $\gamma$, assume that $\gamma$ crosses $T^o$ at $p_1,\cdots,p_c$ sequentially. Denote $p_0=p$ and $p_{c+1}=q$. Suppose that $p_j\in \tau_{i_j}$ for $j=1,\cdots,c$. Given a non-self-folded arc $\alpha\in T^o$, denote $T'^o=\mu_\alpha(T^o)$ and by $\alpha'$ the new arc obtained.

In this section we construct a partition bijection $\varphi^{T^o}_\alpha: \mathcal {P}(G_{T^o,\gamma})\to \mathcal {P}(G_{T'^o,\gamma})$. To this end, we first construct a partition bijection $\phi^{T^o}_\alpha: \mathcal {CP}(T^o,\gamma)\to \mathcal {CP}(T'^o,\gamma)$ between the complete $(T^o,\gamma)$-paths.

We start with the polygon case.

\subsection{Polygon case} In this subsection, let $\Sigma$ be the $n$-gon, i.e., the disk with $n$ marked points labeled $1,\cdots,n$ clockwise on the boundary.

\subsubsection{Partition bijection $\phi^{T^o}_\alpha$}\label{Sec-PB-A} For any complete $(T^o,\gamma)$-path $\overrightarrow\xi$ and $\zeta\in T^o$, recall $m(\overrightarrow\xi,\zeta)$ given in (\ref{equ-m1}). In the case that $\Sigma$ is a polygon, we have $m(\overrightarrow\xi,\zeta)\in \{1,0,-1\}$.

Assume that $\alpha=(a,c)$ is the diagonal of the quadrilateral $(a,b,c,d)$ in $T^o$, where $a,b,c,d$ are in the clockwise order. Then $\alpha'=(b,d)$.

Let $\overrightarrow\xi=(\overrightarrow\xi_{\hspace{-2pt}1},\cdots,\overrightarrow\xi_{\hspace{-2pt}2c+1})$ be a complete $(T^o,\gamma)$-path. We would like to construct a set of complete $(T'^o,\gamma)$-paths as follows.

\medskip

Case I: $\gamma$ crosses with two parallel edges of the quadrilateral $(a,b,c,d)$. We may assume that $\gamma$ crosses $\cdots, (a,b), (a,c), (c,d),\cdots $ sequentially. Then $\xi_{2j-2}=(a,b), \xi_{2j}=(a,c)$ and $\xi_{2j+2}=(c,d)$ for some $j\in \{2,\cdots,c-1\}$.

\begin{figure}[h]

\centerline{\includegraphics{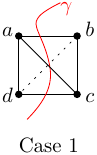}}

\end{figure}

(1)~ If $\overrightarrow\xi_{\hspace{-2pt}2j}=(a,c)$ then $(\overrightarrow\xi_{\hspace{-2pt}2j-2},\overrightarrow\xi_{\hspace{-2pt}2j-1},\overrightarrow\xi_{\hspace{-2pt}2j},\overrightarrow\xi_{\hspace{-2pt}2j+1},\overrightarrow\xi_{\hspace{-2pt}2j+2})=((a,b),(b,a),(a,c),(c,d),(d,c))$. We have $m(\overrightarrow\xi,\alpha)=-1$.

In this case we replace $(\overrightarrow\xi_{\hspace{-2pt}2j-2},\overrightarrow\xi_{\hspace{-2pt}2j-1},\overrightarrow\xi_{\hspace{-2pt}2j},\overrightarrow\xi_{\hspace{-2pt}2j+1},\overrightarrow\xi_{\hspace{-2pt}2j+2})$ by $((a,b),(b,d),(d,b),(b,d),(d,c))$ in the sequence $\overrightarrow\xi$.

(2)~ If $\overrightarrow\xi_{\hspace{-2pt}2j}=(c,a)$ then there are four possibilities of $(\overrightarrow\xi_{\hspace{-2pt}2j-2},\overrightarrow\xi_{\hspace{-2pt}2j-1},\overrightarrow\xi_{\hspace{-2pt}2j},\overrightarrow\xi_{\hspace{-2pt}2j+1},\overrightarrow\xi_{\hspace{-2pt}2j+2})$:\\
$((b,a),(a,c),(c,a),(a,c),(c,d))$, $((b,a),(a,c),(c,a),(a,d),(d,c))$, $((a,b),(b,c),(c,a),(a,c),(c,d))$ and $((a,b),(b,c),(c,a),(a,d),(d,c))$.

(2.1)~ In case $(\overrightarrow\xi_{\hspace{-2pt}2j-2},\overrightarrow\xi_{\hspace{-2pt}2j-1},\overrightarrow\xi_{\hspace{-2pt}2j},\overrightarrow\xi_{\hspace{-2pt}2j+1},\overrightarrow\xi_{\hspace{-2pt}2j+2})=((b,a),(a,c),(c,a),(a,c),(c,d))$, we have $m(\overrightarrow\xi,\alpha)=1$.

In this case we replace $(\overrightarrow\xi_{\hspace{-2pt}2j-2},\overrightarrow\xi_{\hspace{-2pt}2j-1},\overrightarrow\xi_{\hspace{-2pt}2j},\overrightarrow\xi_{\hspace{-2pt}2j+1},\overrightarrow\xi_{\hspace{-2pt}2j+2})$ by $((b,a),(a,d),(d,b),(b,c),(c,d))$ or $((b,a),(a,b),(b,d),(d,c),(c,d))$ in the sequence $\overrightarrow\xi$.

(2.2)~ In case $(\overrightarrow\xi_{\hspace{-2pt}2j-2},\overrightarrow\xi_{\hspace{-2pt}2j-1},\overrightarrow\xi_{\hspace{-2pt}2j},\overrightarrow\xi_{\hspace{-2pt}2j+1},\overrightarrow\xi_{\hspace{-2pt}2j+2})=((b,a),(a,c),(c,a),(a,d),(d,c))$, we have $m(\overrightarrow\xi,\alpha)=0$.

In this case we replace $(\overrightarrow\xi_{\hspace{-2pt}2j-2},\overrightarrow\xi_{\hspace{-2pt}2j-1},\overrightarrow\xi_{\hspace{-2pt}2j},\overrightarrow\xi_{\hspace{-2pt}2j+1},\overrightarrow\xi_{\hspace{-2pt}2j+2})$ by $((b,a),(a,d),(d,b),(b,d),(d,c))$ in the sequence $\overrightarrow\xi$.

(2.3)~ In case $(\overrightarrow\xi_{\hspace{-2pt}2j-2},\overrightarrow\xi_{\hspace{-2pt}2j-1},\overrightarrow\xi_{\hspace{-2pt}2j},\overrightarrow\xi_{\hspace{-2pt}2j+1},\overrightarrow\xi_{\hspace{-2pt}2j+2})=((a,b),(b,c),(c,a),(a,c),(c,d))$,
we have $m(\overrightarrow\xi,\alpha)=0$.

In this case we replace $(\overrightarrow\xi_{\hspace{-2pt}2j-2},\overrightarrow\xi_{\hspace{-2pt}2j-1},\overrightarrow\xi_{\hspace{-2pt}2j},\overrightarrow\xi_{\hspace{-2pt}2j+1},\overrightarrow\xi_{\hspace{-2pt}2j+2})$ by $((a,b),(b,d),(d,b),(b,c),(c,d))$ in the sequence $\overrightarrow\xi$.

(2.4)~ In case $(\overrightarrow\xi_{\hspace{-2pt}2j-2},\overrightarrow\xi_{\hspace{-2pt}2j-1},\overrightarrow\xi_{\hspace{-2pt}2j},\overrightarrow\xi_{\hspace{-2pt}2j+1},\overrightarrow\xi_{\hspace{-2pt}2j+2})=((a,b),(b,c),(c,a),(a,d),(d,c))$, we have $m(\overrightarrow\xi,\alpha)=-1$.

In this case we replace $(\overrightarrow\xi_{\hspace{-2pt}2j-2},\overrightarrow\xi_{\hspace{-2pt}2j-1},\overrightarrow\xi_{\hspace{-2pt}2j},\overrightarrow\xi_{\hspace{-2pt}2j+1},\overrightarrow\xi_{\hspace{-2pt}2j+2})$ by $((a,b),(b,d),(d,b),(b,d),(d,c))$ in the sequence $\overrightarrow\xi$.

In each case, one can see that the new sequences $\overrightarrow\xi'$ obtained are complete $(T'^o,\gamma)$-paths and satisfy
$m(\overrightarrow\xi',\alpha')=-m(\overrightarrow\xi,\alpha).$

\medskip

Case II: $\gamma$ crosses with $(a,c)$ and two adjacent edges of the quadrilateral $(a,b,c,d)$. We may assume that $\gamma$ crosses with $\cdots, (a,b), (a,c), (a,d),\cdots $ sequentially. Then $\xi_{2j-2}=(a,b), \xi_{2j}=(a,c)$ and $\xi_{2j+2}=(a,d)$ for some $j\in \{2,\cdots,c-1\}$.

\begin{figure}[h]

\centerline{\includegraphics{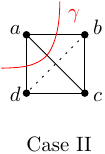}}

\end{figure}

(1)~ If $\overrightarrow\xi_{\hspace{-2pt}2j}=(a,c)$ then $(\overrightarrow\xi_{\hspace{-2pt}2j-2},\overrightarrow\xi_{\hspace{-2pt}2j-1},\overrightarrow\xi_{\hspace{-2pt}2j},\overrightarrow\xi_{\hspace{-2pt}2j+1},\overrightarrow\xi_{\hspace{-2pt}2j+2})=((a,b),(b,a),(a,c),(c,a),(a,d))$ or $((a,b),(b,a),(a,c),(c,d),(d,a))$.

(1.1)~ In case $(\overrightarrow\xi_{\hspace{-2pt}2j-2},\overrightarrow\xi_{\hspace{-2pt}2j-1},\overrightarrow\xi_{\hspace{-2pt}2j},\overrightarrow\xi_{\hspace{-2pt}2j+1},\overrightarrow\xi_{\hspace{-2pt}2j+2})=((a,b),(b,a),(a,c),(c,a),(a,d))$, we have $m(\overrightarrow\xi,\alpha)=0$.

In this case we replace $(\overrightarrow\xi_{\hspace{-2pt}2j-2},\overrightarrow\xi_{\hspace{-2pt}2j-1},\overrightarrow\xi_{\hspace{-2pt}2j},\overrightarrow\xi_{\hspace{-2pt}2j+1},\overrightarrow\xi_{\hspace{-2pt}2j+2})$ by $((a,b),(b,a),(a,d))$ in the sequence $\overrightarrow\xi$.

(1.2)~ In case $(\overrightarrow\xi_{\hspace{-2pt}2j-2},\overrightarrow\xi_{\hspace{-2pt}2j-1},\overrightarrow\xi_{\hspace{-2pt}2j},\overrightarrow\xi_{\hspace{-2pt}2j+1},\overrightarrow\xi_{\hspace{-2pt}2j+2})=((a,b),(b,a),(a,c),(c,d),(d,a))$, we have $m(\overrightarrow\xi,\alpha)=-1$.

In this case we replace $(\overrightarrow\xi_{\hspace{-2pt}2j-2},\overrightarrow\xi_{\hspace{-2pt}2j-1},\overrightarrow\xi_{\hspace{-2pt}2j},\overrightarrow\xi_{\hspace{-2pt}2j+1},\overrightarrow\xi_{\hspace{-2pt}2j+2})$ by $((a,b),(b,d),(d,a))$ in the sequence $\overrightarrow\xi$.

(2)~ If $\overrightarrow\xi_{\hspace{-2pt}2j}=(c,a)$ then $(\overrightarrow\xi_{\hspace{-2pt}2j-2},\overrightarrow\xi_{\hspace{-2pt}2j-1},\overrightarrow\xi_{\hspace{-2pt}2j},\overrightarrow\xi_{\hspace{-2pt}2j+1},\overrightarrow\xi_{\hspace{-2pt}2j+2})=((a,b),(b,c),(c,a),(a,d),(d,a))$ or $((b,a),(a,c),(c,a),(a,d),(d,a))$.

(2.1)~ In case $(\overrightarrow\xi_{\hspace{-2pt}2j-2},\overrightarrow\xi_{\hspace{-2pt}2j-1},\overrightarrow\xi_{\hspace{-2pt}2j},\overrightarrow\xi_{\hspace{-2pt}2j+1},\overrightarrow\xi_{\hspace{-2pt}2j+2})=((a,b),(b,c),(c,a),(a,d),(d,a))$, we have $m(\overrightarrow\xi,\alpha)=-1$.

In this case we replace $(\overrightarrow\xi_{\hspace{-2pt}2j-2},\overrightarrow\xi_{\hspace{-2pt}2j-1},\overrightarrow\xi_{\hspace{-2pt}2j},\overrightarrow\xi_{\hspace{-2pt}2j+1},\overrightarrow\xi_{\hspace{-2pt}2j+2})$ by $((a,b),(b,d),(d,a))$ in the sequence $\overrightarrow\xi$.

(2.2)~ In case $(\overrightarrow\xi_{\hspace{-2pt}2j-2},\overrightarrow\xi_{\hspace{-2pt}2j-1},\overrightarrow\xi_{\hspace{-2pt}2j},\overrightarrow\xi_{\hspace{-2pt}2j+1},\overrightarrow\xi_{\hspace{-2pt}2j+2})=((b,a),(a,c),(c,a),(a,d),(d,a))$, we have $m(\overrightarrow\xi,\alpha)=0$.

In this case we replace $(\overrightarrow\xi_{\hspace{-2pt}2j-2},\overrightarrow\xi_{\hspace{-2pt}2j-1},\overrightarrow\xi_{\hspace{-2pt}2j},\overrightarrow\xi_{\hspace{-2pt}2j+1},\overrightarrow\xi_{\hspace{-2pt}2j+2})$ by $((b,a),(a,d),(d,a))$ in the sequence $\overrightarrow\xi$.

In each case, one can see that the new sequences $\overrightarrow\xi'$ obtained are complete $(T'^o,\gamma)$-paths and satisfy
$m(\overrightarrow\xi',\alpha')=-m(\overrightarrow\xi,\alpha).$

\medskip

Case III:  $\gamma$ crosses two adjacent edges of the quadrilateral $(a,b,c,d)$ but not cross $(a,c)$. We may assume that $\gamma$ crosses with $\cdots, (a,b), (b,c),\cdots $ sequentially. Then $\xi_{2j-2}=(a,b), \xi_{2j}=(b,c)$ for some $j\in \{2,\cdots,c\}$. Then we have $(\overrightarrow\xi_{\hspace{-2pt}2j-2},\overrightarrow\xi_{\hspace{-2pt}2j-1},\overrightarrow\xi_{\hspace{-2pt}2j})=((b,a),(a,b),(b,c))$ or $((a,b),(b,c),(c,b))$ or $((b,a),(a,c),(c,b))$.

\begin{figure}[h]

\centerline{\includegraphics{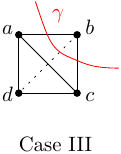}}

\end{figure}

(1)~ If $(\overrightarrow\xi_{\hspace{-2pt}2j-2},\overrightarrow\xi_{\hspace{-2pt}2j-1},\overrightarrow\xi_{\hspace{-2pt}2j})=((b,a),(a,b),(b,c))$, we have $m(\overrightarrow\xi,\alpha)=0$. In this case we replace $(\overrightarrow\xi_{\hspace{-2pt}2j-2},\overrightarrow\xi_{\hspace{-2pt}2j-1},\overrightarrow\xi_{\hspace{-2pt}2j})$ by $((b,a)$, $(a,b),(b,d),(d,b),(b,c))$ in the sequence $\overrightarrow\xi$.

(2)~ If $(\overrightarrow\xi_{\hspace{-2pt}2j-2},\overrightarrow\xi_{\hspace{-2pt}2j-1},\overrightarrow\xi_{\hspace{-2pt}2j})=((a,b),(b,c),(c,b))$, we have $m(\overrightarrow\xi,\alpha)=0$. In this case we replace $(\overrightarrow\xi_{\hspace{-2pt}2j-2},\overrightarrow\xi_{\hspace{-2pt}2j-1},\overrightarrow\xi_{\hspace{-2pt}2j})$ by $((a,b)$, $(b,d),(d,b),(b,c),(c,b))$ in the sequence $\overrightarrow\xi$.

(3)~ If $(\overrightarrow\xi_{\hspace{-2pt}2j-2},\overrightarrow\xi_{\hspace{-2pt}2j-1},\overrightarrow\xi_{\hspace{-2pt}2j})=((b,a),(a,c),(c,b))$, we have $m(\overrightarrow\xi,\alpha)=1$. In this case we replace $(\overrightarrow\xi_{\hspace{-2pt}2j-2},\overrightarrow\xi_{\hspace{-2pt}2j-1},\overrightarrow\xi_{\hspace{-2pt}2j})$ by $((b,a)$, $(a,d),(d,b),(b,c),(c,b))$ or $((b,a),(a,b),(b,d),(d,c),(c,b))$ in the sequence $\overrightarrow\xi$.

In each case, one can see that the new sequences $\overrightarrow\xi'$ obtained are complete $(T'^o,\gamma)$-paths and satisfy
$m(\overrightarrow\xi',\alpha')=-m(\overrightarrow\xi,\alpha).$

\medskip

Case IV: $\gamma$ crosses with $(a,c)$ and exactly one edge of the quadrilateral $(a,b,c,d)$. We may assume that $\gamma$ crosses with $(a,c), (a,d),\cdots $ sequentially. Then $\xi_{2}=(a,c), \xi_{4}=(a,d)$. Thus we have $(\overrightarrow\xi_{\hspace{-2pt}1},\overrightarrow\xi_{\hspace{-2pt}2},\overrightarrow\xi_{\hspace{-2pt}3},\overrightarrow\xi_{\hspace{-2pt}4})=((b,a),(a,c),(c,a),(a,d))$ or $((b,a),(a,c),(c,d),(d,a))$ or $((b,c),(c,a),(a,d),(d,a))$.

\begin{figure}[h]

\centerline{\includegraphics{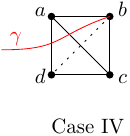}}

\end{figure}

(1)~ If $(\overrightarrow\xi_{\hspace{-2pt}1},\overrightarrow\xi_{\hspace{-2pt}2},\overrightarrow\xi_{\hspace{-2pt}3},\overrightarrow\xi_{\hspace{-2pt}4})=((b,a),(a,c),(c,a),(a,d))$, we have $m(\overrightarrow\xi,\alpha)=0$. In this case we replace $(\overrightarrow\xi_{\hspace{-2pt}1},\overrightarrow\xi_{\hspace{-2pt}2},\overrightarrow\xi_{\hspace{-2pt}3},\overrightarrow\xi_{\hspace{-2pt}4})$ by $((b,a)$, $(a,d))$ in the sequence $\overrightarrow\xi$.

(2)~ If $(\overrightarrow\xi_{\hspace{-2pt}1},\overrightarrow\xi_{\hspace{-2pt}2},\overrightarrow\xi_{\hspace{-2pt}3},\overrightarrow\xi_{\hspace{-2pt}4})=((b,a),(a,c),(c,d),(d,a))$, we have $m(\overrightarrow\xi,\alpha)=-1$. In this case we replace $(\overrightarrow\xi_{\hspace{-2pt}1},\overrightarrow\xi_{\hspace{-2pt}2},\overrightarrow\xi_{\hspace{-2pt}3},\overrightarrow\xi_{\hspace{-2pt}4})$ by $((b,d)$, $(d,a))$ in the sequence $\overrightarrow\xi$.

(3)~ If $(\overrightarrow\xi_{\hspace{-2pt}1},\overrightarrow\xi_{\hspace{-2pt}2},\overrightarrow\xi_{\hspace{-2pt}3},\overrightarrow\xi_{\hspace{-2pt}4})=((b,c),(c,a),(a,d),(d,a))$, we have $m(\overrightarrow\xi,\alpha)=-1$. In this case we replace $(\overrightarrow\xi_{\hspace{-2pt}1},\overrightarrow\xi_{\hspace{-2pt}2},\overrightarrow\xi_{\hspace{-2pt}3},\overrightarrow\xi_{\hspace{-2pt}4})$ by $((b,d)$, $(d,a))$ in the sequence $\overrightarrow\xi$.

In each case, one can see that the new sequences $\overrightarrow\xi'$ obtained are complete $(T'^o,\gamma)$-paths and satisfy
$m(\overrightarrow\xi',\alpha')=-m(\overrightarrow\xi,\alpha).$

\medskip

Case V: $\gamma$ crosses with exactly one edge of the quadrilateral $(a,b,c,d)$ but does not cross with $(a,c)$. We may assume that $\gamma$ crosses with $(b,c),\cdots $ sequentially. Then $\xi_{2}=(b,c)$. Thus we have $(\overrightarrow\xi_{\hspace{-2pt}1},\overrightarrow\xi_{\hspace{-2pt}2})=((a,b),(b,c))$ or $((a,c),(c,b))$.

\begin{figure}[h]

\centerline{\includegraphics{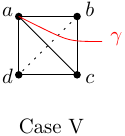}}

\end{figure}

(1)~ If $(\overrightarrow\xi_{\hspace{-2pt}1},\overrightarrow\xi_{\hspace{-2pt}2})=((a,b),(b,c))$, we have $m(\overrightarrow\xi,\alpha)=0$. In this case we replace $(\overrightarrow\xi_{\hspace{-2pt}1},\overrightarrow\xi_{\hspace{-2pt}2})$ by $((a,b),(b,d),(d,b),(b,c))$ in the sequence $\overrightarrow\xi$.

(2)~ If $(\overrightarrow\xi_{\hspace{-2pt}1},\overrightarrow\xi_{\hspace{-2pt}2})=((a,c),(c,b))$, we have $m(\overrightarrow\xi,\alpha)=1$. In this case we replace $(\overrightarrow\xi_{\hspace{-2pt}1},\overrightarrow\xi_{\hspace{-2pt}2})$ by $((a,b),(b,d),(d,c),(c,b))$ or $((a,d),(d,b),(b,c),(c,b))$ in the sequence $\overrightarrow\xi$.

In each case, one can see that the new sequences $\overrightarrow\xi'$ obtained are complete $(T'^o,\gamma)$-paths and satisfy
$m(\overrightarrow\xi',\alpha')=-m(\overrightarrow\xi,\alpha).$

\medskip

From the above discussion, we have the following observation.

\begin{lemma}\label{Lem-mid}
Let $\overrightarrow\xi=(\overrightarrow\xi_{\hspace{-2pt}1},\cdots,\overrightarrow\xi_{\hspace{-2pt}2c+1})$ be a complete $(T^o,\gamma)$-path. Assume that $\xi_j=\alpha=(a,c)$ for some even $j$ such that $\xi_{j-1}\neq \alpha\neq \xi_{j+1}$.
\begin{enumerate}[$(1)$]
\item If $\overrightarrow\xi_{\hspace{-2pt}j}=(a,c)$ then $(\overrightarrow\xi_{\hspace{-2pt}j-1},\overrightarrow\xi_{\hspace{-2pt}j},\overrightarrow\xi_{\hspace{-2pt}j+1})=((b,a),(a,c),(c,d))$ or $((d,a),(a,c),(c,b))$.
\item If $\overrightarrow\xi_{\hspace{-2pt}j}=(c,a)$ then $(\overrightarrow\xi_{\hspace{-2pt}j-1},\overrightarrow\xi_{\hspace{-2pt}j},\overrightarrow\xi_{\hspace{-2pt}j+1})=((b,c),(c,a),(a,d))$ or $((d,c),(c,a),(a,b))$.
\end{enumerate}
\end{lemma}

Therefore, for any complete $(T^o,\gamma)$-path $\overrightarrow\xi$, we obtain a set of complete $(T'^o,\gamma)$-paths, denoted by $\phi^{T^o}_\alpha(\overrightarrow\xi)$, moreover, for any $\overrightarrow\xi'\in \phi^{T^o}_\alpha(\overrightarrow\xi)$, we have $m(\overrightarrow\xi',\alpha')=-m(\overrightarrow\xi,\alpha).$

The construction of $\phi^{T^o}_\alpha$ can be summarized as follows:

\textbf{Construction of $\phi^{T^o}_\alpha$:} given any complete $(T^o,\gamma)$-path $\overrightarrow\xi=(\overrightarrow\xi_{\hspace{-2pt}1},\cdots,\overrightarrow\xi_{\hspace{-2pt}2c+1})$.

If $m(\overrightarrow\xi,\alpha)=1$, then the set $\phi^{T^o}_\alpha(\overrightarrow\xi)$ contains all sequences of oriented arcs $\overrightarrow\xi'$, where $\overrightarrow\xi'$ is obtained via the following recipe:
\begin{enumerate}[$(i)$]
\item First remove a pair $(\overrightarrow\xi_{\hspace{-2pt}j},\overrightarrow\xi_{\hspace{-2pt}j+1})$ in case $\xi_j=\xi_{j+1}=\alpha$ and $\xi_{j-1}\neq \alpha$;
\item Next for any $\xi_j=\alpha$ with odd $j$,
\begin{enumerate}
\item if $\overrightarrow\xi_{\hspace{-2pt}j}=(a,c)$, then replace $\overrightarrow\xi_{\hspace{-2pt}j}$ by $(a,b),(b,d),(d,c)$ or $(a,d),(d,b),(b,c) $;
\item if $\overrightarrow\xi_{\hspace{-2pt}j}=(c,a)$, then replace $\overrightarrow\xi_{\hspace{-2pt}j}$ by $(c,b),(b,d),(d,a)$ or $(c,d),(d,b),(b,a)$;
\end{enumerate}
\end{enumerate}

If $m(\overrightarrow\xi,\alpha)=0$, then the set $\phi^{T^o}_\alpha(\overrightarrow\xi)$ contains all sequences of oriented arcs $\overrightarrow\xi'$, where $\overrightarrow\xi'$ can be obtained via the following recipe:

\begin{enumerate}[$(i)$]
\item First remove a pair $(\overrightarrow\xi_{\hspace{-2pt}j},\overrightarrow\xi_{\hspace{-2pt}j+1})$ in case $\xi_j=\xi_{j+1}=\alpha$ and $\xi_{j-1}\neq \alpha$;
\item Next, if $\gamma$ crosses with $\alpha'=(b,d)$, then add
$(b,d),(d,b)$ or $(d,b),(b,d)$ to make the sequences to complete $(T'^o,\gamma)$-paths;
\end{enumerate}

If $m(\overrightarrow\xi,\alpha)=-1$, then $\xi_j=\alpha$ for some even $j$ such that $\xi_{j-1}\neq \alpha\neq \xi_{j+1}$. The set $\phi^{T^o}_\alpha(\overrightarrow\xi)$ contains all sequences of oriented arcs $\overrightarrow\xi'$, where $\overrightarrow\xi'$ can be obtained via the following recipe:

\begin{enumerate}[$(i)$]
\item First, $(\overrightarrow\xi_{\hspace{-2pt}j-1},\overrightarrow\xi_{\hspace{-2pt}j},\overrightarrow\xi_{\hspace{-2pt}j+1})=((b,a),(a,c),(c,d))$; $((b,c),(c,a),(a,d))$; $((d,a),(a,c),(c,b))$ or $((d,c),(c,a),(a,b))$ by Lemma \ref{Lem-mid}.
\begin{enumerate}[$(a)$]
\item if $(\overrightarrow\xi_{\hspace{-2pt}j-1},\overrightarrow\xi_{\hspace{-2pt}j},\overrightarrow\xi_{\hspace{-2pt}j+1})=((b,a),(a,c),(c,d))$ or $((b,c),(c,a),(a,d))$, then replace it by $(b,d)$;
\item if $(\overrightarrow\xi_{\hspace{-2pt}j-1},\overrightarrow\xi_{\hspace{-2pt}j},\overrightarrow\xi_{\hspace{-2pt}j+1})=((d,a),(a,c),(c,b))$ or $((d,c),(c,a),(a,b))$, then replace it by $(d,b)$;
\end{enumerate}
\item Next, if $\gamma$ crosses with $\alpha'=(b,d)$, then add
$(b,d),(d,b)$ or $(d,b),(b,d)$ to make the sequences to complete $(T'^o,\gamma)$-paths.
\end{enumerate}

\begin{remark}
In particular, if $\gamma$ crosses neither $\alpha$ nor $\alpha'$ then $\mathcal {CP}(T^o,\gamma)=\mathcal {CP}(T'^o,\gamma)$ and $\phi^{T^o}_\alpha$ is the identity map.
\end{remark}

In summary, we obtain the following proposition.

\begin{proposition}\label{Lem-parp}
With the foregoing notation. For any complete $(T^o,\gamma)$-path $\overrightarrow\xi$, we have
\begin{enumerate}[$(1)$]
\item $\phi^{T^o}_\alpha(\overrightarrow \xi)\subset \mathcal{CP}(T'^o,\gamma)$;
\item
$|\phi^{T^o}_\alpha(\overrightarrow \xi)|=2^{[m(\overrightarrow\xi,\alpha)]_+}$,
moreover, if $m(\overrightarrow\xi,\alpha)=1$, assume that $\phi^{T^o}_\alpha(\overrightarrow \xi)=\{\overrightarrow\xi',\overrightarrow\xi''\}$, then $\overrightarrow\xi'$ and $\overrightarrow\xi''$ are related by a twist at $\alpha'$;
\item $\phi^{T^o}_\alpha(\overrightarrow \xi)=\phi^{T^o}_\alpha(\overrightarrow \zeta)$ for some $\zeta\neq \xi$ if and only if
$m(\overrightarrow\xi,\alpha)=-1$,
and $\overrightarrow \xi$ and $\overrightarrow \zeta$ are related by a twist on $\alpha$;
\item for any $\overrightarrow\xi'\in \phi^{T^o}_\alpha(\overrightarrow \xi)$, we have
$m(\overrightarrow{\xi},\alpha)=-m(\overrightarrow\xi',\alpha')$;
\item for any $(a,b), (b,c), (c,d), (a,d), (a,c)\neq \zeta\in T^o$ and $\overrightarrow\xi'\in \phi^{T^o}_\alpha(\overrightarrow \xi)$, we have
$m(\overrightarrow\xi,\zeta)=m(\overrightarrow\xi',\zeta)$.
\end{enumerate}

\end{proposition}

Similarly, we can construct $\phi^{T^o}_{\alpha'}$.

\begin{theorem}\label{parbi1}
With the foregoing notation. $\phi^{T^o}_\alpha$ and $\phi^{T'^o}_{\alpha'}$ are partition bijections between $\mathcal{CP}(T^o,\gamma)$ and  $\mathcal{CP}(T'^o,\gamma)$, moreover, $\phi^{T^o}_\alpha$ and  $\phi^{T'^o}_{\alpha'}$ are inverse to each other.
\end{theorem}

\begin{proof}
We first show that $\phi^{T^o}_\alpha$ and $\phi^{T'^o}_{\alpha'}$ are partition bijections. We shall only consider $\phi^{T^o}_\alpha$. By Proposition \ref{Lem-parp} (2), we have $\phi^{T^o}_\alpha(\overrightarrow\xi)\neq \emptyset$ for any $\overrightarrow\xi\in \mathcal{CP}(T^o,\gamma)$.

Suppose that $\phi^{T^o}_\alpha(\overrightarrow\xi)\cap \phi^{T^o}_\alpha(\overrightarrow\zeta)\neq \emptyset$ for some $\overrightarrow\xi,\overrightarrow\zeta\in \mathcal{CP}(T^o,\gamma)$. Assume that $\overrightarrow\xi'\in \phi^{T^o}_\alpha(\overrightarrow\xi)\cap \phi^{T^o}_\alpha(\overrightarrow\zeta)$. Then $m(\overrightarrow{\xi},\alpha)=m(\overrightarrow{\zeta},\alpha)=-m(\overrightarrow{\xi'},\alpha')$ by Proposition \ref{Lem-parp} (4). If $m(\overrightarrow{\xi},\alpha)=m(\overrightarrow{\zeta},\alpha)\leq 0$ then $\phi^{T^o}_\alpha(\overrightarrow\xi)= \phi^{T^o}_\alpha(\overrightarrow\zeta)=\{\overrightarrow\xi'\}$ by Proposition \ref{Lem-parp} (2). If $m(\overrightarrow{\xi},\alpha)=m(\overrightarrow{\zeta},\alpha)=1$, then $\phi^{T^o}_\alpha(\overrightarrow\xi)= \phi^{T^o}_\alpha(\overrightarrow\zeta)=\{\overrightarrow\xi',\overrightarrow \xi''\}$ by Proposition \ref{Lem-parp} (2), where $\overrightarrow\xi''$ is the complete $(T'^o,\gamma)$-path related to $\overrightarrow\xi'$ by a twist on $\alpha'$. Therefore we have either $\phi^{T^o}_\alpha(\overrightarrow\xi)\cap \phi^{T^o}_\alpha(\overrightarrow\zeta)=\emptyset$ or $\phi^{T^o}_\alpha(\overrightarrow\xi)=\phi^{T^o}_\alpha(\overrightarrow\zeta)$ for any $\overrightarrow\xi,\overrightarrow\zeta\in \mathcal{CP}(T^o,\gamma)$.

According to the construction of $\phi^{T^o}_\alpha$ and $\phi^{T'^o}_{\alpha'}$, we see that $\overrightarrow\xi'\in \phi^{T^o}_\alpha(\overrightarrow\xi)$ if and only if $\overrightarrow\xi\in \phi^{T'^o}_{\alpha'}(\overrightarrow\xi')$ for any $\overrightarrow\xi\in \mathcal{CP}(T^o,\gamma)$ and $\overrightarrow\xi'\in \mathcal{CP}(T'^o,\gamma)$. Thus $\bigcup_{\overrightarrow\xi\in \mathcal{CP}(T^o,\gamma)}\phi^{T^o}_\alpha(\overrightarrow\xi)=\mathcal{CP}(T'^o,\gamma)$. Therefore, $\phi^{T^o}_\alpha$ is a partition bijection by Remark \ref{Rmk-par}. Similarly, $\phi^{T'^o}_\alpha$ is a partition bijection. By Proposition \ref{Pro-parin}, $\phi^{T^o}_\alpha$ and $\phi^{T'^o}_{\alpha'}$ are inverse to each other.
\end{proof}

\subsubsection{Partition bijection $\phi^{T^o}_{\alpha_1,\cdots,\alpha_k}$}
Let $\alpha_1,\cdots,\alpha_k$ be a set of arcs in $T^o$ such that $\alpha_i,\alpha_j$ are not two sides of any triangle in $T^o$ for any $i\neq j$. Denote $T'^o=\mu_{\alpha_k}\circ \cdots \circ \mu_{\alpha_1}(T^o)$. Denote by $\alpha'_i$ the arc obtained from $T^o$ by flip at $\alpha_i$. Assume that $\alpha_i=(a_i,c_i)$ and $\alpha_i$ is a diagonal of the quadrilateral $(a_i,b_i,c_i,d_i)$ in $T^o$, where $a_i,b_i,c_i,d_i$ are in the clockwise order.

For any given complete $(T^o,\gamma)$-path $\overrightarrow\xi=(\overrightarrow\xi_{\hspace{-2pt}1},\cdots,\overrightarrow\xi_{\hspace{-2pt}2c+1})$, we construct a subset $\phi^{T^o}_{\alpha_1,\cdots,\alpha_k}(\overrightarrow\xi)$ of complete $(T'^o,\gamma)$-paths, which is obtained from $\overrightarrow\xi$ via the following receipt:
\begin{enumerate}[$(i)$]
\item First, remove all the pairs $(\overrightarrow\xi_{\hspace{-2pt}j},\overrightarrow\xi_{\hspace{-2pt}j+1})$ such that $\xi_j=\xi_{j+1}=\alpha_i$ for some $i\in \{1,\cdots,k\}$ and $\xi_{j-1}\neq \alpha_i$;
\item Next, for any $i\in \{1,\cdots,k\}$,
      \begin{enumerate}[$(a)$]
      \item if $m(\overrightarrow\xi,\alpha_i)=1$, we have $\xi_{j_i}=\alpha_i$ for some odd $j_i$, then replace $\overrightarrow\xi_{\hspace{-2pt}j_i}$ by $(a_i,b_i),(b_i,d_i),(d_i,c_i)$ or $(a_i,d_i),(d_i,b_i),(b_i,c_i) $ in case $\overrightarrow\xi_{\hspace{-2pt}j_i}=(a,c)$ and replace $\overrightarrow\xi_{\hspace{-2pt}j_i}$ by $(c_i,b_i),(b_i,d_i),(d_i,a_i)$ or $(c_i,d_i),(d_i,b_i),(b_i,a_i)$ in case $\overrightarrow\xi_{\hspace{-2pt}j_i}=(c_i,a_i)$;
      \item if $m(\overrightarrow\xi,\alpha_i)=-1$, we have $\xi_{j_i}=\alpha_i$ for some even $j_i$, then replace $(\overrightarrow\xi_{\hspace{-2pt}j_i-1},\overrightarrow\xi_{\hspace{-2pt}j_i},\overrightarrow\xi_{\hspace{-2pt}j_i+1})$ by $(b_i,d_i)$ in case $(\overrightarrow\xi_{\hspace{-2pt}j_i-1},\overrightarrow\xi_{\hspace{-2pt}j_i},\overrightarrow\xi_{\hspace{-2pt}j_i+1})=((b_i,a_i),(a_i,c_i),(c_i,d_i))$ or $((b_i,c_i),(c_i,a_i),(a_i,d_i))$, and replace $(\overrightarrow\xi_{\hspace{-2pt}j_i-1},\overrightarrow\xi_{\hspace{-2pt}j_i},\overrightarrow\xi_{\hspace{-2pt}j_i+1})$ by $(d_i,b_i)$ in case $(\overrightarrow\xi_{\hspace{-2pt}j_i-1},\overrightarrow\xi_{\hspace{-2pt}j_i},\overrightarrow\xi_{\hspace{-2pt}j_i+1})=((d_i,a_i),(a_i,c_i),(c_i,b_i))$ or $((d_i,c_i),(c_i,a_i),(a_i,b_i))$;
      \end{enumerate}
\item Last, for all the $\alpha'_i=(b_i,d_i)$ such that $\gamma$ crosses with $\alpha'_i$, then add
$(b_i,d_i),(d_i,b_i)$ or $(d_i,b_i),(b_i,d_i)$ to make the sequences to complete $(T'^o,\gamma)$-paths.
\end{enumerate}

As $\alpha_i,\alpha_j$ are not two sides of any triangle in $T^o$ for any $i\neq j$, the above construction is well-defined and we have
\begin{equation}\label{Eq-par}
\phi^{T^o}_{\alpha_1,\cdots,\alpha_k}(\overrightarrow\xi)=\phi^{\mu_{\alpha_{k-1}}\circ \cdots\circ \mu_{\alpha_1}T^o}_{\alpha_k}\left(\cdots\phi^{\mu_{\alpha_1}T^o}_{\alpha_2}(\phi^{T^o}_{\alpha_1}(\overrightarrow\xi))\right),
\end{equation}
moreover, $\phi^{T^o}_{\alpha_1,\cdots,\alpha_k}$ does not depend on the order of $\alpha_1,\cdots,\alpha_k$.

We can also construct a subset $\phi^{T'^o}_{\alpha'_1,\cdots,\alpha'_k}(\overrightarrow \xi')$ for any complete $(T'^o,\gamma)$-path $\overrightarrow \xi'$.

We have the following proposition which generalizes Proposition \ref{Lem-parp}.

\begin{proposition}\label{prop-par11}
Let $\alpha_1,\cdots,\alpha_k$ be a set of arcs in $T^o$ such that $\alpha_i$ and $\alpha_j$ are not two sides of any triangle in $T^o$ for any $i\neq j$. For any complete $(T^o,\gamma)$-path $\overrightarrow\xi$, we have
\begin{enumerate}[$(1)$]
\item $\phi^{T^o}_{\alpha_1,\cdots,\alpha_k}(\overrightarrow \xi)\subset \mathcal{CP}(T'^o,\gamma)$,
\item
$|\phi^{T^o}_{\alpha_1,\cdots,\alpha_k}(\overrightarrow\xi)|=2^{\sum_{i=1}^k[m(\overrightarrow\xi,\alpha_i)]_+}$,
moreover, for any $\alpha_i$ with $m(\overrightarrow\xi,\alpha_i)=1$ and $\overrightarrow\xi'\in \phi^{T^o}_{\alpha_1,\cdots,\alpha_k}(\overrightarrow\xi)$, there exists $\overrightarrow\xi''\in \phi^{T^o}_{\alpha_1,\cdots,\alpha_k}(\overrightarrow\xi)$ such that $\overrightarrow\xi'$ and $\overrightarrow\xi''$ are related by a twist on $\alpha'_i$;
\item if $m(\overrightarrow\xi,\alpha_i)=-1$ for some $\alpha_i$ then we have $\phi^{T^o}_{\alpha_1,\cdots,\alpha_k}(\overrightarrow\xi)=\phi^{T^o}_{\alpha_1,\cdots,\alpha_k}(\overrightarrow\zeta)$, where $\overrightarrow\zeta$ is the complete $(\prod_{i=1}^k\mu_{\alpha_i}(T^o),\gamma)$-path which is related to $\overrightarrow\xi$ by a twist at $\alpha_i$;
\item for any $i\in \{1,\cdots,k\}$ and $\overrightarrow\xi'\in \phi^{T^o}_{\alpha_1,\cdots,\alpha_k}(\overrightarrow\xi)$, we have
$m(\overrightarrow{\xi},\alpha_i)=-m(\overrightarrow\xi',\alpha'_i).$
\item If $\zeta\in T^o$ and $\zeta\neq (a_i,b_i), (b_i,c_i), (c_i,d_i), (a_i,d_i), (a_i,c_i)$ for all $i$, then for any $\overrightarrow\xi'\in \phi^{T^o}_{\alpha_1,\cdots,\alpha_k}(\overrightarrow\xi)$, we have
$m(\overrightarrow\xi,\zeta)=m(\overrightarrow\xi',\zeta).$
\end{enumerate}

\end{proposition}

The following theorem generalizes Theorem \ref{parbi1}.

\begin{theorem}\label{thm-par11}
With the foregoing notation. $\phi^{T^o}_{\alpha_1,\cdots,\alpha_k}$ and $\phi^{T'^o}_{\alpha'_1,\cdots,\alpha'_k}$ are partition bijections between $\mathcal{CP}(T^o,\gamma)$ and  $\mathcal{CP}(T'^o,\gamma)$ inverse to each other.
\end{theorem}

\begin{proof}
It is clear that $\phi^{T^o}_{\alpha_1,\cdots,\alpha_k}(\overrightarrow\xi)\neq \emptyset$ for any $\overrightarrow\xi\in\mathcal{CP}(T^o,\gamma)$.

Suppose that $\phi^{T^o}_{\alpha_1,\cdots,\alpha_k}(\overrightarrow\xi)\cap \phi^{T^o}_{\alpha_1,\cdots,\alpha_k}(\overrightarrow\zeta)\neq \emptyset$ for some $\overrightarrow\xi$ and $\overrightarrow\zeta$. Assume that $\overrightarrow\xi'\in\phi^{T^o}_{\alpha_1,\cdots,\alpha_k}(\overrightarrow\xi)\cap \phi^{T^o}_{\alpha_1,\cdots,\alpha_k}(\overrightarrow\zeta)$. Then $m(\overrightarrow{\xi},\alpha_i)=m(\overrightarrow\zeta,\alpha_i)=-m(\overrightarrow\xi',\alpha'_i)$ for any $i$ by Proposition \ref{prop-par11} (4). By Proposition \ref{prop-par11} (2), we see that $\phi^{T^o}_{\alpha_1,\cdots,\alpha_k}(\overrightarrow\xi)=\phi^{T^o}_{\alpha_1,\cdots,\alpha_k}(\overrightarrow\zeta)$ contains all $\overrightarrow \xi''$ which are related to $\overrightarrow \xi'$ by a sequence of twists on some $\alpha'_i$ with $m(\overrightarrow\xi,\alpha_i)=1$. Therefore, we have
$\phi^{T^o}_{\alpha_1,\cdots,\alpha_k}(\overrightarrow\xi)=\phi^{T^o}_{\alpha_1,\cdots,\alpha_k}(\overrightarrow\zeta)$.

For any $\overrightarrow\xi\in \mathcal{CP}(T^o,\gamma)$ and $\overrightarrow\xi'\in \mathcal {CP}(T'^o,\gamma)$, we show that $\overrightarrow\xi'\in \phi^{T^o}_{\alpha_1,\cdots,\alpha_k}(\overrightarrow\xi)$ if and only if $\overrightarrow\xi\in \phi^{T'^o}_{\alpha'_1,\cdots,\alpha'_k}(\overrightarrow\xi')$ by induction on $k$. If $k=1$, it follows by Theorem \ref{parbi1}. Suppose that it is true for the cases less than $k$. If $\overrightarrow\xi'\in \phi^{T^o}_{\alpha_1,\cdots,\alpha_k}(\overrightarrow\xi)$, then $\overrightarrow\xi'\in \phi^{T^o}_{\alpha_k}(\overrightarrow \eta)$ for some $\overrightarrow \eta\in \phi^{T^o}_{\alpha_1,\cdots,\alpha_{k-1}}(\overrightarrow\xi)$. By induction hypothesis, we have $\overrightarrow \xi\in \phi^{\prod_{i=1}^{k-1}\mu_{\alpha_i}T^o}_{\alpha'_1,\cdots,\alpha'_{k-1}}(\overrightarrow\eta)$ and $\overrightarrow\eta\in \phi^{\prod_{i=1}^{k}\mu_{\alpha_i}T^o}_{\alpha'_k}(\overrightarrow\xi')$. Thus,
$\overrightarrow \xi\in \phi^{\prod_{i=1}^{k-1}\mu_{\alpha_i}T^o}_{\alpha'_1,\cdots,\alpha'_{k-1}}(\phi^{\prod_{i=1}^{k}\mu_{\alpha_i}T^o}_{\alpha'_k}(\overrightarrow\xi'))=\phi^{T'^o}_{\alpha'_1,\cdots,\alpha'_{k}}(\overrightarrow\xi').$ The statement is proved.

By Remark \ref{Rmk-par}, $\phi^{T^o}_{\alpha_1,\cdots,\alpha_k}$ is a partition bijection. Similarly, $\phi^{T'^o}_{\alpha'_1,\cdots,\alpha'_k}$ is a partition bijection. By Proposition \ref{Pro-parin}, $\phi^{T^o}_{\alpha_1,\cdots,\alpha_k}$ and $\phi^{T'^o}_{\alpha'_1,\cdots,\alpha'_k}$ are inverse to each other.
\end{proof}

\begin{remark}\label{rem-eq}
$\phi^{T^o}_{\alpha_1,\cdots,\alpha_k}(\overrightarrow\xi)\cap \phi^{T^o}_{\alpha_1,\cdots,\alpha_k}(\overrightarrow\zeta)\neq \emptyset$ if and only if $\phi^{T^o}_{\alpha_1,\cdots,\alpha_k}(\overrightarrow\xi)= \phi^{T^o}_{\alpha_1,\cdots,\alpha_k}(\overrightarrow\zeta)$ if and only if $\overrightarrow\xi$ and $\overrightarrow\zeta$ are related by a sequence of twists on some $\alpha_i$ such that $m(\overrightarrow\xi,\alpha_i)=-1$.
\end{remark}

\subsubsection{Partition bijection $\varphi^{T^o}_\alpha$} Recall that there are natural bijective maps from $\mathcal {CP}(T^o,\gamma)$ to $\mathcal P(G_{T^o,\gamma})$ and from $\mathcal {CP}(T'^o,\gamma)$ to $\mathcal P(G_{T'^o,\gamma})$, see Subsection \ref{Sec-CPB}. Therefore, the partition bijection $\phi^{T^o}_\alpha$ from $\mathcal {CP}(T^o,\gamma)$ to  $\mathcal {CP}(T'^o,\gamma)$ induces a partition bijection $\varphi^{T^o}_\alpha$ from $\mathcal {P}(G_{T^o,\gamma})$ to  $\mathcal {P}(G_{T'^o,\gamma})$, that is, we have the following commutative diagram.

\medskip

\centerline{\xymatrix{
  & \mathcal {CP}(T^o,\gamma) \ar[d]^\cong \ar[r]^{\phi^{T^o}_{\alpha}}  &  \mathcal {CP}(T'^o,\gamma)\ar[d]_\cong           \\
  & \mathcal {P}(G_{T^o,\gamma}) \ar[r]^{\varphi^{T^o}_{\alpha}}  &  \mathcal {P}(G_{T'^o,\gamma}).                       }}

\medskip

We also have the inverse partition bijection $\varphi^{T'^o}_{\alpha'}:\mathcal {P}(G_{T'^o,\gamma})\to \mathcal {P}(G_{T^o,\gamma})$ of $\varphi^{T^o}_\alpha$.

Under the bijection between $\mathcal {CP}(T^o,\gamma)$ and $\mathcal P(G_{T^o,\gamma})$, for any $P\in \mathcal P(G_{T^o,\gamma})$, denote by $\overrightarrow\xi(P)$ the perfect matching of $G_{T^o,\gamma}$ corresponding to $P$.
For any $\zeta\in T^o$, let
\begin{equation*}
\hat m(P,\zeta)=m(\overrightarrow\xi(P),\zeta),
\end{equation*}
where $m(\overrightarrow\xi(P),\zeta)$ is given by (\ref{equ-m1}).

In view of (\ref{equ-m1}) and (\ref{Equ-mp}) (\ref{Equ-n}), by the bijection between $\mathcal P(G_{T^o,\gamma})$ and $\mathcal {CP}(T^o,\gamma)$, we have
\begin{equation*}
\hat m(P,\zeta)=m(P;\zeta)-n(G_{T^o,\gamma};\zeta).
\end{equation*}

Similarly we can define $\hat m(P',\zeta)$ for any $P'\in \mathcal P(G_{T'^o,\gamma})$ and $\zeta\in T'^o$.

Assume that $\alpha=(a,c)$ and $\alpha$ is the diagonal of the quadrilateral $(a,b,c,d)$ in $T^o$. Suppose that $a,b,c,d$ are in the clockwise order.

\begin{proposition}\label{Lem-tile-1}
With the foregoing notation. Let $P\in \mathcal P(G_{T^o,\gamma})$.
\begin{enumerate}[$(1)$]
\item We have
$|\varphi_{\alpha}^{T^o}(P)|=2^{[\hat m(P,\alpha)]_+}$,
moreover,
\begin{enumerate}
\item if $\hat m(P,\alpha)=1$, then there is a tile $G'$ of $G_{T'^o,\gamma}$ with diagonal labeled $\alpha'$ such that any $P'\in \varphi_{\alpha}^{T^o}$ can twist on $G'$ and $\varphi_{\alpha}^{T^o}$ is closed under the twist on $G'$;
\item if $\hat m(P,\alpha)=-1$, then there is a tile $G$ of $G_{T^o,\gamma}$ with diagonal labeled $\alpha$ such that $P$ can twist on $G$ and $\varphi_{\alpha}^{T^o}(P)=\varphi_{\alpha}^{T^o}(\mu_{G}P)$;
\end{enumerate}
\item For any $P'\in \varphi^{T^o}_\alpha(P)$, we have
$\hat m(P,\alpha)=-\hat m(P',\alpha').$
\item For any $\zeta\in T^o$ such that $\zeta\neq (a,b), (b,c), (c,d), (a,d), (a,c)$ and $P'\in \phi^{T^o}_\alpha(P)$, we have
\begin{equation*}
\hat m(P,\zeta)=\hat m(P',\zeta).
\end{equation*}
\end{enumerate}
\end{proposition}

\begin{proof}
It follows immediately by Proposition \ref{Lem-parp}.
\end{proof}

\begin{proposition}\label{prop-com}
Assume $P$ can twist on a tile $G_l$ with diagonal labeled $\tau_{i_l}$ and $P>\mu_{G_l}P$.
   \begin{enumerate}[$(1)$]
    \item if $\tau_{i_l}\neq \alpha,(a,b),(b,c),(c,d),(a,d)$ then there is a tile $G'_{[l]}$ of $G_{T'^o,\gamma}$ with diagonal labeled $\tau_{i_l}$ such that: any $P'\in \varphi_\alpha^{T^o}(P)$ can twist on $G'_{[l]}$, $P'>\mu_{G'_{[l]}}(P')$, and
        $$\varphi_\alpha^{T^o}(\mu_{G_l}P)=\mu_{G'_{[l]}}(\varphi_\alpha^{T^o}(P)).$$
    \item if $\tau_{i_l}=\alpha$, then $\hat m(P,\alpha)=-1$ and
        $\{P,\mu_{G_l}P\}=\varphi_{\alpha'}^{T'^o}(P'),$
    where $\varphi_{\alpha}^{T^o}(P)=\{P'\}$.
    \item if $\tau_{i_l}\in \{(a,b),(c,d)\}$, then $\hat m(P,\alpha)=\hat m(\mu_{G_l}P,\alpha)-1$.
     For $P'\in \varphi_\alpha^{T^o}(P)$, there are tiles $G'_{[l]}$ and $G'_{[l']}$  of $G_{T'^o,\gamma}$ with diagonal labeled $\tau_{i_l}$ and $(b,d)$, respectively, such that
     \begin{enumerate}
     \item $|[l]-[l']|=1$,
     \item $P'$ can twist on $G'_{[l]}$ and $\mu_{G'_{[l]}}P'$ can twist on $G'_{[l']}$,
     \item $P'>\mu_{G'_{[l]}}P'>\mu_{G'_{[l']}}\mu_{G'_{[l]}}P'$,
     \item $\mu_{G'_{[l]}}P'\in \varphi_\alpha^{T^o}(\mu_{G_l}P)$.
    \end{enumerate}
     \item if $\tau_{i_l}\in \{(b,c),(a,d)\}$, then $\hat m(P,\alpha)=\hat m(\mu_{G_l}P,\alpha)+1$.
     For $P'\in \varphi_\alpha^{T^o}(\mu_{G_l}P)$, there are tiles $G'_{[l]}$ and $G'_{[l']}$  of $G_{T'^o,\gamma}$ with diagonal labeled $\tau_{i_l}$ and $(b,d)$, respectively, such that
     \begin{enumerate}
     \item $|[l]-[l']|=1$,
     \item $P'$ can twist on $G'_{[l]}$ and $\mu_{G'_{[l]}}P'$ can twist on $G'_{[l']}$,
     \item $P'<\mu_{G'_{[l]}}P'<\mu_{G'_{[l']}}\mu_{G'_{[l]}}P'$,
     \item $\mu_{G'_{[l]}}P'\in \varphi_\alpha^{T^o}(P)$.
    \end{enumerate}
      \end{enumerate}
\end{proposition}

\begin{proof}
It follows by the construction of $\phi^{T^o}_\alpha$.
\end{proof}

\begin{example}
Let $\Sigma$ be the hexagon. If $T^o$ and $T'^o$ are triangulations as shown in Figure \ref{Fig:ex},

\begin{figure}[h]
\centerline{\includegraphics{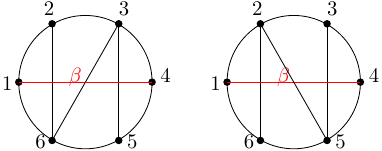}}
\caption{}\label{Fig:ex}
\end{figure}

then the snake graphs $G_{T^o,\gamma}$, $G_{T'^o,\gamma}$ and the partition bijections $\varphi_{\alpha}^{T^o}$, $\varphi_{\alpha'}^{T'^o}$ are shown in Figure \ref{Fig:ex2}.

\begin{figure}[h]
\centerline{\includegraphics{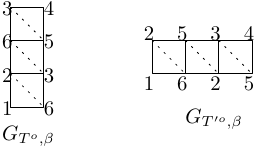}}
\caption{}\label{Fig:ex2}
\end{figure}

If $T^o$ and $T'^o$ are triangulations as shown in Figure \ref{Fig:ex3}, then the snake graphs $G_{T^o,\gamma}$, $G_{T'^o,\gamma}$ and the partition bijections $\varphi_{\alpha}^{T^o}$, $\varphi_{\alpha'}^{T'^o}$ are shown in Figure \ref{Fig:ex3}.

\begin{figure}[h]
\centerline{\includegraphics{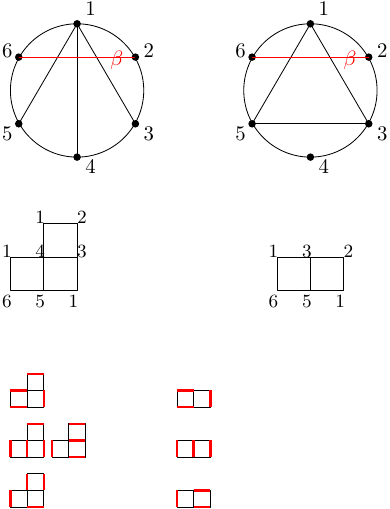}}
\caption{}\label{Fig:ex3}
\end{figure}
\end{example}

According to Lemma \ref{max-min}, we have the following observation.

\begin{lemma}\label{lem-p-}
\begin{enumerate}[$(1)$]
\item If $\hat m(P_-,\alpha)=1$, then one of the following cases happens: (1) $\gamma$ crosses $(a,b), \alpha, (c,d)$ consecutively; (2) $\gamma$ starts from $a$ then crosses $(c,d)$; (3) $\gamma$ starts from $c$ then crosses $(a,b)$.

\item If $\hat m(P_-,\alpha)=-1$, then one of the following cases happens: (1) $\gamma$ crosses $(a,d), \alpha, (b,c)$ consecutively; (2) $\gamma$ starts from $d$ then crosses $\alpha$, $(b,c)$ sequentially; (3) $\gamma$ starts from $b$ then crosses $\alpha, (a,d)$ sequentially.
\end{enumerate}
\end{lemma}

\subsubsection{Partition bijection $\varphi^{T^o}_{\alpha_1,\cdots,\alpha_k}$} Let $\alpha_1,\cdots,\alpha_k$ be a set of arcs in $T^o$ such that $\alpha_i,\alpha_j$ are not two sides of any triangle in $T^o$ for any $i\neq j$. Denote by $\alpha'_i$ the arc obtained from $T^o$ by flip at $\alpha_i$. Assume that $\alpha_i=(a_i,c_i)$ and $\alpha_i$ is the diagonal of the quadrilateral $(a_i,b_i,c_i,d_i)$ in $T^o$, where $a_i,b_i,c_i,d_i$ are in the clockwise order. Denote $T'^o=\mu_{\alpha_k}\circ \cdots \circ \mu_{\alpha_1}(T^o)$.

The partition bijection $\phi^{T^o}_{\alpha_1,\cdots,\alpha_k}:\mathcal {CP}(T^o,\gamma) \to  \mathcal {CP}(T'^o,\gamma)$ induces a partition bijection $\varphi^{T^o}_{\alpha_1,\cdots,\alpha_k}:\mathcal {P}(G_{T^o,\gamma}) \to  \mathcal {P}(G_{T'^o,\gamma})$, i.e., we have the following commutative diagram.

{\centering \xymatrixcolsep{7pc}\xymatrix{
  & \mathcal {CP}(T^o,\gamma) \ar[d]^\cong \ar[r]^{\phi^{T^o}_{\alpha_1,\cdots,\alpha_k}}  &  \mathcal {CP}(T'^o,\gamma)\ar[d]_\cong           \\
  & \mathcal {P}(G_{T^o,\gamma}) \ar[r]^{\varphi^{T^o}_{\alpha_1,\cdots,\alpha_k}}  &  \mathcal {P}(G_{T'^o,\gamma}).                       }}

We have the inverse partition bijection $\varphi^{T'^o}_{\alpha'_1,\cdots,\alpha'_k}\mathcal {P}(G_{T'^o,\gamma})\to \mathcal {P}(G_{T^o,\gamma})$ of $\varphi^{T^o}_{\alpha_1,\cdots,\alpha_k}$.

The following proposition generalizes Proposition \ref{Lem-tile-1}, follows by Proposition \ref{prop-par11}.

\begin{proposition}\label{Lem-tile-2}
With the foregoing notation. Let $P\in \mathcal P(G_{T^o,\gamma})$.
\begin{enumerate}[$(1)$]
\item We have
$|\varphi_{\alpha_1,\cdots,\alpha_k}^{T^o}(P)|=2^{\sum_{i=1}^k[\hat m(P,\alpha_i)]_+}$,
moreover,
\begin{enumerate}
\item if $\hat m(P,\alpha_i)=1$ for some $i$, then there is a tile $G'$ of $G_{T'^o,\gamma}$ with diagonal labeled $\alpha'_i$ such that any $P'\in \varphi_{\alpha_1,\cdots,\alpha_k}^{T^o}$ can twist on $G'$ and $\varphi_{\alpha_1,\cdots,\alpha_k}^{T^o}$ is closed under the twist on $G'$;
\item if $\hat m(P,\alpha_i)=-1$ for some $i$, then there is a tile $G$ of $G_{T^o,\gamma}$ with diagonal labeled $\alpha_i$ such that $P$ can twist on $G$ and $\varphi_{\alpha_1,\cdots,\alpha_k}^{T^o}(P)=\varphi_{\alpha_1,\cdots,\alpha_k}^{T^o}(\mu_{G}P)$;
\end{enumerate}
\item For any $P'\in \varphi^{T^o}_{\alpha_1,\cdots,\alpha_k}(P)$ and $i\in \{1,\cdots,k\}$, we have
%\begin{equation*}
$\hat m(P,\alpha_i)=-\hat m(P',\alpha'_i).$
\item For any $P'\in \phi^{T^o}_{\alpha_1,\cdots,\alpha_k}(P)$ and $\zeta\in T^o$ with $\zeta\neq (a_i,b_i), (b_i,c_i), (c_i,d_i), (a_i,d_i), (a_i,c_i)$ for all $i$, we have
%\begin{equation*}
$\hat m(P,\zeta)=\hat m(P',\zeta).$
\item For any $P'\in \phi^{T^o}_{\alpha_1,\cdots,\alpha_k}(P)$ and $\zeta\in \{(a_i,b_i), (b_i,c_i), (c_i,d_i), (a_i,d_i)\}$, the number of non-$\alpha_i$-mutable edges labeled $\zeta$ in $P$ equals the number of non-$\alpha'_i$-mutable edges labeled $\zeta$ in $P'$.
    %\huang{mutable does not define}
\end{enumerate}
\end{proposition}

The following proposition follows by Proposition \ref{prop-com}.

\begin{proposition}\label{Prop-mutation3}
Assume $P$ can twist on a tile $G_l$ with diagonal labeled $\tau_{i_l}$ and $P>\mu_{G_l}P$.
   \begin{enumerate}[$(1)$]
    \item If $\tau_{i_l}\neq (a_i,b_i), (b_i,c_i), (c_i,d_i), (a_i,d_i), \alpha_i$ for any $i$, then there is a tile $G'_{[l]}$ of $G_{T'^o,\gamma}$ with diagonal labeled $\tau_{i_l}$ such that: for any $P'\in \varphi_{\alpha_1,\cdots,\alpha_k}^{T^o}(P)$, $P'$ can twist on $G'_{[l]}$, $P'>\mu_{G'_{[l]}}(P')$, and
        $\varphi_{\alpha_1,\cdots,\alpha_k}^{T^o}(\mu_{G_l}P)=\mu_{G'_{[l]}}(\varphi_{\alpha_1,\cdots,\alpha_k}^{T^o}(P)).$
    \item If $\tau_{i_l}=\alpha_i$ for some $i$, then
    %\huang{$\hat m(P,\alpha_i)=-1$ and delete}
    $\mu_{G_l}P \in \varphi_{\alpha'_1,\cdots,\alpha'_k}^{T'^o}(P')$
    for any $P'\in \varphi_{\alpha_1,\cdots,\alpha_k}^{T^o}(P)$.
\item If $\tau_{i_l}\in \{(a_i,b_i), (c_i,d_i)\}$ for some $i$, then $\hat m(P,\alpha_i)=\hat m(\mu_{G_l}P,\alpha_i)-1$. Moreover, there is a tile $G'_{[l]}$ of $G_{T'^o,\gamma}$ with diagonal labeled $\tau_{i_l}$ such that: any $P'\in \varphi_{\alpha_1,\cdots,\alpha_k}^{T^o}(P)$ can twist on $G'_{[l]}$ and
    $P'>\mu_{G'_{[l]}}P'\in \varphi_{\alpha_1,\cdots,\alpha_k}^{T^o}(\mu_{G_l}P)$.
\item If $\tau_{i_l}\in \{(b_i,c_i), (a_i,d_i)\}$ for some $i$, then $\hat m(P,\alpha_i)=\hat m(\mu_{G_l}P,\alpha_i)+1$. Moreover, there is a tile $G'_{[l]}$ of $G_{T'^o,\gamma}$ with diagonal labeled $(a_i,b_i)$ such that: any $P'\in \varphi_{\alpha_1,\cdots,\alpha_k}^{T^o}(\mu_{G_l}P)$ can twist on $G'_{[l]}$ and
    $P'<\mu_{G'_{[l]}}P'\in \varphi_{\alpha_1,\cdots,\alpha_k}^{T^o}(\mu_{G_l}P)$.
\end{enumerate}
\end{proposition}

\subsection{General case} In this subsection, we assume that $\Sigma$ is an arbitrary orbifold without orbifold points of weight $2$. Let $T^o$ be an ideal triangulation and $\alpha\in T^o$ be a non-self-folded arc. Denote $T'^o=\mu_\alpha(T^o)$ and by $\alpha'$ the new edge obtained. Let $\gamma$ be an arc in $\Sigma$. Let $P_n$ be the canonical polygon of $\gamma$ in $T^o$ with morphism $f:P_n\to \Sigma$ and triangle $\hat T^o$. Assume that $f(\hat\gamma)=\gamma$ for some diagonal $\hat\gamma$ in $P_n$.

Assume that $\alpha$ is a diagonal of the quadrilateral $(\alpha_1,\alpha_2,\alpha_3,\alpha_4)$ in $T^o$ and $\alpha_1,\alpha_3$ are in the clockwise direction of $\alpha$.

\begin{figure}[h]
\centerline{\includegraphics[width=2cm]{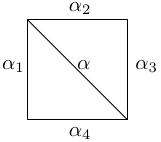}}

%\caption{}%\label{Fig-proof3}

\end{figure}

Denote by $\eta_\alpha$ the cardinality of $f^{-1}(\alpha)$. We may label $f^{-1}(\alpha)$ as $\hat\alpha_1,\cdots,\hat\alpha_{\eta_\alpha}$ in order according to the orientation of $\hat\gamma$. We may further assume that $\hat\alpha_i$ is not a boundary edge in $P_n$ for all $i\in \{1,\cdots,\eta_\alpha\}$, otherwise we may extend $P_n$ to a larger polygon. As $\alpha$ is not a folded arc, $\hat\alpha_i$ and $\hat\alpha_j$ are not in a same triangle in $\hat T^o$ for $i\neq j$. Denote $\hat T'^o=\mu_{\hat\alpha_{\eta_\alpha}}\cdots \mu_{\hat\alpha_1}\hat T^o$.

\subsubsection{Partition bijection $\phi^{T^o}_\alpha$}
%We see that the morphism $f$ induces a bijection between $\mathcal {CP}(\hat T^o,\hat \gamma)$ and $\mathcal {CP}(T^o,\gamma)$, and a bijection between $\mathcal {CP}(\hat T'^o,\hat \gamma)$ and $\mathcal {CP}(T'^o,\gamma)$.

Using Theorem \ref{thm-par11}, as $\hat\alpha_i$ and $\hat\alpha_j$ are not in a same triangle in $\hat T^o$ for $i\neq j$, we have a partition bijection
\begin{equation}\label{Eq-bi1}
\phi^{\hat T^o}_{\hat\alpha_1,\cdots,\hat\alpha_{\eta_\alpha}}: \mathcal {CP}(\hat T^o,\hat \gamma) \to \mathcal {CP}(\hat T'^o,\hat\gamma).
\end{equation}

Since $f$ induces bijections between $\mathcal {CP}(\hat  T^o,\hat  \gamma)$ and $\mathcal {CP}(T^o,\gamma)$, and between $\mathcal {CP}(\hat  T'^o,\hat  \gamma)$ and $\mathcal {CP}(T'^o,\gamma)$, we have $\phi^{\hat T^o}_{\hat\alpha_1,\cdots,\hat\alpha_{\eta_\alpha}}$ induces a partition bijection
\begin{equation}\label{Eq-bi}
\phi^{T^o}_{\alpha}: \mathcal {CP}(T^o,\gamma)\to \mathcal {CP}(T'^o,\gamma).
\end{equation}

That is, we have the following commutative diagram.

{\centering \xymatrixcolsep{7pc}\xymatrix{
  & \mathcal {CP}(\hat T^o,\hat \gamma) \ar[d]^\cong \ar[r]^{\phi^{\hat T^o}_{\hat\alpha_1,\cdots,\hat\alpha_{\eta_\alpha}}}  &  \mathcal {CP}(\hat T'^o,\hat \gamma)\ar[d]_\cong           \\
  & \mathcal {CP}(T^o,\gamma) \ar[r]^{\phi^{T^o}_{\alpha}}  &  \mathcal {CP}(T'^o,\gamma).                       }}

\medskip

For any complete $(T^o,\gamma)$-path $\overrightarrow\xi=(\overrightarrow\xi_{\hspace{-2pt}1},\cdots,\overrightarrow\xi_{\hspace{-2pt}2c+1})$, denote by $\overrightarrow{\hat \xi}=(\overrightarrow{\hat\xi}_{\hspace{-2pt}1},\cdots,\overrightarrow{\hat \xi}_{\hspace{-2pt}2c+1})$ the corresponding complete $(\hat T^o,\hat \gamma)$-path under $f$. We may associate $\overrightarrow\xi$ with an integer vector
\begin{equation}\label{equ-m2}
{\bf m}(\overrightarrow{\xi},\alpha)=(m(\overrightarrow{\hat\xi},\hat\alpha_1),\cdots, m(\overrightarrow{\hat \xi},\hat\alpha_{\eta_\alpha})),
\end{equation}
where each $m(\overrightarrow{\hat \xi},\hat\alpha_i)$ is given by (\ref{equ-m1}).

For any $i\in \{1,\cdots, \eta_\alpha\}$, denote by $\hat\alpha'_i$ the new arc in $P_n$ obtained from $\hat T^o$ by flip at $\hat\alpha_i$. For any complete $(T'^o,\gamma)$-path $\overrightarrow\xi'=(\overrightarrow\xi'_{\hspace{-2pt}1},\cdots,\overrightarrow\xi'_{\hspace{-2pt}2c'+1})$, denote by $\overrightarrow{\hat \xi'}=(\overrightarrow{\hat\xi'}_{\hspace{-2pt}1},\cdots,\overrightarrow{\hat \xi'}_{\hspace{-2pt}2c'+1})$ the corresponding complete $(\mu_{\hat\alpha_{\eta_\alpha}}\cdots \mu_{\hat\alpha_1}\hat T^o,\hat \beta)$-path under $f$. Similarly, we may associate $\overrightarrow\xi'$ with an integer vector
\begin{equation*}
{\bf m}(\overrightarrow\xi',\alpha')=(m(\overrightarrow{\hat \xi'},\hat\alpha'_1),\cdots, m(\overrightarrow{\hat \xi'},\hat\alpha'_{\eta_\alpha})).
\end{equation*}

Similarly, we define the partition bijection
\begin{equation}\label{Eq-biin}
\phi^{T'^o}_{\alpha'}: \mathcal {CP}(T'^o,\gamma)\to \mathcal {CP}(T^o,\gamma),
\end{equation}
which is induced by the partition bijection $\phi^{\hat T'^o}_{\hat\alpha'_1,\cdots,\hat\alpha'_{\eta_\alpha}}$.

\subsubsection{Partition bijection $\varphi^{T^o}_\alpha$}\label{Sec-varphi}
Since $\mathcal {CP}(T^o,\gamma)\cong\mathcal P(G_{T^o,\gamma})$ and $\mathcal {CP}(T'^o,\gamma)\cong\mathcal P(G_{T'^o,\gamma})$, we have the partition bijection $\phi^{T^o}_{\alpha}: \mathcal {CP}(T^o,\gamma)\to \mathcal {CP}(T'^o,\gamma)$ in (\ref{Eq-bi}) induces partition bijections $\varphi^{T^o}_{\alpha}:\mathcal P(G_{T^o,\gamma})\to\mathcal P(G_{T'^o,\gamma})$. Similarly, we have the inverse partition bijection $\varphi^{T'^o}_{\alpha'}:\mathcal P(G_{T'^o,\gamma})\to\mathcal P(G_{T^o,\gamma})$.

In this subsection, we study the properties of $\varphi^{T^o}_{\alpha}$ and $\varphi^{T'^o}_{\alpha'}$.

For any $P\in \mathcal P(G_{T^o,\gamma})$, denote by $\overrightarrow\xi(P)$ the corresponding complete $(T^o,\gamma)$-path and $\hat P$ the corresponding perfect matching of $G_{\hat T^o,\hat\gamma}$, let
\begin{equation*}
{\bf m}(P,\alpha)={\bf m}(\overrightarrow\xi(P),\alpha),
\end{equation*}
where ${\bf m}(\overrightarrow\xi(P),\alpha)$ is given by (\ref{equ-m2}).

Assume that ${\bf m}(P,\alpha)=(m_1(P),\cdots,m_{\eta_\alpha}(P))$. In view of (\ref{Equ-mp}) and (\ref{Equ-n}), we have
$$\sum_{i=1}^{\eta(\alpha)} m_i(P)=m(P;\alpha)-n(G_{T^o,\gamma};\alpha).$$

Similarly we can define ${\bf m}(P';\alpha')$ for any $P'\in \mathcal P(G_{T'^o,\alpha'})$.

Among $f^{-1}(\alpha)=\{\hat\alpha_1,\cdots,\hat\alpha_{\eta_\alpha}\}$, assume that $\hat\gamma$ crosses $\hat\alpha_{i_1},\cdots, \hat\alpha_{i_{l}}$ sequentially,  among $\{\hat\alpha'_1,\cdots,\hat\alpha'_{\eta_\alpha}\}$, assume that $\hat\gamma$ crosses $\hat\alpha'_{j_1},\cdots, \hat\alpha'_{j_{l'}}$ sequentially.

The following two Propositions follows by Proposition \ref{Lem-tile-2} (1) (2).

\begin{proposition}\label{Lem-tile-11}
%\item
 Assume that ${\bf m}(P,\alpha)=(m_1(P),\cdots,m_{\eta_\alpha}(P))$.
\begin{enumerate}[$(1)$]
\item If $m_k(P)=-1$ for some $k\in \{1,\cdots,\eta_\alpha\}$, then there exists a tile $G(k)$ of $G_{T^o,\gamma}$ with diagonal labeled $\alpha$ such that $P$ can twist on the tile $G(k)$.
\item If $m_k(P)=1$ for some $k\in \{1,\cdots,\eta_\alpha\}$, then there exists a tile $G'(k)$ of $G_{T'^o,\gamma}$ with diagonal labeled $\alpha'$ such that any $P'\in \varphi_{\alpha}^{T^o}(P)$ can twist on $G'(k)$.
\end{enumerate}
\end{proposition}

\begin{proposition}\label{Lem-parvar}
For any $P\in \mathcal P(G_{T^o,\gamma})$ with ${\bf m}(P,\alpha)=(m_1(P),\cdots,m_{\eta_\alpha}(P))$, we have
\begin{equation*}
|\varphi^{T^o}_{\alpha}(P)|=2^{\sum_{i=1}^{\eta_\alpha}[m_i(P)]_{+}},
\end{equation*}
more precisely, for any $P'\in \varphi^{T^o}_{\alpha}(P)$,
\begin{enumerate}[$(1)$]
\item we have ${\bf m}(P',\alpha')=-{\bf m}(P,\alpha)$,
\item for any $k$ such that $m_k(P)=1$, $P'$ can twist on the tiles $G'(k)$ and $\mu_{G'(k)}P'\in \varphi^{T^o}_{\alpha}(P)$, where $G'(k)$ are the tiles given in Proposition \ref{Lem-tile-11}.
\end{enumerate}
\end{proposition}

\begin{lemma}\label{Lem-var}
For any $P\in \mathcal \mathcal P(G_{T^o,\gamma})$ and $P'\in \varphi^{T^o}_{\alpha}(P)$, for any $(\alpha\neq )\tau\in T^o$,
\begin{enumerate}[$(1)$]
\item if $\tau\neq \alpha_1,\alpha_2,\alpha_3,\alpha_4$, then we have $\hat m(P;\tau)=\hat m(P';\tau)$;
\item if $\tau\in \{\alpha_1,\alpha_2,\alpha_3,\alpha_4\}$, then the number of edges in $P$ labeled $\tau$ which are not $\alpha$-mutable equals the number of edges in $P'$ labeled $\tau$ which are not $\alpha'$-mutable.
\end{enumerate}
\end{lemma}

\begin{proof}
It follows by Proposition \ref{Lem-tile-2} (3) (4).
\end{proof}

The following Proposition follows by Propositions \ref{Lem-tile-2} and \ref{Prop-mutation3}.

\begin{proposition}\label{Prop-mutation4}
Assume that $P$ can twist on a tile $G_l$ with diagonal labeled $\tau_{i_l}$ and $P>\mu_{G_l}P$. Assume that ${\bf m}(P,\alpha)=(m_1(P),\cdots,m_{\eta_\alpha}(P))$,
    \begin{enumerate}[$(1)$]
    \item if $\tau_{i_l}\neq \alpha, \alpha_1,\alpha_2,\alpha_3,\alpha_4$, then there is a tile $G'_{[l]}$ of $G_{T'^o,\gamma}$ with diagonal labeled $\tau_{i_l}$ such that any $P'\in \varphi_\alpha^{T^o}(P)$ can twist on $G'_{[l]}$, $P'>\mu_{G'_{[l]}}(P')$, and
        $\varphi_\alpha^{T^o}(\mu_{G_l}P)=\mu_{G'_{[l]}}(\varphi_\alpha^{T^o}(P)),$
        moreover, the following equalities hold:\\
     (a) $m^{\pm}(P,G_l;\tau_{i_l})=m^{\pm}(P',G'_{[l]};\tau_{i_l})$;
     (b) $n(G_l^{\pm};\tau_{i_l})=n(G'^{\pm}_{[l]};\tau_{i_l})$.\\
     In fact $G'_{[l]}$ corresponds to the same crossing point of $\gamma$ and $\tau_{i_l}$ as $G_l$.
    \item if $\tau_{i_l}=\alpha$, then $m_k(P)=-1$ and $G_{l}=G(k)$ for some $k$, where $G(k)$ is the tile given in Proposition \ref{Lem-tile-11}, moreover, $\mu_{G_l}P \in \varphi_{\alpha'}^{T'^o}(P')$
    for any $P'\in \varphi_{\alpha}^{T^o}(P)$.
    \item if $\tau_{i_l}\in \{\alpha_1,\alpha_2,\alpha_3,\alpha_4\}$ and $\tau_{i_l}$ is not the radius of any self-folded triangle in $T^o$ or $T'^o$, without loss of generality, we may assume $\tau_{i_l}=\alpha_1$, then there exists a tile $G'_{[l]}$ of $G_{T'^o,\gamma}$ with diagonal labeled $\alpha_1$ such that any $P'\in \varphi_{\alpha}^{T^o}(\mu_{G_l}P)$ can twist on $G'_{[l]}$ and $P'>\mu_{G'_{[l]}}P'\in \varphi_{\alpha}^{T^o}(P)$. Moreover, if the $\alpha'$-mutable edges in $P'$ are labeled $\alpha_{2},\alpha_{4}$ then\\
        (a) $m^{\pm}(P,G_l;\alpha_1)=m^{\pm}(P',G'_{[l]};\alpha_1)$,
        (b) $n(G_l^{\pm};\alpha_1)=n(G'^{\pm}_{[l]};\alpha_1)$.
    \end{enumerate}
\end{proposition}

\begin{proposition}\label{dProp-mutation4}
Assume that $\alpha_1=\alpha_2$ and the diagonal of $G_l$ is labeled $\alpha_1$. We further assume that $rel(G_l,T^o)=1$ and $N(G_l), W(G_l)$ are labeled $\alpha$. If $P$ can twist on $G_l$ and $P>Q:=\mu_{G_l}P$, denote ${\bf m}(P,\alpha)=(m_1(P),\cdots,m_{\eta_\alpha}(P)), {\bf m}(Q,\alpha)=(m_1(Q),\cdots,m_{\eta_\alpha}(Q))$, then $m_k(P)=m_{k}(Q)-1, m_{k+1}(P)=m_{k+1}(Q)+1$ for some $k$ and $m_i(P)=m_{i}(Q)$ for any $i\neq k,k+1$. Consequently, we have $(m_k(P), m_{k+1}(P))=(-1,1)$, $(0,0)$ or $(0,1)$.
\begin{enumerate}[$(1)$]
\item In case $(m_k(P), m_{k+1}(P))=(-1,1)$, we have $(m_k(\mu_{G_l}P), m_{k+1}(\mu_{G_l}P))=(0,0)$. Assume the tile left to $G'(k+1)$ is $G'$. Then the diagonal of $G'$ is labeled $\alpha_1$ and
any $Q'\in \varphi_{\alpha}^{T^o}(Q)$ can twist on $G'$, see Figure \ref{Fig:bigon1}. Moreover,\\
(a) $\mu_{G'}Q',\mu_{G'(k+1)}\mu_{G'}Q'\in \varphi_{\alpha}^{T^o}(P)$,
(b) $Q'<\mu_{G'}Q'<\mu_{G'(k+1)}\mu_{G'}Q'$,\\
(c),\\
(d) $\Omega(Q';G')=\Omega(Q;G_l)+\Omega(P;G_{l+1})$.\\
In particular, $\varphi_{\alpha}^{T^o}(P)=\{\mu_{G'}Q',\mu_{G'(k+1)}\mu_{G'}Q'\mid Q'\in \varphi_{\alpha}^{T^o}(Q)\}$.

\begin{figure}[h]
\centerline{\includegraphics{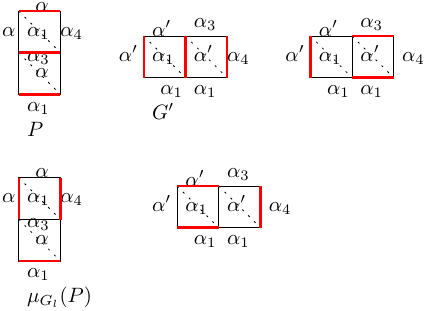}}
\caption{}\label{Fig:bigon1}%{\rm Hasse diagram for $\mathcal L(T^o,\beta^{(p,q)})$ in case $\beta\in   T^o$}}\label{Fig-Hasse}
 %Figure 1
\end{figure}

\item In case $(m_k(P), m_{k+1}(P))=(0,0)$, we have $(m_k(Q), m_{k+1}(\mu_{G_l}P))=(1,-1)$. Assume the tile right to $G'(k)$ is $G'$. Then the diagonal of $G'$ is labeled $\alpha_1$ and
any $P'\in \varphi_{\alpha}^{T^o}(P)$ can twist on $G'$, see Figure \ref{Fig:bigon2}. Moreover,\\
(a) $\mu_{G'}P',\mu_{G'(k)}\mu_{G'}P'\in \varphi_{\alpha}^{T^o}(Q)$, (b) $P'>\mu_{G'}P'>\mu_{G'(k)}\mu_{G'}P'$,\\
(c) $\Omega(Q;G_l)=\Omega(\mu_{G'(k)}\mu_{G'}P';G'(k))+\Omega(\mu_{G'}P';G')$,\\
In particular, $\varphi_{\alpha}^{T^o}(Q)=\{\mu_{G'}P',\mu_{G'(k)}\mu_{G'}P' \mid P'\in \varphi_{\alpha}^{T^o}(P)\}$.

\begin{figure}[h]
\centerline{\includegraphics{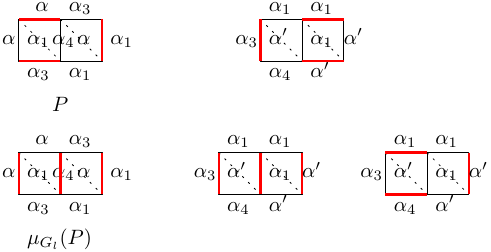}}
\caption{}\label{Fig:bigon2}%{\rm Hasse diagram for $\mathcal L(T^o,\beta^{(p,q)})$ in case $\beta\in   T^o$}}\label{Fig-Hasse}
 %Figure 1
\end{figure}

\item In case $(m_k(P), m_{k+1}(P))=(0,1)$, we have $(m_k(Q), m_{k+1}(Q))=(1,0)$. There is a unique tile $G'$ between $G'(k)$ and $G'(k+1)$ and its diagonal is labeled $\alpha_1$. Assume that $P'\in \varphi_{\alpha}^{T^o}(P)$ with $P'<\mu_{G'(k+1)}P'$. Then $P'$ can twist on $G'$, see Figure \ref{Fig:bigon3}. Moreover,\\
(a) $\mu_{G'} P',\mu_{G'(k)}\mu_{G'}P'\in \varphi_{\alpha}^{T^o}(Q)$, (b) $\mu_{G'(k+1)}P'>P'>\mu_{G'}P'>\mu_{G'(k)}\mu_{G'}P'$,\\
(c) $\Omega(Q;G_l)=\Omega(\mu_{G'(k)}\mu_{G'}P';G'(k))+\Omega(\mu_{G'}P';G')$.

\begin{figure}[h]
\centerline{\includegraphics{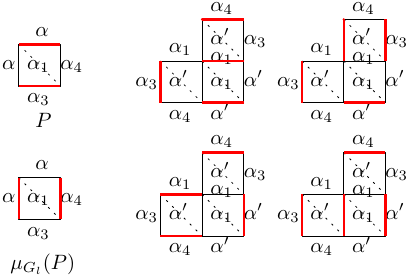}}
\caption{}\label{Fig:bigon3}%{\rm Hasse diagram for $\mathcal L(T^o,\beta^{(p,q)})$ in case $\beta\in   T^o$}}\label{Fig-Hasse}
 %Figure 1
\end{figure}
\item In all cases, for any $Q'\in \varphi_{\alpha}^{T^o}(Q)$ we have
$$m(Q;\alpha_1)-n(G_{T^o,\widetilde\beta};\alpha_1)=m(Q';\alpha_1)-n(G_{T'^o,\widetilde\beta};\alpha_1)+m(Q';\alpha')-n(G_{T'^o,\widetilde\beta};\alpha').$$
\end{enumerate}
%$\Omega(Q;G_l)=\Omega(Q';G')+\Omega(\mu_{G'}Q';G'(k+1))$
\end{proposition}

\begin{proof}
Since $P>Q$, we have $N(G_l)\in P, W(G_l)\in Q$. Thus $m_k(P)=m_{k}(Q)-1, m_{k+1}(P)=m_{k+1}(Q)+1$ for some $k$ and $m_i(P)=m_{i}(Q)$ for any $i\neq k,k+1$. It follows that $m_k(P)\in \{-1,0\}$ and $m_{k+1}(P)\in \{0,1\}$. Since $\widetilde\beta$ can not cross $\alpha,\alpha_1,\alpha$ consecutively, we have $(m_k(P),m_{k+1}(P))\neq (-1,0)$. Therefore, $(m_k(P), m_{k+1}(P))=(-1,1)$, $(0,0)$ or $(0,1)$.

We only consider the case that $(m_k(P), m_{k+1}(P))=(-1,1)$, as the other cases can be proved similarly. We have the diagonal of $G_{l-1}$ is labeled $\alpha$ and $S(G_{l-1})\in P\cap Q$. We may assume $rel(G'(k+1),T'^o)=1$. As $E(G_l)$ is labeled $\alpha_4$, from the construction of snake groups, we have $E(G'(k+1))$ is labeled $\alpha_4$ and the diagonal of $G'$ is labeled $\alpha_1$. For any $Q'\in \varphi_{\alpha}^{T^o}(Q)$, since $m_k(Q)=m_k(P)+1=0$, we have $m_k(Q')=0$ by Proposition \ref{Lem-parvar}. Thus we have $N(G')\in Q'$. As $S(G_{l-1})$ is labeled $\alpha_1$ and $S(G_{l-1})\in Q$, we have $S(G')\in Q'$. Thus $Q'$ can twist on $G'$. For any $P'\in \varphi_{\alpha}^{T^o}(P)$, since $m_k(P)=-1$, we have $m_k(P')=1$ by Proposition \ref{Lem-parvar}. Thus we have $W(G')\in P'$. Therefore, $\mu_{G'}Q'\in \varphi_{\alpha}^{T^o}(P)$. By Proposition \ref{Lem-tile-11} (2), $\mu_{G'}Q'$ can twist on $G'(k+1)$ and $\mu_{G'(k+1)}\mu_{G'}Q'\in \varphi_{\alpha}^{T^o}(P)$.
$Q'<\mu_{G'}Q'<\mu_{G'(k+1)}\mu_{G'}Q'$ is immediate. By Proposition \ref{Lem-tile-2} (3) (4), we have
$$m^{\pm}(Q,G_l;\alpha_1)=m^{\pm}(Q',G';\alpha_1)+m^{\pm}(\mu_{G'}Q',G'(k+1);\alpha')-n((G'(k+1))^{\pm};\alpha'),$$
$$m^{+}(Q',G';\alpha_1)=m^{+}(Q,G_l;\alpha_1)+m^{+}(P,G_{l-1};\alpha)-n(G_{l-1}^{+};\alpha)-1,$$
$$m^{-}(Q',G';\alpha_1)=m^{-}(Q,G_l;\alpha_1)+m^{-}(P,G_{l-1};\alpha)-n(G_{l-1}^{-};\alpha)-1.$$
We have $n(G_l^{\pm};\alpha_1)=n(G'^{\pm};\alpha_1)$.
Therefore, we have $$\Omega(Q;G_l)=\Omega(Q';G')+\Omega(\mu_{G'}Q';G'(k+1)),\;\;\;\Omega(Q';G')=\Omega(Q;G_l)+\Omega(P;G_{l+1})$$ and
$$m(Q;\alpha_1)-n(G_{T^o,\widetilde\beta};\alpha_1)=m(Q';\alpha_1)-n(G_{T'^o,\widetilde\beta};\alpha_1)+m(Q';\alpha')-n(G_{T'^o,\widetilde\beta};\alpha').$$
\end{proof}

\begin{lemma}\label{lem-p-12}
We have $P'_-\in \varphi^{T^o}_{\alpha}(P_-)$.
\end{lemma}

\begin{proof}
Assume $P'$ is the minimum element in $\varphi^{T^o}_{\alpha}(P_-)$. For any tile $G'$ that $P'$ can twist on, as $P_-$ is the minimum perfect matching, we have $P'<\mu_{G'}P'$ by Proposition \ref{Prop-mutation4} and Proposition \ref{dProp-mutation4}. It follows that $P'=P'_-$.
\end{proof}

The following result follows by Lemma \ref{lem-p-}.

\begin{lemma}\label{lem-p--}
For $P_-$, we have $m_i(P_-)m_j(P_-)\geq 0$ for any $i,j\in \{1,\cdots,\eta_\alpha\}$.
\end{lemma}

\begin{proof}
Otherwise, $m_i(P_-)m_j(P_-)<0$ for some $i,j\in \{1,\cdots,\eta_\alpha\}$. We see that $\gamma$ intersects itself by Lemma \ref{lem-p-}, a contradiction.
\end{proof}

\newpage

\section{Partition bijection between triangles}\label{sec:PB}

With the notation in Section \ref{delta1}. Let $q$ be a puncture. Let $\Delta(T^o,q)$ (resp. $\Delta(T'^o,q)$) be the set of triangles incident to $q$ in $T^o$ (resp. $T'^o=\mu_\alpha(T^o)$). In this section, we give a partition bijection $\phi_\alpha^{T^o}:\Delta(T^o,q)\to \Delta(T'^o,q)$.

\subsection{Once-puncture polygon case} We first consider the case that $\Sigma$ is the once-punctured polygon $P_{n,1}$, the disk with $n$ marked points labeled $1,\cdots,n$ clockwise and one puncture labeled $0$. Let $T^o$ be a triangulation of $P_{n,1}$. Assume that $\Delta(T^o,0)=\{\Delta_1,\cdots,\Delta_t\}$ and for any $i\in \{1,\cdots,t\}$ the three edges of $\Delta_i$ are labeled $\tau_{i^-},\tau_i,\tau_{[i]}$, where we denote
\[\begin{array}{ccl} i^+ &=&
\left\{\begin{array}{ll}
i+1, &\mbox{if $1\leq i\leq t-1$}, \\
1, &\mbox{if $i=t$},
\end{array}\right.
\end{array}, \begin{array}{ccl} i^- &=&
\left\{\begin{array}{ll}
i-1, &\mbox{if $2\leq i\leq t$}, \\
t, &\mbox{if $i=1$}.
\end{array}\right.
\end{array}\]

Fix an arc $\alpha\in T^o$, we have one of the following cases happens: (1) $\alpha=\tau_i$ for some $i$; (2) $\alpha=\tau_{[i]}$ for some $i$; (3) the rest cases.

For any $\Delta_j\in \Delta(T^o,0)$, we associate with a subset $\phi_{\alpha}^{T^o,q}(\Delta_j)\subset \Delta(T'^o,0)$ as follows.

\textbf{Case I: $\alpha=\tau_i$ for some $i$.} Then $\Delta(T'^o,0)=\{(\tau_{i^-},\tau_{i^+},\alpha'), \Delta_j\mid j\neq i,i^+\}$. In this case, let
\[\begin{array}{ccl} \phi_{\alpha}^{T^o,q}(\Delta_j) &=&
\left\{\begin{array}{ll}
(\tau_{i^-},\tau_{i^+},\alpha'), &\mbox{if $j=i$ or $i^+$}, \\
\Delta_j, &\mbox{if $j\neq i,i^+$},
\end{array}\right.
\end{array}
\]

\textbf{Case II: $\alpha=\tau_{[i]}$ for some $i$.} Then $\tau_{i-1},\alpha'$ are the two sides of some triangle $\hat\Delta_1$ in $T'^o$ and $\tau_{i},\alpha'$ are the two sides of some triangle $\hat\Delta_2$ in $T'^o$. We have $\Delta(T'^o,0)=\{\hat\Delta_1, \hat\Delta_2, \Delta_j\mid j\neq i\}$.

In this case, let
\[\begin{array}{ccl} \phi_{\alpha}^{T^o,q}(\Delta_j) &=&
\left\{\begin{array}{ll}
\{\hat\Delta_1,\hat\Delta_2\}, &\mbox{if $j=i$}, \\
\Delta_j, &\mbox{if $j\neq i$},
\end{array}\right.
\end{array}
\]

\textbf{Case III: the rest cases.} In this case we have $\Delta(T'^o,0)=\Delta(T^o,0)$. In this case let $\phi_{\alpha}^{T^o,q}(\Delta_j)=\Delta_j$ for all $j\in \{1,\cdots,s\}$.

Similarly we can associate a subset $\phi_{\alpha}^{T'^o,q}(\Delta')$ of $\Delta(T^o,0)$ for any $\Delta'\in \Delta(T'^o,0)$.

Assume that $\Delta(T'^o,0)=\{\Delta'_1,\cdots,\Delta'_{s'}\}$ and for any $i\in \{1,\cdots,s'\}$ the three edges of $\Delta'_i$ are labeled $\tau'_{i^-},\tau'_{i},\tau'_{[i]}$.

\begin{proposition}\label{prop-tri}
With the foregoing notation. $\phi_{\alpha}^{T^o,q}$ is a partition bijection from $\Delta(T^o,0)$ to $\Delta(T'^o,0)$ with inverse $\phi_{\alpha}^{T'^o,q}$. For any $\Delta_i\in \Delta(T^o,0)$ we have
$$|\phi_{\alpha}^{T^o,q}(\Delta_i)|=2^{[m(\Delta_i;\alpha)]_+},$$
where $m(\Delta_i;\alpha)$ is given by (\ref{Eq-mj}), moreover, if $m(\Delta_i;\alpha)=1$, then $\alpha=\tau_{[i]}$ and $\phi_{\alpha}^{T^o,q}(\Delta_i)=\{\Delta'_j, \Delta'_{j+1}\}$ for some $j$ such that $\tau'_j=\alpha'$.
\end{proposition}

\begin{proof}
The result follows easily by checking all the above three cases.
\end{proof}

The following result is clear.

\begin{lemma}\label{lem-minv}
For any $\Delta_i\in \Delta(T^o,0)$ and $\Delta'\in \phi_{\alpha}^{T^o,q}(\Delta_i)$, we have $m(\Delta_i;\alpha)=-m(\Delta';\alpha')$.
\end{lemma}

Next, we generalize the above construction for a sequence of non-self-folded arcs $\alpha_1,\cdots,\alpha_k\in T^o$. Let $\alpha_1,\cdots,\alpha_k\in T^o$ be a sequence of arcs such that $\alpha_i,\alpha_j$ are not two sides of any triangulation in $T^o$ for any $i\neq j$. Denote by $\alpha'_1,\cdots,\alpha'_k$ the new arcs obtained from $T^o$ by flip at $\alpha_1,\cdots,\alpha_k$, respectively. For any triangle $\Delta_i\in \Delta(T^o,0)$, recall that the three sides of $\Delta_i$ are $\tau_{i-1},\tau_i$ and $\tau_{[i]}$, one and exactly one of the following cases happens:
\begin{enumerate}
\item Case I: $\tau_{i-1}=\alpha_a$ for some $a\in \{1,\cdots,k\}$;
\item Case II: $\tau_{i}=\alpha_a$ for some $a\in \{1,\cdots,k\}$;
\item Case III: $\tau_{[i]}=\alpha_a$ for some $a\in \{1,\cdots,k\}$;
\item Case IV: $\tau_{i-1}, \tau_{i},\tau_{[i]}\neq \alpha_a$ for any $a\in \{1,\cdots,k\}$.
\end{enumerate}
We associate a subset $\phi_{\alpha_1,\cdots,\alpha_k}^{T^o,q}(\Delta_i)\subset \Delta(\mu_{\alpha_k}\cdots\mu_{\alpha_1}T^o,0)$ as follows:
\begin{enumerate}
\item In case I, we have $\tau_{i-2},\tau_i,\alpha'_a$ form a triangle in $\mu_{\alpha_k}\cdots\mu_{\alpha_1}T^o$, let
$$\phi_{\alpha_1,\cdots,\alpha_k}^{T^o,q}(\Delta_i)=(\tau_{i-2},\tau_i,\alpha'_a);$$
\item In case II, we have $\tau_{i-1},\tau_{i+1},\alpha'_a$ form a triangle in $\mu_{\alpha_k}\cdots\mu_{\alpha_1}T^o$, let
$$\phi_{\alpha_1,\cdots,\alpha_k}^{T^o,q}(\Delta_i)=(\tau_{i-1},\tau_{i+1},\alpha'_a);$$
\item In case III, we have $\tau_{i-1},\alpha'$ are the two sides of some triangle $\hat\Delta_1$ in $\mu_{\alpha_k}\cdots\mu_{\alpha_1}T^o$ and $\tau_{i},\alpha'$ are the two sides of some triangle $\hat\Delta_2$ in $\mu_{\alpha_k}\cdots\mu_{\alpha_1}T^o$, let
$$\phi_{\alpha_1,\cdots,\alpha_k}^{T^o,q}(\Delta_i)=\{\hat\Delta_1,\hat\Delta_2\};$$
\item In case IV, let
$$\phi_{\alpha_1,\cdots,\alpha_k}^{T^o,q}(\Delta_i)=\Delta_i.$$
\end{enumerate}

Similarly, we can associate a subset $\phi_{\alpha'_1,\cdots,\alpha'_k}^{\mu_{\alpha_k}\cdots\mu_{\alpha_1}T^o,q}(\Delta'_i)\subset \Delta(T^o,0)$ for any $\Delta'_i\in \Delta(\mu_{\alpha_k}\cdots\mu_{\alpha_1}T^o,0)$.

\begin{proposition}\label{Lem-part}
With the foregoing notation. We have $\phi_{\alpha_1,\cdots,\alpha_k}^{T^o,q}$ and $\phi_{\alpha'_1,\cdots,\alpha'_k}^{\mu_{\alpha_k}\cdots\mu_{\alpha_1}T^o,q}$ are partition bijections inverse to each other, moreover, for any $\Delta_i\in \Delta(0)$ we have
$$|\phi_{\alpha_1,\cdots,\alpha_k}^{T^o,q}(\Delta_i)|=2^{\sum_{j=1}^k[m(\Delta_i;\alpha_j)]_+},$$
where $m(\Delta_i;\alpha_j)$ is given by (\ref{Eq-mj}), moreover, if $m(\Delta_i;\alpha_j)=1$ for some $j$, then $\alpha_j=\tau_{[i]}$ and $\phi_{\alpha}^{T^o,q}(\Delta_i)=\{\Delta'_{j'}, \Delta'_{j'+1}\}$ for some $j'$ such that $\tau'_{j'}=\alpha'_i$.
\end{proposition}

\begin{proof}
The result follows easily by checking all the above three cases.
\end{proof}

Note that $\sum_{j=1}^k[m(\Delta_i(q);\alpha_j)]_+=1,0$ or $-1$.

The following result follows immediately by Lemma \ref{lem-minv}.

\begin{lemma}\label{lem-minv1}
For any $\Delta_i\in \Delta(T^o,0)$ and $\Delta'\in \phi_{\alpha_1,\cdots,\alpha_k}^{T^o,q}(\Delta_i)$, we have $m(\Delta_i;\alpha_j)=-m(\Delta';\alpha'_j)$ for any $j\in \{1,\cdots,k\}$.
\end{lemma}

\subsection{General case}

The following is a generalization of ``canonical polygon" in \cite{BR} to ``canonical once - polygon".

\begin{theorem}
Let $T^o$ be an ideal triangulation of $\Sigma$. For any puncture $q$ of $\Delta$, assume that $\Delta(T^o,q)=\{\Delta_1(q),\cdots, \Delta_s(q)\}$ with the three sides of $\Delta_i(q)$ are $\tau_{i-1},\tau_i$ and $\tau_{[i]}$. Then there is a morphism $\pi$ from $P_{s,1}$ to $\Sigma$ such that $\pi(0i)=\tau_i$ for all $i=1,\cdots,s$ and $\pi(i,i^+)=\tau_{[i^+]}$.
\end{theorem}

We call $P_{s,1}$ together with the morphism $\pi$ the \emph{canonical-once punctured polygon} for $q$ concerning $T^o$. It is clear that $\Delta_i(q)$ corresponds to a unique triangle $\Delta_i$ incident to $0$ in $P_{s,1}$ for any $i$.

\begin{example}
In the twice punctured digon with triangulation $T^o$, as shown in the right below figure. We have $\alpha_2=\alpha_4$, $\alpha_{[2]}=\alpha_5$, $\alpha_{[3]}=\alpha_3=\alpha_{[4]}$ and $\alpha_{[5]}=\alpha_1$. Then the canonical once-punctured polygon for $0$ is shown in the left below figure.

\begin{figure}[h]
\centerline{\includegraphics{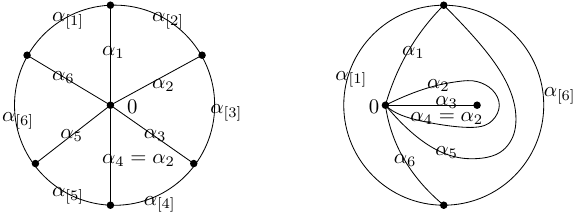}}
\caption{Canonical once-punctured polygon for puncture $0$}%{\rm Hasse diagram for $\mathcal L(T^o,\beta^{(p,q)})$ in case $\beta\in   T^o$}}\label{Fig-Hasse}
 %Figure 1
\end{figure}
\end{example}

For any non-self-folded arc $\alpha\in T^o$, let $T'^o=\mu_{\alpha}(T^o)$, we now construct partition bijection $\phi_{\alpha}^{T^o,q}$ from $\Delta(T^o,q)$ to $\Delta(T'^o,0)$.

In case $\alpha\notin \{\alpha_i,\alpha_{[i]}\mid i=1,\cdots, s\}$, we have $\Delta(T'^o,0)=\Delta(T^o,0)$. Then let $\phi_{\alpha}^{T^o,q}(\Delta_i)=\Delta_i$ for any $i=1,\cdots, s$.

In case $\alpha\in \{\alpha_i,\alpha_{[i]}\mid i=1,\cdots, s\}$, we can extend the canonical once punctured polygon $(P_{s,1},\pi)$ to a bigger once punctured polygon $P_{n,1}$ for some $n>s$ with an ideal triangulation $\hat T^o$ and a morphism $\pi':P_{n,1}\to \Sigma$ such that (1) $\pi'|_{P_{s,1}}=\pi$, (2) the preimages of $\alpha$ are inner arcs, and (3) $\pi'(\Delta)$ is a triangle in $T^o$ for any triangle $\Delta$ in $\hat T^o$.

Assume that $\pi^{-1}(\alpha)=\{\alpha_1,\cdots, \alpha_k\}\subset \hat T^o$. It is easy to see that $\alpha_i,\alpha_j$ are not two sides of any triangulation in $\hat T^o$ for all $i\neq j$ from the construction of the canonical once punctured polygon.

For any $\Delta_i(q)\in \Delta(T^o,q)$, let $\phi_{\alpha}^{T^o,q}(\Delta_i(q))=\pi'(\phi_{\alpha_1,\cdots,\alpha_k}^{T^o,q}(\Delta_i))$, where $\Delta_i$ is the triangle incident to $0$ in the canonical once punctured polygon. Similarly, we can define $\phi_{\alpha}^{T'^o,q}$.

\begin{proposition}\label{Lem-part1}
With the foregoing notation. $\phi_{\alpha}^{T^o,q}$ is a partition bijection from $\Delta(T^o,q)$ to $\Delta(T'^o,q)$ with inverse $\phi_{\alpha}^{T'^o,q}$, moreover, for any $\Delta_i(q)\in \Delta(T^o,q)$ we have
$$|\phi_{\alpha}^{T^o,q}(\Delta_i(q))|=2^{[m(\Delta_i(q);\alpha)]_+},$$
where $m(\Delta_i(q);\alpha)$ is given by (\ref{Eq-mj}).
\end{proposition}

\begin{proof}
It follows immediately by Proposition \ref{Lem-part}.
\end{proof}

The following lemma is clear.

\begin{lemma}\label{lem-deltatodelta}
We have $\Delta_1'(q)\in  \phi_{\alpha}^{T^o,q}(\Delta_1(q))$, where $\Delta'_1(q)$ is the first triangle incident to $q$ in $T'^o$.
\end{lemma}

By Lemma \ref{lem-minv1}, we have

\begin{lemma}\label{lem-minv2}
For any $\Delta\in \Delta(T^o,q)$ and $\Delta'\in \phi_{\alpha}^{T^o,q}(\Delta)$, we have $m(\Delta;\alpha)=-m(\Delta';\alpha')$.
\end{lemma}

%For $k\in \{1,2,3,4\}$, we say that the number of $\alpha$-mutable

\begin{lemma}\label{Lem-var1}
For any $\Delta\in \Delta(T^o,q)$ and $\Delta'\in \phi_{\alpha}^{T^o,q}(\Delta)$, for any $\tau\in T^o$,
\begin{enumerate}[$(1)$]
\item if $\tau\neq \alpha,\alpha_1,\alpha_2,\alpha_3,\alpha_4$, then $m(\Delta;\tau)=m(\Delta';\tau)$;
\item if $\tau\in \{\alpha_1,\alpha_2,\alpha_3,\alpha_4\}$ and $m(\Delta;\alpha)=0$, then $m(\Delta;\tau)=m(\Delta';\tau)$.
\end{enumerate}
\end{lemma}

\newpage

\section*{Notation}

We fix some notation throughout the rest of this paper. Fix an arc $\widetilde\beta$, an ideal triangulation $T^o$ and a non-self-folded $\alpha\in T^o$, let $T'^o=\mu_\alpha T^o$ and $\alpha'$ be the new arc obtained. Suppose that $\alpha$ is a diagonal of the quadrilateral $(\alpha_1,\alpha_2,\alpha_3,\alpha_4)$ in $T^o$ and $\alpha_1,\alpha_3$ are in the clockwise direction of $\alpha$. As $\alpha$ is not folded in $T^o$, $\{\alpha_1,\alpha_3\}\neq \{\alpha_2,\alpha_4\}$.

Denote by $G_{T^o,\widetilde\beta}$ (resp. $G_{T'^o,\widetilde\beta}$) the associated snake graph with tiles $G_1,\cdots,G_c$ (resp. $G'_1,\cdots,G'_{c'}$) in order and $\mathcal P(G_{T^o,\widetilde\beta})$ (resp. $\mathcal P(G_{T'^o,\widetilde\beta})$) the set of all perfect matching. Denote by $P_\pm$ (resp. $P'_\pm$) the maximal/minimal perfect matching. Denote by $\tau_{i_l}$ (resp. $\tau'_{i_l}$) the label of the diagonal of $G_l$ (resp. $G'_l$).

If $t(\widetilde\beta)=q$ is a puncture, denote by $\Delta_1(q),\Delta_2(q),\cdots, \Delta_t(q)$ (resp. $\Delta'_1(q),\Delta'_2(q),\cdots, \Delta'_{t'}(q)$) the triangles incident to $q$ in $T^o$ (resp. $T'^o$) in clockwise order such that $\widetilde\beta$ crosses $\Delta_1(q)$ (resp. $\Delta'_1(q)$) or $\widetilde \beta$ is the common side of $\Delta_1(q)$ and $\Delta_t(q)$ (resp. $\Delta'_1(q)$ and $\Delta'_{t'}(q)$). Denote the common side of $\Delta_{i}(q)$ and $\Delta_{i+1}(q)$ (resp. $\Delta'_{i}(q)$ and $\Delta'_{i+1}(q)$) by $\tau_{i}(q)$ (resp. $\tau'_{i}(q)$). Denote the third side of $\Delta_i(q)$ (resp. $\Delta'_i(q)$) by $\tau_{[i]}(q)$ (resp. $\tau'_{[i]}(q)$).

If $s(\widetilde\beta)=p$ is a puncture, denote by $\Delta_1(p),\Delta_2(p),\cdots, \Delta_s(p)$ (resp. $\Delta'_1(p),\Delta'_2(p),\cdots, \Delta'_{s'}(p)$) the triangles incident to $p$ in $T^o$ (resp. $T'^o$) in clockwise order such that $\widetilde\beta$ crosses $\Delta_1(p)$ (resp. $\Delta'_1(p)$) or $\widetilde \beta$ is the common side of $\Delta_1(p)$ and $\Delta_s(p)$ (resp. $\Delta'_1(p)$ and $\Delta'_{s'}(p)$). Denote the common side of $\Delta_{i}(p)$ and $\Delta_{i+1}(p)$ (resp. $\Delta'_{i}(p)$ and $\Delta'_{i+1}(p)$) by $\tau_{i}(p)$ (resp. $\tau'_{i}(p)$). Denote the third side of $\Delta_i(p)$ (resp. $\Delta'_i(p)$) by $\tau_{[i]}(p)$ (resp. $\tau'_{[i]}(p)$).

If $\widetilde\beta\notin T^o$ and $t(\widetilde\beta)=q$, denote
\begin{equation*}
        E_1(q)=
            \begin{cases}
                E(G_c), & \mbox{if $rel( G_c,  T^o)=1$},\vspace{1mm}\\
                N(G_c), & \mbox{if $rel( G_c,  T^o)=-1$,}
            \end{cases}\;\;\;
        E_2(q)=
            \begin{cases}
                N(G_c), & \mbox{if $rel(G_c,  T^o)=1$},\vspace{1mm}\\
                E(G_c), & \mbox{if $rel(G_c,  T^o)=-1$.}
            \end{cases}
\end{equation*}

If $\widetilde\beta\notin T^o$ and $s(\widetilde\beta)=p$, denote
\begin{equation*}
        E_1(p)=
            \begin{cases}
                W(G_1), & \mbox{if $rel( G_1,  T^o)=1$},\vspace{1mm}\\
                S(G_1), & \mbox{if $rel( G_1,  T^o)=-1$,}
            \end{cases}\;\;\;
        E_2(p)=
            \begin{cases}
                S(G_1), & \mbox{if $rel(G_1,  T^o)=1$},\vspace{1mm}\\
                W(G_1), & \mbox{if $rel(G_1,  T^o)=-1$.}
            \end{cases}
\end{equation*}

If $\widetilde\beta\notin T'^o$ and $t(\widetilde\beta)=q$, denote
\begin{equation*}
        E'_1(q)=
            \begin{cases}
                E(G'_{c'}), & \mbox{if $rel( G'_{c'},  T'^o)=1$},\vspace{1mm}\\
                N(G'_{c'}), & \mbox{if $rel( G'_{c'},  T'^o)=-1$,}
            \end{cases}\;\;\;
        E'_2(q)=
            \begin{cases}
                N(G'_{c'}), & \mbox{if $rel(G'_{c'},  T'^o)=1$},\vspace{1mm}\\
                E(G'_{c'}), & \mbox{if $rel(G'_{c'},  T'^o)=-1$.}
            \end{cases}
\end{equation*}

If $\widetilde\beta\notin T^o$ and $s(\widetilde\beta)=p$, denote
\begin{equation*}
        E'_1(p)=
            \begin{cases}
                W(G'_1), & \mbox{if $rel( G'_1,  T'^o)=1$},\vspace{1mm}\\
                S(G'_1), & \mbox{if $rel( G'_1,  T'^o)=-1$,}
            \end{cases}\;\;\;
        E'_2(p)=
            \begin{cases}
                S(G'_1), & \mbox{if $rel(G'_1,  T'^o)=1$},\vspace{1mm}\\
                W(G'_1), & \mbox{if $rel(G'_1,  T'^o)=-1$.}
            \end{cases}
\end{equation*}

Let $\mathcal L=\mathcal L(T^o,\widetilde\beta), \mathcal L(T^o,\widetilde\beta^{(q)})$ or $\mathcal L(T^o,\widetilde\beta^{(p,q)})$. Let $\mathcal L'=\mathcal L(T'^o,\widetilde\beta), \mathcal L(T'^o,\widetilde\beta^{(q)})$ or $\mathcal L(T'^o,\widetilde\beta^{(p,q)})$ correspondingly. Denote the minimal elements in $\mathcal L$ and $\mathcal L'$ by ${\bf P}_-$ and ${\bf P}'_-$, respectively.

\newpage

\section{Partition bijection between lattices $\mathcal L$ and $\mathcal L'$}\label{sec-pb}

For any pair ${\bf P}, {\bf Q}\in \mathcal L$, assume that ${\bf P}$ covers ${\bf Q}$, then we have that ${\bf P}$ and ${\bf Q}$ are either related by a twist on a tile $G_l$ or related by two adjacent triangles incident to $p$ or $q$, assume $\tau$ is the label of the diagonal or the common side of the triangles. In both cases, we say that ${\bf P}$ covers ${\bf Q}$ and \emph{related by $\tau$}. We denote by
$$\Omega({\bf P}, {\bf Q})=w({\bf P})-w({\bf Q}).$$

We say that ${\bf P}, {\bf Q}\in \mathcal L$ are \emph{related by a twist at $\tau$} if either ${\bf P}$ covers ${\bf Q}$ and related by $\tau$ or ${\bf Q}$ covers ${\bf P}$ and related by $\tau$.

In this section, we construct a partition bijection from $\mathcal L$ to $\mathcal L'$ via the partition bijections $\varphi^{T^o}_\alpha: \mathcal P(G_{T^o,\widetilde\beta})\to \mathcal P(G_{T'^o,\widetilde\beta})$ and $\phi_{\alpha}^{T^o,q}: \Delta(T^o,q)\to \Delta(T'^o,q)$, which are given in Section \ref{Sec-parcom}.

\subsubsection*{Partition bijection $\pi:\mathcal L(T^o,\widetilde\beta)\to \mathcal L(T'^o,\widetilde\beta)$} Herein, we assume that $\mathcal L=\mathcal L(T^o,\widetilde\beta)$ and $\mathcal L'=\mathcal L(T'^o,\widetilde\beta)$.

Recall that for any $P\in \mathcal P(G_{T^o,\widetilde\beta})$ we have ${\bf m}(P,\alpha)=(m_1(P),\cdots,m_{\eta_\alpha}(P))\in \{-1,0,1\}^{\eta_\alpha}$.

\begin{definition}\label{Def-pair}\cite{H,H1}
For any $P\in \mathcal P(G_{T^o,\widetilde\beta})$, we choose pairs of indices in $\{1,\cdots,\eta_\alpha\}$ via the following algorithm:
\begin{enumerate}[$(i)$]
\item If $m_i(P)\geq 0$ for all $i=1,\cdots,\eta_\alpha$ or $m_i(P)\leq 0$ for all $i=1,\cdots,\eta_\alpha$, then nothing is chosen;
\item If $m_i(P)m_j(P)=-1$ for some $i,j$, choose the pair of indices $(a,b)$ with $a<b$ such that (1) $m_a(P)m_b(P)=-1$, (2) $m_\ell(P)=0$ for all $a<\ell<b$, (3) $m_a(P)m_\ell(P)\geq 0$ for all $\ell<a$.
\item Delete $m_a(P)$ and $m_b(P)$ from ${\bf m}(P,\alpha)$. Then return to step (i).
\end{enumerate}
We call any above chosen pair $(a,b)$ of indices an \emph{${\bf m}(P,\alpha)$-pair}.
\end{definition}

\begin{example}
If ${\bf m}(P,\alpha)=(1,1,0,-1,0,1,-1,1)$. Then all of the ${\bf m}(P,\alpha)$-pairs are $(2,4)$ and $(6,7)$. If ${\bf m}(P,\alpha)=(1,1,-1,0-1,1)$. Then all the the ${\bf m}(P,\alpha)$-pairs are $(2,3)$ and $(1,5)$.
\end{example}

The following is immediate.

\begin{lemma}\label{Lem-pair}
For ${\bf m}(P,\alpha)$, we have
\begin{enumerate}[$(1)$]
\item ${\bf m}(P,\alpha)$-pairs and $-{\bf m}(P,\alpha)$-pairs coincide.
\item
For any ${\bf m}(P,\alpha)$-pair $(a,b)$ with $a<b$ we have
$\sum_{a<\ell<b}m_\ell(P)=0.$
\item If $\sum m_i(P)\geq 0$ then for any $a$ with $m_a(P)<0$ we have $a$ is in some ${\bf m}(P,\alpha)$-pair.
\item If $\sum m_i(P)\leq 0$ then for any $a$ with $m_a(P)>0$ we have $a$ is in some ${\bf m}(P,\alpha)$-pair.
\end{enumerate}
\end{lemma}

Let $(a,b)$ be an ${\bf m}(P,\alpha)$-pair. Then either $m_a(P)=1,m_b(P)=-1$ or $m_a(P)=-1,m_b(P)=1$. By Proposition \ref{Lem-tile-11}, in case $m_a(P)=1,m_b(P)=-1$, we have $P$ can twist on the tile $G(b)$ and any $P'\in \varphi_\alpha^{T^o}(P)$ can twist on $G'(a)$; in case $m_a(P)=-1,m_b(P)=1$, we have $P$ can twist on the tile $G(a)$ and any $P'\in \varphi_\alpha^{T^o}(P)$ can twist on $G'(b)$, where $G(a),G(b)$ (resp. $G'(a), G'(b)$) are the tiles with diagonal labeled $\alpha$ (resp. $\alpha'$) given in Proposition \ref{Lem-tile-11}.

\emph{Construction of $\pi$.} For any $P\in \mathcal L$,
$\pi(P)$ contains all $P'\in \varphi_\alpha^{T^o}(P)$ satisfy the following condition: for any ${\bf m}(P,\alpha)$-pair $(a,b)$,

$\bullet$ if $m_a(P)=1,m_b(P)=-1$ then the labels of $P'\cap edge(G'(a))$ and $P\cap edge(G(b))$ coincide;

$\bullet$ if $m_a(P)=-1,m_b(P)=1$ then the labels of $P'\cap edge(G'(b))$ and $P\cap edge(G(a))$ coincide.

Similarly, we can construction a set $\pi'(P')\subset \mathcal L$ for any $P'\in \mathcal L'=\mathcal P(G_{T'^o,\widetilde\beta})$.

\begin{proposition}\label{Lem-pi}
For any $P\in \mathcal L=\mathcal P(G_{T^o,\widetilde\beta})$ we have
$$|\pi(P)|=2^{[\sum m_i(P)]_+},$$
to be precise, for any $P'\in \pi(P)$, we have
\begin{enumerate}[$(1)$]
\item ${\bf m}(P',\alpha')=-{\bf m}(P,\alpha)$,
\item if $m_k(P)=1$ and $k$ is not in any ${\bf m}(P,\alpha)$-pair then $P'$ can twist on the tiles $G'(k)$ and $\mu_{G'(k)}P'\in \pi(P)$.
\end{enumerate}
where $G'(k)$ are the tiles given in Proposition \ref{Lem-tile-11}.
\end{proposition}

\begin{proof}
It follows by Proposition \ref{Lem-parvar} and the construction of $\pi(P)$.
\end{proof}

\begin{remark}\label{Rem-2}
Let $P\in \mathcal L$.
\begin{enumerate}[$(1)$]
\item Denote $r=[\sum m_i(P)]_+$. Denote by $\{k_1,\cdots,k_{r}\}\subset \{1,\cdots,\eta_\alpha \}$ the set of indices that are not in any ${\bf m}(P,\alpha)$-pair and $m_{k_j}=1$. Assume $k_1<k_2<\cdots<$ $k_{r}$. By Proposition \ref{Lem-pi}, we may write $\pi(P)$ as
$\{P(\vec{c})\mid \vec{c}=(c_1,\cdots,c_{r})\in \{0,1\}^{r}\}.$ To be precise, for any $j\in \{1,\cdots,r\}$ the edges labeled $\alpha_1,\alpha_3$ of $G'(k_j)$ are in $P(\vec{c})$ if $c_j=1$ and the the edges labeled $\alpha_2,\alpha_4$ of $G'(k_j)$ are in $P(\vec{c})$ if $c_j=0$.

\item Under the order of $\mathcal L'=\mathcal P(G_{T'^o,\widetilde\beta})$ in Section \ref{Sec-threesets}, we have $P(\vec{c})\leq P(\vec{c}\hspace{2pt}')$ if and only if $c_j\leq c'_j$ for all $j\in \{1,\cdots,r\}$.

%\item If $\sum m_i<0$, then $\pi(P)=\{P'\}$ for some $P'$, we also denote $$P(0,\cdots,0)=P(1,\cdots,1)=P'.$$
\end{enumerate}
\end{remark}

\subsubsection*{Partition bijection $\pi:\mathcal L(T^o,\widetilde\beta^{(q)})\to \mathcal L(T'^o,\widetilde\beta^{(q)})$.} Herein, we assume that $\mathcal L=\mathcal L(T^o,\widetilde\beta^{(q)})$ and $\mathcal L'=\mathcal L(T'^o,\widetilde\beta^{(q)})$.
%\subsubsection{Construction of $\pi:\mathcal L(T^o,\widetilde\beta^{(q)})\to \mathcal L(T'^o,\widetilde\beta^{(q)})$}
Recall in (\ref{Eq-mj}) that
\begin{equation*}
\begin{array}{rcl}
m(\Delta_j(q);\alpha)\hspace{-2mm}&= &\hspace{-2mm}
\text{number of edges labeled } \alpha \text{ in } \{\tau_{[j]}(q)\}\vspace{2mm}\\
\hspace{-2mm}&-&\hspace{-2mm}
\text{number of edges labeled } \alpha \text{ in } \{\tau_{j-1}(q),\tau_{j}(q)\}.
\end{array}
\end{equation*}
Note that if $m(\Delta_j(q);\alpha)=1$ then $\alpha=\tau_{[j]}(q)$ and thus either $\tau_{j-1}(q)=\alpha_1, \tau_j(q)=\alpha_4$ or $\tau_{j-1}(q)=\alpha_3, \tau_j(q)=\alpha_2$.

For ${\bf P}=(P,\Delta_j(q))\in \mathcal L$, denote ${\bf m}({\bf P}; \alpha)=({\bf m}(P,\alpha), m(\Delta_j(q);\alpha))$. As $\alpha$ is not a self-folded arc in $T^o$, we have $m(\Delta_j(q);\alpha)\in \{-1,0,1\}$ and thus ${\bf m}({\bf P}; \alpha)\in \{-1,0,1\}^{\eta_\alpha+1}$.

\begin{definition}\label{Def-pair1}
Assume that ${\bf m}({\bf P}; \alpha)=(m_1({\bf P}),\cdots,m_{\eta_\alpha}({\bf P}),m_{\eta_\alpha+1}({\bf P}))$. We choose pairs of indices $(a,b)$ via the following algorithm:
\begin{enumerate}[$(i)$]
\item If $m_{\eta_\alpha}({\bf P})m_{\eta_\alpha+1}({\bf P})=-1$, then first choose $(\eta_\alpha,\eta_\alpha+1)$ and all of the $(m_1({\bf P}),\cdots,m_{\eta_\alpha-1}({\bf P}))$-pairs as in Definition \ref{Def-pair};
\item If $m_{\eta_\alpha}({\bf P})m_{\eta_\alpha+1}({\bf P})\neq -1$, then we choose all of the $(m_1({\bf P}),\cdots,m_{\eta_\alpha+1}({\bf P})$-pairs as in Definition \ref{Def-pair}.
\end{enumerate}
The chosen pairs are called \emph{${\bf m}({\bf P}; \alpha)$-pairs}.
\end{definition}

The following is immediate.

\begin{lemma}\label{Lem-pair1}
For ${\bf m}(P,\alpha)$, we have
\begin{enumerate}[$(1)$]
\item ${\bf m}(P,\alpha)$-pairs and $-{\bf m}(P,\alpha)$-pairs coincide.
\item
For any ${\bf m}(P,\alpha)$-pair $(a,b)$ with $a<b$ we have
$\sum_{a<\ell<b}m_\ell(P)=0.$
\item If $\sum m_i(P)\geq 0$ then for any $a$ with $m_a(P)<0$ we have $a$ is in some ${\bf m}(P,\alpha)$-pair.
\item If $\sum m_i(P)\leq 0$ then for any $a$ with $m_a(P)>0$ we have $a$ is in some ${\bf m}(P,\alpha)$-pair.
\end{enumerate}
\end{lemma}

\emph{Construction of $\pi$.} For ${\bf P}=(P,\Delta_j(q))\in \mathcal L$,
$\pi({\bf P})$ contains all $(P',\Delta')\in \varphi_\alpha^{T^o}(P)\times \phi_\alpha^{T^o,q}(\Delta_j(q))$ satisfy the following condition: for any ${\bf m}(P,\Delta_j(q);\alpha)$-pair\footnote{We use ${\bf m}(P,\Delta_j(q);\alpha)$-pair as in Definition \ref{Def-pair1} to ensure Proposition \ref{prop-cover} holds} $(a,b)$,

$\bullet$ if $m_a({\bf P})=1,m_b({\bf P})=-1$ with $b\neq \eta_\alpha+1$ then the labels of $P'\cap edge(G'(a))$ and $P\cap edge(G(b))$ coincide;

$\bullet$ if $m_a({\bf P})=-1,m_b({\bf P})=1$ with $b\neq \eta_\alpha+1$ then the labels of $P'\cap edge(G'(b))$ and $P\cap edge(G(a))$ coincide;

$\bullet$ if $m_a({\bf P})=1,m_b({\bf P})=-1$ with $b=\eta_\alpha+1$ then the edges labeled $\tau_{j-2}(q)$ and $\tau_{[j]}(q)$ of $edge(G'(a))$ are in $P'$ in case $\tau_{j-1}(q)=\alpha$, or the edges labeled $\tau_{j+1}(q)$ and $\tau_{[j]}(q)$ of $edge(G'(a))$ are in $P'$ in case $\tau_{j}(q)=\alpha$;

$\bullet$ if $m_a({\bf P})=-1,m_b({\bf P})=1$ with $b=\eta_\alpha+1$,
\begin{enumerate}[$(i)$]
\item in case $\tau_{j-1}(q)=\alpha_1, \tau_j(q)=\alpha_4$, then $\Delta'=\{\alpha',\alpha_3,\alpha_4\}$ when the edges in $P\cap edge(G(a))$ labeled $\alpha_1$ and $\alpha_3$, or $\Delta'=\{\alpha',\alpha_1,\alpha_2\}$ when the edges in $P\cap edge(G(a))$ labeled $\alpha_2$ and $\alpha_4$.
\item in case $\tau_{j-1}(q)=\alpha_3, \tau_j(q)=\alpha_2$, then $\Delta'=\{\alpha',\alpha_1,\alpha_2\}$ when the edges in $P\cap edge(G(a))$ labeled $\alpha_1$ and $\alpha_3$, or $\Delta'=\{\alpha',\alpha_3,\alpha_4\}$ when the edges in $P\cap edge(G(a))$ labeled $\alpha_2$ and $\alpha_4$.
\end{enumerate}

Similarly, we can construction a set $\pi'({\bf P}')\subset \mathcal L$ for any ${\bf P}'\in \mathcal L'=\mathcal P(G_{T'^o,\widetilde\beta^{(q)}})$.

By Propositions \ref{Lem-parvar}, \ref{Lem-part1}, from the construction, the following proposition follows.

\begin{proposition}\label{Lem-pi1}
For any ${\bf P}=(P,\Delta_j(q))\in \mathcal L$ we have
$$|\pi({\bf P})|=2^{[\sum m_i({\bf P})]_+},$$
to be precise, for any ${\bf P}'=(P',\Delta')\in \pi({\bf P})$, we have
\begin{enumerate}[$(1)$]
\item ${\bf m}({\bf P}';\alpha')=-{\bf m}({\bf P};\alpha)$;
\item if $k\neq \eta_\alpha+1$ is not in any ${\bf m}({\bf P};\alpha)$-pair and $m_k({\bf P})=1$, then $P'$ can twist on $G'(k)$ and $(\mu_{G'(k)}P',\Delta')\in \pi({\bf P})$, where $G'(k)$ is the tile given in Proposition \ref{Lem-tile-11};
\item if $\eta(\alpha)+1$ is not in any ${\bf m}({\bf P};\alpha)$-pair and $m_{\eta_\alpha+1}({\bf P})=1$, then $(P',\Delta'')\in \pi({\bf P})$, where $\phi_\alpha^{T^o,q}(\Delta_j(q))=\{\Delta',\Delta''\}$.
\end{enumerate}
\end{proposition}

\begin{lemma}\label{lem-pj}
For any ${\bf P}=(P,\Delta_j(q))\in \mathcal L$, if $m(\Delta_j(q);\alpha)=1$, then $\phi_{\alpha}^{T^o,q}(\Delta_j(q))=\{\Delta'_1,\Delta'_2\}$ satisfies that $(P',\Delta'_1)<(P',\Delta'_2)$ for any $P'\in \varphi_{\alpha}^{T^o}(P)$.
\end{lemma}

\begin{proof}
Otherwise, by Lemma \ref{Lem-cover2}, we have $\Delta'_1(q)\in \phi_{\alpha}^{T^o,q}(\Delta_j(q))$ and there exist $P',Q'\in \varphi_{\alpha}^{T^o}(P)$ such that $E'_1(q)\in P', E'_2(q)\in Q'$. Thus by Proposition \ref{Lem-parvar}, the diagonal of $G'_{c'}$ is labeled $\alpha'$.
%, any $R\in \varphi_{\alpha'}^{T'^o}(P')$ can twist on $G_c$ and $\varphi_{\alpha'}^{T'^o}(P')$ is closed under twist on $G_c$. As the diagonal of $G_c$ is labeled $\alpha$,
%we see that
Thus $m(\Delta'_1(q),\alpha')=1$. It contradicts to $m(\Delta_j(q);\alpha)=1$ and Lemma \ref{lem-minv2}.
\end{proof}

\begin{remark}\label{Rem-3}
Denote $r=[\sum m_i({\bf P})]_+$. Denote by $\{k_1,\cdots,k_{r}\}\subset \{1,\cdots,\eta_\alpha+1\}$ the set of indices that are not in any ${\bf m}({\bf P};\alpha)$-pair and $m_{k_\ell}({\bf P})=1$ for $\ell=1,\cdots,r$. Suppose that $k_1<k_2<\cdots <k_{r}$. By Proposition \ref{Lem-pi1}, we may write $\pi({\bf P})$ as
$\{{\bf P}(\vec{c})\mid \vec{c}\in \{0,1\}^{r}\},$ to be precise, assume that ${\bf P}(\vec{c})=(P',\Delta'),$ for any $\ell\in \{1,\cdots,r\}$,
\begin{enumerate}[$(i)$]
\item if $k_\ell\neq \eta_\alpha+1$ then the edges labeled $\alpha_1,\alpha_3$ of $G'(k_\ell)$ are in $P$ in case $c_\ell=1$ and the the edges labeled $\alpha_2,\alpha_4$ of $G'(k_\ell)$ are in $P$ in case $c_\ell=0$;
\item if $k_\ell=\eta_\alpha+1$, then $\Delta'=\Delta'_1$ in case $c_\ell=0$ and $\Delta'=\Delta'_2$ in case $c_\ell=1$, where $\Delta'_1$ and $\Delta'_2$ are given by Lemma \ref{lem-pj}.
%assume that $\phi_\alpha^{T^o,q}(\Delta_j(q))=\{\Delta',\Delta''\}$ such that $(P',\Delta')<(P',\Delta'')$ for any $P'\in \varphi_\alpha^{T^o}(P)$ (by ), .
\end{enumerate}
\end{remark}

\begin{lemma}\label{lem-nu-pair}
Suppose that $\widetilde\beta\notin T'^o$. For any ${\bf P}=(P,\Delta_j(q))\in \mathcal L$,
%assume that ${\bf m}({\bf P};\alpha)=(m_1(P',\Delta'),\cdots,m_{\eta_\alpha+1}(P',\Delta'))$.
\begin{enumerate}[$(1)$]
\item if $\phi_{\alpha}^{T^o,q}(\Delta_j(q))=\{\Delta'_1(q),\Delta'_2(q)\}$ and $E'_2(q)\in P'$ for some $P'\in \varphi_{\alpha}^{T^o}(P)$, then $m_{\eta_\alpha}({\bf P})=-1, m_{\eta_\alpha+1}({\bf P})=1$;
\item If $\phi_{\alpha}^{T^o,q}(\Delta_j(q))=\{\Delta'_{t'}(q),\Delta'_1(q)\}$ and $E'_1(q)\in P'$ for some $P'\in \varphi_{\alpha}^{T^o}(P)$, then $m_{\eta_\alpha}({\bf P})=-1, m_{\eta_\alpha+1}({\bf P})=1$.
\end{enumerate}
Consequently, we have $(\eta_\alpha, \eta_\alpha+1)$ is an ${\bf m}({\bf P};\alpha)$-pair.
\end{lemma}

\begin{proof}
We shall only prove the first statement, as the second one can be proved similarly. As $\phi_{\alpha}^{T^o,q}(\Delta)=\{\Delta'_1(q),\Delta'_2(q)\}$, we see that $\tau'_1(q)=\alpha$ and the diagonal of $G'_{c'}$ is labeled $\alpha_1$ or $\alpha_3$. It follows that $m(\Delta'_1(q))=-1$ and $E'_2(q)$ is labeled $\alpha'$. Thus $m(\Delta_j(q))=1$ and $m_{\eta_\alpha}(P')=1$. It follows that $m_{\eta_\alpha}({\bf P})=-m_{\eta_\alpha}(P')=-1$ and $m_{\eta_\alpha+1}({\bf P})=m(\Delta_j(q))=1$.
\end{proof}

\begin{proposition}\label{prop-cover}
Suppose that $r=\sum m_i({\bf P})\geq 0$. Given any $k$, for $\vec{c}, \vec{c}\hspace{.8mm}'\in \{0,1\}^{r}$ such that $c_k=1,c'_k=0$ and $c_a=c'_a$ for $a\neq k$, then ${\bf P}(\vec{c})$ covers ${\bf P}(\vec{c}\hspace{.8mm}')$.
\end{proposition}

\begin{proof}
Assume that ${\bf P}(\vec{c}\hspace{.8mm}')=(P'_1,\Delta'_1)$ and ${\bf P}(\vec{c})=(P'_2,\Delta'_2)$. By Remark \ref{Rem-3}, we have either $P'_1=P'_2$ or $\Delta'_1=\Delta'_2$.

(Case I) $P'_1=P'_2$. We have $(P'_2,\Delta'_2)=(P'_1,\Delta'_2)>(P'_1,\Delta'_1)$ by Remark \ref{Rem-3} (2). By Lemma \ref{Lem-cover2}, $(P'_1,\Delta'_2)$ does not cover $(P'_1,\Delta'_1)$ if and only if either $E'_2(q)\in P'_1,\Delta'_2=\Delta'_1(q), \Delta'_1=\Delta'_2(q)$ or $E'_1(q)\in P'_1, \Delta'_2=\Delta'_{t'}(q), \Delta'_1=\Delta'_1(q)$.
We have $(\eta_\alpha,\eta_\alpha+1)$ is an ${\bf m}({\bf P};\alpha)$-pair by Lemma \ref{lem-nu-pair}. From the construction of $\pi$, we see that $(P'_1,\Delta'_1)$ and $(P'_1,\Delta'_2)$ can not belong to $\pi({\bf P})$, a contradiction.

(Case II) $\Delta'_1=\Delta'_2$. We have $P'_1<P'_2=\mu_{G'_l}P'_1$ for some tile $G'_l$ with diagonal labeled $\alpha'$ by Remark \ref{Rem-3} (1). If $(P'_2,\Delta'_2)=(\mu_{G'_l}P'_1,\Delta'_1)$ does not cover $(P'_1,\Delta'_1)$, by Lemma \ref{Lem-cover}, we have $\Delta'_1=\Delta'_1(q)$ and $G'_l=G'_{c'}$ is the last tile of $G_{T'^o,\widetilde\beta}$. It follows that $m_{\eta_\alpha+1}(P'_1,\Delta'_1)=1$ and $m_{\eta_\alpha}(P'_1,\Delta'_1)=-1$. Thus $m_{\eta_\alpha+1}({\bf P})=-1$ and $m_{\eta_\alpha}({\bf P})=1$. It follows that $(\eta_\alpha,\eta_\alpha+1)$ is an ${\bf m}({\bf P};\alpha)$-pair. From the construction of $\pi$, we see that $(P'_1,\Delta'_1)$ and $(P'_2,\Delta'_1)$ can not belong to $\pi({\bf P})$, a contradiction.
\end{proof}

\subsubsection*{Partition bijection $\pi_\alpha^{T^o,p,q}:\mathcal L(T^o,\widetilde\beta^{(p,q)})\to \mathcal L(T'^o,\widetilde\beta^{(p,q)})$}
 Herein we assume that $\mathcal L=\mathcal L(T^o,\widetilde\beta^{(p,q)})$ and $\mathcal L'=\mathcal L(T'^o,\widetilde\beta^{(p,q)})$.

Recall in (\ref{Eq-mi}) that
\begin{equation*}
\begin{array}{rcl}
m(\Delta_i(p);\alpha)\hspace{-2mm}&= &\hspace{-2mm}
\text{number of edges labeled } \alpha \text{ in } \{\tau_{[i]}(p)\}\vspace{2mm}\\
\hspace{-2mm}&-&\hspace{-2mm}
\text{number of edges labeled } \alpha \text{ in } \{\tau_{i-1}(p),\tau_{i}(p)\}.
\end{array}
\end{equation*}

\begin{lemma}\label{Lem-de1}
If $m(\Delta_i(p);\alpha)=1$ and $m(\Delta_j(q);\alpha)=-1$, assume that $\phi_\alpha^{T^o,q}(\Delta_j(q))=\{\Delta''\}$, then there is a unique $\Delta'\in \phi_\alpha^{T^o,p}(\Delta_i(p))$ such that
$$x^{T}(\Delta_i(p))x^{T}(\Delta_j(q))=x^{T'}(\Delta')x^{T'}(\Delta'').$$
\end{lemma}

\begin{proof}
Since $m(\Delta_j(q);\alpha)=-1$, we have the following cases: $\Delta_j(q)=\{\alpha_1,\alpha_4,\alpha\}$, $\tau_{[j]}(q)=\alpha_1$; $\Delta_j(q)=\{\alpha_1,\alpha_4,\alpha\}$, $\tau_{[j]}(q)=\alpha_4$; $\Delta_j(q)=\{\alpha_2,\alpha_3,\alpha\}$, $\tau_{[j]}(q)=\alpha_2$; $\Delta_j(q)=\{\alpha_2,\alpha_3,\alpha\}$, $\tau_{[j]}(q)=\alpha_3$. We shall only consider the case that $\Delta_j(q)=\{\alpha_1,\alpha_4,\alpha\}$, $\tau_{[j]}(q)=\alpha_1$, as the remaining cases can be proved similarly. Then $x^{T}(\Delta_j(q))=x_{\alpha_1}x^{-1}_{\alpha_4}x^{-1}_{\alpha}$, $\Delta''=\{\alpha_3,\alpha_4,\alpha'\}$ with $m(\Delta'';\alpha')=1$ and thus
$x^{T'}(\Delta'')=x_{\alpha'}x^{-1}_{\alpha_3}x^{-1}_{\alpha_4}$

Since $m(\Delta_i(p);\alpha)=1$, we have $\tau_{[i]}(p)=\alpha$ and $\Delta_i(p)=\{\alpha_1,\alpha_4,\alpha\}$ or $\{\alpha_2,\alpha_3,\alpha\}$.

In case $\Delta_i(p)=\{\alpha_1,\alpha_4,\alpha\}$, then $x^{T}(\Delta_i(p))=x_{\alpha}x^{-1}_{\alpha_1}x^{-1}_{\alpha_4}$. Thus $\Delta'=\{\alpha_3,\alpha_4,\alpha'\}$ with $m(\Delta';\alpha_3)=1$. Therefore $x^{T}(\Delta_i(p))x^{T}(\Delta_j(q))=\frac{1}{x^2_{\alpha_4}}=x^{T'}(\Delta')x^{T'}(\Delta'').$ The uniqueness of $\Delta'$ follows as $\alpha_2=\alpha_3$ and $\alpha_1=\alpha_4$ can not hold simultaneous.

\begin{figure}[h]
\centerline{\includegraphics{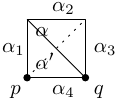}}
%\caption{}\label{Fig-proof}

\end{figure}

In case $\Delta_i(p)=\{\alpha_2,\alpha_3,\alpha\}$, then $x^{T}(\Delta_i(p))=x_{\alpha}x^{-1}_{\alpha_2}x^{-1}_{\alpha_3}$. Thus $\Delta'=\{\alpha_1,\alpha_2,\alpha'\}$ with $m(\Delta';\alpha_1)=1$. Therefore $x^{T}(\Delta_i(p))x^{T}(\Delta_j(q))=\frac{x_{\alpha_1}}{x_{\alpha_2}x_{\alpha_3}x_{\alpha_4}}=x^{T'}(\Delta')x^{T'}(\Delta'').$ The uniqueness of $\Delta'$ follows as $\alpha_2=\alpha_3$ and $\alpha_1=\alpha_4$ can not hold simultaneous.
\end{proof}

Similarly, we have the following result.

\begin{lemma}\label{Lem-de2}
If $m(\Delta_i(p);\alpha)=-1$ and $m(\Delta_j(q);\alpha)=1$, assume that $\phi_\alpha^{T^o,p}(\Delta_i(p))=\{\Delta'\}$, then there is a unique $\Delta''\in \phi_\alpha^{T^o,q}(\Delta_j(q))$ such that
$$x^{T}(\Delta_i(p))x^{T}(\Delta_j(q))=x^{T'}(\Delta')x^{T'}(\Delta'').$$
\end{lemma}

For ${\bf P}=(\Delta_i(p),P,\Delta_j(q))\in \mathcal L$, denote ${\bf m}({\bf P}; \alpha)=(m(\Delta_i(p);\alpha),{\bf m}(P,\alpha), m(\Delta_j(q);\alpha))$. As $\alpha$ is not a self-folded arc in $T^o$, we have $m(\Delta_j(q);\alpha)\in \{-1,0,1\}$ and thus ${\bf m}({\bf P}; \alpha)\in \{-1,0,1\}^{\eta_\alpha+1}$. Denote ${\bf m}({\bf P}; \alpha)=(m_0({\bf P}),m_1({\bf P}),\cdots,m_{\eta_\alpha}({\bf P}),m_{\eta_\alpha+1}({\bf P}))$.

\begin{definition}\label{Def-pair1}
For ${\bf m}({\bf P}; \alpha)$, we choose pairs of indices $(a,b)$ via the following algorithm:
\begin{enumerate}[$(i)$]
\item If $m_{0}({\bf P})m_{1}({\bf P})=-1$, then choose $(0,1)$ and let ${\bf m}'=(m_2({\bf P}),\cdots,m_{\eta_\alpha+1}({\bf P}))$, otherwise let ${\bf m}'={\bf m}({\bf P}; \alpha)$;
\item In ${\bf m}'$, if $m_{\eta_\alpha}({\bf P})m_{\eta_\alpha+1}({\bf P})=-1$, then choose $(\eta_\alpha,\eta_\alpha+1)$ and let ${\bf m}''$ be obtained by delete $m_{\eta_\alpha}({\bf P}),m_{\eta_\alpha+1}({\bf P})$ from ${\bf m}'$, otherwise let ${\bf m}''={\bf m}'$;
\item Choose all of the ${\bf m}''$-pairs as in Definition \ref{Def-pair}.
\end{enumerate}
The chosen pairs are called \emph{${\bf m}({\bf P}; \alpha)$-pairs}.
\end{definition}

The following is immediate.

\begin{lemma}\label{Lem-pair2}
For ${\bf m}(P,\alpha)$, we have
\begin{enumerate}[$(1)$]
\item ${\bf m}(P,\alpha)$-pairs and $-{\bf m}(P,\alpha)$-pairs coincide.
\item
For any ${\bf m}(P,\alpha)$-pair $(a,b)$ with $a<b$ we have
$\sum_{a<\ell<b}m_\ell(P)=0.$
\item If $\sum m_i(P)\geq 0$ then for any $a$ with $m_a(P)<0$ we have $a$ is in some ${\bf m}(P,\alpha)$-pair.
\item If $\sum m_i(P)\leq 0$ then for any $a$ with $m_a(P)>0$ we have $a$ is in some ${\bf m}(P,\alpha)$-pair.
\end{enumerate}
\end{lemma}

Note that if $m_{0}({\bf P})=1$ then $\alpha=\tau_{[i]}(p)$ and thus either either $\tau_{j-1}(p)=\alpha_1, \tau_j(p)=\alpha_4$ or $\tau_{j-1}(p)=\alpha_3, \tau_j(p)=\alpha_2$.
Similarly, if $m_{\eta_\alpha+1}({\bf P})=1$ then $\alpha=\tau_{[j]}(q)$ and thus either $\tau_{j-1}(q)=\alpha_1, \tau_j(q)=\alpha_4$ or $\tau_{j-1}(q)=\alpha_3, \tau_j(q)=\alpha_2$.

\emph{Construction of $\pi$.} For ${\bf P}=(\Delta_i(p),P,\Delta_j(q))\in \mathcal L$,
$\pi({\bf P})$ contains all $(\Delta',P',\Delta'')\in \phi_\alpha^{T^o,p}(\Delta_i(p))\times \varphi_\alpha^{T^o}(P)\times \phi_\alpha^{T^o,q}(\Delta_j(q))$ satisfy the following condition: for any ${\bf m}({\bf P};\alpha)$-pair $(a,b)$,

$\bullet$ if $m_a({\bf P})=1,m_b({\bf P})=-1$ with $a\neq 0, b\neq \eta_\alpha+1$ then the labels of $P'\cap edge(G'(a))$ and $P\cap edge(G(b))$ coincide;

$\bullet$ if $m_a({\bf P})=-1,m_b({\bf P})=1$ with $a\neq 0, b\neq \eta_\alpha+1$ then the labels of $P'\cap edge(G'(b))$ and $P\cap edge(G(a))$ coincide;

$\bullet$ if $m_a({\bf P})=1,m_b({\bf P})=-1$ with $a=0,b\neq \eta_\alpha+1$,
\begin{enumerate}[$(i)$]
\item in case $\tau_{j-1}(p)=\alpha_1, \tau_j(p)=\alpha_4$, then $\Delta'=\{\alpha',\alpha_3,\alpha_4\}$ when the edges in $P\cap edge(G(b))$ labeled $\alpha_1$ and $\alpha_3$, or $\Delta'=\{\alpha',\alpha_1,\alpha_2\}$ when the edges in $P\cap edge(G(b))$ labeled $\alpha_2$ and $\alpha_4$.
\item in case $\tau_{j-1}(p)=\alpha_3, \tau_j(p)=\alpha_2$, then $\Delta'=\{\alpha',\alpha_1,\alpha_2\}$ when the edges in $P\cap edge(G(b))$ labeled $\alpha_1$ and $\alpha_3$, or $\Delta'=\{\alpha',\alpha_3,\alpha_4\}$ when the edges in $P\cap edge(G(b))$ labeled $\alpha_2$ and $\alpha_4$.
\end{enumerate}

$\bullet$ if $m_a({\bf P})=-1,m_b({\bf P})=1$ with $a=0,b\neq \eta_\alpha+1$, then the edges labeled $\tau_{i-2}(p)$ and $\tau_{[i]}(p)$ of $edge(G'(b))$ are in $P'$ in case $\tau_{i-1}(p)=\alpha$, or the edges labeled $\tau_{i+1}(p)$ and $\tau_{[i]}(p)$ of $edge(G'(a))$ are in $P'$ in case $\tau_{i}(p)=\alpha$;

$\bullet$ if $m_a({\bf P})=1,m_b({\bf P})=-1$ with $a\neq 0,b=\eta_\alpha+1$ then the edges labeled $\tau_{j-2}(q)$ and $\tau_{[j]}(q)$ of $edge(G'(a))$ are in $P'$ in case $\tau_{j-1}(q)=\alpha$, or the edges labeled $\tau_{j+1}(q)$ and $\tau_{[j]}(q)$ of $edge(G'(a))$ are in $P'$ in case $\tau_{j-1}(q)=\alpha$;

$\bullet$ if $m_a({\bf P})=-1,m_b({\bf P})=1$ with $a\neq 0,b=\eta_\alpha+1$,
\begin{enumerate}[$(i)$]
\item in case $\tau_{j-1}(q)=\alpha_1, \tau_j(q)=\alpha_4$, then $\Delta'=\{\alpha',\alpha_3,\alpha_4\}$ when the edges in $P\cap edge(G(a))$ labeled $\alpha_1$ and $\alpha_3$, or $\Delta'=\{\alpha',\alpha_1,\alpha_2\}$ when the edges in $P\cap edge(G(a))$ labeled $\alpha_2$ and $\alpha_4$.
\item in case $\tau_{j-1}(q)=\alpha_3, \tau_j(q)=\alpha_2$, then $\Delta'=\{\alpha',\alpha_1,\alpha_2\}$ when the edges in $P\cap edge(G(a))$ labeled $\alpha_1$ and $\alpha_3$, or $\Delta'=\{\alpha',\alpha_3,\alpha_4\}$ when the edges in $P\cap edge(G(a))$ labeled $\alpha_2$ and $\alpha_4$.
\end{enumerate}

$\bullet$ if $m_a({\bf P})=1,m_b({\bf P})=-1$ with $a=0,b=\eta_\alpha+1$, then $\Delta'$ is given by Lemma \ref{Lem-de1}.

$\bullet$ if $m_a({\bf P})=-1,m_b({\bf P})=1$ with $a=0,b=\eta_\alpha+1$, then $\Delta''$ is given by Lemma \ref{Lem-de2}.

Similarly, we can construction a set $\pi'({\bf P}')\subset \mathcal L$ for any ${\bf P}'\in \mathcal L'=\mathcal P(G_{T'^o,\widetilde\beta^{(p,q)}})$.

By Propositions \ref{Lem-parvar}, \ref{Lem-part1}, from the construction, the following proposition follows.

\begin{proposition}\label{Lem-pi2}
For any ${\bf P}=(\Delta_i(p),P,\Delta_j(q))\in \mathcal L$ we have
$$|\pi({\bf P})|=2^{[\sum m_a({\bf P})]_+},$$
more precisely, for any ${\bf P}'=(\Delta',P',\Delta'')\in \pi({\bf P})$,

\begin{enumerate}[$(1)$]
\item ${\bf m}({\bf P}')=-{\bf m}({\bf P})$;
\item if $k(\neq 0, \eta_\alpha+1)$ is not in any ${\bf m}({\bf P};\alpha)$-pair and $m_k({\bf P})=1$, then $P'$ can twist on $G'(k)$ and $(\Delta',\mu_{G'(k)}P',\Delta'')\in \pi({\bf P})$;
\item if $0$ is not in any ${\bf m}({\bf P};\alpha)$-pair and $m_{0}({\bf P})=1$, then there is a triangle $\widetilde\Delta'\in \Delta_p(T'^o)$ such that $\widetilde\Delta'$ and $\Delta'$ share the common side $\alpha'$, and $(\widetilde\Delta',P',\Delta'')\in \pi({\bf P})$.
\item if $\eta(\alpha)+1$ is not in any ${\bf m}({\bf P};\alpha)$-pair and $m_{\eta_\alpha+1}({\bf P})=1$, then there is a triangle $\widetilde\Delta''\in \Delta_q(T'^o)$ such that $\widetilde\Delta''$ and $\Delta''$ share the common side $\alpha'$, and $(\Delta',P',\widetilde\Delta'')\in \pi({\bf P})$.
\end{enumerate}
where $G'(k)$ are the tiles given in Proposition \ref{Lem-tile-11}.
\end{proposition}

The following lemmas are similar to Lemma \ref{lem-pj}, the proofs are also similar, we omit them.

\begin{lemma}\label{lem-pj1}
For any ${\bf P}=(\Delta_i(p),P,\Delta_j(q))\in \mathcal L$, if $m(\Delta_j(q);\alpha)=1$, then $\phi_{\alpha}^{T^o,q}(\Delta_j(q))=\{\Delta''_1,\Delta''_2\}$ satisfies that $(\Delta',P',\Delta''_1)<(\Delta',P',\Delta''_2)$ for any $P'\in \varphi_{\alpha}^{T^o}(P)$ and $\Delta'\in \phi_{\alpha}^{T^o,p}(\Delta_i(p))$.
\end{lemma}

\begin{lemma}\label{lem-pj2}
For any ${\bf P}=(\Delta_i(p),P,\Delta_j(q))\in \mathcal L$, if $m(\Delta_i(p);\alpha)=1$, then $\phi_{\alpha}^{T^o,p}(\Delta_i(p))=\{\Delta'_1,\Delta'_2\}$ satisfies that $(\Delta'_1,P',\Delta')<(\Delta'_2,P',\Delta')$ for any $P'\in \varphi_{\alpha}^{T^o}(P)$ and $\Delta'\in \phi_{\alpha}^{T^o,q}(\Delta_j(q))$.
\end{lemma}

\begin{remark}\label{Rem-4}
Denote $r=[\sum m_i({\bf P})]_+$. Denote by $\{k_1,\cdots,k_{r}\}\subset \{1,\cdots,\eta_\alpha+1\}$ the set of indices that are not in any ${\bf m}({\bf P};\alpha)$-pair and $m_{k_\ell}({\bf P})=1$ for $\ell=1,\cdots,r$. Suppose that $k_1<k_2<\cdots <k_{r}$. By Proposition \ref{Lem-pi2}, we may write $\pi({\bf P})$ as
$\{{\bf P}(\vec{c})\mid \vec{c}\in \{0,1\}^{r}\},$ to be precise, assume that ${\bf P}(\vec{c})=(\Delta',P',\Delta''),$ for any $\ell\in \{1,\cdots,r\}$,
\begin{enumerate}[$(i)$]
\item if $k_\ell\neq 0,\eta_\alpha+1$, then the edges labeled $\alpha_1,\alpha_3$ of $G'(k_j)$ are in $P'$ in case $c_j=1$ and the the edges labeled $\alpha_2,\alpha_4$ of $G'(k_j)$ are in $P'$ in case $c_j=0$;
\item if $k_\ell=0$, then $\Delta'=\Delta'_1$ in case $c_\ell=0$ and $\Delta'=\Delta'_2$ in case $c_\ell=1$, where $\Delta'_1$ and $\Delta'_2$ are given by Lemma \ref{lem-pj2};
\item if $k_\ell=\eta_\alpha+1$, then $\Delta''=\Delta''_1$ in case $c_\ell=0$ and $\Delta''=\Delta''_2$ in case $c_\ell=1$, where $\Delta''_1$ and $\Delta''_2$ are given by Lemma \ref{lem-pj1}.
\end{enumerate}

%Under the order of $\mathcal L({T'^o},\widetilde\beta^{(p,q)})$ in Section \ref{Sec-threesets}, we have $(\Delta_i(p),P,\Delta_j(q))(\vec{x})\leq (\Delta_i(p),P,\Delta_j(q))(\vec{y})$ if and only if $x_j\leq y_j$ for all $j\in \{1,\cdots,[\sum m_a(\Delta_i(p),P,\Delta_j(q))]_+\}$.
\end{remark}

\begin{proposition}\label{prop-cover1}
Suppose that $r=\sum m_a({\bf P})\geq 0$. Given any $k$, assume that $\vec{c}, \vec{c}\hspace{.8mm}'$ in $\{0,1\}^{r}$ satisfy $c_k=1,c'_k=0$ and $c_a=c'_a$ for $a\neq k$, then ${\bf P}(\vec{c})$ covers ${\bf P}(\vec{c}\hspace{.8mm}')$.
\end{proposition}

\begin{proof}
Assume that ${\bf P}(\vec{c}\hspace{.8mm}')=(\Delta'_1,P'_1,\Delta''_1)$ and ${\bf P}(\vec{c})=(\Delta'_2,P'_2,\Delta''_2)$. By Remark \ref{Rem-4}, one of the following three cases happens: (1) $P'_1=P'_2$ and $\Delta''_1=\Delta''_2$, (2) $\Delta'_1=\Delta'_2$ and $\Delta''_1=\Delta''_2$, (3) $\Delta'_1=\Delta'_2$ and $P'_1=P'_2$.

In case $P'_1=P'_2$ and $\Delta''_1=\Delta''_2$, we have $(\Delta'_2,P''_2,\Delta''_2)=(\Delta'_2,P''_1,\Delta''_2)>(\Delta'_1,P'_1,\Delta'_1)$ by Remark \ref{Rem-4} (2). By Lemma \ref{Lem-cover3}, $(\Delta'_2,P''_2,\Delta''_2)=(\Delta'_2,P''_1,\Delta''_2)$ covers $(\Delta'_1,P'_1,\Delta'_1)$.

Similarly, $(\Delta'_2,P''_2,\Delta''_2)$ covers $(\Delta'_1,P'_1,\Delta'_1)$ in case $\Delta'_1=\Delta'_2$ and $P'_1=P'_2$.

In case $\Delta'_1=\Delta'_2$ and $\Delta''_1=\Delta''_2$, by Remark \ref{Rem-4} (1), $P'_1<P'_2=\mu_{G'_l}P'_1$ for some tile $G'_l$ with diagonal labeled $\alpha'$. Otherwise, if $(\Delta'_1,\mu_{G'_l}P'_1,\Delta''_1)$ does not cover $(\Delta'_1,P'_1,\Delta''_1)$, by Lemma \ref{Lem-cover1}, we have either $\Delta'_1=\Delta'_1(p)$ is the first triangle in $\Delta(T'^o,p)$ and $G'_l$ is the first tile of $G_{T'^o,\widetilde\beta}$ $\Delta''_1=\Delta'_1(q)$ is the first triangle in $\Delta(T'^o,q)$ and $G'_l$ is the last tile of $G_{T'^o,\widetilde\beta}$.

Suppose that $\Delta'_1=\Delta'_1(p)$ is the first triangle in $\Delta(T'^o,p)$ and $G'_l$ is the first tile of $G_{T'^o,\widetilde\beta}$. It follows that the diagonal $G'_l=G'_1$ is labeled $\alpha'$ and $m_{0}(\Delta'_1, P'_1,\Delta''_1)=1$. Thus $m_{0}({\bf P})=-1$ and $m_{1}({\bf P})=1$, thus $(0,1)$ is an ${\bf m}({\bf P};\alpha)$-pair.
From the construction of $\pi$, we see that $(\Delta'_1,\mu_{G'_l}P'_1,\Delta''_1)$ and $(\Delta'_1,P'_1,\Delta''_1)$ can not belong to $\pi({\bf P})$, a contradiction.

Suppose that $\Delta''_1=\Delta'_1(q)$ is the first triangle in $\Delta(T'^o,q)$ and $G'_l$ is the last tile of $G_{T'^o,\widetilde\beta}$. Similarly, $m_{\eta_\alpha}({\bf P})=1$, $m_{\eta_\alpha+1}({\bf P})=-1$ and $(\eta_\alpha,\eta_\alpha+1)$ can not be an ${\bf m}({\bf P};\alpha)$-pair. Thus, $\widetilde\beta=\alpha$, $\eta_\alpha=1$ and $m_{0}({\bf P})=-1$ with $(0,1=\eta_\alpha)$ is an ${\bf m}({\bf P};\alpha)$-pair. It follows that $r=\sum m_a({\bf P})=-1+1-1=-1$, contradicts to $r=\sum m_a({\bf P})\geq 0$.

The proof is complete.
\end{proof}

In summary of Propositions \ref{Lem-pi}, \ref{Lem-pi1}, \ref{Lem-pi2}, \ref{prop-cover}, \ref{prop-cover1} and Remarks \ref{Rem-2}, \ref{Rem-3}, \ref{Rem-4}, we obtain the following theorem.

\begin{theorem}\label{thm:pi}
For any ${\bf P}\in \mathcal L$, we have
\begin{enumerate}[$(1)$]
\item
$|\pi({\bf P})|=2^{[\sum m_i({\bf P})]_+}$ and ${\bf m}(\bf P',\alpha')=-{\bf m}(\bf P,\alpha)$ for all $\bf P'\in \pi(\bf P)$;
\item There is a bijective map $\{0, 1\}^{[\sum m_i({\bf P})]_+}\to \pi({\bf P}), \vec{c}\mapsto {\bf P}(\vec c)$ such that ${\bf P}(\vec c)$ covers ${\bf P}(\vec c\hspace{.8mm}')$ in $\mathcal L$ if and only if $\vec c$ covers $\vec c\hspace{.8mm}'$ in the lattice $\{0, 1\}^{[\sum m_i({\bf P})]_+}$.
\end{enumerate}
\end{theorem}

We need the following lemma for later use.

\begin{lemma}\label{lem:ome1}
For any ${\bf P}\in \mathcal L$ with $r=\sum m_i({\bf P};\alpha)>0$, for any $\ell\in \{1,\cdots,r\}$,
assume that $\vec{c}, \vec{c}\hspace{.8mm}'$ in $\{0,1\}^{r}$ satisfy $c_\ell=1,c'_\ell=0$ and $c_i=c'_i$ for $i\neq \ell$,
we have
    \begin{equation*}
    \Omega({\bf P}(\vec{c}\hspace{.8mm}'),{\bf P}(\vec c))=d(\alpha')(2\ell-1-r).
    \end{equation*}
\end{lemma}

\begin{proof}
We have $\Omega({\bf P}(\vec{c}\hspace{.8mm}'),{\bf P}(\vec c))=d(\alpha')\left(-(r-\ell)-(-\ell-1)\right)=d(\alpha')(2\ell-1-r).$
\end{proof}

\begin{proof}
It follows by Lemma \ref{Lem-var} and Lemma \ref{Lem-var1}.
\end{proof}

The main result in this section is the following.

\begin{theorem}\label{thm:partitionbij}
$\pi$ is a partition bijection with inverse $\pi'$.
\end{theorem}

\begin{proof}
By Theorem \ref{thm:pi}, $\pi({\bf P})\neq\emptyset$ for any ${\bf P}\in \mathcal L$.

If $\pi({\bf P})\cap\pi({\bf Q})\neq \emptyset$ for some ${\bf P},{\bf Q}$, from the construction of $\pi$, ${\bf P}$ and ${\bf Q}$ are related by a sequence of twists at $\alpha$. By Theorem \ref{thm:pi}, we have ${\bf m}(\bf P,\alpha)={\bf m}(\bf Q,\alpha)$. Thus $\pi(P)=\pi(Q)$ by the construction of $\pi$.

For any ${\bf P}\in \mathcal L, {\bf P}'\in \pi({\bf P})$, we have ${\bf m}({\bf P},\alpha)$-pairs and ${\bf m}({\bf P}',\alpha')$-pairs coincide by Lemmas \ref{Lem-pair}, \ref{Lem-pair1}, \ref{Lem-pair2} and Theorem \ref{thm:pi}. Thus ${\bf P}\in \pi'({\bf P}')$ iff ${\bf P}'\in \pi({\bf P})$. Therefore, $\bigcup_{{\bf P}\in \mathcal L}\pi({\bf P})=\mathcal L'$.

Therefore, by Remark \ref{Rmk-par}, $\pi$ is a partition bijection. Similarly, $\pi'$ is a partition bijection. Furthermore, $\pi'$ is the inverse of $\pi$ by Proposition \ref{Pro-parin}.
\end{proof}

\newpage

\section{Compatibility of the partition bijections and the lattice structures}\label{sec:CPP}
%Throughout, let $\mathcal L=\mathcal L(T^o,\widetilde\beta), \mathcal L(T^o,\widetilde\beta^{(q)})$ or $\mathcal L(T^o,\widetilde\beta^{(p,q)})$. Correspondingly, let $\mathcal L'=\mathcal L(T'^o,\widetilde\beta), \mathcal L(T'^o,\widetilde\beta^{(q)})$ or $\mathcal L(T'^o,\widetilde\beta^{(p,q)})$. Denote the minimal elements in $\mathcal L$ and $\mathcal L'$ by ${\bf P}_-$ and ${\bf P}'_-$, respectively.

For any pair ${\bf P}, {\bf Q}\in \mathcal L$, assume that ${\bf P}$ covers ${\bf Q}$, then we have that ${\bf P}$ and ${\bf Q}$ are either related by a twist on a tile $G_l$ or related by two adjacent triangles incident to $p$ or $q$, assume $\tau$ is the label of the diagonal or the common side of the triangles. In both cases, we say that ${\bf P}$ covers ${\bf Q}$ and \emph{related by $\tau$}. We denote
$$\Omega({\bf P}, {\bf Q})=w({\bf P})-w({\bf Q}).$$

We say that ${\bf P}, {\bf Q}\in \mathcal L$ are \emph{related by a twist at $\alpha$} if either ${\bf P}$ covers ${\bf Q}$ and related by $\tau$ or ${\bf Q}$ covers ${\bf P}$ and related by $\tau$.

{\bf Assumption}: We always assume that $T$ and $T'$ contain no arc tagged notched at $q$ in case $\mathcal L=\mathcal L(T^o,\widetilde\beta^{(q)})$, $T$ and $T'$ contain no arc tagged notched at $p$ or $q$ in case $\mathcal L=\mathcal L(T^o,\widetilde\beta^{(p,q)})$ if there is no other state.

\begin{proposition}\label{lem-p-5}
We have ${\bf P}'_-\in \pi({\bf P}_-)$. In particular, we have ${\bf P}'_-={\bf P}_-(0,\cdots, 0)$.
\end{proposition}

\begin{proof}
By Lemma \ref{lem-p-12}, $P'_-\in \varphi_\alpha^{T^o}(P_-)$. By Proposition \ref{Lem-parvar} we may assume $\sum m_i(P_-)\geq 0$. By Lemma \ref{lem-p--}, $m_i(P_-)\geq 0$ for any $i$. We have $\Delta'_1(q)\in \phi_{\alpha}^{T^o,q}(\Delta_1(q))$, $\Delta'_1(p)\in \phi_{\alpha}^{T^o,p}(\Delta_1(p))$ from Lemma \ref{lem-deltatodelta}.

The situation that $\mathcal L=\mathcal L(T^o,\widetilde\beta)$ follows by Lemma \ref{lem-p-12} and Lemma \ref{lem-p--}.

We then consider the situation that $\mathcal L=\mathcal L(T^o,\widetilde\beta^{(q)})$.

Case 1. If $m(\Delta_1(q);\alpha)\geq 0$, as $\Delta'_1(q)\in \phi_{\alpha}^{T^o,q}(\Delta_1(q))$, we see that $(P'_-,\Delta'_1(q))\in \pi(P_-,\Delta_1(q))$.

Case 2. If $m(\Delta_1(q);\alpha)<0$, then $m(\Delta_1(q);\alpha)=-1$ and $\tau_1(q)=\alpha$ or $\tau_t(q)=\alpha$. Thus we have $\phi_{\alpha}^{T^o,q}(\Delta_1(q))=\{\Delta'_1(q)\}$.

Case 2.1. $\tau_1(q)=\alpha$. We may assume that $\tau_t(q)=\alpha_4$ and $\tau_{[1]}(q)=\alpha_1$. We have either $\widetilde\beta=\tau_t(q)$ or $\widetilde\beta$ crosses $\alpha_1$ and ending at $q$.

Case 2.1.1. $\widetilde\beta=\tau_t(q)$. Then $\widetilde\beta=\alpha_4\in T^o\cap T'^o$. Thus $P_{-}=P'_-=P_{\widetilde\beta}$ and $\varphi_\alpha^{T^o}(P_-)=\{P'_-\}$ and thus $(P'_-,\Delta'_1(q))\in \pi(P_-,\Delta_1(q))$.

Case 2.1.2. $\widetilde\beta$ crosses $\alpha_1$ and ending at $q$.
If $\sum m_i(P_-)=0$, then $\varphi_\alpha^{T^o}(P_-)=\{P'_-\}$ and thus $(P'_-,\Delta'_1(q))\in \pi(P_-,\Delta_1(q))$. If $\sum m_i(P_-)>0$, then $m_i(P_-)=1$ for some $i$. By Lemma \ref{lem-p-}, one the following happens: $\widetilde\beta$ crosses $\alpha_2,\alpha,\alpha_4$ consecutively; or $\widetilde\beta$ starts from the common endpoint of $\alpha$ and $\alpha_1$ then crosses $\alpha_4$; or $\widetilde\beta$ starts from $q$ then crosses $\alpha_2$, see Figure \ref{Fig-proof}. We have either $\widetilde\beta$ intersects itself or the two endpoints are $q$, a contradiction.

\begin{figure}[h]
\centerline{\includegraphics{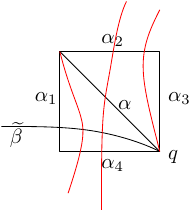}}

\caption{}\label{Fig-proof}

\end{figure}

Case 2.2. $\tau_t(q)=\alpha$, we may assume that $\tau_1(q)=\alpha_1$ and $\tau_{[1]}(q)=\alpha_4$. We have either $\widetilde\beta=\alpha$ or $\widetilde\beta$ crosses $\alpha_4$ and ending at $q$.

Case 2.2.1. $\widetilde\beta=\alpha$. Thus we have $\phi_{\alpha}^{T^o,q}(\Delta_1(q))=\{\Delta'_1(q)\}$, $P_{-}=P_{\widetilde\beta}$ and $\varphi_\alpha^{T^o}(P_-)=\{P'_-,P'_+\}$. For any $(P',\Delta')\in \pi(P_-,\Delta_1(q))$, from the construction of $\pi$, the edges labeled $\alpha_2$ and $\tau_{[1]}(q)=\alpha_4$ are in $P'$. Thus $(P'_-,\Delta'_1(q))\in \pi(P_-,\Delta_1(q))$.

Case 2.2.2. $\widetilde\beta$ crosses $\alpha_4$ and ending at $q$. Then $m_{\eta_\alpha}(P_-)=1$ and thus $(\eta_\alpha,\eta_\alpha+1)$ is an ${\bf m}({\bf P}_-;\alpha)$-pair. For any $(P',\Delta')\in \pi({\bf P}_-)$, from the construction of $\pi$, we have the edges labeled $\tau_{t-1}(q)=\alpha_2$ and $\tau_{[1]}(q)=\alpha_4$ of $G'(\eta_\alpha)$ are in $P'$. Thus $(P'_-,\Delta'_1(q))\in \pi(P_-,\Delta_1(q))$.

Last, we consider the situation that $\mathcal L=\mathcal L(T^o,\widetilde\beta^{(p,q)})$.

Case 1. $m(\Delta_1(p);\alpha),m(\Delta_1(q);\alpha)\geq 0$. As $\Delta'_1(p)\in \phi_{\alpha}^{T^o,p}(\Delta_1(p))$ and $\Delta'_1(q)\in \phi_{\alpha}^{T^o,q}(\Delta_1(q))$, we see that $(\Delta'_1(p),P'_-,\Delta'_1(q))\in \pi(\Delta_1(p),P_-,\Delta_1(q))$.

Case 2. $m(\Delta_1(p);\alpha)\geq 0, m(\Delta_1(q);\alpha)< 0$. Then $m(\Delta_1(q);\alpha)=-1$ and $\tau_1(q)=\alpha$ or $\tau_t(q)=\alpha$. Thus we have $\phi_{\alpha}^{T^o,q}(\Delta_1(q))=\{\Delta'_1(q)\}$.

Case 2.1. $\tau_1(q)=\alpha$. We may assume that $\tau_t(q)=\alpha_4$ and $\tau_{[1]}(q)=\alpha_1$. We have either $\widetilde\beta=\tau_t(q)$ or $\widetilde\beta$ crosses $\alpha_1$ and ending at $q$.

Case 2.1.1. $\widetilde\beta=\tau_t(q)$. Then $\widetilde\beta=\alpha_4\in T^o\cap T'^o$. Thus $P_{-}=P'_-=P_{\widetilde\beta}$ and $m(\Delta_1(p),\alpha)=0$.
Therefore, $\varphi_\alpha^{T^o}(P_-)=\{P'_-\}$ and thus $(\Delta'_1(p),P'_-,\Delta'_1(q))\in \pi(\Delta_1(p),P_-,\Delta_1(q))$.

Case 2.1.2. $\widetilde\beta$ crosses $\alpha_1$ and ending at $q$.

If $\sum m_i(P_-)>0$, then $m_i(P_-)=1$ for some $i$. From the above discussion, we have $p=q$, $\widetilde\beta$ starts from $q$ then crosses $\alpha_2$ and $m_i(P_-)=0$ for all $i$ with $1<i<\eta_\alpha$. We see that $\tau_{[1]}(p)=\alpha_2$, $m(\Delta_1(p);\alpha)=-1$ and $m_1(P_-)=1$. Contradicts to the assumption that $m(\Delta_1(p);\alpha)\geq 0$.

%Thus $(0,1)$ is an ${\bf m}({\bf P}_-;\alpha)$-pair and $\phi_{\alpha}^{T^o,p}(\Delta_1(p))=\{\Delta'_1(p)\}$. For any $(\Delta',P',\Delta'')\in \pi({\bf P}_-)$, from the construction of $\pi$, we have the edges labeled $\tau_{[1]}(p)=\alpha_2$ and $\alpha_4$ of $G'(1)$ are in $P'$. Thus $(\Delta'_1(p),P'_-,\Delta'_1(q))\in \pi({\bf P}_-)$.

%If $\sum m_i(P_-)=0$ and $m(\Delta_1(p);\alpha)\leq 0$, then $\varphi_\alpha^{T^o}(P_-)=\{P'_-\}$ and $\phi_{\alpha}^{T^o,p}(\Delta_1(p))=\{\Delta'_1(p)\}$. Thus $\pi({\bf P}_-)=\{{\bf P}'_-\}$.

If $\sum m_i(P_-)=0$ and $m(\Delta_1(p);\alpha)=1$, then $\varphi_\alpha^{T^o}(P_-)=\{P'_-\}$ and $\widetilde\beta$ starts from the common endpoint of $\alpha_2$ and $\alpha_3$ then crosses $\alpha$ and $\alpha_1$, see Figure \ref{Fig-proof1}. Then we have $\tau_{i_1}=\alpha, \tau_{i_2}=\alpha_1$, thus $m_1(P_-)=-1$, see Figure \ref{Fig-proof1}, contradicts to $\sum m_i(P_-)=0$.
%Thus $(1,\eta_\alpha+1)$ is an ${\bf m}({\bf P}_-;\alpha)$-pair and $\Delta'_1(q)=\{\tau_1'(p)=\alpha_2,\tau_{s'}(p)=\alpha',\tau_{[1]}(p)=\alpha_1\}$. From the construction of $\pi$, we see that $\pi({\bf P}_-)=\{{\bf P}'_-\}$.
\begin{figure}[h]
\centerline{\includegraphics{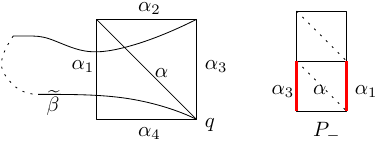}}

\caption{}\label{Fig-proof1}

\end{figure}

Case 2.2. $\tau_t(q)=\alpha$, we may assume that $\tau_1(q)=\alpha_1$ and $\tau_{[1]}(q)=\alpha_4$. We have either $\widetilde\beta=\alpha$ or $\widetilde\beta$ crosses $\alpha_4$ and ending at $q$.

Case 2.2.1. $\widetilde\beta=\alpha$. Then we have $\phi_{\alpha}^{T^o,q}(\Delta_1(q))=\{\Delta'_1(q)\}$, $P_{-}=P_{\widetilde\beta}$ and $\Delta_1(p)=\{\tau_s(p)=\alpha,\tau_1(p)=\alpha_3,\tau_{[1]}(p)=\alpha_2\}$. Thus $\varphi_\alpha^{T^o}(P_-)=\{P'_-,P'_+\}$, $\phi_{\alpha}^{T^o,p}(\Delta_1(p))=\{\Delta'_1(p)\}$ and $(1,2)$ is an ${\bf m}({\bf P}_-;\alpha)$-pair.
For any $(\Delta',P',\Delta'')\in \pi({\bf P})$, from the construction of $\pi$, the edges labeled $\alpha_2=\tau_{[1]}(p)$ and $\alpha_4$ are in $P'$. Thus $(\Delta'_1(p),P'_-,\Delta'_1(q))\in \pi({\bf P}_-)$.

Case 2.2.2. $\widetilde\beta$ crosses $\alpha_4$ and ending at $q$. Then $m_{\eta_\alpha}(P_-)=1$ and thus $\sum m_i(P_-)>0$. As $m(\Delta_1(p);\alpha)\geq 0$, we have $(\eta_\alpha,\eta_\alpha+1)$ is an ${\bf m}({\bf P}_-;\alpha)$-pair. For any $(\Delta',P',\Delta'')\in \pi({\bf P}_-)$, from the construction of $\pi$, we have the edges labeled $\tau_{t-1}(q)=\alpha_2$ and $\tau_{[1]}(q)=\alpha_4$ of $G'(\eta_\alpha)$ are in $P'$. Thus $(\Delta'_1(p),P'_-,\Delta'_1(q))\in \pi({\bf P}_-)$.

Case 3. $m(\Delta_1(p);\alpha)<0, m(\Delta_1(q);\alpha)\geq 0$. This case is similar to case 2.

Case 4. $m(\Delta_1(p);\alpha), m(\Delta_1(q);\alpha)<0$. Thus we have $\phi_{\alpha}^{T^o,q}(\Delta_1(q))=\{\Delta'_1(q)\}$ and $\phi_{\alpha}^{T^o,p}(\Delta_1(p))=\{\Delta'_1(p)\}$.

If $\sum m_i(P_-)\leq 0$, then $\varphi_\alpha^{T^o}(P_-)=\{P'_-\}$. Thus $\{(\Delta'_1(p),P'_-,\Delta'_1(q))\}=\pi({\bf P}_-)$.

If $\sum m_i(P_-)> 0$, then by Lemma \ref{lem-p-} we see that $\tau_{[1]}(p),\tau_{[1]}(q)\in \{\alpha_2,\alpha_4\}$. For any $(\Delta',P',\Delta'')\in \pi({\bf P}_-)$, from the construction of $\pi$, we have the edges labeled $\alpha_2$ and $\alpha_4$ of $G'(a)$ are in $P'$, where $a$ is in some ${\bf m}({\bf P}_-;\alpha)$-pair. Thus $(\Delta'_1(p),P'_-,\Delta'_1(q))\in \pi({\bf P}_-)$.

The proof is complete.
\end{proof}

\begin{thm}\label{thm-compatible1}
Under the {\bf Assumption}, assume that ${\bf P}$ covers ${\bf Q}$ and related by $\tau$. Denote ${\bf P}(0,\cdots,0)={\bf P}'$ and ${\bf Q}(0,\cdots,0)={\bf Q}'$.
\begin{enumerate}[$(1)$]
\item If $\tau\neq \alpha,\alpha_1,\alpha_2,\alpha_3,\alpha_4$, then ${\bf P}'$ covers ${\bf Q}'$ and related by $\tau$. Moreover, we have $$\Omega({\bf P}, {\bf Q})=\Omega({\bf P}', {\bf Q}').$$
\item If $\tau=\alpha$, then we have either

$\bullet$ ${\bf P}'={\bf Q}'$ or

$\bullet$ ${\bf P}'$ is covered by ${\bf Q}'$ and related by $\alpha'$ with $\Omega({\bf P}, {\bf Q})=-\Omega({\bf P}', {\bf Q}').$
\end{enumerate}
\end{thm}

\begin{proof}
(1) As $\tau\neq \alpha,\alpha_1,\alpha_2,\alpha_3,\alpha_4$, we have ${\bf m}({\bf P};\alpha)={\bf m}({\bf Q};\alpha)$.

We first consider the situation that $\mathcal L=\mathcal L(T^o,\widetilde\beta)$. Then ${\bf P}=P$ and ${\bf Q}=\mu_{G_l}P$. By Proposition \ref{Prop-mutation4}, there is a tile $G'_{[l]}$ of $G_{T'^o,\widetilde\beta}$ with diagonal labeled $\tau$ such that any $P'\in \varphi_\alpha^{T^o}(P)$ can twist on $G'_{[l]}$, $P'>\mu_{G'_{[l]}}(P')$, and $\varphi_\alpha^{T^o}({\bf Q})=\mu_{G'_{[l]}}(\varphi_\alpha^{T^o}(P))$. Thus ${\bf P}'$ covers ${\bf Q}'$ and related by $\tau$. As ${\bf m}(P,\alpha)= {\bf m}(\mu_{G_l}P,\alpha)$, we see that $\pi({\bf Q})=\mu_{G'_{[l]}}(\pi(P))$. By Proposition \ref{Prop-mutation4} (1), we have $m^{\pm}(P,G_l;\tau_{i_l})=m^{\pm}({\bf P}',G'_{[l]};\tau_{i_l})$ and $n(G_l^{\pm};\tau_{i_l})=n(G'^{\pm}_{[l]};\tau_{i_l})$, thus $\Omega({\bf P}',{\bf Q}')=\Omega({\bf P},{\bf Q})$.

We then consider the situation that $\mathcal L=\mathcal L(T^o,\widetilde\beta^{(q)})$. Assume that ${\bf P}=(P,\Delta_j(q))$.

If ${\bf Q}=(\mu_{G_l}P,\Delta_j(q))$, then by Proposition \ref{Prop-mutation4} there is a tile $G'_{[l]}$ of $G_{T'^o,\widetilde\beta}$ with diagonal labeled $\tau$ such that
$\pi(Q)=\{(\mu_{G'_{[l]}}(P'),\Delta')\mid (P',\Delta')\in \pi({\bf P})\}$ and $(P',\Delta')>(\mu_{G'_{[l]}}(P'),\Delta')$. In particular, ${\bf P}'>{\bf Q}'$. Next, we show ${\bf P}'$ covers ${\bf Q}'$. Otherwise, by Lemma \ref{Lem-cover} we have $G'_{[l]}$ is the last tile of $G_{T'^o,\widetilde\beta}$ and $\Delta'=\Delta'_1(q)$ is the first triangle incidents to $q$ in $T'^o$. As $\tau\neq \alpha,\alpha_1,\alpha_2,\alpha_3,\alpha_4$, we see that $G_l$ is the last tile of $G_{T^o,\widetilde\beta}$ and $m(\Delta';\alpha)=0$ and thus $\Delta_j(q)=\Delta'=\Delta_1(q)$ is the first triangle incidents to $q$ in $T^o$. By Lemma \ref{Lem-cover}, ${\bf P}=(P,\Delta_j(q))$ does not cover ${\bf Q}=(\mu_{G_{l}}P,\Delta_j(q))$, a contradiction. Therefore, ${\bf P}'$ covers ${\bf Q}'$. Assume that ${\bf P}'=(P',\Delta')$. Then ${\bf Q}'=(\mu_{G'_{[l]}}(P'),\Delta')$. From the previous case, we have $\Omega(P,\mu_{G_l}P)=\Omega(P',\mu_{G'_{[l]}}(P'))$. By Lemma \ref{Lem-var1}, we see that $\Omega({\bf P}',{\bf Q}')=\Omega({\bf P},{\bf Q})$.

If ${\bf Q}=(P,\Delta_{j-1}(q))$, as $\tau\neq \alpha,\alpha_1,\alpha_2,\alpha_3,\alpha_4$, we have $m(\Delta_{j-1}(q);\alpha)=m(\Delta_{j}(q);\alpha)=0$. Thus ${\bf P}'=(P',\Delta_j(q))$ for some $P'$ and ${\bf Q}'=(P',\Delta_{j-1}(q))$. In case neither $\Delta_{j-1}(q)$ nor $\Delta_{j}(q)$ is the first triangle incident to $q$ in $T'^o$, by Lemma \ref{Lem-cover2} ${\bf P}'$ covers ${\bf Q}'$. In case either $\Delta_{j-1}(q)$ or $\Delta_{j}(q)$ is the first triangle incident to $q$ in $T'^o$, as $\tau=\tau_{j-1}(q)\neq \alpha,\alpha_1,\alpha_2,\alpha_3,\alpha_4$, we see the last tiles of $G_{T^o,\widetilde\beta}$ and $G_{T'^o,\widetilde\beta}$ are the same and $E_1(p)\in P$ iff $E'_1(p)\in P'$. As ${\bf P}$ coves ${\bf Q}$, by Lemma \ref{Lem-cover2} we have ${\bf P}'$ covers ${\bf Q}'$. By Lemma \ref{Lem-var}, we have $\Omega({\bf P}',{\bf Q}')=\Omega(P,Q)$.

The situation that $\mathcal L=\mathcal L(T^o,\widetilde\beta^{(p,q)})$ is similar to the case $\mathcal L=\mathcal L(T^o,\widetilde\beta^{(q)})$.

(2) As $\tau=\alpha$, we have ${\bf m}({\bf P};\alpha)={\bf m}({\bf Q};\alpha)$. Suppose that ${\bf P}'\neq {\bf Q}'$.

We first consider the situation that $\mathcal L=\mathcal L(T^o,\widetilde\beta)$. Then ${\bf P}=P$ and ${\bf Q}=\mu_{G_l}P$. By Proposition \ref{Prop-mutation4} (2), we have $m_k(P)=-1$ and $G_{l}=G(k)$ for some $k$.
As ${\bf P}'\neq {\bf Q}'$, $k$ is in some ${\bf m}(P,\alpha)={\bf m}(\mu_{G_l}P,\alpha)$-pair. Assume that $(k,k')$ is an ${\bf m}(P,\alpha)$-pair, then $m_{k'}(\mu_{G_l}P)=m_{k'}(P)=1$. Thus the labels of ${\bf P}'\cap edge(G'(k'))$ and $P\cap edge(G(k))$ coincide, the labels of ${\bf Q}'\cap edge(G'(k'))$ and $\mu_{G_l}P\cap edge(G(k))$ coincide. Therefore ${\bf Q}'=\mu_{G'(k')}{\bf P}'$ and ${\bf Q}'> {\bf P}'$. Denote $G'_{[l]}=G'(k')$. We have
\begin{equation*}
m^{+}(P,G_l;\alpha)-n(G^+_l;\alpha)=\sum_{i>k}m_i(P),\hspace{5mm} m^{-}(P,G_l;\alpha)-n(G^-_l;\alpha)=\sum_{i<k}m_i(P),
\end{equation*}
\begin{equation*}
m^{+}({\bf P}',G'_{[l]};\alpha')-n(G'^+_{[l]};\alpha')=\sum_{i>k'}m_i({\bf P}'),\hspace{5mm} m^{-}({\bf P}',G'_{[l]};\alpha')-n({\bf G}'^-_{[l]};\alpha')=\sum_{i<k'}m_i(P').
\end{equation*}
It follows that
$$\Omega(P,\mu_{G_l}P)=d(\alpha)(\sum_{i>k}m_i(P)-\sum_{i<k}m_i(P)),\hspace{5mm }\Omega({\bf P}',{\bf Q}')=d(\alpha')(\sum_{i>k'}m_i({\bf P}')-\sum_{i<k'}m_i({\bf P}')).$$ As ${\bf m}(P,\alpha)=-{\bf m}({\bf P}',\alpha')$ and $(k,k')$ is an ${\bf m}(P,\alpha)$-pair, we obtain $\Omega(P,\mu_{G_l}P)=-\Omega({\bf P}',{\bf Q}')$.

We then consider the situation that $\mathcal L=\mathcal L(T^o,\widetilde\beta^{(q)})$. Assume that ${\bf P}=(P,\Delta_j(q))$.

If ${\bf Q}=(\mu_{G_l}P,\Delta_j(q))$, then by Proposition \ref{Prop-mutation4} (2) we have $m_k({\bf P})=-1$ and $G_{l}=G(k)$ for some $k$.
As ${\bf P}'\neq {\bf Q}'$, $k$ is in some ${\bf m}({\bf P},\alpha)$-pair. Assume that $(k,k')$ is an ${\bf m}({\bf P},\alpha)$-pair, then $m_{k'}({\bf Q})=m_{k'}({\bf P})=1$. Assume ${\bf P}'=(P',\Delta')$.

In case $k'\neq \eta_\alpha+1$, we have ${\bf Q}'=(\mu_{G'(k')}(P'),\Delta')>{\bf P}'$. We claim ${\bf Q}'$ covers ${\bf P}'$. Otherwise, by Lemma \ref{Lem-cover} $G'(k')$ is the last tile of $G_{T'^o,\widetilde\beta}$ and $\Delta'=\Delta'_1(q)$. It follows $k'=\eta_\alpha$ and $m_{\eta_\alpha}({\bf P}')=-1, m_{\eta_\alpha+1}({\bf P}')=1$. Thus $(k'=\eta_\alpha,\eta_\alpha+1)$ is an ${\bf m}({\bf P},\alpha)$-pair, contradicts to $(k,k')$ is an ${\bf m}({\bf P},\alpha)$-pair. Therefore, ${\bf Q}'$ covers ${\bf P}'$. Similar to the case $\mathcal L=\mathcal L(T^o,\widetilde\beta)$, we have $\Omega({\bf P},{\bf Q})=-\Omega({\bf P}',{\bf Q}')$.

In case $k'=\eta_\alpha+1$, we have $m_{\eta_\alpha+1}({\bf P})=1$, we may assume that $\tau_{j-1}(q)=\alpha_1, \tau_j(q)=\alpha_4$. Thus, $\Delta'=\{\alpha',\alpha_1,\alpha_2\}$ and ${\bf Q}'=(P',\Delta'')$, where $\Delta''=\{\alpha',\alpha_3,\alpha_4\}$, see Figure \ref{Fig-proof3}. We claim ${\bf Q}'$ covers ${\bf P}'$. Otherwise, by Lemma \ref{Lem-cover2}, we have either $\Delta'=\Delta'_1(q), E_2'(q)\in P'$ or $\Delta''=\Delta'_1(q), E_1'(q)\in P'$. In both cases, we have $\Delta_j(q)=\Delta_1(q)$ and $\tau_{i_c}=\alpha$. If $\Delta'=\Delta'_1(q), E_2'(q)\in P'$, then $E'_2(q)$ is labeled $\alpha'$. If $\Delta''=\Delta'_1(q), E_1'(q)\in P'$, then $E'_1(q)$ is labeled $\alpha'$. Thus $m_{\eta_\alpha}({\bf P'})=1$ and $m_{\eta_\alpha}({\bf P})=-1$. It follows that $G_l=G_c$ is the last tile. It contradicts to ${\bf P}=(P,\Delta_j(q))$ covers ${\bf Q}=(\mu_{G_l}P,\Delta_j(q))$. Therefore, ${\bf Q}'$ covers ${\bf P}'$. Similar to the case $\mathcal L=\mathcal L(T^o,\widetilde\beta)$, we have $\Omega({\bf P},{\bf Q})=-\Omega({\bf P}',{\bf Q}')$.
\begin{figure}[h]
\centerline{\includegraphics{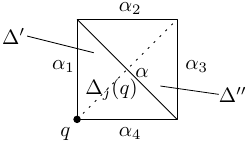}}

\caption{}\label{Fig-proof3}

\end{figure}

If ${\bf Q}=(P,\Delta_{j-1}(q))$, then $m_{\eta_\alpha+1}({\bf P})=-1$. As ${\bf P}'\neq {\bf Q}'$, $\eta_\alpha+1$ is in some ${\bf m}({\bf P},\alpha)$-pair. Assume that $(k,\eta_\alpha+1)$ is an ${\bf m}({\bf P},\alpha)$-pair and ${\bf P}'=(P',\Delta')$. Then ${\bf Q}'=(\mu_{G'(k')}P',\Delta')$. Dual to the case $k'=\eta_\alpha+1$ for ${\bf Q}=(\mu_{G_l}P,\Delta_j(q))$, we can prove ${\bf Q}'$ covers ${\bf P}'$ and $\Omega({\bf P},{\bf Q})=-\Omega({\bf P}',{\bf Q}')$.

The situation that $\mathcal L=\mathcal L(T^o,\widetilde\beta^{(p,q)})$ is similar to the case $\mathcal L=\mathcal L(T^o,\widetilde\beta^{(q)})$.
\end{proof}

\begin{thm}\label{thm-compatible2}
Under the {\bf Assumption}, let ${\bf P}=P,(P,\Delta_j(q))$ or $(\Delta_i(p),P,\Delta_j(q))$. Assume that ${\bf P}$ covers some ${\bf Q}$ and related by $\alpha_k$ for some $k\in \{1,2,3,4\}$ with $\alpha_k\neq \alpha_{k-1},\alpha_{k+1}$. Suppose that all $\alpha$-mutable edges of $P$ are labeled $\alpha_{k-1}$ or $\alpha_{k+1}$.
%Denote ${\bf P}(0,\cdots,0)={\bf P}'$ and ${\bf Q}(0,\cdots,0)={\bf Q}'$.
\begin{enumerate}[$(1)$]
\item Assume that $k\in\{1,3\}$, denote ${\bf P}(0,\cdots,0)={\bf P}'$ and ${\bf Q}(0,\cdots, 0)={\bf Q}'$. Then except for the following cases,

 $\bullet$ ${\bf P}=(P,\Delta_j(q)), {\bf Q}=(\mu_{G_l}P,\Delta_j(q))$ such that $\tau_j(q)=\alpha$ and $(P,\Delta_{j+1}(q))$ covers ${\bf P}$,

 $\bullet$ ${\bf P}=(\Delta_i(p),P,\Delta_j(q)), {\bf Q}=(\Delta_i(p),\mu_{G_l}P,\Delta_j(q))$ such that $\tau_i(p)\neq \alpha, \tau_j(q)=\alpha$ and $(\Delta_i(p),P,\Delta_{j+1}(q))$ covers ${\bf P}$,

 $\bullet$ ${\bf P}=(\Delta_i(p),P,\Delta_j(q)), {\bf Q}=(\Delta_i(p),\mu_{G_l}P,\Delta_j(q))$ such that $\tau_j(q)\neq \alpha, \tau_i(p)=\alpha$ and $(\Delta_{i+1}(p),P,\Delta_{j}(q))$ covers ${\bf P}$,

 $\bullet$ ${\bf P}=(\Delta_i(p),P,\Delta_j(q)), {\bf Q}=(\Delta_i(p),\mu_{G_l}P,\Delta_j(q))$ such that $\tau_i(p)=\tau_j(q)=\alpha$, $(\Delta_{i+1}(p),P,\Delta_{j}(q))$ covers ${\bf P}$ and $(\Delta_{i+1}(p),P,\Delta_{j+1}(q))$ covers $(\Delta_{i+1}(p),P,\Delta_{j}(q))$,

 $\bullet$ ${\bf P}=(\Delta_{i+1}(p),P,\Delta_j(q)), {\bf Q}=(\Delta_i(p),P,\Delta_j(q))$ such that $\tau_j(q)=\alpha$ and $(\Delta_{i+1}(p),P,\Delta_{j+1}(q))$ covers ${\bf P}$,

 $\bullet$ ${\bf P}=(\Delta_{i}(p),P,\Delta_{j+1}(q)), {\bf Q}=(\Delta_i(p),P,\Delta_j(q))$ such that $\tau_i(p)=\alpha$ and $(\Delta_{i+1}(p),P,\Delta_{j+1}(q))$ covers ${\bf P}$,

we have ${\bf P}'$ covers ${\bf Q}'$ and related by $\alpha_k$. Moreover, we have ${\bf P}={\bf P}'(1,\cdots,1)$ covers ${\bf Q}={\bf Q}'(1,\cdots,1)$ and related by $\alpha_k$, and $$\Omega({\bf P}, {\bf Q}_1)=\Omega({\bf P}', {\bf Q}').$$
\item Assume that $k\in\{2,4\}$, denote ${\bf P}(1,\cdots,1)={\bf P}'$ and ${\bf Q}(1,\cdots, 1)={\bf Q}'$. Then except for the following cases,

 $\bullet$ ${\bf P}=(P,\Delta_j(q)), {\bf Q}=(\mu_{G_l}P,\Delta_j(q))$ such that $\tau_{j-1}(q)=\alpha$ and ${\bf P}$ covers $(P,\Delta_{j-1}(q))$,

 $\bullet$ ${\bf P}=(\Delta_i(p),P,\Delta_j(q)), {\bf Q}=(\Delta_i(p),\mu_{G_l}P,\Delta_j(q))$ such that $\tau_{i-1}(p)\neq \alpha, \tau_{j-1}(q)=\alpha$ and ${\bf P}$ covers $(\Delta_i(p),P,\Delta_{j-1}(q))$,

 $\bullet$ ${\bf P}=(\Delta_i(p),P,\Delta_j(q)), {\bf Q}=(\Delta_i(p),\mu_{G_l}P,\Delta_j(q))$ such that $\tau_{j-1}(q)\neq \alpha, \tau_{i-1}(p)=\alpha$ and ${\bf P}$ covers $(\Delta_{i-1}(p),P,\Delta_{j}(q))$,

 $\bullet$ ${\bf P}=(\Delta_i(p),P,\Delta_j(q)), {\bf Q}=(\Delta_i(p),\mu_{G_l}P,\Delta_j(q))$ such that $\tau_{i-1}(p)=\tau_{j-1}(q)=\alpha$, ${\bf P}$ covers $(\Delta_{i-1}(p),P,\Delta_{j}(q))$ and $(\Delta_{i-1}(p),P,\Delta_{j}(q))$ covers $(\Delta_{i-1}(p),P,\Delta_{j-1}(q))$,

 $\bullet$ ${\bf P}=(\Delta_{i+1}(p),P,\Delta_j(q)), {\bf Q}=(\Delta_i(p),P,\Delta_j(q))$ such that $\tau_{j-1}(q)=\alpha$ and ${\bf P}$ covers $(\Delta_{i+1}(p),P,\Delta_{j-1}(q))$,

 $\bullet$ ${\bf P}=(\Delta_{i}(p),P,\Delta_{j+1}(q)), {\bf Q}=(\Delta_i(p),P,\Delta_j(q))$ such that $\tau_{i-1}(p)=\alpha$ and ${\bf P}$ covers $(\Delta_{i-1}(p),P,\Delta_{j+1}(q))$,

 we have ${\bf P}'$ covers ${\bf Q}'$ and related by $\alpha_k$. Moreover, assume that ${\bf P}_1={\bf P}'(0,\cdots,0), {\bf Q}_1={\bf Q}'(0,\cdots,0)$, then ${\bf P}_1$ covers ${\bf Q}_1$ and related by $\alpha_k$, and $$\Omega({\bf P}_1, {\bf Q}_1)=\Omega({\bf P}', {\bf Q}').$$
\end{enumerate}
\end{thm}

\begin{proof}
We shall prove (1) as (2) can be proved similarly. %We have $\sum m_i({\bf P})-\sum m_i({\bf Q})=b_{\alpha,\alpha_1}^{T^o}>0$.

We first consider the situation that ${\bf P}=P$. Then ${\bf Q}=\mu_{G_l}P$ for some tile $G_l$. As all the $\alpha$-mutable edges in $P$ are labeled $\alpha_{2}$ or $\alpha_{4}$, we see that all the $\alpha$-mutable edges in $\mu_{G_l}P$ and $\alpha'$-mutable edges in ${\bf P}'$ are labeled $\alpha_2$ or $\alpha_4$. By Proposition \ref{Prop-mutation4} (3), there is a tile $G'_{[l]}$ of $G_{T'^o,\widetilde\beta}$ with diagonal labeled $\alpha_1$ such that ${\bf P}'$ can twist on $G'_{[l]}$ and ${\bf P}'>\mu_{G'_{[l]}}{\bf P}'\in \varphi_{\alpha}^{T^o}(\mu_{G_l}P)$. It follows that all the $\alpha'$-mutable edges in $\mu_{G'_{[l]}}{\bf P}'$ are labeled $\alpha_2$ or $\alpha_4$. Therefore, we have $\mu_{G'_{[l]}}{\bf P}'\in \pi(\mu_{G_l}P)$ and $\mu_{G'_{[l]}}{\bf P}'=(\mu_{G_l}P)(0,\cdots,0),(\mu_{G'_{[l]}}{\bf P}')(1,\cdots,1)=\mu_{G_l}(P)$. By Proposition \ref{Prop-mutation4} (3), we have $\Omega({\bf P}';G'_{[l]})=\Omega(P;G_{l})$.

We then consider the situation that ${\bf P}=(P,\Delta_j(q))$.

Case 1. ${\bf Q}=(\mu_{G_l}P,\Delta_j(q))$. As the case that $\tau_j(q)=\alpha$ and $(P,\Delta_{j+1}(q))$ covers ${\bf P}$ is excepted, we shall consider the following cases:

Case 1.1. $\tau_{j}(q)\neq \alpha$;
% $m(\Delta_j(q);\alpha)\geq 0$ or $m(\Delta_j(q);\alpha)=-1$ with $\tau_{j-1}(q)=\alpha$.
Case 1.2. $\tau_{j}(q)=\alpha$ but $(P,\Delta_{j+1}(q))$ does not cover ${\bf P}$.

In Case 1.1, we have either $m(\Delta_j(q);\alpha)\geq 0$ or $m(\Delta_j(q);\alpha)=-1$ with $\tau_{[j]}(q)\in \{\alpha_2,\alpha_4\}$, the case can be proved similar to the case that ${\bf P}=P$.

In Case 1.2, by Lemma \ref{lem:cop} we have $G_l\neq G_c$ and $(\eta_\alpha,\eta_\alpha+1)$ is an ${\bf m}({\bf P};\alpha)$ and ${\bf m}({\bf Q};\alpha)$-pair. Assume that ${\bf P}'=(P',\Delta')$. Thus $P'\cap edge(G'(\eta_\alpha))$ is labeled $\alpha_1,\alpha_3$. By Proposition \ref{Prop-mutation4} (3), there is a tile $G'_{[l]}$ of $G_{T'^o,\widetilde\beta}$ with diagonal labeled $\alpha_1$ such that ${\bf P}'$ can twist on $G'_{[l]}$ and ${\bf P}'>\mu_{G'_{[l]}}{\bf P}'\in \varphi_{\alpha}^{T^o}(\mu_{G_l}P)$. Since $G_l\neq G_c$, we have $\mu_{G'_{[l]}}P'\cap edge(G'(\eta_\alpha))$ is labeled $\alpha_1,\alpha_3$. Thus $(\mu_{G'_{[l]}}P',\Delta')\in \pi({\bf Q})$ and ${\bf Q}'=(\mu_{G'_{[l]}}P',\Delta')$. By Proposition \ref{Prop-mutation4} (3), we have $\Omega({\bf P}';G'_{[l]})=\Omega(P;G_{l})$.

Case 2. ${\bf Q}=(P,\Delta_{j-1}(q))$. Since all $\alpha$-mutable edges of $P$ are labeled $\alpha_{2}$ or $\alpha_{4}$, this case can be proved similar to the case that ${\bf P}=P$.

The situation that ${\bf P}=(\Delta_i(p), P,\Delta_j(q))$ can be proved similarly to situation that ${\bf P}=(P,\Delta_j(q))$ by using Proposition \ref{Prop-mutation4} (3) and Lemmas \ref{lem:cop1}, \ref{lem:cop2}, \ref{lem:cop3}.
\end{proof}

The case that $\alpha_k=\alpha_{k-1}$ or $\alpha_{k+1}$ need more careful discussion.

\begin{thm}\label{thm-compatible3}
Under the {\bf Assumption}, assume that ${\bf P}$ covers ${\bf Q}$ and related by $\alpha_k$ for some $k\in \{1,2,3,4\}$ with $\alpha_k=\alpha_{k-1}$ or $\alpha_{k+1}$, without loss of generality, we may assume $k=1$ and $\alpha_1=\alpha_2$.
%Then except for the following case:
%$\bullet$ ${\bf P}=(P,\Delta_j(q)), {\bf Q}=(\mu_{G_l}P,\Delta_j(q))$ \huang{XXXXXXXXXXXX}
Then one of the following holds:
\begin{enumerate}[$(i)$]
\item there exist ${\bf P}_0, {\bf Q}_0\in \mathcal L$ and ${\bf P}',{\bf R}',{\bf Q}'\in \mathcal L'$ satisfying
\begin{enumerate}[$(1)$]
\item ${\bf P}_0$ and ${\bf P}$ are related by a sequence of twists at $\alpha$;
\item ${\bf Q}_0$ and ${\bf Q}$ are related by a sequence of twists at $\alpha$;
\item ${\bf P}_0$ covers ${\bf Q}_0$ and related by $\alpha_1$;
\item ${\bf P}'={\bf P}_0(0,\cdots,0)$ or ${\bf P}_0(1,\cdots,1)$ and ${\bf Q}'={\bf Q}_0(0,\cdots,0)$ or ${\bf Q}_0(1,\cdots,1)$;
\item ${\bf P}_0={\bf P}'(0,\cdots,0)$ or ${\bf P}'(1,\cdots,1)$ and ${\bf Q}_0={\bf Q}'(0,\cdots,0)$ or ${\bf Q}'(1,\cdots,1)$;
\item ${\bf P}'$ covers ${\bf R}'$ and ${\bf R}'$ covers ${\bf Q}'$.
\end{enumerate}

Moreover, we have
%assume ${\bf P}'$ covers ${\bf R}'$ and related by $\tau_1$, ${\bf R}'$ covers ${\bf Q}'$ and related by $\tau_2$, then $\{\tau_1,\tau_2\}=\{\alpha_1,\alpha'\}$ and
$\Omega({\bf P}_0,{\bf Q}_0)=\Omega({\bf P}',{\bf R}')+\Omega({\bf R}',{\bf Q}').$

\item there exist ${\bf P}_0,{\bf R}_0, {\bf Q}_0\in \mathcal L$ and ${\bf P}',{\bf Q}'\in \mathcal L'$ satisfying
\begin{enumerate}[$(1)$]
\item ${\bf P}_0$ and ${\bf P}$ are related by a sequence of twists at $\alpha$;
\item ${\bf Q}_0$ and ${\bf Q}$ are related by a sequence of twists at $\alpha$;
\item ${\bf P}_0$ covers ${\bf R_0}$ and ${\bf R}_0$ covers ${\bf Q}_0$;
\item ${\bf P}'={\bf P}_0(0,\cdots,0)$ or ${\bf P}_0(1,\cdots,1)$ and ${\bf Q}'={\bf Q}_0(0,\cdots,0)$ or ${\bf Q}_0(1,\cdots,1)$;
\item ${\bf P}_0={\bf P}'(0,\cdots,0)$ or ${\bf P}'(1,\cdots,1)$ and ${\bf Q}_0={\bf Q}'(0,\cdots,0)$ or ${\bf Q}'(1,\cdots,1)$;
\item ${\bf P}'$ covers ${\bf Q}'$.
\end{enumerate}

Moreover, we have
%assume ${\bf P}'$ covers ${\bf R}'$ and related by $\tau_1$, ${\bf R}'$ covers ${\bf Q}'$ and related by $\tau_2$, then $\{\tau_1,\tau_2\}=\{\alpha_1,\alpha'\}$ and
$\Omega({\bf P}_0,{\bf R}_0)+\Omega({\bf R}_0,{\bf Q}_0)=\Omega({\bf P}',{\bf Q}').$

\end{enumerate}
\end{thm}

\begin{proof}
We first consider the situation that ${\bf P}=P$. Assume that $Q=\mu_{G_l}P$ and the conditions in Proposition \ref{dProp-mutation4} hold.

In case $(m_k(P),m_{k+1}(P))=(-1,1)$, do twists for $P$ on the tiles $G_a$ with diagonal labeled $\alpha$ and $P\cap Edge(G_a)$ are labeled $\alpha_1,\alpha_4$, we obtain a perfect matching $P_0$. Thus all $\alpha$-twist-able edges in $P_0$ and $Q_0$ are labeled $\alpha_1,\alpha_3$. Since $m_k(P)=-1$, we have $\tau_{i_{l-1}}=\alpha$. We see that $Q$ can twist on these tiles $G_a$. Do the same twists for $Q$ on these $G_a$, we obtain $Q_0$. Therefore, $Q_0=\mu_{G_l}P_0<P_0$.
Let $Q'=Q_0(1,\cdots,1)$. By Proposition \ref{dProp-mutation4}, we have $P'=\mu_{G'(k+1)}\mu_{G'}Q'\in \varphi^{T^o}_\alpha(P_0)$. Since all $\alpha$-twist-able edges in $P_0$ are labeled $\alpha_1,\alpha_3$ and all $\alpha'$-twist-able edges in $P'$ are labeled $\alpha_1,\alpha_3$, we have $P'=P_0(1,\cdots,1)$. Let $R'=\mu_{G'}Q'$. By Proposition \ref{dProp-mutation4} (1) (c), we have $\Omega(P_0, Q_0)=\Omega( P', R')+\Omega( R',Q').$

The case $(m_k(P),m_{k+1}(P))=(0,0)$ is similar to the case that $(m_k(P),m_{k+1}(P))=(-1,1)$ as $(m_k(Q),m_{k+1}(Q))=(1,-1)$.

In case $(m_k(P),m_{k+1}(P))=(0,1)$, do twist for $P$ on the tiles $G_a$ with diagonal labeled $\alpha$ such that $P\cap Edge(G_a)$ are labeled $\alpha_1,\alpha_4$, we obtain a perfect matching $P_0$. Thus all $\alpha$-twist-able edges in $P_0$ and $Q_0$ are labeled $\alpha_1,\alpha_3$. Since $m_k(P)=0$, we see that $Q$ can twist on these tiles $G_a$. Do the same twists for $Q$ on these $G_a$, we obtain $Q_0$. Therefore, $Q_0=\mu_{G_l}P_0<P_0$.
Let $P'=P_0(1,\cdots,1)$. Then $P'>\mu_{G'(k+1)}P'$. By Proposition \ref{dProp-mutation4}, $\mu_{G'(k+1)}P'$ can twist on $G'$ and $Q':=\mu_{G'}\mu_{G'(k+1)}P'\in \varphi^{T^o}_\alpha(Q_0)$. Since all $\alpha$-twist-able edges in $Q_0$ are labeled $\alpha_1,\alpha_3$ and all $\alpha'$-twist-able edges in $Q'$ are labeled $\alpha_1,\alpha_3$, we have $Q'=Q_0(1,\cdots,1)$. Let $R'=\mu_{G'(k+1)}P'$. By Proposition \ref{dProp-mutation4} (1) (c), we have $\Omega(P_0, Q_0)=\Omega( P', R')+\Omega( R',Q').$

We then consider the situation that ${\bf P}=(P,\Delta_j(q))$. We have the following two cases: ${\bf Q}=(P,\Delta_{j-1}(q))$ or $(\mu_{G_l}P,\Delta_j(q))$.

Case 1. ${\bf Q}=(P,\Delta_{j-1}(q))$. Since $T'$ contains no arc tagged notched at $q$, we have $\tau_{[j-1]}(q)=\tau_{[j]}(q)=\alpha, \tau_{j-2}(q)=\alpha_3, \tau_{j}(q)=\alpha_4$. Do twist for $P$ on the tiles $G_a(\neq G_c)$ with diagonal labeled $\alpha$ such that $P\cap Edge(G_a)$ are labeled $\alpha_1,\alpha_4$, we obtain a perfect matching $P_0$. Let ${\bf P}_0=(P_0,\Delta_j(q))$ and ${\bf Q}_0=(P_0,\Delta_{j-1}(q))$. Since $G_a\neq G_c$, we have ${\bf P}_0$ and ${\bf P}$ are related by a sequence of twists at $\alpha$, ${\bf Q}_0$ and ${\bf Q}$ are related by a sequence of twists at $\alpha$ and ${\bf P}_0$ covers ${\bf Q}_0$. Let ${\bf P}'_0={\bf P}_0(1,\cdots,1)$ and ${\bf Q}'_0={\bf Q}_0(1,\cdots,1)$. Assume ${\bf P}'_0=(P',\Delta'_{j'}(q))$.

If $\tau_{i_c}=\alpha$ and $P$ can twist on $G_c$ with $P\cap edge(G_c)$ are labeled $\alpha_1,\alpha_4$. Then $(\eta_\alpha,\eta_\alpha+1)$ is an $m({\bf P}_0;\alpha)$ and $m({\bf Q}_0;\alpha)$-pair. Thus $\Delta'_{j'}(q)=(\alpha_1,\alpha',\alpha_1)$ with $\tau'_{j'}(q)=\alpha', \tau'_{j'-1}(q)=\alpha_1$ and ${\bf Q}'_0=(P',\Delta'_{j'-2}(q))$ with $\Delta'_{j'-2}(q)=(\alpha',\alpha_3,\alpha_4)$ and $\tau'_{j'-2}(q)=\alpha', \tau'_{j'-3}(q)=\alpha_3$. Therefore, $\Delta'_{j'-1}(q)=(\alpha_1,\alpha',\alpha_1)$. Let ${\bf R}'=(P',\Delta'_{j'-1}(q))$. Since all $\alpha$-twist-able edges in $P_0$ (except for the edges in $G_c$) are labeled $\alpha_1,\alpha_3$ and $(\eta_\alpha,\eta_\alpha+1)$ is an $m({\bf P}_0;\alpha)$ and $m({\bf Q}_0;\alpha)$-pair, we have ${\bf P}_0={\bf P}'(0,\cdots,0)$ and ${\bf Q}_0={\bf Q}'(0,\cdots,0)$.
\begin{figure}[h]
\centerline{\includegraphics{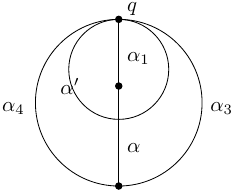}}

%\caption{}\label{Fig-proof3}

\end{figure}

Otherwise, we have $\Delta'_{j'}(q)=(\alpha_4,\alpha',\alpha_3)$ with $\tau'_{j'}(q)=\alpha_4, \tau'_{j'-1}(q)=\alpha'$ and ${\bf Q}'_0=(P',\Delta'_{j'-2}(q))$ with $\Delta'_{j'-2}(q)=(\alpha_1,\alpha',\alpha_1)$ and $\tau'_{j'-2}(q)=\alpha_1, \tau'_{j'-3}(q)=\alpha'$. Thus $\Delta'_{j'-1}(q)=(\alpha',\alpha_1,\alpha_1)$. Let ${\bf R}'=(P',\Delta'_{j'-1}(q))$. Since all $\alpha$-twist-able edges in $P_0$ are labeled $\alpha_1,\alpha_3$,  we have ${\bf P}_0={\bf P}'(0,\cdots,0)$ and ${\bf Q}_0={\bf Q}'(0,\cdots,0)$.

In both cases, one see that $E_1(q)\in P$ iff $E'_1(q)\in P'$. Since ${\bf P}$ covers ${\bf Q}$, we have ${\bf P}'$ covers ${\bf R}'$ and ${\bf R}'$ covers ${\bf Q}'$ by Lemma \ref{Lem-cover2}. We have
$\Omega({\bf P}_0, {\bf Q}_0)=\Omega({\bf P}', {\bf R}')+\Omega({\bf R}',{\bf Q}')$ by Proposition \ref{dProp-mutation4} (4).

Case 2. ${\bf Q}=(\mu_{G_l}P,\Delta_j(q))$. We assume that the conditions in Proposition \ref{dProp-mutation4} hold since the other cases can be proved similarly.

Case 2.1. $(m_k(P),m_{k+1}(P))=(-1,1)$.

Case 2.1.1. We first assume that $G_l=G_c$. Since the conditions in Proposition \ref{dProp-mutation4} hold, we have $\Delta_1(q)=(\alpha,\alpha_4,\alpha_1)$ with $m(\Delta_1(q);\alpha_1)=1$. Do twist for $P$ on the tiles $G_a$ with diagonal labeled $\alpha$ such that $P\cap Edge(G_a)$ are labeled $\alpha_1,\alpha_3$, we obtain a perfect matching $P_0$. Do twist for $\mu_{G_l}(P)$ on these tiles $G_a\neq G_{c-1}$, we obtain a perfect matching $Q_0$. Thus, $Q_0<\mu_{G_{c-1}} P_0<P_0$ and all $\alpha$-twist-able edges in $P_0$ and $Q_0$ are labeled $\alpha_1,\alpha_4$. Let ${\bf P}_0=(P_0,\Delta_j(q))$, ${\bf R}_0=(\mu_{G_{c-1}}P_0,\Delta_j(q))$ and ${\bf Q}_0=(Q_0,\Delta_j(q))$. By Lemma \ref{Lem-cover}, we have ${\bf P}_0$ and ${\bf P}$ are related by a sequence of twist at $\alpha$, ${\bf Q}_0$ and ${\bf Q}$ are related by a sequence of twist at $\alpha$, and ${\bf P}_0$ covers ${\bf R}_0$ and ${\bf R}_0$ covers ${\bf Q}_0$. Let ${\bf P}'={\bf P}_0(0,\cdots,0)$ and ${\bf Q}'={\bf Q}_0(0,\cdots,0)$. From Proposition \ref{dProp-mutation4} (1), we see that ${\bf P}'$ covers ${\bf Q}'$ and related by $\alpha_1$. Since ${\bf P}$ covers ${\bf Q}$, by Lemma \ref{Lem-cover}, we have $\Delta_j(q)\neq \Delta_1(q)$. Therefore, $\Delta_j(q)=\Delta_2(q)=(\alpha,\alpha_3,\alpha_1)$ with $m(\Delta_2(q);\alpha_1)=1$ in case $m_{\eta_\alpha+1}({\bf P};\alpha)=-1$ (i.e., $(\eta_\alpha,\eta_\alpha+1)$ is an ${\bf m}({\bf P};\alpha)$-pair). As all $\alpha$-twist-able edges in $P_0$ and $Q_0$ are labeled $\alpha_1,\alpha_4$, we have ${\bf P}_0={\bf P}'(1,\cdots,1)$ and ${\bf Q}_0={\bf Q}'(1,\cdots,1)$. By Proposition \ref{dProp-mutation4} (1) (d), we have
$\Omega({\bf P}_0, {\bf R}_0)+ \Omega({\bf R}_0, {\bf Q}_0)=\Omega({\bf P}', {\bf Q}')$.

Case 2.1.2. We then assume that $G_l\neq G_c$.

Case 2.1.2.1. $\tau_{i_c}\neq \alpha$ and $m(\Delta_j(q);\alpha)=-1$ with $m(\Delta_j(q);\alpha_3)=-1$. Then do twist for $P$ on the tiles $G_a$ with diagonal labeled $\alpha$ such that $P\cap Edge(G_a)$ are labeled $\alpha_1,\alpha_3$, we obtain a perfect matching $P_0$. Do twist for $\mu_{G_l}(P)$ on these tiles $G_a\neq G_{l-1}$, we obtain a perfect matching $Q_0$. Thus, $Q_0<\mu_{G_{l-1}} P_0<P_0$ and all $\alpha$-twist-able edges in $P_0$ and $Q_0$ are labeled $\alpha_1,\alpha_4$. Let ${\bf P}_0=(P_0,\Delta_j(q))$, ${\bf R}_0=(\mu_{G_{l-1}}P_0,\Delta_j(q))$, ${\bf Q}_0=(Q_0,\Delta_j(q))$, ${\bf P}'={\bf P}_0(0,\cdots,0)$ and ${\bf Q}'={\bf Q}_0(0,\cdots,0)$. This case can be proved similarly to the case $G_l=G_c$.

Case 2.1.2.2. $\tau_{i_c}\neq \alpha$ and $m(\Delta_j(q);\alpha)\geq 0$ or $m(\Delta_j(q);\alpha)=-1$ with $m(\Delta_j(q);\alpha_4)=-1$. Then do twist for $P$ on the tiles $G_a$ with diagonal labeled $\alpha$ such that $P\cap Edge(G_a)$ are labeled $\alpha_1,\alpha_4$, we obtain a perfect matching $P_0$. Do twist for $\mu_{G_l}(P)$ on these tiles $G_a$, we obtain a perfect matching $Q_0$. Thus, $Q_0<P_0$ and all $\alpha$-twist-able edges in $P_0$ and $Q_0$ are labeled $\alpha_1,\alpha_3$. Let ${\bf P}_0=(P_0,\Delta_j(q))$, ${\bf Q}_0=(Q_0,\Delta_j(q))$, ${\bf P}'={\bf P}_0(1,\cdots,1)$ and ${\bf Q}'={\bf Q}_0(1,\cdots,1)$. Assume that ${\bf Q}'=(Q',\Delta')$. By Proposition \ref{dProp-mutation4}, we see that ${\bf P}'=(\mu_{G'(k+1)}\mu_{G'}Q',\Delta')$. Let ${\bf R}'=(\mu_{G'}Q',\Delta')$. This case can be proved similarly to the situation that ${\bf P}=P$.

Case 2.1.2.3. $\tau_{i_c}=\alpha$ and $\Delta_j(q)\neq \Delta_1(q)$.

Case 2.1.2.3.1. If $m(\Delta_j(q);\alpha)\geq 0$ or $m(\Delta_j(q);\alpha)=-1$ with $m(\Delta_j(q);\alpha_4)=-1$, then it can be proved similarly to the case that $\tau_{i_c}\neq \alpha$ and $m(\Delta_j(q);\alpha)\geq 0$ or $m(\Delta_j(q);\alpha)=-1$ with $m(\Delta_j(q);\alpha_4)=-1$.

Case 2.1.2.3.2. If $m(\Delta_j(q);\alpha)=-1$ with $m(\Delta_j(q);\alpha_3)=-1$, then it can be proved similarly to the case that $\tau_{i_c}\neq \alpha$ and $m(\Delta_j(q);\alpha)=-1$ with $m(\Delta_j(q);\alpha_3)=-1$.

Case 2.1.2.4. $\tau_{i_c}=\alpha$ and $\Delta_j(q)=\Delta_1(q)$. Then $(\eta_\alpha,\eta_\alpha+1)$ is an ${\bf m}({\bf P};\alpha)$ and ${\bf m}({\bf Q};\alpha)$-pair. Then do twist for $P$ on the tiles $G_a\neq G_c$ with diagonal labeled $\alpha$ such that $P\cap Edge(G_a)$ are labeled $\alpha_1,\alpha_4$, we obtain a perfect matching $P_0$. Do twist for $\mu_{G_l}(P)$ on these tiles $G_a\neq G_c$, we obtain a perfect matching $Q_0$. Thus, $Q_0<P_0$. Let ${\bf P}_0=(P_0,\Delta_j(q))$, ${\bf Q}_0=(Q_0,\Delta_j(q))$, ${\bf P}'={\bf P}_0(1,\cdots,1)$ and ${\bf Q}'={\bf Q}_0(1,\cdots,1)$. Assume that ${\bf Q}'=(Q',\Delta')$. By Proposition \ref{dProp-mutation4}, we see that ${\bf P}'=(\mu_{G'(k+1)}\mu_{G'}Q',\Delta')$. Let ${\bf R}'=(\mu_{G'}Q',\Delta')$. This case can be proved similarly to the situation that ${\bf P}=P$.

Case 2.2. $(m_k(P),m_{k+1}(P))=(0,0)$. The case is similar to the case that $(m_k(P),m_{k+1}(P))=(-1,1)$ as $(m_k(Q),m_{k+1}(Q))=(1,-1)$.

Case 2.3. $(m_k(P),m_{k+1}(P))=(0,1)$.

Case 2.3.1. $G_l=G_c$. The case can be proved similarly to the Case 2.1.1.

Case 2.3.2. $G_l\neq G_c$ and $\tau_{i_c}\neq \alpha$. The case can be proved similarly to the Case 2.1.2.2.

Case 2.3.3. $G_l\neq G_c$ and $\tau_{i_c}=\alpha$.

Case 2.3.3.1. $\Delta_j(q)\neq \Delta_1(q)$. The case can be proved similarly to the Case 2.1.2.3.1

Case 2.3.3.2. $\Delta_j(q)=\Delta_1(q)$. The case can be proved similarly to the Case 2.1.2.4.

The situation that ${\bf P}=(\Delta_i(p),P,\Delta_j(q))$ can be proved similarly to the situation that ${\bf P}=(P,\Delta_j(q))$.
\end{proof}

\newpage

\section{Two equivalent relations in $\mathcal L$.}\label{sec:TEQ}

{\bf Assumption}: We always assume that $T$ and $T'$ contain no arc tagged notched at $q$ in case $\mathcal L=\mathcal L(T^o,\widetilde\beta^{(q)})$, $T$ and $T'$ contain no arc tagged notched at $p$ or $q$ in case $\mathcal L=\mathcal L(T^o,\widetilde\beta^{(p,q)})$ if there is no other state.

In this section, we introduce two equivalent relations in $\mathcal L$ and show any two elements of $\mathcal L$ are equivalent.

For any ${\bf P}\in \mathcal L$, we denote $\pi'\pi({\bf P})_{\pm}$ the maximum/minimum element in $\pi'\pi({\bf P})$ and $\pi({\bf P})_{\pm}$ the maximum/minimum element in $\pi({\bf P})$.

\begin{lemma}\label{lem:pm}
For any ${\bf P}\in \mathcal L$, we have
$$w(\pi'\pi({\bf P}_{+}))=w(\pi'\pi({\bf P}_{-})), \hspace{5mm} w(\pi({\bf P}_{+}))=w(\pi({\bf P}_{-})).$$
\end{lemma}

\begin{proof}
It follows by Lemma \ref{lem:ome1}.
\end{proof}

\begin{definition}
For any ${\bf P},{\bf Q}\in \mathcal L$,
\begin{enumerate}[$(1)$]
\item we say that ${\bf P}$ and ${\bf Q}$ are \emph{$y$-equivalent}, denoted as ${\bf P}\sim {\bf Q}$, if in $\mathbb P$,
$$\frac{\bigoplus_{{\bf R}\in \pi'\pi({\bf P})}y^{T}({\bf R})}
{\bigoplus_{{\bf R}'\in \pi({\bf P})}y^{T'}({\bf R}')}
=\frac{\bigoplus_{{\bf R}\in \pi'\pi({\bf Q})}y^{T}({\bf R})}
{\bigoplus_{{\bf R}'\in \pi({\bf Q})}y^{T'}({\bf R}')}.$$
\item we say that ${\bf P}$ and ${\bf Q}$ are \emph{$w$-equivalent}, denoted as ${\bf P}\approx {\bf Q}$, if
$$w(\pi'\pi({\bf P})_{\pm})-w(\pi({\bf P})_{\pm})=w(\pi'\pi({\bf Q})_{\pm})-w(\pi({\bf Q})_{\pm}).$$
\end{enumerate}
\end{definition}

It is clear that $\sim$, $\approx$ are equivalence relations on $\mathcal L(T^o,\widetilde\beta^{(q)})$.

\begin{lemma}\label{basic4}
For any ${\bf P}\in \mathcal L$ such that $r=\sum m_i({\bf P})\geq 0$, for ${\bf P}'={\bf P}(\vec c)\in \pi({\bf P})$ corresponding to some $\vec c\in \{0,1\}^{r}$ as in Theorem \ref{thm:pi}.
%\huang{Assume that there is no arc tagged notched at $q$.}
Then in $\mathbb P$ we have

\begin{enumerate}[$(1)$]
\item $\bigoplus_{{\bf Q}'\in \pi({\bf P})}y^{T'}({\bf Q}')=y^{T'}({\bf P}')\cdot(1\oplus y^{T'^o}_{\alpha'})^{r}$ if $\vec c=(0,0,\cdots,0)$.
\item $\bigoplus_{{\bf Q}'\in \pi({\bf P})}y^{T'}({\bf Q}')=y^{T'}({\bf P}')\cdot(1\oplus y^{T^o}_{\alpha})^{r}$ if $\vec c=(1,1,\cdots,1)$.
\end{enumerate}
\end{lemma}

Summarize Lemmas \ref{Lem-yf}, \ref{lem-ycover}, \ref{lem-ycover1}, \ref{lem-dcover1},  \ref{lem-dcover2} and  \ref{lem-dcover3}, we obtain the following result.

\begin{lemma}\label{lem:ycover}
%Suppose that $T$ and $T'$ contain no arc tagged notched at $p$ or $q$.
Assume that ${\bf P}$ covers ${\bf Q}$ and related by $\tau$. Then we have
$$\frac{y^T({\bf Q})}{y^T({\bf P})}=y^{T^o}_{\tau}.$$
\end{lemma}

The main result in this section is the following.

\begin{theorem}\label{thm-sim3}
For any ${\bf P}, {\bf Q}\in \mathcal L$, we have ${\bf P}\sim {\bf Q}$ and ${\bf P}\approx {\bf Q}$.
\end{theorem}

In the rest of this section, we give the proof of Theorem \ref{thm-sim3}

\begin{lemma}\label{lem-na}
Assume that ${\bf P}$ covers ${\bf Q}$ and related by $\tau$. If $\tau\neq \alpha_1,\alpha_2,\alpha_3,\alpha_4$, then we have ${\bf P}\sim {\bf Q}$ and ${\bf P}\approx {\bf Q}$.
\end{lemma}

\begin{proof}
As $\tau\neq \alpha_1,\alpha_2,\alpha_3,\alpha_4$, we have $\sum m_i({\bf P})=\sum m_i ({\bf Q})$. Without loss of generality, we assume that $\sum m_i({\bf P})=\sum m_i ({\bf Q})\geq 0$. Then we have
$\pi'\pi({\bf P})=\{{\bf P}\}$ and $\pi'\pi({\bf Q})=\{{\bf Q}\}$ by Theorem \ref{thm:pi}. Denote ${\bf P}(0,\cdots,0)={\bf P}'$ and ${\bf Q}(0,\cdots,0)={\bf Q}'$.

By Lemma \ref{lem:ycover}, we have
\begin{equation}\label{eq-sum1}
\frac{\bigoplus_{{\bf R}\in \pi'\pi({\bf P})}y^{T}({\bf R})}{\bigoplus_{{\bf R}\in \pi'\pi({\bf Q})}y^{T}({\bf R})}
=\frac{y^{T}({\bf P})}{y^{T}({\bf Q})}=y^{T^o}_{\tau}.
\end{equation}

In case $\tau\neq \alpha$, by Theorem \ref{thm-compatible1} (1), ${\bf P}'$ covers ${\bf Q}'$ and related by $\tau$, and
\begin{equation}\label{eq-ome5}
\Omega({\bf P},{\bf Q})=\Omega({\bf P}',{\bf Q}').
\end{equation}

By Lemmas \ref{basic4} and \ref{lem:ycover}, we have
\begin{equation}\label{eq-sum2}
\frac{\bigoplus_{{\bf R}'\in \pi({\bf P})}y^{T'}({\bf R}')}{\bigoplus_{{\bf R}'\in \pi({\bf Q})}y^{T'}({\bf R}')}
=\frac{y^{T'}({\bf P}')\cdot(1\oplus y^{T'^o}_{\alpha'})^{\sum m_i({\bf P})}}{y^{T'}({\bf Q}')\cdot(1\oplus y^{T'^o}_{\alpha'})^{\sum m_i({\bf Q})}}=y^{T'^o}_{\tau}.
\end{equation}

Then ${\bf P}\sim {\bf Q}$ follows by (\ref{eq-sum1}), (\ref{eq-sum2}) and Lemma \ref{lem-yf1} (1).

By (\ref{eq-ome5}), $w({\bf P})-w({\bf Q})=w({\bf P}')-w({\bf Q}')=w({\bf P}(0,\cdots,0))-w({\bf Q}(0,\cdots,0))$. Thus ${\bf P}\approx {\bf Q}$.

In case $\tau=\alpha$, by Theorem \ref{thm-compatible1} (2), ${\bf P}'$ is covered ${\bf Q}'$ and related by $\alpha'$, and
\begin{equation}\label{eq-ome6}
\Omega({\bf P},{\bf Q})=-\Omega({\bf P}',{\bf Q}').
\end{equation}

By Lemmas \ref{basic4} and \ref{lem:ycover}, we have
\begin{equation}\label{eq-sum3}
\frac{\bigoplus_{{\bf R}'\in \pi({\bf P})}y^{T'}({\bf R}')}{\bigoplus_{{\bf R}'\in \pi({\bf Q})}y^{T'}({\bf R}')}
=\frac{y^{T'}({\bf P}')\cdot(1\oplus y^{T'^o}_{\alpha'})^{\sum m_i({\bf P})}}{y^{T'}({\bf Q}')\cdot(1\oplus y^{T'^o}_{\alpha'})^{\sum m_i({\bf Q})}}=y^{T'^o}_{\alpha'}.
\end{equation}

Then ${\bf P}\sim {\bf Q}$ follows by (\ref{eq-sum1}), (\ref{eq-sum3}) and Lemma \ref{lem-yf1} (2).

By (\ref{eq-ome6}), $w({\bf P})-w({\bf Q})=w({\bf P}')-w({\bf Q}')=w({\bf P}(0,\cdots,0))-w({\bf Q}(0,\cdots,0))$. Thus ${\bf P}\approx {\bf Q}$.

The proof is complete.
\end{proof}

\begin{lemma}\label{lem-a1}
Assume that ${\bf P}$ covers ${\bf Q}$ and related by $\alpha_k$ for some $k\in \{1,2,3,4\}$ with $\alpha_k=\alpha_{k-1}$ or $\alpha_{k+1}$. Then we have ${\bf P}\sim {\bf Q}$ and ${\bf P}\approx {\bf Q}$.
\end{lemma}

\begin{proof}
We may assume that $k=1$ and $\alpha_1=\alpha_2$.

We shall only consider the statement (i) of Theorem \ref{thm-compatible3} holds, as it can be proved similarly if the statement (ii) holds.

Let ${\bf P}_0$, ${\bf Q}_0$, ${\bf P}', {\bf R}'$ and ${\bf Q}'$ be the elements given in Theorem \ref{thm-compatible3}. Then ${\bf P}_0$ and ${\bf P}$ are related by a sequence of twists at $\alpha$. By Lemma \ref{lem-na} we have ${\bf P}_0\sim {\bf P},{\bf P}_0\approx {\bf P}$. Similarly, ${\bf Q}_0\sim {\bf Q},{\bf Q}_0\approx {\bf Q}$. Thus it suffices to prove ${\bf P}_0\sim {\bf Q}_0$ and ${\bf P}_0\approx {\bf Q}_0$.

From Theorem \ref{thm-compatible3}, we have $\Omega({\bf P}_0,{\bf Q}_0)=\Omega({\bf P}',{\bf R}')+\Omega({\bf R}',{\bf Q}')$. Thus, $w({\bf P}_0)-w({\bf Q}_0)=w({\bf P}')-w({\bf Q}')$. Therefore, ${\bf P}_0\approx {\bf Q}_0$ follows by Theorem \ref{thm-compatible3} (i) (4), (5).

We now prove ${\bf P}_0\sim {\bf Q}_0$.

As ${\bf P}_0$ covers ${\bf Q}_0$ and related by $\alpha_1$, we see that $\sum m_i({\bf P}_0)-\sum m_i({\bf Q}_0)=b^{T^o}_{\alpha \alpha_1}=0$. We may assume that $\sum m_i({\bf P}_0)=\sum m_i({\bf Q}_0)\geq 0$. Then
$\pi'\pi({\bf P}_0)=\{{\bf P}_0\}$ and $\pi'\pi({\bf Q}_0)=\{{\bf Q}_0\}$ by Theorem \ref{thm:pi}.

By Lemma \ref{lem:ycover}, we have
\begin{equation}\label{eq-sum8}
\frac{\bigoplus_{{\bf R}\in \pi'\pi({\bf P}_0)}y^{T}({\bf R})}{\bigoplus_{{\bf R}\in \pi'\pi({\bf Q}_0)}y^{T}({\bf R})}
=\frac{y^{T}({\bf P}_0)}{y^{T}({\bf Q}_0)}=y^{T^o}_{\alpha_1}.
\end{equation}

By Theorem \ref{thm-compatible3} and Lemmas \ref{basic4}, \ref{lem:ycover}, we have
\begin{equation}\label{eq-sum9}
\frac{\bigoplus_{{\bf S}'\in \pi({\bf P}_0)}y^{T'}({\bf S}')}{\bigoplus_{{\bf S}'\in \pi({\bf Q}_0)}y^{T'}({\bf S}')}
=\frac{y^{T'}({\bf P}')}{y^{T'}({\bf Q}')}=\frac{y^{T'}({\bf P}')}{y^{T'}({\bf R}')}\frac{y^{T'}({\bf R}')}{y^{T'}({\bf Q}')}=y^{T'^o}_{\alpha_1}y^{T'^o}_{\alpha'}.
\end{equation}

As $\alpha_1=\alpha_2$, $\alpha_1$ is the radius of the self-folded triangle $(\alpha_1,\alpha_1,\alpha')$ in $T'^o$. By Lemma \ref{lem-yf1} (3) (b), $y^{T^o}_{\alpha_1}=y^{T'^o}_{\alpha_1}y^{T'^o}_{\alpha'}$. Thus, ${\bf P}\sim {\bf Q}$ follows by (\ref{eq-sum8}), (\ref{eq-sum9}).

The proof is complete.
\end{proof}

\begin{lemma}\label{lem-a}
Let ${\bf P}=P,(P,\Delta_j(q))$ or $(\Delta_i(p),P,\Delta_j(q))$. Assume that ${\bf P}$ covers some ${\bf Q}$ and related by $\alpha_k$ for some $k\in \{1,2,3,4\}$ with $\alpha_k\neq \alpha_{k-1},\alpha_{k+1}$. If all the $\alpha$-mutable edges of $P$ are labeled $\alpha_{k-1}$ or $\alpha_{k+1}$, then we have ${\bf P}\sim {\bf Q}$ and ${\bf P}\approx {\bf Q}$.
\end{lemma}

\begin{proof}
We may assume that $k=1$. Then $\sum m_i({\bf P})-\sum m_i({\bf Q})=b^{T^o}_{\alpha \alpha_1}>0$. We may further assume that $\sum m_i({\bf P})\geq 0$. Then we have
$\pi'\pi({\bf P})=\{{\bf P}\}$ by Theorem \ref{thm:pi}. Denote ${\bf P}(0,\cdots,0)={\bf P}'$ and ${\bf Q}(0,\cdots,0)={\bf Q}'$.

We divide the proof into two parts according to the conditions in Theorem \ref{thm-compatible2} (1).

1) The cases except for the exceptional cases in Theorem \ref{thm-compatible2} (1).

As $\alpha_1\neq \alpha_2,\alpha_4$, by Theorem \ref{thm-compatible2} (1), we see that $$w({\bf P}'(1,\cdots,1))-w({\bf Q}'(1,\cdots,1))=w({\bf P}(0,\cdots,0))-w({\bf Q}(0,\cdots,0)).$$ Thus ${\bf P}\approx {\bf Q}$.

We now show that ${\bf P}\sim {\bf Q}$.

We first consider the case that $\sum m_i({\bf P})\geq \sum m_i({\bf Q})\geq 0$. Then by Theorem \ref{thm:pi} $\pi'\pi({\bf Q})=\{{\bf Q}\}$.

By Lemma \ref{lem:ycover}, we have
\begin{equation}\label{eq-sum4}
\frac{\bigoplus_{{\bf R}\in \pi'\pi({\bf P})}y^{T}({\bf R})}{\bigoplus_{{\bf R}\in \pi'\pi({\bf Q})}y^{T}({\bf R})}
=\frac{y^{T}({\bf P})}{y^{T}({\bf Q})}=y^{T^o}_{\alpha_1}.
\end{equation}

By Lemmas \ref{basic4} and \ref{lem:ycover}, we have
\begin{equation}\label{eq-sum5}
\frac{\bigoplus_{{\bf R}'\in \pi({\bf P})}y^{T'}({\bf R}')}{\bigoplus_{{\bf R}'\in \pi({\bf Q})}y^{T'}({\bf R}')}
=\frac{y^{T'}({\bf P}')\cdot(1\oplus y^{T'^o}_{\alpha'})^{\sum m_i({\bf P})}}{y^{T'}({\bf Q}')\cdot(1\oplus y^{T'^o}_{\alpha'})^{\sum m_i({\bf Q})}}=y^{T'^o}_{\alpha_1}(1\oplus y^{T'^o}_{\alpha'})^{b_{\alpha\alpha_1}^{T^o}}.
\end{equation}

By (\ref{eq-sum4}), (\ref{eq-sum5}) and Lemma \ref{lem-yf1} (3) (a), we obtain ${\bf P}\sim {\bf Q}$.

We then consider the case that $\sum m_i({\bf P})\geq 0> \sum m_i({\bf Q})$. By Theorem \ref{thm:pi}, we have $\pi({\bf Q})=\{{\bf Q}'\}$.

By Lemmas \ref{basic4} and \ref{lem:ycover}, we have
\begin{equation}\label{eq-sum6}
\frac{\bigoplus_{{\bf R}\in \pi'\pi({\bf P})}y^{T}({\bf R})}{\bigoplus_{{\bf R}\in \pi'\pi({\bf Q})}y^{T}({\bf R})}
=\frac{y^{T}({\bf P})}{y^{T}({\bf Q})\cdot(1\oplus y^{T'^o}_{\alpha'})^{-\sum m_i({\bf Q})}}=y^{T^o}_{\alpha_1}\cdot (1\oplus y^{T'^o}_{\alpha'})^{\sum m_i({\bf Q})}.
\end{equation}

By Lemmas \ref{basic4} and \ref{lem:ycover}, we have
\begin{equation}\label{eq-sum7}
\frac{\bigoplus_{{\bf R}'\in \pi({\bf P})}y^{T'}({\bf R}')}{\bigoplus_{{\bf R}'\in \pi({\bf Q})}y^{T'}({\bf R}')}
=\frac{y^{T'}({\bf P}')\cdot(1\oplus y^{T'^o}_{\alpha'})^{\sum m_i({\bf P})}}{y^{T'}({\bf Q}')}=y^{T'^o}_{\alpha_1}(1\oplus y^{T'^o}_{\alpha'})^{\sum m_i({\bf P})}.
\end{equation}

By (\ref{eq-sum6}), (\ref{eq-sum7}) and Lemma \ref{lem-yf1} (3) (a), we obtain ${\bf P}\sim {\bf Q}$.

2) The exceptional cases in Theorem \ref{thm-compatible2} (1).

Case 1. ${\bf P}=(P,\Delta_j(q)), {\bf Q}=(\mu_{G_l}P,\Delta_j(q))$ such that $\tau_j(q)=\alpha$ and $(P,\Delta_{j+1}(q))$ covers ${\bf P}$. By Lemma \ref{lem:diamond1}, we have $(P,\Delta_{j+1}(q))$ covers $(\mu_{G_l}P,\Delta_{j+1}(q))$ and $(\mu_{G_l}P,\Delta_{j+1}(q))$ covers ${\bf Q}$. By the previous discussion, we have $(P,\Delta_{j+1}(q))\star \mu_{G_l}P,\Delta_{j+1}(q))$ for $\star\in \{\sim,\approx\}$. Since $\tau_j(q)=\alpha$, we have $(P,\Delta_{j+1}(q))\star {\bf P}$ and $(\mu_{G_l}P,\Delta_{j+1}(q))\star {\bf Q}$ for $\star\in \{\sim,\approx\}$ by Lemma \ref{lem-na}. Therefore, we have ${\bf P}\sim {\bf Q}$ and ${\bf P}\approx {\bf Q}$.

The remaining exceptional cases in Theorem \ref{thm-compatible2} (1) can be proved similarly to Case 1 by using Lemmas \ref{lem:diamond4}, \ref{lem:diamond5}, \ref{lem:diamond2} and \ref{lem:diamond3}.

The proof is complete.
\end{proof}

%The following Lemmas can be proved by using Theorems \ref{thm-compatible21} and \ref{thm-compatible22}.

\begin{lemma}\label{lem:covera}
Assume that ${\bf P}=(P,\Delta_{j+1}(q))$ (resp. $(\Delta_i(p),P,\Delta_{j+1}(q))$) covers ${\bf Q}=(P,\Delta_{j}(q))$ (resp. $(\Delta_i(p),P,\Delta_{j}(q))$) with $\tau_j(q)=\alpha_k$ for some $k\in \{1,2,3,4\}$. Then we have ${\bf P}\sim {\bf Q}$ and ${\bf P}\approx {\bf Q}$.
\end{lemma}

\begin{proof}
We may assume that $k=1$. We can do twists on the tiles with diagonals labeled $\alpha$ for $P$ to obtain a perfect matching $R$ such that all $\alpha$-mutable edges are labeled $\alpha_2$ or $\alpha_4$. We obtain a sequence of perfect matchings $P=P_0<P_1<\cdots<P_n=R$ such that $P_{\ell}$ covers $P_{\ell-1}$ for all $\ell=1,\cdots, n$.

We first consider the case that ${\bf P}=(P,\Delta_{j+1}(q))$.  By Lemma \ref{lem:diamond1}, we have $(P_\ell,\Delta_{j+1}(q))$ covers $(P_{\ell-1},\Delta_{j+1}(q))$, $(P_\ell,\Delta_{j}(q))$ covers $(P_{\ell-1},\Delta_{j}(q))$ for all $\ell$ with $1\leq \ell\leq n$ and $(R,\Delta_{j+1}(q))$ covers $(R,\Delta_{j}(q))$. Thus by Lemmas \ref{lem-na}, \ref{lem-a} and \ref{lem-a1}, for $\star\in \{\sim,\approx\}$, we obtain $${\bf P}\star (R,\Delta_{j+1}(q))\star (R,\Delta_j(q))\star {\bf Q}.$$

The case that ${\bf P}=(\Delta_i(p),P,\Delta_{j+1}(q))$ can be proved similarly by Lemmas \ref{lem:diamond2}, \ref{lem-na}, \ref{lem-a} and \ref{lem-a1}.
\end{proof}

The following result can be proved similarly by Lemmas \ref{lem:diamond3}, \ref{lem-na}, \ref{lem-a1} and \ref{lem-a}.

\begin{lemma}\label{lem:covera4}
Assume that ${\bf P}=(\Delta_{i+1}(p),P,\Delta_{j}(q))$ covers ${\bf Q}=(\Delta_i(p),P,\Delta_{j}(q))$ with $\tau_j(q)=\alpha_k$ for some $k\in \{1,2,3,4\}$. Then we have ${\bf P}\sim {\bf Q}$ and ${\bf P}\approx {\bf Q}$.
\end{lemma}

\begin{lemma}\label{lem:covera0}
Suppose that $\mathcal L=\mathcal L(T^o,\widetilde\beta)$. Assume that $P\in \mathcal L$ can twist on $G_l$ such that $P$ covers $Q=\mu_{G_l}P$. If $\tau_{i_l}=\alpha_k$ for some $k\in \{1,2,3,4\}$, then we have $P\sim Q$ and $P\approx Q$.
\end{lemma}

\begin{proof}
We may assume that $k=1$. Then we have $\alpha_1=\tau_{i_l}$. By Lemma \ref{lem-a1}, it suffices to consider the case that $\alpha_1\neq \alpha_2,\alpha_4$. We can do twists on the tiles with diagonals labeled $\alpha$ for $P$ to obtain a perfect matching $R$ such that all $\alpha$-mutable edges are labeled $\alpha_2$ or $\alpha_4$. We obtain a sequence of perfect matchings $P=P_0<P_1<\cdots<P_n=R$ such that $P_{\ell}$ covers $P_{\ell-1}$ for all $\ell=1,\cdots, n$.

As $P>\mu_{G_l}P$, we see that $P_\ell$ can twist on $G_l$ and $P_\ell>\mu_{G_l}P$ for any $\ell\in \{1,2,\cdots, n\}$. Thus $R$ covers $\mu_{G_l}R$ in $\mathcal L$.
For any $\star\in\{\sim,\approx\}$, we have $P\star R$ and $\mu_{G_l}P \star \mu_{G_l}R$ by Lemma \ref{lem-na} and $R \star \mu_{G_l}R$ by Lemma \ref{lem-a}.

Therefore, $P\sim Q=\mu_{G_l}P$ and $P\approx Q=\mu_{G_l}P$.
\end{proof}

The following proposition follows immediately by Lemmas \ref{lem-na} and \ref{lem:covera0}.

\begin{proposition}\label{prop-ga}
For any ${\bf P},{\bf Q}\in \mathcal L(T^o,\widetilde\beta)$, we have ${\bf P}\sim {\bf Q}$ and ${\bf P}\approx {\bf Q}$.
\end{proposition}

We now consider the case that $\mathcal L=\mathcal L(T^o,\widetilde\beta^{(q)})$.

The following is a corollary of Lemmas \ref{lem-na} and \ref{lem:covera}.

\begin{corollary}\label{cor:puncturecover}
Suppose that $\mathcal L=\mathcal L(T^o,\widetilde\beta^{(q)})$. For any ${\bf P}=(P,\Delta_{i}(q))$ and ${\bf Q}=(P,\Delta_{j}(q))$, we have ${\bf P}\sim {\bf Q}$ and ${\bf P}\approx {\bf Q}$.
\end{corollary}

\begin{lemma}\label{lem:covera1}
Suppose that $\mathcal L=\mathcal L(T^o,\widetilde\beta^{(q)})$.
If ${\bf P}=(P,\Delta_j(q))$ covers ${\bf Q}=(\mu_{G_l}P,\Delta_j(q))$ with $\tau_{i_l}=\alpha_k$ for some $k\in \{1,2,3,4\}$, then ${\bf P}\sim {\bf Q}$ and ${\bf P}\approx {\bf Q}$.
\end{lemma}

\begin{proof}
We may assume that $k=1$. By Lemma \ref{lem-a1}, it suffices to consider the case that $\alpha_1\neq \alpha_2,\alpha_4$. We can do twists on the tiles with diagonals labeled $\alpha$ for $P$ to obtain a perfect matching $R$ such that all $\alpha$-mutable edges are labeled $\alpha_2$ or $\alpha_4$. We obtain a sequence of perfect matchings $P=P_0<P_1<\cdots<P_n=R$ such that $P_{\ell}$ covers $P_{\ell-1}$ for all $\ell=1,\cdots, n$. From the proof of Lemma \ref{lem:covera0}, $R$ can twist on $G_l$.

%Since $P>\mu_{G_l}P$, we have $P$ can not twist on $G_{l\pm 1}$ if the $\tau_{i\pm 1}=\alpha$.

We first prove that
\begin{equation}\label{eq:*eq1}
(P,\Delta_j(q))\star (R,\Delta_j(q)), \;\;\;\; (\mu_{G_l}P,\Delta_j(q))\star (\mu_{G_l}R,\Delta_j(q))
\end{equation}
 for $\star=\sim, \approx$. It suffices to prove that, for all $\ell=1,\cdots, n$,
\begin{equation*}
(P_{\ell},\Delta_{j}(q))\star (P_{\ell-1},\Delta_{j}(q)),\;\;\;\;\;(\mu_{G_l}P_{\ell},\Delta_{j}(q))\star (\mu_{G_l} P_{\ell-1},\Delta_{j}(q)).
\end{equation*}

As $T$ contains no arc tagged notched at $q$, we have $|\Delta(T^o,q)|\geq 2$, it implies $\Delta_2(q)\in \Delta(T^o,q)$. For any $\ell\in \{1,\cdots,n\}$, by Lemma \ref{Lem-cover}, we have $(P_{\ell},\Delta_{2}(q))$ covers $(P_{\ell-1},\Delta_{2}(q))$. Therefore, by Lemma \ref{lem-na} and Corollary \ref{cor:puncturecover}, we have
$$(P_{\ell},\Delta_{j}(q))\star (P_{\ell},\Delta_{2}(q))\star (P_{\ell-1},\Delta_{2}(q))\star (P_{\ell-1},\Delta_{j}(q)).$$

Similarly, $(\mu_{G_l}P_{\ell},\Delta_{j}(q))\star (\mu_{G_l}P_{\ell-1},\Delta_{j}(q))$ for all $\ell=1,\cdots, n$.

We then prove that $(R,\Delta_j(q))\star (\mu_{G_l}R,\Delta_j(q))$. As $T$ and $T'$ contain no arc tagged notched at $q$, we can choose $\Delta_\ell(q)$ such that $m(\Delta_{\ell}(q);\alpha)\geq 0$. Then  by Lemma \ref{lem-a} and Corollary \ref{cor:puncturecover}, we have
\begin{equation}\label{eq:*eq2}
(R,\Delta_{j}(q))\star (R,\Delta_{\ell}(q))\star (\mu_{G_l}R,\Delta_{\ell}(q))\star (\mu_{G_l}R,\Delta_{j}(q)).
\end{equation}

Therefore, we obtain $(P,\Delta_{j}(q))\star (\mu_{G_l}P,\Delta_{j}(q))$ by (\ref{eq:*eq1}) (\ref{eq:*eq2}).
%Since otherwise we have $\tau_{i_c}=\alpha_1=\tau_{[\ell']}(q)$. Then by (1) and Lemmas \ref{lem-na21}, \ref{lem-na3}, we have
%$$(R,\Delta_{j}(q))\star (R,\Delta_{\ell'}(q))\star (\mu_{G_l}R,\Delta_{\ell'}(q))\star (\mu_{G_l}R,\Delta_{j}(q)).$$
\end{proof}

The following proposition follows immediately by Lemmas \ref{lem-na}, \ref{lem:covera} and \ref{lem:covera1}.

\begin{proposition}\label{prop-ga1}
For any ${\bf P},{\bf Q}\in \mathcal L(T^o,\widetilde\beta^{(q)})$, we have ${\bf P}\sim {\bf Q}$ and ${\bf P}\approx {\bf Q}$.
\end{proposition}

We now turn to the case that $\mathcal L=\mathcal L(T^o,\widetilde\beta^{(p,q)})$.

The following is a corollary of Lemmas \ref{lem-na}, \ref{lem:covera} and \ref{lem:covera4}.

\begin{corollary}\label{cor:puncturecover1}
Suppose that $\mathcal L=\mathcal L(T^o,\widetilde\beta^{(p,q)})$.
%\begin{enumerate}[$(1)$]
For any ${\bf P}=(\Delta_{i}(p),P,\Delta_{j}(q))$ and ${\bf Q}=(\Delta_{i'}(p),P,\Delta_{j'}(q))$, we have ${\bf P}\sim {\bf Q}$ and ${\bf P}\approx {\bf Q}$.
%\item For any ${\bf P}=(\Delta_{i}(p),P,\Delta_{j}(q))$ and ${\bf Q}=(\Delta_{i'}(p),P,\Delta_{j}(q))$, we have ${\bf P}\sim {\bf Q}$ and ${\bf P}\approx {\bf Q}$.
%\end{enumerate}
\end{corollary}

\begin{lemma}\label{lem:covera3}
Suppose that $\mathcal L=\mathcal L(T^o,\widetilde\beta^{(p,q)})$.
If ${\bf P}=(\Delta_{i}(p),P,\Delta_{j}(q))$ covers ${\bf Q}=(\Delta_{i}(p),\mu_{G_l}P,\Delta_{j}(q))$ with $\tau_{i_l}=\alpha_k$ for some $k\in \{1,2,3,4\}$, then ${\bf P}\sim {\bf Q}$ and ${\bf P}\approx {\bf Q}$.
\end{lemma}

\begin{proof}
We may assume that $k=1$. By Lemma \ref{lem-a1}, it suffices to consider the case that $\alpha_1\neq \alpha_2,\alpha_4$. We can do twists on the tiles with diagonals labeled $\alpha$ from $P$ to obtain a perfect matching $R$ such that all $\alpha$-mutable edges are labeled $\alpha_2$ or $\alpha_4$. We obtain a sequence of perfect matchings $P=P_0<P_1<\cdots<P_n=R$ such that $P_{\ell}$ covers $P_{\ell-1}$ for all $\ell=1,\cdots, n$.

We first prove that
\begin{equation}\label{eq:*eq3}
(\Delta_{i}(p),P,\Delta_j(q))\star (\Delta_{i}(p),R,\Delta_j(q)), \;\;\;\; (\Delta_{i}(p),\mu_{G_l}P,\Delta_j(q))\star (\Delta_{i}(p),\mu_{G_l}R,\Delta_j(q))
\end{equation}
 for $\star=\sim, \approx$. It suffices to prove that, for all $\ell=1,\cdots, n$,
\begin{equation*}
(\Delta_{i}(p),P_{\ell},\Delta_{j}(q))\star (\Delta_{i}(p),P_{\ell-1},\Delta_{j}(q)),\;\;\;\;\;(\Delta_{i}(p),\mu_{G_l}P_{\ell},\Delta_{j}(q))\star (\Delta_{i}(p),\mu_{G_l} P_{\ell-1},\Delta_{j}(q)).
\end{equation*}

As $T$ contains no arc tagged notched at $p$ or $q$, we have $|\Delta(T^o,p)|, |\Delta(T^o,q)|\geq 2$, it implies $\Delta_2(p)\in \Delta(T^o,p),\Delta_2(q)\in \Delta(T^o,q)$. For any $\ell\in \{1,\cdots,n\}$, by Lemma \ref{Lem-cover1}, we have $(\Delta_{2}(p),P_{\ell},\Delta_{2}(q))$ covers $(\Delta_{2}(p),P_{\ell-1},\Delta_{2}(q))$. Therefore, by Lemma \ref{lem-na} and Corollary \ref{cor:puncturecover1}, we have
$$(\Delta_{i}(p),P_{\ell},\Delta_{j}(q))\star (\Delta_{2}(p),P_{\ell},\Delta_{2}(q))\star (\Delta_{2}(p),P_{\ell-1},\Delta_{2}(q))\star (\Delta_{i}(p),P_{\ell-1},\Delta_{j}(q)).$$

Similarly, $(\Delta_{i}(p),\mu_{G_l}P_{\ell},\Delta_{j}(q))\star (\Delta_{i}(p),\mu_{G_l}P_{\ell-1},\Delta_{j}(q))$ for all $\ell=1,\cdots, n$.

We then prove that $(\Delta_{i}(p),R,\Delta_j(q))\star (\Delta_{i}(p),\mu_{G_l}R,\Delta_j(q))$. As $T$ and $T'$ contain no arc tagged notched at $q$, we can choose $\Delta_\ell(p),\Delta_{\ell'}(q)$ such that $m(\Delta_{\ell}(p);\alpha), m(\Delta_{\ell'}(q);\alpha)\geq 0$. Then by Lemma \ref{lem-a} and Corollary \ref{cor:puncturecover1}, we have
\begin{equation}\label{eq:*eq4}
(\Delta_{i}(p),R,\Delta_{j}(q))\star (\Delta_{\ell}(p),R,\Delta_{\ell'}(q))\star (\Delta_{\ell}(p),\mu_{G_l}R,\Delta_{\ell'}(q))\star (\Delta_{i}(p),\mu_{G_l}R,\Delta_{j}(q)).
\end{equation}

Therefore, we obtain $(\Delta_{i}(p),P,\Delta_{j}(q))\star (\Delta_{i}(p),\mu_{G_l}P,\Delta_{j}(q))$ by (\ref{eq:*eq3}) (\ref{eq:*eq4}).
%Since otherwise we have $\tau_{i_c}=\alpha_1=\tau_{[\ell']}(q)$. Then by (1) and Lemmas \ref{lem-na21}, \ref{lem-na3}, we have
%$$(R,\Delta_{j}(q))\star (R,\Delta_{\ell'}(q))\star (\mu_{G_l}R,\Delta_{\ell'}(q))\star (\mu_{G_l}R,\Delta_{j}(q)).$$
\end{proof}

The following proposition follows immediately by Lemmas \ref{lem-na}, \ref{lem:covera3} and Corollary \ref{cor:puncturecover1}.

\begin{proposition}\label{prop-ga2}
For any ${\bf P},{\bf Q}\in \mathcal L(T^o,\widetilde\beta^{(p,q)})$, we have ${\bf P}\sim {\bf Q}$ and ${\bf P}\approx {\bf Q}$.
\end{proposition}

\emph{Proof of Theorem \ref{thm-sim3}}
It follows by Propositions \ref{prop-ga}, \ref{prop-ga1} and \ref{prop-ga2}.
\endproof

\newpage

\section{Proof for the expansion formulas}\label{sec:PROOF}

Let $x_0,x_1$ be two variables such that $x_0x_1=v^{2\lambda} x_1x_0$ for some $\lambda\in \mathbb Z$. Denote $\Lambda(x_0,x_1)=-\Lambda(x_1,x_0)=\lambda$ and $\Lambda(x_0,x_0)=\Lambda(x_1,x_1)=0$. Let $r\in \mathbb N$. For any ${\vec c}=(c_1,\cdots,c_r)\in \{0,1\}^r$, denote
\begin{equation}\label{eq-lambda}
\lambda({\vec c})=\sum_{i< j} \Lambda(x_{c_i},x_{c_j}),\;\;\;\;x^{\vec c}= v^{-\lambda({\vec c})}x_{c_1}x_{c_2}\cdots x_{c_r}.
\end{equation}
Thus, we have
\begin{equation}\label{eq-bino}
(x_0+x_1)^r=\sum_{{\vec c}\in \{0,1\}^r}v^{\lambda ({\vec c})} x^{\vec c}.
\end{equation}

We have the following observation.

\begin{lemma}\label{lem1} Let $r\in \mathbb N$.

\begin{enumerate}[$(1)$]
    \item For any $i$ with $1\leq i\leq r$, given ${\vec c,\vec c\hspace{1.5pt}'}\in \{0,1\}^r$ with $c_i=1, c'_i=0$ and $c_j=c'_j$ for $j\neq i$, we have
    $\lambda({\vec c})-\lambda({\vec c\hspace{1.5pt}'})=(2i-r-1)\lambda.$
    \item For any $i$ with $1\leq i\leq n-1$, given ${\vec c, \vec c\hspace{1.5pt}'}\in \{0,1\}^r$ with $c_i=c'_{i+1}=0, c_{i+1}=c'_i=1$ and $c_j=c'_j$ for $j\neq i,i+1$, we have
    $\lambda({\vec c})-\lambda({\vec c\hspace{1.5pt}'})=2\lambda.$
\end{enumerate}

\end{lemma}

\begin{proof}
(1) As $c_i=1, c'_i=0$ and $c_j=c'_j$ for $j\neq i$, we have
\begin{align*}
&\lambda({\vec c})-\lambda({\vec c\hspace{1.5pt}'})\\
={}& \sum_{s<t} \Lambda(x_{c_s},x_{c_t})-\sum_{s< t} \Lambda(x_{c'_s},x_{c'_t})\\
={}& \sum_{s<i} (\Lambda(x_{c_s},x_1)-\Lambda(x_{c_s},x_0))+\sum_{i<s}  (\Lambda(x_{c_s},x_1)-\Lambda(x_{c_s},x_0))\\
={}& (i-1)\lambda-(r-i)\lambda\\
={}& (2i-r-1)\lambda.
\end{align*}

(2) As $c_i=c'_{i+1}=0, c_{i+1}=c'_i=1$ and $c_j=c'_j$ for $j\neq i,i+1$,
\begin{align*}
&\lambda({\vec c})-\lambda({\vec c\hspace{1.5pt}'})\\
={}& \sum_{s<t} \Lambda(x_{c_s},x_{c_t})-\sum_{s< t} \Lambda(x_{c'_s},x_{c'_t})\\
={}& \Lambda(x_0,x_1)-\Lambda(x_1,x_0)\\
={}& 2\lambda.
\end{align*}

\end{proof}

Throughout this section, denote by $b^{T'^o}_{\alpha'}$ the $\alpha'$-th column of the extended exchange matrix $\widetilde B(T'^o)$ and $(b^{T'^o}_{\alpha'})_{\pm}$ are the positive and negative part, respectively, of $b^{T'^o}_{\alpha'}$.

Let $\mathcal L=\mathcal L(T^o,\widetilde\beta), \mathcal L(T^o,\widetilde\beta^{(q)})$ or $\mathcal L(T^o,\widetilde\beta^{(p,q)})$. Correspondingly, let $\mathcal L'=\mathcal L(T'^o,\widetilde\beta), \mathcal L(T'^o,\widetilde\beta^{(q)})$ or $\mathcal L(T'^o,\widetilde\beta^{(p,q)})$, $\pi=\pi, \pi$ or $\pi_\alpha^{T^o,p,q}$ and $\pi'=\pi_{\alpha'}^{T'^o}, \pi_{\alpha'}^{T'^o,q}$ or $\pi_{\alpha'}^{T'^o,p,q}$. Denote the minimal elements in $\mathcal L$ and $\mathcal L'$ by ${\bf P}_-$ and ${\bf P}'_-$, respectively.

As a corollary of Theorem \ref{thm-sim3}, by the definition of the $y$-equivalent, we have the following.

\begin{proposition}\label{prop-any}
For any ${\bf P}\in \mathcal L$, we have
$$\frac{\bigoplus_{{\bf Q}\in \pi'\pi({\bf P})}y^{T}({\bf Q})} {\bigoplus_{{\bf Q}'\in \pi({\bf P})}y^{T'}({\bf Q'})}
=\frac{\bigoplus_{{\bf Q}\in \mathcal L}y^{T}({\bf Q})} {\bigoplus_{{\bf Q}'\in \mathcal L'}y^{T'}({\bf Q'})}.$$
\end{proposition}

\begin{proposition}\label{prop-c1}
For any ${\bf P}\in \mathcal L$ with $r=\sum m_i({\bf P})\geq 0$, we have
\begin{equation*}
X^{T}({\bf P})|_{v=1}=\sum_{{\bf P}'\in \pi({\bf P})}X^{T'}({\bf P}')|_{v=1}.
\end{equation*}
\end{proposition}

\begin{proof}
As $r\geq 0$, we have that in the cluster algebra $\mathcal A_v(\Sigma)|_{v=1}$
\begin{equation}\label{Eq-xp11}
    X^{T}({\bf P})|_{v=1}=\frac{M\cdot(x_{\alpha})^{r}\cdot y^{T}({\bf P})}{\bigoplus_{{\bf Q}\in \mathcal L}y^{T}({\bf Q})},
\end{equation}
where $M$ is the cluster Laurent monomial corresponding to the edges not labeled $\alpha$ in $P$, the diagonals not labeled $\alpha$ of $G_{T^o,\widetilde\beta}$ and the sides of $\Delta$ not equal $\alpha$.

For each $\vec c\in \{0,1\}^r$, by Theorem \ref{thm:pi}, Lemmas \ref{Lem-var}, \ref{Lem-var1}, \ref{lem-ycover} and \ref{lem-ycover1}, we have
$$x^{T'}({\bf P}(\vec c))=M\cdot (x_{\alpha'})^{-r}\cdot (x_{\alpha_2}x_{\alpha_4})^{r-\sum_i c_i}\cdot (x_{\alpha_1}x_{\alpha_3})^{\sum_i c_i},$$
$$y^{T'}({\bf P}(\vec c))=y^{T'}({\bf P}(0,\cdots,0))(y^{T'^o}_{\alpha'})^{\sum_i c_i}.$$
Therefore, by Definitions \ref{Def-weight}, \ref{Def-wei1}, \ref{Def-wei2}, we have
$$\begin{array}{rcl}
& & \sum_{\vec c\in \{0,1\}^r}X^{T'}({\bf P}(\vec c))|_{v=1} \vspace{4pt}\\

& = & \sum_{\vec c}\frac{M\cdot (x_{\alpha'})^{-r}\cdot (x_{\alpha_2}x_{\alpha_4})^{r-\sum_i c_i}\cdot (x_{\alpha_1}x_{\alpha_3})^{\sum_i c_i}\cdot  y^{T'}({\bf P}(0,\cdots,0))(y^{T'^o}_{\alpha'})^{\sum_i c_i}}{\oplus_{{\bf Q}'\in \mathcal L'}y^{T'}({\bf Q}')}\vspace{4pt}\\

& = & M\cdot (x_{\alpha'})^{-r}\cdot y^{T'}({\bf P}(0,\cdots,0)) \frac{\textstyle\sum_{\vec c\in \{0,1\}^r}(x_{\alpha_2}x_{\alpha_4})^{r-\sum c_i}\cdot (x_{\alpha_1}x_{\alpha_3}y^{T'^o}_{\alpha'})^{\sum c_i}}{\bigoplus_{{\bf Q}'\in \mathcal L'}y^{T'}({\bf Q}')} \vspace{4pt} \\

& = & M\cdot (x_{\alpha'})^{-r}\cdot y^{T'}({\bf P}(0,\cdots,0)) \frac{(x_{\alpha_2}x_{\alpha_4}+x_{\alpha_1}x_{\alpha_3}y^{T'^o}_{\alpha'})^r}{\bigoplus_{{\bf Q}'\in \mathcal L'}y^{T'}({\bf Q}')}, \vspace{2pt} \\
\end{array}
$$

By Lemma \ref{lem-yf1} (4), $x_{\alpha}=\frac{x_{\alpha_2}x_{\alpha_4}+x_{\alpha_1}x_{\alpha_3}y^{T'^o}_{\alpha'}}{x_{\alpha'}\cdot(1\oplus y^{T'^o}_{\alpha'})}$. Thus we have

\begin{equation*}
     \sum_{\vec c\in \{0,1\}^r}X^{T'}({\bf P}(\vec c))|_{v=1}= \frac{M\cdot (x_{\alpha})^r \cdot y^{T'}({\bf P}(0,\cdots,0))\cdot (1\oplus y^{T'^o}_{\alpha'})^r}{\bigoplus_{{\bf Q}'\in \mathcal L'}y^{T'}({\bf Q}')}.
\end{equation*}

By Lemma \ref{basic4}, we obtain

\begin{equation}\label{Eq-xp12}
     \sum_{\vec c\in \{0,1\}^r}X^{T'}({\bf P}(\vec c))|_{v=1}= \frac{M\cdot (x_{\alpha})^r \cdot \bigoplus_{{\bf Q}'\in \pi({\bf P})}y^{T'}({\bf Q}')}{\bigoplus_{{\bf Q}'\in \mathcal L}y^{T'}({\bf Q}')}.
\end{equation}
Then the result follows by (\ref{Eq-xp11}) (\ref{Eq-xp12}) and Proposition \ref{prop-any}.
\end{proof}

From the proof, we see the following result.

\begin{lemma}\label{lem-plam1}
For any ${\bf P}\in \mathcal L$ with $r=\sum m_i({\bf P})\geq 0$, assume that $X^T({\bf P})=(X^T)^{\overrightarrow{{\bf P}}+re_{\alpha}}$, then for any $\vec c\in \{0,1\}^{r}$ we have
$$X^{T'}({\bf P}(\vec c))=(X^{T'})^{\overrightarrow{{\bf P}}-r e_{\alpha'}+\sum c_i (b_{\alpha'}^{T'^o})_-+(r-\sum c_i)(b_{\alpha'}^{T'^o})_+}.$$
%where $(b_{\alpha'}^{T'^o})_\pm$ are the positive and negative part respectively, of $b_{\alpha'}^{T'^o}$, the $\alpha'$-th column of the extended exchange matrix $\widetilde B(T'^o)$.
\end{lemma}

\begin{proposition}
We have
$$\sum_{{\bf P}\in \mathcal L} X^{T}({\bf P})|_{v=1}=\sum_{{\bf P}'\in \mathcal L'} X^{T'}({\bf P}')|_{v=1}.$$
\end{proposition}

\begin{proof}
The result follows by Proposition \ref{prop-c1} and $\pi$ is a partition bijection.
\end{proof}

It follows that the commutative version of Theorems \ref{Thm-1}, \ref{Thm-2}, \ref{Thm-M3} hold.

\begin{theorem}
For $X=X_{\beta}, X_{\beta^{(q)}}$ or $X_{\beta^{(p,q)}}$, in the commutative cluster algebra $\mathcal A_v(\Sigma)\mid_{v=1}$ we have
$$X|_{v=1}=\sum_{{\bf P}\in \mathcal L} X^{T}({\bf P})|_{v=1}.$$
\end{theorem}

\medskip

We now turn to the quantum case.

%The following lemma is similar to Lemma \ref{lem-wp}.

\begin{lemma}\label{lem-wp1}
For any ${\bf P}\in \mathcal L$ with $r=\sum m_i({\bf P})\geq 0$, the following are equivalent:
\begin{enumerate}[$(i)$]
\item $v^{w({\bf P})}X^{T}({\bf P})=\sum_{{\bf P}'\in \pi({\bf P})}v^{w({\bf P}')}X^{T'}({\bf P}'),$
\item $w({\bf P})=w({\bf P}(0,\cdots,0))$,
\item $w({\bf P})=w({\bf P}(1,\cdots,1)).$
\end{enumerate}
\end{lemma}

\begin{proof}
As in Lemma \ref{lem-plam1}, we may write $X^{T}({\bf P})$ as
$(X^T)^{re_{\alpha}+\overrightarrow{{\bf P}}}$. Hence the coordinates of $e_{\alpha}$ and $e_{\alpha'}$ in $\overrightarrow{{\bf P}}$ are zero. It follows that $(X^T)^{\overrightarrow{{\bf P}}}=(X^{T'})^{\overrightarrow{{\bf P}}}$.

Thus
\begin{equation}\label{eq-xxxp11}
\begin{array}{rcl}

& & X^{T}({\bf P}) =  (X^T)^{re_{\alpha}+\overrightarrow{{\bf P}}} =  v^{-\Lambda^{T}(\overrightarrow{{\bf P}},re_{\alpha})}(X^{T})^{\overrightarrow{{\bf P}}}\cdot (X^{T})^{re_{\alpha}} \vspace{2pt} \\

& = & v^{-\Lambda^{T}(\overrightarrow{{\bf P}},re_{\alpha})}(X^{T'})^{\overrightarrow{{\bf P}}}\cdot \left((X^{T'})^{-e_{\alpha'}+(b^{T'^o}_{\alpha'})_+}+(X^{T'})^{-e_{\alpha'}+(b^{T'^o}_{\alpha'})_-}\right)^{re_{\alpha}} \vspace{2pt} \\

& = & v^{-\Lambda^{T}(\overrightarrow{{\bf P}},re_{\alpha})}(X^{T'})^{\overrightarrow{{\bf P}}}\cdot\textstyle\sum_{\vec c\in \{0,1\}^r}
v^{\lambda(\vec c)}(X^{T'})^{-re_{\alpha'}+\sum c_i(b^{T'^o}_{\alpha'})_-+(r-\sum c_i)(b^{T'^o}_{\alpha'})_+} \vspace{2pt} \\

& = & v^{-\Lambda^{T}(\overrightarrow{{\bf P}},re_{\alpha})}v^{\Lambda^{T'}(\overrightarrow{{\bf P}}, -re_{\alpha'}+r(b^{T'^o}_{\alpha'})_+)}\textstyle\sum_{\vec c\in \{0,1\}^r}
v^{\lambda(\vec c)}(X^{T'})^{\overrightarrow{{\bf P}}-re_{\alpha'}+\sum c_i(b^{T'^o}_{\alpha'})_-+(r-\sum c_i)(b^{T'^o}_{\alpha'})_+} \vspace{2pt} \\

& = & \textstyle\sum_{\vec c\in \{0,1\}^r}
v^{\lambda(\vec c)}(X^{T'})^{\overrightarrow{{\bf P}}-re_{\alpha'}+\sum c_i(b^{T'^o}_{\alpha'})_-+(r-\sum c_i)(b^{T'^o}_{\alpha'})_+},
\end{array}
\end{equation}
where the fourth equality follows by (\ref{eq-bino}) and $\lambda(\vec c)$ is given by (\ref{eq-lambda}) under the convention that $x_0=(X^{T'})^{-e_{\alpha'}+(b^{T'^o}_{\alpha'})_-}, x_1=(X^{T'})^{-e_{\alpha'}+(b^{T'^o}_{\alpha'})_+}$.

On the other hand, by Remark \ref{Rem-3}, Proposition \ref{prop-cover} and Lemma \ref{lem-plam1}, we have
$$\sum_{{\bf P}'\in \pi({\bf P})}v^{w({\bf P}')}X^{T'}({\bf P}')=\sum_{\vec c\in \{0,1\}^r}v^{w({\bf P}(\vec c))} (X^{T'})^{\overrightarrow{{\bf P}}-r e_{\alpha'}+\sum c_i (b_{\alpha'}^{T'^o})_-+(r-\sum c_i)(b_{\alpha'}^{T'^o})_+}.$$

By Lemma \ref{lem:ome1} and Lemma \ref{lem1} (1), for any $i\in \{1,\cdots, r\}$ and ${\vec c,\vec c\hspace{1.5pt}'}\in \{0,1\}^r$ with $c_i=1, c'_i=0$ and $c_j=c'_j$ for $j\neq i$, we have
$$w({\bf P}(\vec c))-w({\bf P}(\vec c\hspace{1.5pt}'))=(2i-r-1)d(\alpha')=\lambda({\vec c})-\lambda({\vec c\hspace{1.5pt}'}),$$
by (\ref{eq-lambda}), we have $\lambda(0,\cdots, 0)=\lambda(1,\cdots,1)=0$,
it follows that
\begin{equation}\label{eq-xxxp21}
\begin{array}{rcl} & & \sum_{{\bf P}'\in \pi({\bf P})}v^{w({\bf P}')}X^{T'}({\bf P}')\vspace{2pt} \\

& = & v^{w({\bf P}(0,\cdots,0))}\sum_{\vec c\in \{0,1\}^r}v^{\lambda(\vec c)} (X^{T'})^{\overrightarrow{{\bf P}}-r e_{\alpha'}+\sum c_i (b_{\alpha'}^{T'^o})_-+(r-\sum c_i)(b_{\alpha'}^{T'^o})_+} \vspace{2pt} \\

& = & v^{w({\bf P}(1,\cdots,1))}\sum_{\vec c\in \{0,1\}^r}v^{\lambda(\vec c)} (X^{T'})^{\overrightarrow{{\bf P}}-r e_{\alpha'}+\sum c_i (b_{\alpha'}^{T'^o})_-+(r-\sum c_i)(b_{\alpha'}^{T'^o})_+}.
\end{array}
\end{equation}

Then the result follows by (\ref{eq-xxxp11}) and (\ref{eq-xxxp21}).
\end{proof}

As a corollary of Lemma \ref{lem-wp1} and Proposition \ref{lem-p-5}, we have the following.

\begin{lemma}\label{lem-p-4}
With the foregoing notation. We have
\begin{equation*}
\sum_{{\bf P}\in \pi'\pi({\bf P}_-)}v^{w({\bf P})}X^{T}({\bf P})=\sum_{{\bf P}'\in \pi({\bf P}_-)}v^{w({\bf P}')}X^{T'}({\bf P}').
\end{equation*}
\end{lemma}

As a consequence of Theorem \ref{thm-sim3} and Lemma \ref{lem-p-4}, we obtain the following result.

\begin{proposition}\label{prop-all211}
With the foregoing notation. For any ${\bf P}\in \mathcal L$ we have
\begin{equation*}
\sum_{{\bf Q}\in \pi'\pi({\bf P})}v^{w({\bf Q})}X^{T}({\bf Q})=\sum_{{\bf Q}'\in \pi({\bf P})}v^{w({\bf Q}')}X^{T'}({\bf Q}').
\end{equation*}
\end{proposition}

It follows that

\begin{theorem}\label{thm-sum4}
With the foregoing notation. We have
\begin{equation*}
\sum_{{\bf P}\in \mathcal L}v^{w({\bf P})}X^{T}({\bf P})=\sum_{{\bf P}'\in \mathcal L'}v^{w({\bf P}')}X^{T'}({\bf P}').
\end{equation*}
\end{theorem}

\subsection{Proofs of Theorems \ref{Thm-1}, \ref{Thm-2}}
We can now give the proofs of Theorems \ref{Thm-1}, \ref{Thm-2}.

\emph{Proof of Theorem \ref{Thm-1}:}
If $\beta\in T^o$, then $\mathcal P(G_{T^o,\widetilde\beta})$ contains only one perfect matching $P$ with $X^{T}(P)=X_{\beta}$ and $w(P)=0$. Thus the result holds in case $\beta\in T^o$. If $\beta\not\in T^o$, as any two ideal triangulations are related by a sequence of flips, the result follows by Theorem \ref{thm-sum4} and the case that $\beta\in T^o$. The proof is complete.
\endproof

\begin{lemma}\label{lem:step1}
Assume that $p\neq q$, for an ideal triangulation $T_0^o$ such that $\beta,l_q(\beta)\in T_0^o$, let $\alpha=l_q(\beta)$ and $T'^o_0=\mu_\alpha(T_0^o)$. Then we have
$X^{T_0}(P_{\beta},\Delta)=\sum_{{\bf P}'\in \mathcal L(T'^o_0,\beta)}v^{w({\bf P}')}X^{T'_0}({\bf P}')$.
\end{lemma}

\emph{Proof of Theorem \ref{Thm-2}:}
Choose an ideal triangulation $T_0^o$ such that $\beta,l_q(\beta)\in T_0^o$. Then $\mathcal L(T_0^o,\widetilde \beta^{(q)})=\{(P_{\beta},\Delta)\}$ contains a unique element, where $\Delta=(\beta,\beta,l_q(\beta))$. In this case, we have $w(P_{\beta},\Delta)=0$. From Definition \ref{Def-weight}, we see that $X^{T_0}(P_{\beta},\Delta)=X_{\beta^{(q)}}$. By Lemma \ref{lem:step1}, we see the result holds for $T'_0$. For any ideal triangulation $T^o$ such that the corresponding tagged triangulation $T$ contains no arc tagged notched at $q$. We see the result holds for $T$ by Lemma \ref{lem:flip-p} and Theorem \ref{thm-sum4}. The proof is complete.
\endproof

The rest of this section is devoted to giving the proof of Theorem \ref{Thm-M3}.

Given a quantum cluster algebra $\mathcal A_v(\Sigma)$ and a puncture $p$, recall the quantum cluster algebra $\mathcal A^{(p)}_v(\Sigma)$ and the isomorphism $\sigma: \mathcal A_v(\Sigma)\to \mathcal A^{(p)}_v(\Sigma)$ in Section \ref{Sec-iso}. Denote by $\sigma\mid_{v=1}$ the corresponding isomorphism between $\mathcal A_v(\Sigma)\mid_{v=1}$ and $\mathcal A^{(p)}_v(\Sigma)\mid_{v=1}$. For a given tagged triangulation $T$, for any $\Delta_j(p)\in \Delta(T^o,p)$, denote by $x^{T,(p)}(P_\beta,\Delta_j(p))$, $h^{T,(p)}(P_\beta,\Delta_j(p))$, $y^{T,(p)}(P_\beta,\Delta_j(p))$ and $X^{T,(p)}(P_\beta,\Delta_j(p))$ the weight, height monomial, specialized height monomial and quantum weight, respectively of $(P_\beta,\Delta_j(p))$ in $\mathcal A^{(p)}_v(\Sigma)$, which are given in Definition \ref{Def-wei1}.

We divide the proof of Theorem \ref{Thm-M3} into two cases: $p\neq q$ or $p=q$.

\subsection{Proof of Theorem \ref{Thm-M3} in case $p\neq q$} As $p\neq q$, we have $\widetilde\beta=\beta$, $\beta$ and $\beta^{(p)}$ are compatible. We may choose a tagged triangulation $T_0$ such that $\beta,\beta^{(p)}\in T_0$ and any arc $\gamma(\neq \beta^{(p)}) \in T_0$ does not tagged notched. Thus $T_0^{(p)}=T_0$. Let $T_0^o=l(T_0)$ be the corresponding ideal triangulation. Thus $l_p(\beta)\in T_0^o$ and $\Delta=(\beta,\beta,l_p(\beta))$ is a self-folded triangle in $T^o_0$. Denote $\Delta(T_0^o,q)=\{\Delta_1(q),\cdots, \Delta_t(q)\}$. Then $\tau_{t-1}(q)=\tau_1(q)=l_p(\beta), \tau_t(q)=\tau_{[1]}(q)=\tau_{[t]}(q)=\beta$ and $\tau_{j}[q]\neq \beta,l_p(\beta)$ for any $j\in \{2,\cdots, t-1\}$.
Thus we have
$$\mathcal L:=\mathcal L(T^o_0,\beta^{(p,q)})=\{\Delta\}\times \{P_\beta\}\times \{\Delta_1(q),\Delta_3(q),\cdots, \Delta_{t-3}(q),\Delta_{t}(q)\},$$
$$\mathcal L(T^o_0,\beta^{(q)})=\{P_\beta\}\times \{\Delta_1(q),\Delta_2(q),\cdots, \Delta_{t-1}(q),\Delta_{t}(q)\}.$$

\begin{figure}[h]
\centerline{\includegraphics{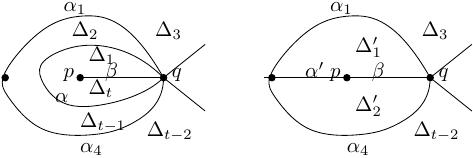}}

\caption{} \label{Fig:proof4}
\end{figure}

For $j<t$, let $$y^{T_0}(\Delta,P_\beta,\Delta_j(q))=y^{T_0}_\beta y^{T_0}_{\tau_2(q)}\cdots y^{T_0}_{\tau_{j-1}(q)},$$
 and
$$y^{T_0}(\Delta,P_\beta,\Delta_t(q))=y^{T_0}_\beta y^{T_0}_{\tau_2(q)}\cdots y^{T_0}_{\tau_{t-2}(q)}y^{T_0}_\beta.$$

We have the following lemma.

\begin{lemma}\label{lem-p-q1} For any $(\Delta,P_\beta,\Delta_j(q))\in \mathcal L$, we have  %\huang{XXXXXXXXXXXXXXXXXX}
\begin{enumerate}[$(1)$]
\item $\sigma\mid_{v=1}\left(x^{T_0}(\Delta,P_\beta,\Delta_j(q))\right)=x^{T_0,(p)}(P_\beta,\Delta_j(q))$,
\item $\sigma\mid_{v=1}\left(y^{T_0}(\Delta,P_\beta,\Delta_j(q))\right)=y^{T_0,(p)}(P_\beta,\Delta_j(q))$,
%\item $\sigma\mid_{v=1}(\sum_{{\bf Q}\in \pi^{-1}\pi({\bf P})}X^{T'_0}({\bf Q}))=X^{T_0,(p)}(\pi({\bf P}))$,
\item $\sigma\left(X^{T_0}(\Delta,P_\beta,\Delta_j(q))\right)=X^{T_0,(p)}(P_\beta,\Delta_j(q)).$
\end{enumerate}
\end{lemma}

\begin{proof}
(1) For $j=1,t$, we have $x^{T_0}(\Delta,P_\beta,\Delta_j(q))=(x_{l_p(\beta)}x_{\beta}^{-2})x_\beta(x_\beta x_\beta^{-1}x_{l_p(\beta)}^{-1})=x_\beta^{-1}$. $x^{T_0,(p)}(P_\beta,\Delta_j(q))=x^{(p)}_{\beta}(x^{(p)}_\beta (x^{(p)}_\beta)^{-1}(x^{(p)}_{l_p(\beta)})^{-1})=(x^{(p)}_{\beta^{(p)}})^{-1}$. Thus the result holds for $j=1,t$. For $j\neq 1,t$, we have $x^{T_0}(\Delta,P_\beta,\Delta_j(q))=(x_{l_p(\beta)}x_{\beta}^{-2})x_\beta x^{T_0}(\Delta_j(q))=x_{\beta^{(p)}}x^{T_0}(\Delta_j(q))$. Thus $\sigma\mid_{v=1}\left(x^{T_0}(\Delta,P_\beta,\Delta_j(q))\right)=x^{(p)}_{\beta}x^{T_0,(p)}(\Delta_j(q))=x^{T_0,(p)}(P_\beta,\Delta_j(q))$.

(2) Since $\tau_{j}[q]\neq \beta,l_p(\beta)$ for any $j\in \{2,\cdots, t-1\}$, we have for $j<t$
$$\sigma\mid_{v=1}\left(y^{T_0}(\Delta,P_\beta,\Delta_j(q)\right)=y^{T_0,(p)}_{\beta^{(p)}} y^{T_0,(p)}_{\tau_2(q)}\cdots y^{T_0,(p)}_{\tau_{j-1}(q)}=y^{T_0,(p)}(P_\beta,\Delta_j(q)),$$
$$\sigma\mid_{v=1}\left(y^{T_0}(\Delta,P_\beta,\Delta_t(q)\right)=y^{T_0,(p)}_{\beta^{(p)}} y^{T_0,(p)}_{\tau_2(q)}\cdots y^{T_0,(p)}_{\tau_{t-2}(q)}y^{T_0,(p)}_{\beta^{(p)}}=y^{T_0,(p)}(P_\beta,\Delta_t(q)).$$

(3) It follows by (1) and (2).
\end{proof}

The following proposition follows immediately by Lemma \ref{lem-p-q1}.

\begin{proposition}\label{prop-equ4}
With the foregoing notation. We have
$$\sum_{{\bf P}=(\Delta,P_{\beta},\Delta_j(q))\in \mathcal L} v^{w(P_{\beta},\Delta_j(q))}\sigma\left(X^{T_0}({\bf P})\right)
=\sum_{{\bf P}\in \mathcal L(T_0^o,\beta^{(q)})} v^{w({\bf P})}X^{T_0,(p)}({\bf P}).$$
\end{proposition}

\begin{proposition}\label{lem:step2}
Assume that $p\neq q$, for any ideal triangulation $T_0^o$ such that $\beta,l_q(\beta)\in T_0^o$, let $\alpha=l_q(\beta)$ and $T_0'^o=\mu_\alpha(T^o_0)$. Then we have
$$\sum_{{\bf P}=(\Delta,P_{\beta},\Delta_j(q))\in \mathcal L(T_0^o,\beta^{(p,q)})} v^{w(P_{\beta},\Delta_j(q))}X^{T_0}({\bf P})
=\sum_{{\bf P}\in \mathcal L(T_0'^o,\beta^{(p,q)})} v^{w({\bf P})}X^{T'_0}({\bf P}).$$
\end{proposition}

\begin{proof}
Assume $\alpha_2=\alpha_3=\beta$. Denote $\Delta'_1=\{\beta,\alpha_1,\alpha'\}$ and $\Delta'_2=\{\beta,\alpha_4,\alpha'\}$. Assume that the endpoints of $\gamma$ are $p$ and $p'$. We have the following two cases $p'\neq q$ and $p'=q$.

Case I: $p'\neq q$. Then we have $\Delta(T'^o_0,q)=\{\Delta'_1,\Delta_3(q),\cdots, \Delta_{t-2}(q),\Delta'_2\}$ and $\Delta(T'^o_0,p)=\{\Delta'_2,\Delta'_1\}$. See Figure \ref{Fig:proof4}. Then the following assignments
$$\pi(\Delta,P_\beta,\Delta_1(q))=(\Delta'_2,P_\beta,\Delta'_2),\;\;\; \pi(\Delta,P_\beta,\Delta_2(q))=(\Delta'_2,P_\beta,\Delta'_1),$$ $$\pi(\Delta,P_\beta,\Delta_j(q))=\{(\Delta'_2,P_\beta,\Delta_j(q)),(\Delta'_1,P_\beta,\Delta_j(q))\}$$ for $3\leq j\leq t-2$,
$$\pi(\Delta,P_\beta,\Delta_{t-1}(q))=(\Delta'_1,P_\beta,\Delta'_2),\;\;\; \pi(\Delta,P_\beta,\Delta_t(q))=(\Delta'_1,P_\beta,\Delta'_1)$$
give a partition bijection $\pi:\mathcal L(T_0^o,\beta^{(p,q)})\to \mathcal L(T'^o_0,\beta^{(p,q)})$.

We have
$$x^{T_0}(\Delta,P_\beta,\Delta_1(q))=\frac{x_{\alpha}}{x_\beta^2}x_\beta \frac{x_\beta}{x_\alpha x_\beta}=\frac{1}{x_\beta},\;\; y^{T_0}(\Delta,P_\beta,\Delta_1(q))=1,$$
$$x^{T'_0}(\Delta'_2,P_\beta,\Delta'_2)=\frac{x_{\alpha_4}}{x_\beta x_{\alpha'}}x_\beta \frac{x_{\alpha'}}{x_{\alpha_4} x_\beta}=\frac{1}{x_\beta},\;\; y^{T'_0}(\Delta'_2,P_\beta,\Delta'_2)=1.$$
Thus $x^{T_0}(\Delta,P_\beta,\Delta_1(q))y^{T_0}(\Delta,P_\beta,\Delta_1(q))=x^{T'_0}(\Delta'_2,P_\beta,\Delta'_2)y^{T'_0}(\Delta'_2,P_\beta,\Delta'_2).$

We have
$$x^{T_0}(\Delta,P_\beta,\Delta_2(q))=\frac{x_{\alpha}}{x_\beta^2}x_\beta \frac{x_{\alpha_4}}{x_\alpha x_{\alpha_1}}=\frac{x_{\alpha_4}}{x_\beta x_{\alpha_1}},\;\; y^{T_0}(\Delta,P_\beta,\Delta_2(q))=y^{T_0}_\beta,$$
$$x^{T'_0}(\Delta'_2,P_\beta,\Delta'_1)=\frac{x_{\alpha_4}}{x_\beta x_{\alpha'}}x_\beta \frac{x_{\alpha'}}{x_{\alpha_1} x_\beta}=\frac{x_{\alpha_4}}{x_\beta x_{\alpha_1}},\;\; y^{T'_0}(\Delta'_2,P_\beta,\Delta'_1)=y^{T'_0}_\beta.$$
Thus $x^{T_0}(\Delta,P_\beta,\Delta_2(q))y^{T_0}(\Delta,P_\beta,\Delta_2(q))=x^{T'_0}(\Delta'_2,P_\beta,\Delta'_1)y^{T'_0}(\Delta'_2,P_\beta,\Delta'_1).$

For $3\leq j\leq t-2$, we have

\begin{equation*}
\begin{array}{rcl} x^{T_0}(\Delta,P_\beta,\Delta_j(q))y^{T_0}(\Delta,P_\beta,\Delta_j(q))&=&y^{T_0}(\Delta,P_\beta,\Delta_j(q))\frac{x_{\alpha}}{x_\beta^2}x_\beta \frac{x_{[j]}(q)}{x_{\tau_{j-1}(q)} x_{\tau_j(q)}}\vspace{2pt} \\
& = & y^{T_0}(\Delta,P_\beta,\Delta_j(q)) \frac{x_\alpha}{x_\beta}\frac{x_{[j]}(q)}{x_{\tau_{j-1}(q)} x_{\tau_j(q)}}.
\end{array}
\end{equation*}

\begin{equation*}
\begin{array}{rcl} \sum\limits_{{\bf P'}\in \pi(\Delta,P_\beta,\Delta_j(q))} x^{T'_0}({\bf P'})y^{T'_0}({\bf P'})&=&y^{T'_0}(\Delta'_2,P_\beta,\Delta_j(q))(\frac{x_{\alpha_4}}{x_\beta x_{\alpha'}}+y^{T'_0}_{\alpha'}\frac{x_{\alpha_1}}{x_\beta x_{\alpha'}}) x_\beta \frac{x_{[j]}(q)}{x_{\tau_{j-1}(q)} x_{\tau_j(q)}}\vspace{2pt} \\
& = & y^{T'_0}(\Delta'_2,P_\beta,\Delta_j(q))\frac{x_\alpha(1+y^{T'_0}_{\alpha'})}{x_\beta^2} x_\beta \frac{x_{[j]}(q)}{x_{\tau_{j-1}(q)} x_{\tau_j(q)}}.
\end{array}
\end{equation*}

By Lemma \ref{lem-yf1} (3)(a), we have $y^{T_0}(\Delta,P_\beta,\Delta_j(q))=y^{T'_0}(\Delta'_2,P_\beta,\Delta_j(q))(1+y^{T'_0}_{\alpha'})$. Thus
$$ x^{T_0}(\Delta,P_\beta,\Delta_j(q))y^{T_0}(\Delta,P_\beta,\Delta_j(q))=\sum\limits_{{\bf P'}\in \pi(\Delta,P_\beta,\Delta_j(q))} x^{T'_0}({\bf P'})y^{T'_0}({\bf P'}).$$

We have
$$x^{T_0}(\Delta,P_\beta,\Delta_{t-1}(q))=\frac{x_{\alpha}}{x_\beta^2}x_\beta \frac{x_{\alpha_1}}{x_\beta x_{\alpha_4}}=\frac{x_{\alpha_1}}{x_\beta x_{\alpha_4}},$$
$$y^{T_0}(\Delta,P_\beta,\Delta_{t-1}(q))=y^{T_0}_\beta y^{T_0}_{\tau_2(q)}\cdots y^{T_0}_{\tau_{t-2}(q)},$$
$$x^{T'_0}(\Delta'_1,P_\beta,\Delta'_2)=\frac{x_{\alpha_1}}{x_\beta x_{\alpha'}}x_\beta \frac{x_{\alpha'}}{x_{\alpha_4} x_\beta}=\frac{x_{\alpha_1}}{x_\beta x_{\alpha_4}},$$
$$y^{T'_0}(\Delta'_1,P_\beta,\Delta'_2)=y^{T'_0}_\beta y^{T'_0}_{\tau_2(q)}\cdots y^{T'_0}_{\tau_{t-2}(q)} y^{T'_0}_{\alpha'}.$$
Since $\tau_2(q)=\alpha_1,\tau_{t-2}(q)=\alpha_4$, we have $y^{T_0}(\Delta,P_\beta,\Delta_{t-1}(q))=y^{T'_0}(\Delta'_1,P_\beta,\Delta'_2)$ by Lemma \ref{lem-yf1} (3)(a). Thus $x^{T_0}(\Delta,P_\beta,\Delta_{t-1}(q))y^{T_0}(\Delta,P_\beta,\Delta_{t-1}(q))=x^{T'_0}(\Delta'_1,P_\beta,\Delta'_2)y^{T'_0}(\Delta'_1,P_\beta,\Delta'_2).$

We have
$$x^{T_0}(\Delta,P_\beta,\Delta_{t}(q))=\frac{x_{\alpha}}{x_\beta^2}x_\beta \frac{x_{\beta}}{x_\beta x_{\alpha}}=\frac{1}{x_\beta},$$
$$y^{T_0}(\Delta,P_\beta,\Delta_{t}(q))=y^{T_0}_\beta y^{T_0}(\Delta,P_\beta,\Delta_{t-1}(q)),$$
$$x^{T'_0}(\Delta'_1,P_\beta,\Delta'_1)=\frac{x_{\alpha_1}}{x_\beta x_{\alpha'}}x_\beta \frac{x_{\alpha'}}{x_{\alpha_1} x_\beta}=\frac{1}{x_\beta},$$
$$y^{T'_0}(\Delta'_1,P_\beta,\Delta'_1)=y^{T'_0}_\beta y^{T'_0}(\Delta'_2,P_\beta,\Delta'_1).$$
Thus $y^{T_0}(\Delta,P_\beta,\Delta_{t}(q))=y^{T'_0}(\Delta'_2,P_\beta,\Delta'_1)$ and
$$x^{T_0}(\Delta,P_\beta,\Delta_t(q))y^{T_0}(\Delta,P_\beta,\Delta_t(q))=x^{T'_0}(\Delta'_1,P_\beta,\Delta'_1)y^{T'_0}(\Delta'_1,P_\beta,\Delta'_1).$$

From the above, we see that $\oplus_{{\bf P}\in \mathcal L(T_0^o,\beta^{(p,q)})}y^{T_0}({\bf P})=\oplus_{{\bf P'}\in \mathcal L(T'^o_0,\beta^{(p,q)})}y^{T'_0}({\bf P'})$. Thus, we have $X^{T_0}({\bf P})=\sum_{{\bf P'}\in \pi({\bf P})} X^{T'_0}({\bf P'})$ for any ${\bf P}\in \mathcal L(T^o_0,\beta^{(p,q)})$. To end the proof, it remains to prove that $w(P_\beta,\Delta_j(q))=w({\bf P'})$ for any $(\Delta,P_\beta,\Delta_j(q))\in \mathcal L(T^o_0,\beta^{(p,q)})$ and ${\bf P'}\in \pi(\Delta,P_\beta,\Delta_j(q))$.

It is easy to see that $w(P_\beta,\Delta_j(q))=0$ for all $j\in \{1,2,\cdots, t\}$ and $w({\bf P}')=0$ for all ${\bf P}'\in \mathcal L(T'^o_0,\beta^{(p,q)})$.

Case II. $p'=q$. Then there exists $k$ such that $\tau_k(q)=\alpha_1$ and $\tau_{k+1}(q)=\alpha_4$. See Figure \ref{Fig:proof5}. Thus we have $\Delta(T'^o_0,q)=\{\Delta'_1,\Delta_3(q),\cdots,\Delta_k(q),\Delta'_{k+1},\Delta''_{k+1}, \Delta_{k+2}(q),\cdots, \Delta_{t-2}(q),\Delta'_2\}$ and $\Delta(T'^o_0,p)=\{\Delta'_2,\Delta'_1\}$, where $\Delta'_{k+1}=\Delta'_1$ with $m(\Delta'_{k+1};\alpha')=-1$ and $\Delta''_{k+1}=\Delta_2'$ with $m(\Delta''_{k+1};\alpha')=-1$.
 Then the following assignments
$$\pi(\Delta,P_\beta,\Delta_1(q))=(\Delta'_2,P_\beta,\Delta'_2),\;\;\; \pi(\Delta,P_\beta,\Delta_2(q))=(\Delta'_2,P_\beta,\Delta'_1),$$
$$\pi(\Delta,P_\beta,\Delta_{k+1}(q))=\{(\Delta'_1,P_\beta,\Delta'_{k+1}),(\Delta'_1,P_\beta,\Delta''_{k+1}),(\Delta'_2,P_\beta,\Delta'_{k+1}),(\Delta'_2,P_\beta,\Delta''_{k+1})\},$$
$$\pi(\Delta,P_\beta,\Delta_j(q))=\{(\Delta'_2,P_\beta,\Delta_j(q)),(\Delta'_1,P_\beta,\Delta_j(q))\}$$ for $3\leq j\leq k$ or $k+2\leq j\leq t-2$,
$$\pi(\Delta,P_\beta,\Delta_{t-1}(q))=(\Delta'_1,P_\beta,\Delta'_2),\;\;\; \pi(\Delta,P_\beta,\Delta_t(q))=(\Delta'_1,P_\beta,\Delta'_1)$$
give a partition bijection $\pi:\mathcal L(T_0^o,\beta^{(p,q)})\to \mathcal L(T'^o_0,\beta^{(p,q)})$.

\begin{figure}[h]
\centerline{\includegraphics[width=10cm]{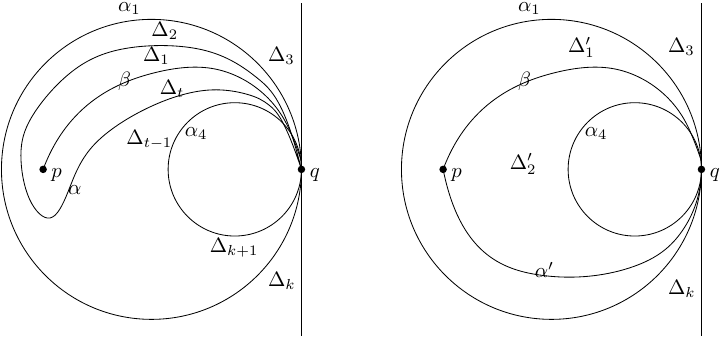}}

\caption{} \label{Fig:proof5}
\end{figure}

As Case I, we can similarly prove $X^{T_0}({\bf P})=\sum_{{\bf P'}\in \pi({\bf P})} X^{T'_0}({\bf P'})$ for any ${\bf P}\in \mathcal L(T^o_0,\beta^{(p,q)})$ and $w(P_\beta,\Delta_j(q))=0$ for all $j\in \{1,2,\cdots, t\}$ and $w({\bf P}')=0$ for all ${\bf P}'\in \mathcal L(T'^o_0,\beta^{(p,q)})$.

The proof is complete.
\end{proof}

\medskip

\emph{Proof of Theorem \ref{Thm-M3} in case $p\neq q$} Choose a tagged triangulation $T_0$ contains $\beta,\beta^{(p)}$ and any arc $\gamma(\neq \beta^{(p)}) \in T_0$ does not tagged notched. By Theorem \ref{Thm-2}, we have
$$\sum_{(P_\beta,\Delta_j(q))\in \mathcal L(T_0^o,\beta^{(q)})} v^{w(P_\beta,\Delta_j(q))}X^{(p)}(P_\beta,\Delta_j(q))=X^{(p)}_{\beta^{(q)}}.$$

By Propositions \ref{prop-equ4} and \ref{lem:step2}, we have Theorem \ref{Thm-M3} holds for $T'_0$. For any ideal triangulation $T^o$ such that the corresponding tagged triangulation $T$ contains no arc tagged notched at $p$ or $q$, we see Theorem \ref{Thm-M3} holds for $T$ by Lemma \ref{lem:flip-pq} and Theorem \ref{thm-sum4}.
\endproof

\subsection{Proof of Theorem \ref{Thm-M3} in case $p=q$} As $(S,M)$ is not a once-punctured closed surface, we may choose an arc $\gamma$ incident to $p$ with another endpoint different from $p$ and compatible with $\widetilde\beta$. Then we have $\gamma$ and $\gamma^{(p)}$ are compatible. We may choose a tagged triangulation $T_0$ such that $\gamma,\gamma^{(p)}\in T_0$. Thus $T_0^{(p)}=T_0$. Let $T_0^o=l(T_0)$ be the corresponding ideal triangulation. Thus $l_p(\gamma)\in T_0^o$ and $\Delta=(\gamma,\gamma,l_p(\gamma))$ is a self-folded triangle in $T_0^o$.

Thus
$$\mathcal L(T^o_0,\widetilde\beta^{(p,q)})=\{\Delta\}\times \mathcal P(G_{T^o_0,\widetilde\beta})\times \{\Delta\},$$
$$\mathcal L(T^o_0,\widetilde\beta)=\mathcal P(G_{T^o_0,\widetilde\beta}).$$

For any $(\Delta,P,\Delta)\in \mathcal L(T_0^o, \widetilde\beta^{(p,q)})$, let $y^{T_0}(\Delta,P,\Delta)=\sigma\mid_{v=1}(y^{T_0,(p)}(P))$.

The following result follows by Lemma \ref{Lem-yf}.

\begin{lemma}\label{lem:coverl}
Assume that $P>\mu_{G_l}P$. Then we have
\[\begin{array}{ccl} \frac{y^{T_0}(\Delta,P,\Delta)}{y^{T_0}(\Delta,Q,\Delta)} &=&
\left\{\begin{array}{ll}
y^{T_0}_\gamma, &\mbox{if $\tau_{i_l}=l_q(\gamma)$}, \vspace{2.5pt}\\
y^{T^o_0}_{\tau_{i_1}}, &\mbox{otherwise}.
\end{array}\right.
\end{array}\]
\end{lemma}

\begin{lemma}\label{lem-p=q4} For any $(\Delta,P,\Delta)\in \mathcal L(T_0^o, \widetilde\beta^{(p,q)})$, we have  %\huang{XXXXXXXXXXXXXXXXXX}
\begin{enumerate}[$(1)$]
\item $\sigma\mid_{v=1}\left(x^{T_0}(\Delta,P,\Delta)\right)=x^{T_0,(p)}(P)$,
\item $\sigma\mid_{v=1}\left(y^{T_0}(\Delta,P,\Delta)\right)=y^{T_0,(p)}(P)$,
%\item $\sigma\mid_{v=1}(\sum_{{\bf Q}\in \pi^{-1}\pi({\bf P})}X^{T'_0}({\bf Q}))=X^{T_0,(p)}(\pi({\bf P}))$,
\item $\sigma\left(X^{T_0}(\Delta,P,\Delta)\right)=X^{T_0,(p)}(P)$.
\end{enumerate}
\end{lemma}

\begin{proof}
(1) We have the west edge $E_1$ and south edge $E_2$ of the first tile of $G_{T_0^o,\widetilde\beta}$ are labeled $\gamma$. The east edge $F_1$ and north edge $F_2$ of the last tile of $G_{T_0^o,\widetilde\beta}$ are labeled $\gamma$. For any $P\in \mathcal P(G_{T_0^o,\widetilde\beta})$, we have either $E_1\in P$ or $E_2\in P$, either $F_1\in P$ or $F_2\in P$. Thus $x^{T_0}(\Delta,P,\Delta)=\frac{x_{l_p(\gamma)}}{x_{\gamma}^2}x_\gamma X x_\gamma \frac{x_{l_p(\gamma)}}{x_{\gamma}^2}=x_{\gamma^{(p)}}^2 X$, where $X$ is the cluster Laurent monomial corresponding to the edges except $E_1,E_2,F_1,F_2$ in $P$ and the diagonals of $G_{T_0^o,\widetilde\beta}$. Since $x_\gamma$ is not a factor of $X$, we have
$\sigma\mid_{v=1}\left(x^{T_0}(\Delta,P,\Delta)\right)=(x^{(p)}_{\gamma})^2 X^{(p)}=x^{T_0,(p)}(P)$.

(2) It follows by the definition of $y^{T_0}(\Delta,P,\Delta)$.

(3) It follows by (1) and (2).
\end{proof}

The following proposition follows immediately by Lemma \ref{lem-p=q4}.

\begin{proposition}\label{prop-equ5}
With the foregoing notation. We have
$$\sum_{(\Delta,P,\Delta)\in \mathcal L(T_0^o,\widetilde\beta^{(p,q)})} v^{w(P)}\sigma\left(X^{T_0}(\Delta,P,\Delta)\right)
=\sum_{P\in \mathcal L(T_0^o,\widetilde\beta)} v^{w(P)}X^{T_0,(p)}(P).$$
\end{proposition}

%Let $\alpha=l_q(\gamma)$ and $T'^o_0=\mu_\alpha(T^o_0)$. Assume that $\gamma$

\begin{proposition}\label{lem:step3}
Assume that $p=q$, let $\alpha=l_q(\gamma)$ and $T'^o_0=\mu_\alpha(T^o_0)$. Then we have
$$\sum_{(\Delta,P,\Delta)\in \mathcal L(T_0^o,\widetilde\beta^{(p,q)})} v^{w(P)}X^{T_0}(\Delta,P,\Delta)
=\sum_{{\bf P}\in \mathcal L(T_0'^o,\beta^{(p,q)})} v^{w({\bf P})}X^{T'_0}({\bf P}).$$
\end{proposition}

We will prove Proposition \ref{lem:step3} in Section \ref{sec:step3}.

\medskip

\emph{Proof of Theorem \ref{Thm-M3} in case $p=q$.} By Theorem \ref{Thm-1}, we have
$$\sum_{P\in \mathcal L(T^o_0,\widetilde\beta)}v^{w(P)}X^{T_0,(p)}(P)=X^{(p)}_{\widetilde\beta}.$$
By Proposition \ref{prop-equ5}, we have
%$$\sum_{{\bf P}\in \mathcal L(T_0'^o,\widetilde\beta^{(p)})} v^{w({\bf P})}\left(X^{T'_0}({\bf P})\right)=X_{\widetilde\beta^{(p)}}.$$
%Therefore,
Theorem \ref{Thm-M3} holds for $T_0$. By Proposition \ref{lem:step3}, Theorem \ref{Thm-M3} holds for $T'_0$.
For any ideal triangulation $T^o$ such that the corresponding tagged triangulation $T$ contains no arc tagged notched at $p$, we see Theorem \ref{Thm-M3} holds for $T$ by Lemma \ref{lem:flip-pq} and Theorem \ref{thm-sum4}.

The proof of Theorem \ref{Thm-M3} is complete.
\endproof

\subsection{Proof of Proposition \ref{lem:step3}.}\label{sec:step3}
%Let $\mathcal L=\mathcal L(T_0^o,\widetilde\beta^{(p,q)})$, $\mathcal L'\mathcal L(T'^o_0,\widetilde\beta^{(p,q)})$ and $\pi:\mathcal L\to \mathcal L'$ be the partition bijection with inverse $\pi'$, which are given in Section \ref{sec-pb}. We may assume that $\alpha_2=\alpha_3=\gamma$. XXXXXXXXXXXXX
Assume that the tiles of $G_{T'^o_0,\widetilde\beta}$ are $G'_1,\cdots,G'_c$ with diagonals labeled $\tau'_{i_1},\cdots,\tau'_{i_c}$. Then $\tau'_{i_1},\tau'_{i_c}\in \{\alpha_1,\alpha_4\}$. Since $\gamma,\widetilde\beta$ are compatible, $\widetilde\beta$ crosses $T_0^o$ $c+2$ times. Assume that the tiles of $G_{T_0^o,\widetilde\beta}$ are $G_0,G_1,\cdots,G_c,G_{c+1}$ with diagonals labeled $\tau_{i_j}, j=0,\cdots,c+1$. Then $\tau_{i_j}=\tau'_{i_j}$ for $j=1,\cdots, c$ and $\tau_{i_0},\tau_{i_{c+1}}=\alpha$. We may further assume that $rel(T_0^o,G_1)=rel(T_0^o,G_c)=rel(T'^o_0,G'_1)=rel(T'^o_0,G'_c)=1$. (Here the assumption implies $c$ is odd, the result can be proved similarly in case $c$ is even, the assumption fixes the labels of the edges of these tiles.)

%For any ${\bf P}\in \mathcal L$, since $\gamma$ and $\widetilde\beta$ are compatible, we have ${\bf m}({\bf P})=(1,m_1,m_2,1)$ with $m_1,m_2\in \{0,-1\}$.

Denote $\mathcal L=\mathcal L(T_0^o,\widetilde\beta^{(p,q)})$, $\mathcal L'=\mathcal L(T'^o_0,\widetilde\beta^{(p,q)})$.

By symmetry, we shall consider the following cases. Case I: $\tau'_{i_1}=\alpha_1,\tau'_{i_c}=\alpha_4$; Case II: $\tau'_{i_1}=\alpha_1,\tau'_{i_c}=\alpha_1$ and Case III: $\tau'_{i_1}=\alpha_4,\tau'_{i_c}=\alpha_4$.

We consider the Case I, as the other cases can be proved similarly.

\begin{figure}[h]
\centerline{\includegraphics{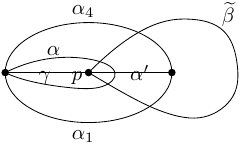}}
%\caption{}\label{Fig:proof6}
\end{figure}

In this case, the subgraph of $G_{T_0^o,\widetilde\beta}$ formed by $N(G_1),E(G_1), S(G_c), W(G_c)$ and $G_2,\cdots,G_{c-1}$ equals the subgraph of $G_{T'^o_0,\widetilde\beta}$ formed by $N(G'_1),E(G'_1), S(G'_c), W(G'_c)$ and $G'_2,\cdots,G'_{c-1}$. See Figure \ref{Fig:proof6}. Denote this common subgraph by $H$.

\begin{figure}[h]
\centerline{\includegraphics{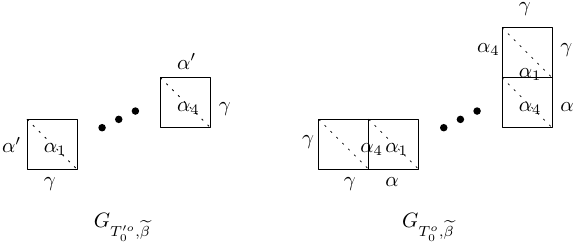}}
\caption{}\label{Fig:proof6}
\end{figure}

The following lemma is immediate.

\begin{lemma}
For any $P\in \mathcal P(G_{T_0^o,\widetilde\beta})$, there is a unique $P'\in \mathcal P(G_{T'^o_0,\widetilde\beta})$ such that $P\cap edge(H)=P'\cap edge(H)$. We denote $P\mid_{G_{T'^o_0,\widetilde\beta}}=P'$. In particular, $P_-\mid_{G_{T'^o_0,\widetilde\beta}}=P'_-$ is the minimal perfect matching.
\end{lemma}

Denote $\Delta_1=\{\gamma,\alpha',\alpha_1\}$ and $\Delta_2=\{\gamma,\alpha',\alpha_4\}$. Then the maximal/minimal element in $\mathcal L'$ is $(\Delta_1,P'_\pm,\Delta_2)$. In this case the partition bijection $\pi:\mathcal L\to \mathcal L'$ has the following explicit construction: for any ${\bf P}=(\Delta,P,\Delta)\in \mathcal L$,

\begin{enumerate}[$(i)$]
\item if $S(G_1), E(G_c)\in P$, then $\pi({\bf P})=\{(\Delta_1,P',\Delta_1),(\Delta_1,P',\Delta_2),(\Delta_2,P',\Delta_1),(\Delta_2,P',\Delta_2)\}$;
\item If $S(G_1)\in P, E(G_c)\notin P$, then $\pi({\bf P})=\{(\Delta_1,P',\Delta_1),(\Delta_2,P',\Delta_1)\}$ in case $N(G_{c+1})\in P$ and $\pi({\bf P})=\{(\Delta_1,P',\Delta_2),(\Delta_2,P',\Delta_2)\}$ in case $N(G_{c+1})\notin P$;
\item If $S(G_1)\notin P, E(G_c)\in P$, then $\pi({\bf P})=\{(\Delta_1,P',\Delta_1),(\Delta_1,P',\Delta_2)\}$ in case $W(G_{0})\notin P$ and $\pi({\bf P})=\{(\Delta_2,P',\Delta_1),(\Delta_2,P',\Delta_2)\}$ in case $W(G_{0})\in P$;
\item If $S(G_1), E(G_c)\notin P$, then $\pi({\bf P})=\{(\Delta_1,P',\Delta_1)\}$ in case $W(G_{0})\notin P, N(G_{c+1})\in P$, $\pi({\bf P})=\{(\Delta_2,P',\Delta_1)\}$ in case $W(G_{0}), N(G_{c+1})\in P$, $\pi({\bf P})=\{(\Delta_1,P',\Delta_2)\}$ in case $W(G_{0}), N(G_{c+1})\notin P$ and $\pi({\bf P})=\{(\Delta_2,P',\Delta_2)\}$ in case $W(G_{0})\in P, N(G_{c+1})\notin P$,
\end{enumerate}
where $P'=P\mid_{G_{T'^o_0,\widetilde\beta}}$.

The following is immediate.

\begin{lemma}\label{lem:ll}
$\pi(\Delta,P_-,\Delta)$ contains the minimal element $(\Delta_1,P'_-,\Delta_2)$.
\end{lemma}

\begin{lemma}\label{lem:lll}
Assume that ${\bf P}=(\Delta,P,\Delta)\in \mathcal L$ covers ${\bf Q}=(\Delta,\mu_{G_l}P,\Delta)\in \mathcal L$.
\begin{enumerate}[$(1)$]
\item If $l=0$, then for any $(\Delta',P',\Delta'')\in \pi({\bf P})$ we have $\Delta'=\Delta_2$ and
$$\pi({\bf Q})=\{(\Delta_1,P',\Delta'')\mid (\Delta',P',\Delta'')\in \pi({\bf P})\}.$$
Moreover, ${\bf P}(0,\cdots,0)$ covers ${\bf Q}(0,\cdots,0)$ and related by $\gamma$, and
$$\Omega(P,\mu_{G_l}P)=\Omega({\bf P}(0,\cdots,0),{\bf Q}(0,\cdots,0));$$
\item If $l=1$, then for any $(\Delta',Q',\Delta'')\in \pi({\bf Q})$ we have $\Delta'=\Delta_2$, $Q'$ can twist on $G'_1$ with $Q'<\mu_{G'_1}(Q')$ and
$$\pi({\bf P})=\{(\Delta_1,\mu_{G'_1}Q',\Delta''),(\Delta_2,\mu_{G'_1}Q',\Delta'')\mid (\Delta',Q',\Delta'')\in \pi({\bf Q})\}.$$
Moreover, ${\bf P}(0,\cdots,0)$ covers ${\bf Q}(0,\cdots,0)$ and related by $\alpha_1$, and
$$\Omega(P,\mu_{G_l}P)=\Omega({\bf P}(0,\cdots,0),{\bf Q}(0,\cdots,0));$$
\item If $2\leq l\leq c-1$, then for any $(\Delta',P',\Delta'')\in \pi({\bf P})$ we have $P'$ can twist on $G'_l$ with $P'>\mu_{G'_l}(P')$ and
$$\pi({\bf Q})=\{(\Delta',\mu_{G_l}P',\Delta'')\mid (\Delta',P',\Delta'')\in \pi({\bf P})\}.$$
Moreover, ${\bf P}(0,\cdots,0)$ covers ${\bf Q}(0,\cdots,0)$ and related by $\tau_{i_l}$, and
$$\Omega(P,\mu_{G_l}P)=\Omega({\bf P}(0,\cdots,0),{\bf Q}(0,\cdots,0));$$
\item If $l=c$, then for any $(\Delta',P',\Delta'')\in \pi({\bf P})$ we have $\Delta''=\Delta_1$, $P'$ can twist on $G'_c$ with $P'>\mu_{G'_c}(P')$ and
$$\pi({\bf Q})=\{(\Delta',\mu_{G'_c}P',\Delta_1),(\Delta',\mu_{G'_c}P',\Delta_2)\mid (\Delta',Q',\Delta'')\in \pi({\bf P})\}.$$
Moreover, ${\bf P}(1,\cdots,1)$ covers ${\bf Q}(1,\cdots,1)$ and related by $\alpha_1$, and
$$\Omega(P,\mu_{G_l}P)=\Omega({\bf P}(1,\cdots,1),{\bf Q}(1,\cdots,1));$$
\item If $l=c+1$, then for any $(\Delta',P',\Delta'')\in \pi({\bf P})$ we have $\Delta''=\Delta_2$ and
$$\pi({\bf Q})=\{(\Delta',P',\Delta_1)\mid (\Delta',P',\Delta'')\in \pi({\bf P})\}.$$
Moreover, ${\bf P}(0,\cdots,0)$ covers ${\bf Q}(0,\cdots,0)$ and related by $\gamma$, and
$$\Omega(P,\mu_{G_l}P)=\Omega({\bf P}(0,\cdots,0),{\bf Q}(0,\cdots,0)).$$
\end{enumerate}

\end{lemma}

\begin{proof}
The relation between $\pi({\bf P})$ and $\pi({\bf Q})$ follows immediately by the construction of $\pi$.

(1) Since $W(G'_1)\in P'$, we have ${\bf P}(0,\cdots,0)$ covers ${\bf Q}(0,\cdots,0)$ and related by $\gamma$. If $E(G_c)\in P$, then $\Omega(P,\mu_{G_l}P)=0$ and $E(G'_c)\in P'$. Thus $$\Omega({\bf P}(0,\cdots,0),{\bf Q}(0,\cdots,0))=1+m(\Delta_2;\gamma)=1-1=0.$$ If $E(G_c)\notin P$, then $\Omega(P,\mu_{G_l}P)=-1$ and $S(G'_c)\in P'$. Thus $$\Omega({\bf P}(0,\cdots,0),{\bf Q}(0,\cdots,0))=0+m(\Delta_1;\gamma)=0-1=-1.$$

(2) Since $l=1$, we have $S(G'_1)\in P'$. Thus ${\bf P}(0,\cdots,0)$ covers ${\bf Q}(0,\cdots,0)$ and related by $\alpha_1$. If $N(G_c)\in P$, then $\Omega(P,\mu_{G_l}P)=1$ and $N(G_{c+1})\in P, N(G'_c)\in P'$. Thus ${\bf P}(0,\cdots,0)=(\Delta_2,\mu_{G'_1}Q',\Delta_1)$ and $$\Omega({\bf P}(0,\cdots,0),{\bf Q}(0,\cdots,0))=m(\Delta_1;\alpha_1)-m(\Delta_2;\alpha_1)=1-0=1.$$ If $N(G_c)\notin P$, then $\Omega(P,\mu_{G_l}P)=0$ and $E(G_c)\in P$ in case $N(G_{c+1})\in P$. Thus ${\bf P}(0,\cdots,0)=(\Delta_2,\mu_{G'_1}Q',\Delta_2)$ and $$\Omega({\bf P}(0,\cdots,0),{\bf Q}(0,\cdots,0))=m(\Delta_2;\alpha_1)-m(\Delta_2;\alpha_1)=0.$$

(3) Since $2\leq l\leq c-1$, we have ${\bf P}(0,\cdots,0)$ covers ${\bf Q}(0,\cdots,0)$ and related by $\tau_{i_l}$, and
$$\Omega(P,\mu_{G_l}P)=\Omega({\bf P}(0,\cdots,0),{\bf Q}(0,\cdots,0))=0.$$

(4) and (5) are similar to (2) and (1).
\end{proof}

\begin{lemma}\label{lem:llll}
Assume that ${\bf P}=(\Delta,P,\Delta)\in \mathcal L$ covers ${\bf Q}=(\Delta,\mu_{G_l}P,\Delta)\in \mathcal L$. Then
$$\frac{y^{T_0}({\bf P})}{y^{T_0}({\bf Q})}=
\frac{\bigoplus_{{\bf P}'\in \pi({\bf P})}y^{T'_0}({\bf P}')}
{\bigoplus_{{\bf Q}'\in \pi({\bf Q})}y^{T'_0}({\bf Q}')}.$$
\end{lemma}

\begin{proof}
If $l=0$, then by Lemmas \ref{lem:coverl}, \ref{lem-dcover1} and \ref{lem:lll} we have
$$\frac{y^{T_0}({\bf P})}{y^{T_0}({\bf Q})}=y^{T_0}_\gamma=y^{T'_0}_\gamma=\frac{\bigoplus_{{\bf P}'\in \pi({\bf P})}y^{T'_0}({\bf P}')}
{\bigoplus_{{\bf Q}'\in \pi({\bf Q})}y^{T'_0}({\bf Q}')}.$$

If $l=1$, then by Lemmas \ref{lem:coverl}, \ref{lem-yf1}, \ref{lem-dcover1} and \ref{lem:lll} we have
$$\frac{y^{T_0}({\bf P})}{y^{T_0}({\bf Q})}=y^{T^o_0}_{\alpha_1}=y^{T'^o_0}_{\alpha_1}(1\oplus y^{T'^o_0}_{\alpha'})
=\frac{\bigoplus_{{\bf P}'\in \pi({\bf P})}y^{T'_0}({\bf P}')}
{\bigoplus_{{\bf Q}'\in \pi({\bf Q})}y^{T'_0}({\bf Q}')}.$$

If $2\leq l\leq c-1$, then by Lemmas \ref{lem-dcover1} and \ref{lem:lll} we have
$$\frac{y^{T_0}({\bf P})}{y^{T_0}({\bf Q})}=y^{T^o_0}_{\tau_{i_l}}=y^{T'^o_0}_{\tau_{i_l}}=\frac{\bigoplus_{{\bf P}'\in \pi({\bf P})}y^{T'_0}({\bf P}')}
{\bigoplus_{{\bf Q}'\in \pi({\bf Q})}y^{T'_0}({\bf Q}')}.$$

If $l=c$, then by Lemmas \ref{lem:coverl}, \ref{lem-yf1}, \ref{lem-dcover1} and \ref{lem:lll} we have
$$\frac{y^{T_0}({\bf P})}{y^{T_0}({\bf Q})}=y^{T^o_0}_{\alpha_4}=y^{T'^o_0}_{\alpha_4}\frac{1}{y^{T'^o_0}_{\alpha'}}(1\oplus y^{T'^o_0}_{\alpha'})
=\frac{\bigoplus_{{\bf P}'\in \pi({\bf P})}y^{T'_0}({\bf P}')}
{\bigoplus_{{\bf Q}'\in \pi({\bf Q})}y^{T'_0}({\bf Q}')}.$$

If $l=c+1$, then by Lemmas \ref{lem:coverl}, \ref{lem-dcover1} and \ref{lem:lll} we have
$$\frac{y^{T_0}({\bf P})}{y^{T_0}({\bf Q})}=y^{T_0}_\gamma=y^{T'_0}_\gamma=\frac{\bigoplus_{{\bf P}'\in \pi({\bf P})}y^{T'_0}({\bf P}')}
{\bigoplus_{{\bf Q}'\in \pi({\bf Q})}y^{T'_0}({\bf Q}')}.$$

The proof is complete.
\end{proof}

\begin{lemma}
Assume that ${\bf P}=(\Delta,P,\Delta)\in \mathcal L$ covers ${\bf Q}=(\Delta,\mu_{G_l}P,\Delta)\in \mathcal L$. Then we have
$w(P)=w({\bf P}(0,\cdots,0))=w({\bf P}(1,\cdots,1))$.
\end{lemma}

\begin{proof}
It follows by Lemmas \ref{lem:ll} and \ref{lem:lll}.
\end{proof}

\begin{proposition}\label{prop:ll}
For any ${\bf P}=(\Delta,P,\Delta)\in \mathcal L$, we have
 $$v^{w(P)}X^{T_0}({\bf P})=\sum\limits_{{\bf P'}\in \pi({\bf P})} v^{w({\bf P'})}X^{T'_0}({\bf P'}).$$
\end{proposition}

\begin{proof}
For ${\bf P}=(\Delta,P_-,\Delta)$, we have $W(G_0),S(G_1)\notin P, E(G_c),N(G_{c+1})\in P$. Then $$x^{T_0}(\Delta,P,\Delta)=\frac{x_\alpha}{x^2_\gamma}x_\gamma x^{-1}_{\alpha} x(P\mid_H)x^{-1}_{\alpha_4} x_\gamma \frac{x_\alpha}{x^2_\gamma}=x(P\mid_H) \frac{x_\alpha}{x^2_\gamma x_{\alpha_4}},\;\;y^{T_0}(\Delta,P,\Delta)=1,$$
$$x^{T'_0}(\Delta_1,P'_-,\Delta_1)=\frac{x_{\alpha_1}}{x_\gamma x_{\alpha'}} x_{\alpha'}x^{-1}_{\alpha_1} x(P'_-\mid_H) x^{-1}_{\alpha_4} x_{\gamma} \frac{x_{\alpha_1}}{x_\gamma x_{\alpha'}}=x(P'_-\mid_H) \frac{x_{\alpha_1}}{x_{\alpha_4}x_\gamma x_{\alpha'}},$$
$$x^{T'_0}(\Delta_1,P'_-,\Delta_2)=x^{T'_0}(\Delta_1,P'_-,\Delta_1)\frac{x_{\alpha_4}}{x_{\alpha_1}}=x(P'_-\mid_H) \frac{1}{x_\gamma x_{\alpha'}},$$
$$y^{T_0}(\Delta_1,P'_-,\Delta_1)=y^{T_0}_{\alpha'},y^{T_0}(\Delta_1,P'_-,\Delta_2)=1.$$
where $x(P\mid_H)$ (resp. $x(P'_-\mid_H)$) is the Laurent monomial determined by the edges $P\cap edge(H)$ (resp. $P'_-\cap edge(H)$) and the diagonals of $H$.
Thus, $$x^{T_0}({\bf P})y^{T_0}({\bf P})=\frac{1}{\oplus_{{\bf P'}\in \pi({\bf P})}y^{T'_0}({\bf P'})}\sum_{{\bf P'}\in \pi({\bf P})}x^{T'_0}({\bf P'})y^{T'_0}({\bf P'}).$$

Then the proof is similar to Proposition \ref{prop-all211} by using Lemma \ref{lem:llll}.
\end{proof}

\emph{Proof of Proposition \ref{lem:step3}.} It follows immediately by Proposition \ref{prop:ll}.
\newpage

%
%{\bf Acknowledgements:}\;
%This project is supported by
%the National Natural Science Foundation of China (No.12101617) (No.12071422)(No.11801043).

\end{document}